\documentclass[a4paper,12pt,reqno]{amsart}

\usepackage{amsmath, amsfonts, amssymb, amsthm, mathrsfs, mathtools, mathalfa, setspace}
\usepackage{comment}
\usepackage{yfonts}
\usepackage{enumerate}
\usepackage{enumitem}
\usepackage{stmaryrd}
\usepackage{fancyhdr}
\usepackage{bm}
\usepackage[utf8]{inputenc}
\usepackage[pdfencoding=auto, hyperfootnotes=false]{hyperref}
\usepackage{tikz,tikz-cd}
\usepackage[hang,flushmargin]{footmisc}
\usepackage{xr}
\usepackage{extarrows}
\usepackage{todonotes}
\usepackage{faktor}
\usepackage{xfrac}
\usepackage{times}
\usepackage{color}
\usepackage{graphicx}

\usepackage{xurl}

\makeatletter
\def\@tocline#1#2#3#4#5#6#7{\relax
  \ifnum #1>\c@tocdepth 
  \else
    \par \addpenalty\@secpenalty\addvspace{#2}%
    \begingroup \hyphenpenalty\@M
    \@ifempty{#4}{%
      \@tempdima\csname r@tocindent\number#1\endcsname\relax
    }{%
      \@tempdima#4\relax
    }%
    \parindent\z@ \leftskip#3\relax \advance\leftskip\@tempdima\relax
    \rightskip\@pnumwidth plus4em \parfillskip-\@pnumwidth
    #5\leavevmode\hskip-\@tempdima
      \ifcase #1
       \or\or \hskip 1em \or \hskip 2em \else \hskip 3em \fi%
      #6\nobreak\relax
    \dotfill\hbox to\@pnumwidth{\@tocpagenum{#7}}\par
    \nobreak
    \endgroup
  \fi}
\makeatother

\newcommand*{\compl}[1]{#1^\dagger}

\newtheorem{theorem}{Theorem}[section]
\newtheorem{lemma}[theorem]{Lemma}

\newtheorem{corollary}[theorem]{Corollary}
\newtheorem{proposition}[theorem]{Proposition}

\theoremstyle{definition}
\newtheorem{defn}[theorem]{Definition}
\newtheorem{remark}[theorem]{Remark}
\newtheorem{example}[theorem]{Example}

\def\supp{{\rm supp}}

\def\<{\langle}
\def\>{\rangle}

\def\to{\rightarrow}

\def\im{{\rm im}}

\newcommand*{\sbr}[1]{\scalebox{0.8}{$(#1)$}}
\newcommand*{\db}[1]{\llbracket #1\rrbracket}

\newcommand{\mc}{\mathcal}

\newcommand{\mb}{\mathbb}

\newcommand{\wh}{\widehat}
\newcommand{\wt}{\widetilde}

\newcommand{\id}{\mathrm{id}}

\DeclareMathOperator{\stab}{Stab}

\DeclareMathOperator{\ab}{Z}

\DeclareMathOperator{\tran}{\Theta}

\DeclareMathOperator{\Lip}{Lip}

\DeclareMathOperator{\codim}{codim}

\DeclareMathOperator{\q}{c}
\DeclareMathOperator{\ns}{X}
\DeclareMathOperator{\nss}{Y}
\DeclareMathOperator{\co}{\circ\hspace{-0.02 cm}}
\DeclareMathOperator{\cu}{C}
\DeclareMathOperator{\cor}{Cor}
\DeclareMathOperator{\cs}{s}

\DeclareMathOperator{\bnd}{B}

\DeclareMathOperator{\diam}{diam}

\newcommand*{\smallcorner}[7]{
\begin{tikzpicture}[baseline={([yshift=-.5ex]current bounding box.center)},vertex/.style={anchor=base,
    circle,fill=black!25,minimum size=18pt,inner sep=2pt}]
  \draw[thick](0,1,0)--(0,1,1)--(1,1,1);
  \draw[thick](1,0,0)--(1,0,1)--(0,0,1)--(0,1,1);
  \draw[thick](1,1,1)--(1,0,1);
  \draw[thick](1,0,0)--(0,0,0)--(0,1,0);
  \draw[thick](0,0,0)--(0,0,1);
  \draw[thick](0.2,0.2,0) node{$#3$};
  \draw(1.1,0.2,0) node{$#4$};
  \draw(0.1,1.2,0) node{$#7$};
  \draw(1.1,-0.2,1) node{$#2$};
  \draw(0.1,-0.2,1) node{$#1$};
  \draw(-0.1,1.2,1) node{$#5$};
  \draw(1,1.2,1) node{$#6$};
\end{tikzpicture}
}

\title[Free nilspaces, double-coset nilspaces, and Gowers norms]{Free nilspaces, double-coset nilspaces,\\ and Gowers norms}

\author{Pablo Candela}
\thanks{Pablo Candela: Instituto de Ciencias Matem\'aticas, Calle Nicol\'as Cabrera 13-15, Madrid 28049, Spain. 
\texttt{pablo.candela@icmat.es}, ORCiD: 0000-0002-7261-5536.}

\author{Diego González-Sánchez}
\thanks{Diego González-Sánchez: HUN-REN Alfréd Rényi Institute of Mathematics, Reáltanoda utca 13-15, Budapest, Hungary, H-1053. \texttt{diegogs@renyi.hu}, ORCiD: 0000-0003-4388-9148.}

\author{Balázs Szegedy}
\thanks{Balázs Szegedy: HUN-REN Alfréd Rényi Institute of Mathematics, Reáltanoda utca 13-15, Budapest, Hungary, H-1053. \texttt{szegedyb@gmail.com}, ORCiD: 0009-0009-6682-3361.}


\begin{document}
\singlespacing

\vspace*{-1cm}

\maketitle

\bigskip

\vspace*{-1cm}
\begin{abstract} 
 Compact finite-rank (\textsc{cfr}) nilspaces have become central in the nilspace approach to higher-order Fourier analysis, notably through their role in a general form of the inverse theorem for the Gowers norms. This paper studies these nilspaces per se, and in connection with further refinements of this inverse theorem that have been conjectured recently. Our first main result states that every \textsc{cfr} nilspace is obtained by taking a \emph{free} nilspace (a nilspace based on an abelian group of the form $\mb{Z}^{r}\times \mb{R}^s$) and quotienting this by a discrete group action of a specific type, describable in terms of polynomials. We call these group actions \emph{higher-order lattice actions} as they generalize actions of lattices in $\mb{Z}^r\times \mb{R}^s$. The second main result (relying on the first one) represents every \textsc{cfr} nilspace as a \emph{double-coset space} $K\backslash G / \Gamma$ where $G$ is a nilpotent Lie group of a specific kind. Our third main result extends the aforementioned results to $k$-step compact nilspaces (not necessarily of finite rank), by representing any such nilspace as a quotient of infinite products of free nilspaces and also as double coset spaces $K\backslash G/\Gamma$ where $G$ is a degree-$k$ nilpotent pro-Lie group. These results require developing the theory of topological \emph{non-compact} nilspaces, for which we provide groundwork in this paper. Applications include new inverse theorems for Gowers norms on any finite abelian group. These theorems  are purely group theoretic in  that the correlating harmonics are based on double-coset spaces. This yields progress towards the Jamneshan--Tao conjecture.
\end{abstract}

\section{Introduction}
Since its inception in the seminal work of Gowers \cite{GSz}, higher-order Fourier analysis has yielded many notable results, including the celebrated proof by Green and Tao that the primes contain arbitrarily long arithmetic progressions \cite{GTprimes}. A central topic in this area is that of inverse theorems and regularity lemmas (or structure theorems) for the Gowers norms \cite{G-gen}. An inverse theorem for all Gowers norms on finite cyclic groups was proved by Green, Tao and Ziegler in \cite{GTZ} (see also previous partial results in this direction in \cite{GT08} and \cite{GTZ-U4}). Analogous results on vector spaces over finite fields were proved in \cite{BTZ} and \cite{TZ-High, T&Z-Low}. 

In ergodic theory, Host and Kra introduced uniformity seminorms for measure-preserving systems \cite{HK-non-conv}, analogous to the Gowers norms. Host and Kra then  initiated an axiomatic approach to the objects that appear as characteristic factors for these seminorms  \cite{HK-par}, notably by introducing the notion of parallelepiped structures. Inspired by this, Antolín Camarena and Szegedy introduced the concept of nilspaces \cite{CamSzeg}. These objects provide a general notion of cubic structures and have proved useful in addressing various questions related to uniformity norms and seminorms, in arithmetic combinatorics, ergodic theory, topological dynamics, and probability theory. The literature on this theory and its applications includes \cite{Cand:Notes1,CGSS, CScouplings, CSinverse, GMV1, GMV2, GMV3}.

\emph{Compact Finite-Rank}  (\textsc{cfr}) nilspaces arise naturally in this context. In particular, the general inverse theorem for Gowers norms in \cite{CSinverse} shows that these nilspaces suffice to obtain correlating harmonics for functions with non-trivial Gowers norms on an ample family of spaces admitting such norms, including all compact abelian groups and nilmanifolds. Moreover, these specific nilspaces are relevant to an interesting recent conjecture of Jamneshan and Tao concerning the inverse theorem \cite[Conjecture 1.11]{J&T}. Indeed, this conjecture posits that in the case of finite abelian groups, the underlying object generating the correlating harmonics in the inverse theorem can always be taken to be a (not necessarily connected) nilmanifold. Given the inverse theorem from \cite{CSinverse}, this conjecture can be approached as a problem of giving an appropriate description of \textsc{cfr} nilspaces in terms of nilmanifolds (we illustrate this in particular in the 2-step case, as explained below in Subsection \ref{subsec:intro2step}).

This relevance of \textsc{cfr} nilspaces motivates their further study per se. One of our aims in this paper is to widen the set of viewpoints from which these objects can be analyzed, and especially to open this study to areas more classical than nilspace theory, such as the theory of groups and group actions. The first two main results in this paper go in this direction, as follows. 

The first main result establishes that \textsc{cfr} nilspaces can always be represented as quotients of spaces of the form $\mb{Z}^{r}\times \mb{R}^s$ by discrete group actions of a specific type, which can be described in terms of polynomials; see Theorem \ref{thm:gpcongrep-intro}. 

The second main result (building on the first one) represents every \textsc{cfr} nilspace as a double-coset space $K\backslash G / \Gamma$ where $G,K$ are specific nilpotent Lie groups and $\Gamma$ is a discrete subgroup of $G$ (see Theorem \ref{thm:cfr=double-coset-intro}). The relevance of double-coset spaces to nilspace theory emerged in
connection with applications in dynamics. Specifically, in 2014, Gutman,
Manners and Varjú announced the result that minimal nilspace systems of finite order where the acting group is a compactly generated abelian group are isomorphic to double coset spaces (private
communication). In
particular, they expressed a well-known example of Rudolph from \cite{Rudolph} 
in terms of a double-coset space. More recent uses of double-coset spaces in related directions include the work of Shalom \cite[Theorem 1.21]{Shalom} and that of Jamneshan, Shalom and Tao \cite[Theorem 1.8]{JST1}. These results motivated the belief (shared in the area of nilspace theory\footnote{We thank in particular Yonatan Gutman, Frederick Manners and Or Shalom for exchanges on this topic.}) that, while compact nilspaces are not always coset spaces (see \cite[Example 6]{HK-par}), they might always be expressible as \emph{double}-coset spaces. Theorem \ref{thm:cfr=double-coset-intro} below confirms this belief for \textsc{cfr} nilspaces in general. 

Furthermore, our third main result extends the above picture to all compact nilspaces (not necessarily of finite rank), proving that every such nilspace is also a double-coset space, involving a \emph{pro-Lie} group $G$. This extension to general compact nilspaces is also of interest beyond nilspace theory; in particular, it is relevant to the study of Host--Kra factors, since it was established in \cite[Theorem 5.11]{CScouplings} that Host--Kra factors (for general nilpotent group actions) are always compact nilspace systems, i.e.\ measure-preserving systems defined on compact nilspaces \cite[Definition 5.10]{CScouplings}. The extension in question builds on the fact that every $k$-step compact nilspace is an inverse limit of $k$-step \textsc{cfr} nilspaces  \cite[Theorem 2.7.3]{Cand:Notes2}; combining this with the first two main results above, we obtain that an inverse limit of $k$-step \textsc{cfr} nilspaces is isomorphic to a double-coset nilspace $K\backslash G/\Gamma$, where $G$ is a degree-$k$ filtered pro-Lie group, $K$ is a closed subgroup of $G$ and $\Gamma$ is a pro-discrete (hence closed) subgroup of $G$; see Theorem \ref{thm:cpct-nil-as-dou-coset}. In particular, this implies that every inverse limit of compact nilmanifolds\footnote{Meaning that for each $i\in \mb{N}$ we have a degree-$k$ filtered Lie group $G_i$, a lattice $\Gamma_i\le G_i$, and a continuous surjective homomorphism $\varphi_{i,i+1}:G_{i+1}\to G_i$ such that $\varphi_{i,i+1}(\Gamma_{i+1})\subset\Gamma_i$; the corresponding inverse system of nilmanifolds $G_i/\Gamma_i$ is defined by the maps $\overline{\varphi_{i,i+1}}:G_{i+1}/\Gamma_{i+1}\to G_i/\Gamma_i$, $g\Gamma_{i+1}\mapsto \varphi_{i,i+1}(g)\Gamma_i$, for $i\in \mb{N}$.}, when viewed in the category of compact nilspaces, is isomorphic to a double-coset space of this type (recall, for comparison, that Rudolph's example \cite{Rudolph} shows that an inverse limit of 2-step nilmanifolds is not, in general, the homogeneous space of some locally compact nilpotent group).

Proving the above results required further development, in this paper, of notions that are of interest for nilspace theory, such as \emph{free nilspaces} and \emph{congruences} on nilspaces. These notions in turn required non-trivial groundwork in the theory of topologically \emph{non-compact} nilspaces.

We obtain applications by upgrading the general inverse theorem for Gowers norms from \cite{CSinverse} using the new information on \textsc{cfr} nilspaces obtained in this paper. In particular we prove an inverse theorem for finite abelian groups in terms of double-coset spaces, which can be viewed as a step towards \cite[Conjecture 1.11]{J&T}. We also obtain a new proof, in the 2-step case, of a recent result of Jamneshan and Tao \cite[Theorem 1.10]{J&T}.

The next four brief subsections explain in more detail our main results, the related concepts, and the main applications. After that, this introduction is closed with an outline of the paper.

\subsection{\textsc{cfr} nilspaces as quotients of $\mb{Z}^r\times \mb{R}^s$ by higher-order lattice actions}\hfill\\
To motivate our first main result, let us start by recalling the following description of characters from classical Fourier analysis. Given any compact abelian group $\ab$, a character $\chi$ on $\ab$ can be viewed as the composition of two maps, a continuous homomorphism $\phi:\ab\to \mb{R}/\mb{Z}$ and the map $e:\mb{R}/\mb{Z}\to \mb{C}$, $x\mapsto \exp(2\pi i x)$:
\[
\chi:\ab \xrightarrow{\quad\phi\quad} \mb{R}/\mb{Z} \xrightarrow{\quad e\quad} \mb{C}.
\]
Allowing $\mb{R}/\mb{Z}$ to be replaced by a torus $\mb{R}^s/\mb{Z}^s$ and replacing $e$ by a linear combination of such exponential functions, we obtain trigonometric polynomials. 

The main result in question here extends this picture to higher-order Fourier analysis, replacing the torus with a quotient of $\mb{Z}^r\times\mb{R}^s$ by the action of a discrete nilpotent group $\Gamma$ acting on $\mb{Z}^r\times\mb{R}^s$ by polynomial maps. The homomorphism $\phi$ is then replaced by a nilspace morphism\footnote{Note that the quotient notation here takes the wider meaning of quotienting by the \emph{action} of $\Gamma$ ($\Gamma$ is not necessarily a subgroup of $\mb{Z}^r\times\mb{R}^s$).} $\varphi:\ab\to (\mb{Z}^r\times\mb{R}^s)/\Gamma$, and the map $e$ by a Lipschitz function $F:(\mb{Z}^r\times\mb{R}^s)/\Gamma\to \mb{C}$, thus obtaining what we call a \emph{nilspace polynomial} on $\ab$:
\[
\chi:\ab \xrightarrow{\quad\varphi\quad} (\mb{Z}^r\times\mb{R}^s)/\Gamma \xrightarrow{\quad F\quad} \mb{C}. 
\]
\begin{example}
A basic non-abelian example is provided by the Heisenberg group. This group has appeared already in many works in higher-order Fourier analysis since \cite{GT08}; recall that it is the group of $3\times 3$ upper unitriangular real matrices $\mc{H}= \begin{psmallmatrix} 1 & \mb{R} & \mb{R}\\[0.1em]0  & 1 & \mb{R}\\[0.1em] 0 & 0 & 1 \end{psmallmatrix}\vspace{0.05cm}$. A special case of what we prove in this paper is that, as a nilspace, the Heisenberg group is isomorphic to a nilspace expressible in terms of structures which are among the most basic ones in nilspace theory, and which play a key role in this paper, namely the nilspaces $\mc{D}_k(\mb{R})$. (Given any abelian group $\ab$ and $k\in\mb{N}$, we denote by $\mc{D}_k(\ab)$ the $k$-step nilspace associated with the filtered group $(\ab,\ab_\bullet)$, where the filtration $\ab_\bullet=(\ab_i)_{i\geq 0}$ consists of $\ab_i=\ab$ for $i\leq k$ and $\ab_i=\{\id\}$ otherwise; see \cite[\S 2.2.4]{Cand:Notes1}.) Indeed, the group nilspace consisting of $\mc{H}$ equipped with the Host-Kra cubes (relative to the lower-central series) is isomorphic to the product nilspace  $F=\mc{D}_1(\mb{R}^2)\times \mc{D}_2(\mb{R})$ (in this special case, this can be checked directly, an isomorphism being $\varphi:(x,y,z)\in F \mapsto \begin{psmallmatrix} 1 & x & z\\[0.1em]0  & 1 & y\\[0.1em] 0 & 0 & 1 \end{psmallmatrix}\in \mc{H}\vspace{0.05cm}$). The Heisenberg nilmanifold is the quotient space $\mc{H}/\Gamma$ where $\Gamma$ is the \emph{discrete} Heisenberg group $\begin{psmallmatrix} 1 & \mb{Z} & \mb{Z}\\[0.1em]0  & 1 & \mb{Z}\\[0.1em] 0 & 0 & 1 \end{psmallmatrix}\vspace{0.05cm}$. Using this isomorphism, it is easily seen that $\mc{H}/\Gamma$ is isomorphic (as a nilspace) to the quotient of $F$ by the action of the group generated by the following transformations on $F$, which are specific examples of \emph{nilspace translations} (see \cite[\S 3.2.4]{Cand:Notes1}): $\alpha(x,y,z)=(x+1,y,z)$ and $\beta(x,y,z)=(x,y+1,z+x)$. Indeed, these two translations correspond (via the isomorphism $\varphi$) to the right multiplications in $\mc{H}$ by the elements $\begin{psmallmatrix} 1 & 1 & 0\\[0.1em]0  & 1 & 0\\[0.1em] 0 & 0 & 1 \end{psmallmatrix}$ and $\begin{psmallmatrix} 1 & 0 & 0\\[0.1em] 0  & 1 & 1\\[0.1em] 0 & 0 & 1 \end{psmallmatrix}$ respectively, and these elements generate $\Gamma$.
\end{example}
\noindent The central objects involved in our first main result  are what we call \emph{free nilspaces}. The nilspace $F$ above is a simple example, and the general definition is as follows.
\begin{defn}\label{def:free-nil-intro}
A \emph{free nilspace} is a direct product (in the nilspace category) of finitely many components of the form $\mc{D}_i(\mb{R})$ and $\mc{D}_i(\mb{Z})$ where $i\in \mb{N}$. We say that a free nilspace is \emph{discrete} if it is a direct product of components of the form $\mathcal{D}_i(\mathbb{Z})$, and that it is \emph{continuous} if it is a direct product of components of the form $\mc{D}_i(\mb{R})$.
\end{defn}

\noindent A very useful aspect of free nilspaces is the  explicit description one can give of the translations on these nilspaces, in terms of multivariate polynomials; see Theorem \ref{thm:decrip-trans-group}.

As mentioned above, we want to take the quotient of any free nilspace $F$ by the action of a discrete subgroup $\Gamma$ of the translation group $\tran(F)$. However, such a quotient space is not necesarily a well-defined nilspace. We identify the following sufficient condition for the quotient to be a nilspace  (purely algebraically, without any topological considerations for now).
\begin{defn}[Fiber-transitive group of translations]
Let $F$ be a $k$-step free nilspace and let $\Gamma$ be a subgroup of the translation group $\tran(F)$. We say that $\Gamma$ is a \emph{fiber-transitive group on} $F$ if the following holds: for all $x,y\in F$, if there exists $\gamma\in \Gamma$ and $i\in[k]$ such that $\gamma(x)=y$ and $\pi_i(x)=\pi_i(y)$, then there exists $\gamma'\in \Gamma\cap \tran_{i+1}(F)$ such that $\gamma'(x)=y$.
\end{defn}
\noindent This is in fact a special case of a concept introduced in this paper which provides a general setting in which quotienting a nilspace by an equivalence relation yields again a nilspace. We call such equivalence relations \emph{nilspace congruences}, and the congruences corresponding to fiber-transitive groups are called \emph{groupable congruences}; see Section \ref{sec:gpcongs}. We also give examples of subgroups of translation groups that are not fiber-transitive; see Remark \ref{rem:noncongorbit}.

When considering quotients by fiber-transitive groups in settings where a topology is added, we face certain issues similar to standard phenomena in topological group theory. For example, a quotient group such as $\mb{R}/\mb{Q}$ is well-defined algebraically, but when we equip $\mb{R}$ with the standard topology, quotienting by $\mb{Q}$ does not yield a useful \emph{Hausdorff} topology in the quotient, whereas quotienting by subgroups such as $\mb{Z}$ does yield such a topology in the quotient, thanks to the proper discontinuity of the action of $\mb{Z}$ on $\mb{R}$. We shall need an analogously convenient property for quotients of free nilspaces. To formulate this, we first need to recall the notion of the canonical projections onto characteristic factors of a nilspace. Specifically, if $F$ is a free nilspace $\prod_{i=1}^k \mc{D}_i(\mb{Z}^{a_i}\times \mb{R}^{b_i})$, the projection $\pi_j$ is defined as the map that just deletes the coordinates with index $i > j$, namely $\pi_j:F=\prod_{i=1}^k \mc{D}_i(\mb{Z}^{a_i}\times \mb{R}^{b_i})\to F_j=\prod_{i=1}^j \mc{D}_i(\mb{Z}^{a_i}\times \mb{R}^{b_i})$. Similarly, note that a translation $\alpha\in \tran(F)$ can be seen as a translation in $F_j$ just by forgetting about its action on the higher-order components. Thus the projection $\pi_j$ induces a map $\eta_j:\tran(F)\to \tran(F_j)$, which is in fact a group homomorphism. Finally, recall (using for example  \cite[Lemma 3.2.37]{Cand:Notes1}) that for every $k$-step free nilspace $F$ we have $\tran_k(F)\cong \mb{Z}^{a_k}\times \mb{R}^{b_k}$.
\begin{defn}[Fiberwise discrete and cocompact actions on free nilspaces]\label{def:FDCA}
Let $F$ be a $k$-step free nilspace and let $\Gamma$ be a fiber-transitive group on $F$. We say that $\Gamma$ is \emph{fiber-discrete} if for every $j\in[k]$, the group $\eta_j(\Gamma)\cap \tran_j(F_j)$ is a discrete subgroup of $\tran_j(F_j)\cong \mb{Z}^{a_j}\times \mb{R}^{b_j}$. We say that $\Gamma$ is \emph{fiber-cocompact} if for every $j\in[k]$, the group $\eta_j(\Gamma)\cap \tran_j(F_j)$ is cocompact in $\tran_j(F_j)$.
\end{defn}
\noindent The above notions are special cases, for free nilspaces, of notions for more general non-compact topological nilspaces, introduced in Section \ref{sec:gpcongs} (see Definition \ref{def:FDCA-general}).
\begin{remark}
Note that given a fiber-transitive group $\Gamma$, if it is \emph{fiber}-discrete then it is a discrete subgroup of $\tran(F)$ in the usual sense; see Lemma \ref{lem:closedconseq}. However, note that $\Gamma$ being fiber-cocompact does \emph{not} imply that $\Gamma$ is a cocompact subgroup of $\tran(F)$; see Example \ref{ex:fiber-co-comp-not-co-comp}. On the other hand $\Gamma$ being fiber-cocompact does imply that the topological space $F/\Gamma$ is compact; this is clear in Example \ref{ex:fiber-co-comp-not-co-comp}, and in general it follows from Lemma \ref{lem:cong-equi} and Corollary \ref{cor:quo-cfr-nil}.
\end{remark}
\noindent We can now define a generalization, for free nilspaces, of the actions of standard lattices on $\mb{R}^n$.
\begin{defn}[Higher-order lattice actions]
Let $F$ be a $k$-step free nilspace on $\mb{Z}^r\times \mb{R}^s$ (i.e.\ $F=\prod_{i=1}^k \mc{D}_i(\mb{Z}^{a_i}\times \mb{R}^{b_i})$ where $\sum_i a_i=r$ and $\sum_i b_i=s$). A \emph{$k$-th order lattice action} on $F$ (or on $\mb{Z}^r\times \mb{R}^s$) is an action by a subgroup $\Gamma$ of the translation group $\tran(F)$ such that $\Gamma$ is fiber-transitive and fiber-discrete. We say that this action is \emph{cocompact} if $\Gamma$ is also fiber-cocompact.
\end{defn}
Let us now state our first main result, which we prove in Section \ref{sec:groupequivrep} (see Theorem \ref{thm:groupequivrep}).
\begin{theorem}\label{thm:gpcongrep-intro}
Let $\ns$ be a $k$-step compact finite-rank nilspace. Then there exists a $k$-step free nilspace $F$, and a $k$-th order lattice action on $F$ by some group $\Gamma\subset \tran(F)$, such that $\ns\cong F/\Gamma$.
\end{theorem}

\subsection{\textsc{cfr} nilspaces as double-coset spaces}\hfill\\
The quotient structure describing \textsc{cfr} nilspaces in Theorem \ref{thm:gpcongrep-intro} is similar in several ways to a nilmanifold construction, but note that in this theorem $\ns$ is not obtained as the quotient of a Lie \emph{group} (i.e.\ we are not obtaining $\ns$ as a \emph{coset} nilspace), but rather as the quotient of the free nilspace $F$ by the \emph{action} of $\Gamma$. Nevertheless, from this description it is possible to deduce a characterization of $\ns$ as a \emph{double}-coset nilspace. Indeed, fix any $x_0\in F$ and let $K=\stab(x_0)=\{\alpha\in \tran(F):\alpha(x_0)=x_0\}$. We can then prove that the nilspace $F$ is isomorphic to the coset nilspace $K\backslash\!\tran(F)$. We can thus view the action of $\Gamma$ as an action by right multiplication on $ K\backslash\!\tran(F)$, and then prove that $\ns \cong F/\Gamma \cong K\backslash\!\tran(F)/\Gamma$. To confirm all this rigorously,  we must clarify, in particular, under what conditions a double-coset space $K\backslash G/\Gamma$ becomes a nilspace when equipped with the natural image (under the double quotient map) of the cube structure on the ambient filtered group $G$. We identify the following symmetric, purely algebraic property which suffices for this purpose.
\begin{defn}[Groupable nilpair]
Let $(G,G_\bullet)$ be a filtered group of degree $k$ and let $K,\Gamma$ be subgroups of $G$. We say that $(K,\Gamma)$ is a \emph{groupable nilpair} in $(G,G_\bullet)$ if any of the following equivalent properties is satisfied:
\begin{enumerate}
    \item For every $x\in G$ and every $i\ge 0$ we have $(Kx\Gamma)\cap(G_ix\Gamma)=(K\cap G_i)x\Gamma$.
    \item For every $x\in G$ and every $i\ge 0$ we have $(Kx\Gamma)\cap(KxG_i)=Kx(\Gamma\cap G_i)$.
\end{enumerate}
\end{defn}
\noindent We shall need useful conditions, analogous to those in Definition \ref{def:FDCA}, to ensure that a double coset nilspace $K\backslash G/\Gamma$ (equipped with the images of the Host--Kra cubes on $(G,G_\bullet)$) is not only a nilspace purely algebraically, but also a \textsc{cfr} nilspace topologically. The following definition provides sufficient conditions of this sort, when the group $G$ is a compactly-generated Lie group (all Lie groups considered in this paper are  assumed to be compactly generated).
\begin{defn}\label{def:CRDCnilpair}
Let $(G,G_\bullet)$ be a degree-$k$ filtered Lie group such that for each $i\in[k]$ the subgroup $G_i$ is closed in $G$. Let $(K,\Gamma)$ be a groupable nilpair in $(G,G_\bullet)$. For each $i\in [k]$ let $K_i=K\cap G_i$, $\Gamma_i=\Gamma\cap G_i$. We say the nilpair $(K,\Gamma)$ is \emph{closed right-discrete} if for every $i\in[k]$ the group $(KG_{i+1})/G_{i+1}$ is a closed subgroup of $G/G_{i+1}$ and $K_i \Gamma_i G_{i+1} /(K_iG_{i+1})$ is a discrete subgroup of $G_i/(K_iG_{i+1})$. We say the nilpair $(K,\Gamma)$ is \emph{fiberwise cocompact} if for every $i\in[k]$, the group $K_i\Gamma_iG_{i+1}/(K_iG_{i+1})$ is a cocompact subgroup of $G_i/(K_iG_{i+1})$.
\end{defn}
\noindent We prove the following result showing that the above properties indeed suffice as mentioned.
\begin{theorem}\label{thm:suffDC}
Let $(G,G_\bullet)$ be a degree-$k$ filtered Lie group such that for each $i\in [k]$ the subgroup $G_i$ is closed in $G$. Let $(K,\Gamma)$ be a groupable nilpair in $(G,G_\bullet)$ that is closed right-discrete and cocompact. Then the double-coset nilspace $\ns=K\backslash G/\Gamma$ is a $k$-step \textsc{cfr} nilspace. Moreover, for each $i\in [k]$ the $i$-th structure group of $\ns$ is $G_i/(\Gamma_i K_i G_{i+1})$.
\end{theorem}

\noindent We can now state the second main result of this paper. 
\begin{theorem}\label{thm:cfr=double-coset-intro}
Let $\ns$ be a $k$-step compact finite-rank nilspace. Then there exists a degree-$k$ filtered Lie group $(G,G_\bullet)$, and closed subgroups $K,\Gamma$ of $G$ forming a groupable nilpair $(K,\Gamma)$ in $(G,G_\bullet)$ that is closed right-discrete and cocompact, such that $\ns\cong K\backslash G/\Gamma$.
\end{theorem}

\subsection{Compact nilspaces as quotients of pro-free nilspaces and as double-coset spaces}\label{subsec:introCompExt}\hfill\\
Given Theorem \ref{thm:cfr=double-coset-intro}, it is  natural to wonder whether a similar result holds for more general compact nilspaces (not necessarily of finite rank). As mentioned earlier, this question is also highly relevant (via \cite[Theorem 5.11]{CScouplings}) to the study of Host--Kra factors of measure-preserving group actions (however, pursuing this  connection in detail is outside the scope of this paper).  

It is known that every $k$-step compact nilspace is an inverse limit of $k$-step \textsc{cfr} nilspaces \cite[Theorem 2.7.3]{Cand:Notes2}. This, together with Theorems \ref{thm:gpcongrep-intro} and \ref{thm:cfr=double-coset-intro}, implies that every compact $k$-step nilspace can be represented as an inverse limit of either free nilspaces modulo $k$-th order lattice actions, or double-coset nilspaces. Here we want to go further and obtain single objects out of these inverse systems; in particular, we want a single double-coset space representing the original compact nilspace. To do so, we shall use the following concept.
\begin{defn}[Graded pro-free nilspaces]\label{def:free-graded-nilspace} Let $d\in \mb{N}\cup\{\omega\}$.
For every $i< d$ let $F_i$ be a free nilspace. We call the product nilspace $F:=\prod_{i<d} F_i$ a \emph{pro-free} nilspace, and together with the sequence $(F_i)_{i< d}$ we call it a \emph{graded pro-free nilspace}. We say $F$ is \emph{$k$-step} if each  $F_i$ is $k$-step.
\end{defn}
\noindent One of the main points of adding such a grading is that it enables us to identify, within the potentially very large translation group of a pro-free nilspace, certain subgroups that we call \emph{graded} translation groups, which we can ensure to be \emph{pro-Lie} groups. We leave the technical definition of graded translation groups to Section \ref{sec:cpct-nil} (see Definition \ref{def:gr-tran}).

Using this, we shall prove the following result.
\begin{theorem}\label{thm:cpct-as-quo-of-free-inf}
Let $\ns$ be a $k$-step compact nilspace. There exists a $k$-step graded pro-free nilspace $F$, and a fiber-transitive subgroup $\Gamma$ of the corresponding graded translation group on $F$, such that $\ns$ is isomorphic as a compact nilspace to $F/\Gamma$ \textup{(}with the quotient topology\textup{)}.
\end{theorem}
\noindent With this in hand, an argument similar to the proof of Theorem \ref{thm:cfr=double-coset-intro} will enable us to turn Theorem \ref{thm:cpct-as-quo-of-free-inf} into the following general double-coset representation.
\begin{theorem}\label{thm:cpct-nil-as-dou-coset}
Let $\ns$ be a $k$-step compact nilspace. There exists a $k$-step pro-free nilspace $F$, a pro-Lie subgroup $G$ of $\tran(F)$, and a fiber-transitive group $\Gamma\subset G$ such that, letting $K:=\stab_G(x_0)=\{\alpha\in G:\alpha(x_0)=x_0\}$ for some $x_0\in \ns$, we have that $\ns$ is isomorphic as a compact nilspace to $K\backslash G/\Gamma$.
\end{theorem}

\subsection{Main applications}\label{subsec:intro2step}\hfill\\
The main applications in this paper are refinements of the principal results in higher-order Fourier analysis, namely the inverse theorem and regularity lemma for the Gowers norms. The refinements also provide additional perspectives on these results. These applications follow from combining the above-mentioned main results, namely Theorem \ref{thm:gpcongrep-intro} and Theorem  \ref{thm:cfr=double-coset-intro}, with the inverse theorem and regularity lemma from \cite{CSinverse}. Although these results in \cite{CSinverse} apply to more general compact abelian groups, for the applications here we focus on finite abelian groups, especially since this is the setting of the Jamneshan--Tao conjecture.

Let $\ab$ be a finite abelian group. Then there is a natural surjective homomorphism $\varphi:\mb{Z}^n\to \ab$; namely, letting $\ab=\prod_{j=1}^n\mb{Z}/k_j\mb{Z}$ be the invariant factor decomposition (where $k_j$ divides $k_{j+1}$ for all $j<n$), we let $\varphi:\mb{Z}^n\to \ab$ be the homomorphism that takes the quotient by $k_j\mb{Z}$ in the $j$-th component of $\mb{Z}^n$. Suppose now that $\ns,\nss$ are compact nilspaces and that we have a morphism $m:\ab\to \ns$ and a fibration $\psi:\nss \to \ns$. Then the map $g:=m\co\varphi:\mb{Z}^n\to \ns$ is a morphism in $\hom(\mc{D}_1(\mb{Z}^n),\ns)$, and it can then be lifted to $\nss$, in the sense that there exists $g'\in \hom(\mc{D}_1(\mb{Z}^n),\nss)$ such that $\psi\co g' = g$ (see \cite[Corollary A.7]{CGSS-p-hom}). If $\nss$ happens to have a simpler structure than $\ns$, then the morphism $g'$ can be described in a  more useful way than $g$. Examples of such descriptions include the Taylor expansion of a polynomial map $g'\in\hom(\mc{D}_1(\mb{Z}^n),\nss)$ when $\nss$ is a \emph{group nilspace} corresponding to a filtered group $(G,G_\bullet)$ (see for instance \cite[Lemma B.9]{GTZ}). In the special case of this where $\nss$ is a \emph{free} nilspace (and therefore $G$ is an \emph{abelian} group of the form $\mb{Z}^m\times \mb{R}^n$ and $G_\bullet$ is a specific type of filtration of finite degree), note that this Taylor expansion (e.g.\ from \cite[Lemma B.9]{GTZ}) reduces to an expression of $g'$ as a polynomial in a classical sense (a polynomial in $n$ integer variables and of degree at most $k$).
The statements of the inverse theorems will also use the following terms.
\begin{defn}
A \emph{family of double-coset nilspaces} is an array $\mc{N}=((G,K,\Gamma)_{k,i})_{k,i\in\mb{N}}$ where for each $k,i$ the triple $ (G,K,\Gamma)_{k,i}$ consists of a degree-$k$ filtered Lie group $(G,G_\bullet)$ and subgroups $K,\Gamma$ of $G$ forming a closed right-discrete and cocompact groupable nilpair.  Similarly, a \emph{family of higher-order lattice actions} is an array $\mc{M}=((F,\Gamma)_{k,i})_{k,i\in\mb{N}}$ where for each $k,i$ the pair $ (F,\Gamma)_{k,i}$ consists of a free $k$-step nilspace $F$ and a subgroup $\Gamma$ of $\tran(F)$ with a $k$-th order lattice action on $F$. A \emph{metrization} of $\mc{N}$ (resp.\ $\mc{M}$) is a sequence $D=(d_{k,i})_{k,i\in\mb{N}}$ where for each $k,i$ we have that $d_{k,i}$ is a compatible metric on the double-coset space $K\backslash G /\Gamma$ corresponding to the triple $(G,K,\Gamma)_{k,i}$ (resp.\ on the quotient space $F/\Gamma$ corresponding to the pair $(F,\Gamma)_{k,i})$.
\end{defn}

We now state the refinements of the inverse theorem from \cite{CSinverse} that we obtain for finite abelian groups (for the regularity lemma, similar refinements are obtained).
\begin{theorem}[Inverse theorem with double-coset nilspaces]\label{thm:inv-double-coset-intro} There exists a family of double-coset nilspaces $\mc{N}$, such that for any metrization $D$ of $\mc{N}$ the following holds. For any  $\epsilon>0$ and $k\in\mb{N}$,  there exists $M_{k,\epsilon}>0$ such that for any finite abelian group $\ab$ and any 1-bounded  function $f:\ab\to \mb{C}$  with $\|f\|_{U^{k+1}(\ab)}\ge \epsilon$, for some $i\leq M_{k,\epsilon}$ the triple $(G,K,\Gamma)_{k,i}$ in $\mc{N}$  satisfies the following:
\setlength{\leftmargini}{0.7cm}
\begin{enumerate}
    \item there is a polynomial map $g\in \hom(\mc{D}_1(\mb{Z}^n),G_\bullet)$ such that, letting $\pi_{K,\Gamma}:G\to K\backslash G/\Gamma$ be the quotient map and $\varphi:\mb{Z}^n\to\ab$ be the natural surjective homomorphism, there exists a morphism $m:\mc{D}_1(\ab)\to K\backslash G/\Gamma$ with $\pi_{K,\Gamma}\co g = m\co \varphi$,
    \item there is a continuous function $F:K\backslash G/\Gamma\to \mb{C}$ with Lipschitz norm $O_{k,\epsilon}(1)$,
\end{enumerate}
and letting $D\subset \mb{Z}^n$ be any fundamental domain of $\mb{Z}^n/\ker(\varphi)$, we have
\[
\left|\mb{E}_{x\in D} f(x) \overline{F\left(K g(x) \Gamma\right)}\right|\gg_{k,\epsilon} 1.
\]
\end{theorem}
\begin{remark}
Our proof gives additional information on the Lie group $G$ and subgroups $K,\Gamma$ in the above result. Indeed we obtain that $G$ can be taken to be the translation group of a free $k$-step nilspace $F$ (which is  proved to be a Lie group), that $K$ is the stabilizer in $G$ of a point in $F$, and that $\Gamma$ is a discrete subgroup of $G$ such that the pair $(K,\Gamma)$ satisfies Definition \ref{def:CRDCnilpair}. The polynomial map $g$ can also be described in more detail via Taylor expansion \cite[Lemma B.9]{GTZ}. 
\end{remark}

\begin{theorem}[Inverse theorem with higher-order lattice actions]\label{thm:inv-high-ord-latti-intro}
There exists a family of higher-order lattice actions $\mc{M}$, such that for any metrization $D$ of $\mc{M}$ the following holds.
For any $\epsilon>0$ and $k\in\mb{N}$, there exists $N_{k,\epsilon}>0$ such that for any finite abelian group $\ab$  and any 1-bounded function $f:\ab\to \mb{C}$ with $\|f\|_{U^{k+1}(\ab)}\ge \epsilon$, for some $i\leq N_{k,\epsilon}$ the pair $(F,\Gamma)_{i,k}$ in $\mc{M}$ satisfies the following:
\setlength{\leftmargini}{0.7cm}
\begin{enumerate}
    \item there is a polynomial map $g:\mb{Z}^n\to F$ such that, letting $\pi_{\Gamma}:F\to F/\Gamma$ be the quotient fibration and $\varphi:\mb{Z}^n\to\ab$ be the natural surjective homomorphism, there exists a morphism $m:\mc{D}_1(\ab)\to F/\Gamma$ with $\pi_{\Gamma}\co g = m\co \varphi$,
    \item there is a continuous function $\phi:F/\Gamma\to \mb{C}$ with Lipschitz norm $O_{k,\epsilon}(1)$,
\end{enumerate}
and letting $D\subset \mb{Z}^n$ be any fundamental domain of $\mb{Z}^n/\ker(\varphi)$, we have
\[
\left|\mb{E}_{x\in D} f(x) \overline{\phi(\pi_\Gamma(g(x)))}\right|\gg_{k,\epsilon} 1.
\]
\end{theorem}

\noindent From these results we can also deduce the following theorem, giving a new confirmation (alternative to \cite[Theorem 1.10]{J&T}) of Conjecture 1.11 from \cite{J&T}, specifically for the $U^3$-norm.
\begin{theorem}[Inverse theorem for the $U^3$-norm on finite abelian groups] 
Let $\epsilon>0$, let $\ab$ be a finite abelian group, and let $f:\ab\to \mb{C}$ be a 1-bounded function with $\|f\|_{U^{3}(\ab)}\ge \epsilon$. Then there exists a 2-step connected filtered nilmanifold $G/\Gamma$ drawn from some finite collection $\mc{N}_\epsilon$ (where each such nilmanifold is endowed with an arbitrary compatible metric), a continuous function $F:G/\Gamma\to \mb{C}$ of Lipschitz norm $O_\epsilon(1)$, and a nilspace morphism $g:\mc{D}_1(\ab)\to G/\Gamma$, such that $|\mb{E}_{x\in \ab} f(x) \overline{F(g(x))}|\gg_{\epsilon} 1$.
\end{theorem}

\subsection{Outline of the paper}\hfill\smallskip\\
Proving Theorem \ref{thm:gpcongrep-intro} occupies most of the efforts in this paper. The proof argues by induction on the step of the \textsc{cfr} nilspace $\ns$. The induction requires working with a fiber product of $\ns$ with a free $(k-1)$-step nilspace $F_{k-1}$. Since this free nilspace is non-compact, this requires us to handle non-compact topological nilspaces. We develop groundwork theory on these \emph{locally-compact Hausdorff} nilspaces (or \textsc{lch} nilspaces) in Section \ref{sec:lch}. In Section \ref{sec:freenilspaces} we focus on the main type of \textsc{lch} nilspaces that we shall need, namely \emph{free nilspaces}, and establish some of their main properties used in later sections. Section \ref{sec:split-ext} presents an important ingredient for our proof of  Theorem \ref{thm:gpcongrep-intro}. This ingredient, Theorem \ref{thm:splitext}, establishes that the fiber-product nilspace in the above-mentioned induction is a split nilspace extension. This result is of independent interest as a non-trivial generalization of classical results on homological algebra in the category of locally-compact abelian groups \cite{Mosk}. As a consequence, the fiber-product in question is shown to be nilspace-isomorphic to a product nilspace of the form $F_{k-1}\times\mc{D}_k(\ab_k)$, for some inductively-given $(k-1)$-step free nilspace $F_{k-1}$ and the (compact abelian Lie) structure group $\ab_k$ of $\ns$. This is useful because this product nilspace can then be shown to admit a fibration from a $k$-step free nilspace, namely $F:=F_{k-1}\times\mc{D}_k(\ab_k')$ where $\ab_k'$ is an abelian Lie group having a continuous surjective homomorphism onto $\ab_k$. Obtaining this fibration from this free nilspace is a major step towards Theorem \ref{thm:gpcongrep-intro}, and is recorded (in greater generality than  for \textsc{cfr} nilspaces) in Theorem \ref{thm:lcfr-factor-of-free}. Section \ref{sec:gpcongs} introduces the notion of groupable congruence and its main properties, in preparation for Section \ref{sec:groupequivrep}, where the first main result Theorem \ref{thm:gpcongrep-intro}  is proved (see Theorem \ref{thm:gpcongrep}). In Section \ref{sec:DC}, we first gather the necessary background on double-coset nilspaces (both algebraically and topologically).  We then prove our second main result, Theorem \ref{thm:cfr=double-coset-intro}, and prove the above-mentioned applications. Finally, in Section \ref{sec:cpct-nil} we extend our results to compact (not necessarily finite-rank) nilspaces.

\section{Locally-compact Hausdorff nilspaces}\label{sec:lch}

In this section we consider nilspaces equipped with a  locally compact, Hausdorff and second-countable topology (we shall abbreviate this from now on by referring to these topologies as ``\textsc{lch} topologies"). In particular all such spaces are metrizable. This will extend the topological theory of nilspaces from the case of  \emph{compact} nilspaces (studied in \cite{CamSzeg,Cand:Notes2}), to the present not-necessarily compact setting. Note that second-countability and the Hausdorff property will be assumed throughout this paper (just as they were in \cite{CamSzeg,Cand:Notes2}).

Let us start by recalling that a \emph{compact nilspace} is a nilspace $\ns$ equipped with a compact topology such that for every $n\ge 0$ the cube set $\cu^n(\ns)$ is a closed subset of the product (compact) space $\ns^{\db{n}}$. The non-compact generalization considered here is the following.
\begin{defn}[\textsc{lch} nilspace]\label{def:top-nil-open-maps-version}
We say that a nilspace $\ns$ is an \textsc{lch}\emph{ nilspace} if $\ns$ is a locally-compact Haursdorff second-countable topological space such that the following property holds for every integer $n\geq 0$:
\begin{enumerate}
    \item The cube set $\cu^n(\ns)$ is closed in the product topology on $\ns^{\db{n}}$.
    \item The coordinate projection $p^{\db{n}}:\cu^n(\ns)\to \cor^n(\ns)$ is an open map.
\end{enumerate}
\end{defn}

\noindent Similar notions have been studied before in the context of topological cube structures. For instance, in \cite[Ch. 7 \S 2.2, Proposition 7 (iii)]{HKbook} such objects are described and their main properties include continuity of a completion function for cubes. In the present nilspace-theoretic context, this can be phrased as follows. For a $k$-step nilspace $\ns$, since completion of $k+1$ corners is unique, there is a well-defined function $\mc{K}:\cor^{k+1}(\ns)\to \ns$ that maps any given corner to the point in $\ns$ that completes this corner\footnote{We will denote this completion map by $\mc{K}_{k+1}$ when there are many completion functions for different dimensions in the same argument.}. The assumption here analogous to that of \cite{HKbook} is that if we endow $\ns$ with some topology, then the map $\mc{K}$ should be continuous. Let us note that this continuity is an immediate consequence of Definition \ref{def:top-nil-open-maps-version} when $\ns$ is a $k$-step nilspace.
\begin{lemma}\label{lem:cont-if-open}
Let $\ns$ be a $k$-step \textsc{lch} nilspace. Then the completion function $\mc{K}:\cor^{k+1}(\ns)\to\ns$ is continuous.
\end{lemma}
\begin{proof}
The map $p^{\db{k+1}}:\cu^{k+1}(\ns)\to \cor^{k+1}(\ns)$ is open and bijective by assumption. Hence its inverse $(p^{\db{k+1}})^{-1}$ is continuous. Composing this inverse with the projection to the coordinate indexed by $1^{k+1}$, we obtain precisely the map $\mc{K}$, which is therefore continuous.
\end{proof}
It is natural to wonder whether condition $(ii)$ in Definition \ref{def:top-nil-open-maps-version} is necessary in this non-compact setting. Indeed, in \cite{CamSzeg} or \cite[Definition 1.0.2]{Cand:Notes2}, the definition of \emph{compact} nilspaces only requires the cube set $\cu^n(\ns)$ to be a closed subset of $\ns^{\db{n}}$ for each $n\geq 0$ (in addition to the requirement that the topology of $\ns$ be \textsc{lch} and compact). In the present more general (non-compact) setting, it turns out that closure of the cube sets is not enough to obtain a satisfactory theory. Indeed, such a theory should yield in particular that every 1-step LCH nilspace is a topological abelian group for any addition operation $+$ compatible with the group structure (``compatible" meaning that the cubes on this nilspace are precisely the standard abelian cubes relative to $+$). But there are examples where this fails if we only assume closure of the cubes. Let us detail this with the following construction inspired by an example of Eric Wofsey \cite{WofStackEx}.
\begin{example}
On the abelian group $(\mb{Z},+)$ let us define the topology $\tau$ generated by all the sets $\{n\}$ for $n\in \mb{Z}\setminus\{0\}$ and the sets $\{0,3^m,3^{m+1},\ldots\}$ for $m\in \mb{N}$. It can be checked that this topology is \textsc{lch}. Hence, in particular this is a metric space. We claim that $(\mb{Z},\tau)$ satisfies that for all $n\ge 0$, the set $\cu^n(\mc{D}_1(\mb{Z}))$ is closed in $\mb{Z}^{\db{n}}$, but $(\mb{Z},\tau,+)$ is not a topological group. 

In order to prove that the cube sets are closed, note that it suffices to prove that $\cu^2(\mc{D}_1(\mb{Z}))$ is closed (since every higher-dimensional cube set is an intersection of preimages of 2-dimensional cube sets by continuous coordinate-projection maps). Let $((x_n,y_n,z_n,t_n))_n$ be a sequence in $\cu^2(\mc{D}_1(\mb{Z}))$ (that is, satisfying $x_n+y_n=z_n+t_n$ for each $n$) and converging to some limit $(x,y,z,t)$. We want to prove that $x+y=z+t$. Note that in this topology there are two types of convergent sequences. If $b\in \mb{Z}\setminus\{0\}$, then $b_n\to  b$ means that $b_n=b$ for $n$ sufficiently large. If $b=0$, then $b_n\to b$ means that either $b_n=0$ for $n$ large enough, or there exists a subsequence $(b_{n_m})_m$ such that $b_{n_m}=3^{\iota(m)}$ for some strictly increasing function $\iota:\mb{N}\to \mb{N}$. Thus, if $(b_n)$ converges then either $b_n$ is eventually constant or we can restrict to a subsequence such that $b_{n_m}=3^{\iota(m)}$ for an increasing $\iota$. Coming back to the convergent sequence $(x_n,y_n,z_n,t_n)$, we have the following exhaustive cases. In the first case,  all four component sequences are eventually constant, and we then have the desired closure. In a second case, only one of the sequences has a strictly increasing subsequence, and then (relabeling the subsequence), we have $x_n=3^{\iota(n)}$, and $y_n,z_n,t_n$ are eventually constant. We would then have $3^{\iota(n)}$ equal to the constant $z+t-y$, which is impossible for $n$ large enough. In the third case, two of the component sequences in the tuple $(x_n,y_n,z_n,t_n)$ have a strictly increasing subsequence. Then note that the only subcase here that is not obviously impossible is when these two components are on different sides of the equation $x_n+y_n = z_n+t_n$. Suppose that these components are $x_n$ and $z_n$. Passing to a subsequence if necessary, we may assume that $x_n=3^{\iota_1(n)}$. If, on this subsequence, the component $z_n$ becomes eventually constant (zero), then we are back into the previous case, yielding a contradiction. Otherwise we can pass to a further subsequence in which $x_n=3^{\iota_1(n)}$ and $z_n=3^{\iota_2(n)}$. Then, as $y_n$ and $t_n$ are eventually constant, for $n$ large enough they are equal to $y$ and $t$ respectively. But now, if we write the equation $3^{\iota_1(n)}+y=3^{\iota_2(n)}+t$ in base 3, it follows that as $\iota_1$ and $\iota_2$ are strictly increasing, the equation can hold for all such large enough $n$ only if $y=t$ and $\iota_1(n)=\iota_2(n)$. Thus, the limit of $(x_n,y_n,z_n,t_n)$ is $(0,y,0,y)$,  which is in $\cu^2(\mc{D}_1(\mb{Z}))$. The case when three of the component sequences have infinite increasing subsequences (and the fourth one is eventually constant) is treated similarly, considering subcases, in each of which either one obtains a contradiction or a conclusion confirming the desired closure. The case when all four sequences have a strictly increasing subsequence leads to the limit $(0,0,0,0)$, which is in $\cu^2(\mc{D}_1(\mb{Z}))$ as required. Finally, note that $(\mb{Z},\tau,+)$ is not a topological group, because the sequence $(3^n,3^n)$ converges to $(0,0)$ but $3^n+3^n=2\cdot3^n$ diverges as $n\to\infty$.
\end{example}
\noindent Our next aim is to prove that the basic constructions in nilspace theory also preserve the properties of \textsc{lch} nilspaces. We start by proving that the structure groups of an \textsc{lch} nilspace $\ns$ are \textsc{lch} topological groups when equipped with the restriction of the topology on $\ns$. This involves the following lemma, showing that the fibers of the factor projections $\pi_i:\ns\to\ns_i$ are closed.
\begin{lemma}\label{lem:closure}
Let $\ns$ be a $k$-step nilspace endowed with an \textsc{lch} topology such that for every $n\ge 0$ the cube set $\cu^n(\ns)$ is a closed subset of $\ns^{\db{n}}$. Then for each $i\in [k]$ the equivalence classes of $\sim_i$ are closed subsets of $\ns$. Moreover, the graph of $\sim_i$ is a closed subset of $\ns\times \ns$.
\end{lemma}
\begin{proof}
Fix any $x\in \ns$ and let $(y_n)$ be any sequence of points $y_n\sim_i x$ converging to some $y\in \ns$. We just need to show that $y\sim_i x$. This follows by definition of $\sim_i$ and the cube-set closure assumptions; indeed there is an $(i+1)$-cube $\q_n$ equal to $x$ everywhere except at one vertex where it equals $y_n$, and the convergence of $y_n$ to $y$ implies that $\q_n$ converges to a map $g$ that equals $x$ everywhere except at one vertex where it equals $y$. By the closure of $\cu^{i+1}(\ns)$ we deduce that $g$ is a cube, which implies that $y\sim_i x$ as required.

To see that the graph of $\sim_i$ is a closed subset of $\ns\times \ns$, the argument is similar: suppose that $(x_n,y_n)$ is a convergent sequence in this graph, that is we have $x_n\sim_i y_n$ for all $n$ and there is $(x,y)\in \ns\times\ns$ such that $(x_n,y_n)\to (x,y)$ in $\ns\times \ns$. Now $x_n\sim_i y_n$ means there is an $(i+1)$-cube $\q_n$ equal to $x_n$ at all vertices except at $1^{i+1}$ where it equals $y_n$. The convergence $(x_n,y_n)\to (x,y)$ implies that $(\q_n)$ converges in $\ns^{\db{i+1}}$ to a map $g$ equal to $x$ at every vertex except at $1^{i+1}$ where it equals $y$. By closure of $\cu^{i+1}(\ns)$ we have that $g\in \cu^{i+1}(\ns)$, so $(x,y)$ is in the graph of $\sim_i$ as required.
\end{proof}
\begin{proposition}\label{prop:k-th-group-action}
Let $\ns$ be a $k$-step nilspace endowed with an \textsc{lch} topology such that for every $n\ge 0$ the cube set $\cu^n(\ns)$ is closed in $\ns^{\db{n}}$ and the completion function on $\cor^{k+1}(\ns)$ is continuous. Then the structure group $\ab_k$ is an \textsc{lch} abelian group acting continuously on $\ns$.
\end{proposition}
\begin{proof}
By Lemma \ref{lem:closure}, each fiber $F$ of $\pi_{k-1}$ is a closed subset of $\ns$. It is straightforwardly seen that then the subspace topology on $F$ is also locally-compact second-countable Hausdorff. Fix any such fiber $F$ on $\ns$, i.e. an equivalence class of $\sim_{k-1}$. Recall from \cite[Corollary 3.2.16]{Cand:Notes1} that we can regard $F$ as an abelian group by fixing an arbitrary $e\in F$. The group operation can be defined as follows: for any $x,y\in F$, let $\q'=\q'_{x,y}\in\cor^{k+1}(\ns)$ be the corner such that $\q'(1,0,1^{k-1})=x$, $\q'(0,1^k)=y$ and $\q'(v)=e$ otherwise. Letting $\mc{K}:\cor^{k+1}(\ns)\to \ns$ denote the (continuous) completion function, we define $x+y:=\mc{K}(\q_{x,y})$. For $k=2$ we can represent this for instance as follows:
$(x,y)\mapsto
\smallcorner{e}{e}{e}{e}{e}{x}{y}\mapsto \mc{K}(\q_{x,y})$. 

Letting $\iota:F\times F\to \cor^{k+1}(\ns)$,  $(x,y)\mapsto \q_{x,y}$, we have that $x+y:=\mc{K}\co\iota (x,y)$. By \cite[Lemma 3.2.7]{Cand:Notes1} we have $x+y\in F$. Since both $\iota$ and $\mc{K}$ are continuous,  the addition operation on $F$ is also continuous. A similar argument shows that inversion is also a continuous operation (see \cite[Proposition 2.4.1]{Cand:Notes1}).

We have thus shown that on any particular fiber $F$ of $\pi_{k-1}$, the structure and topology of $\ns$ induces a structure of \textsc{lch} abelian group on $F$. Next, we prove that every two such fibers are homeomorphic. To see this, let  $F,F'$ be any two fibers of $\pi_{k-1}$ and fix any points $e\in F,e'\in F'$, which we shall use as identity elements in the respective groups. For any $x\in F$, define $\q_x\in\cor^{k+1}(\ns)$ by $\q_x(1^k,0)=x$, $\q_x(v,0)=e$ for $v\in \db{k}$ and $\q_x(v,1)=e'$ for $v\in \db{k}\setminus\{1^k\}$. We then define the map $\theta:F\to F'$, $\theta(x):=\mc{K}(\q_x)$. For $k=2$, again we can visualize this as follows:
$x\mapsto \q_x:=\smallcorner{e}{e}{e}{x}{e'}{e'}{e'}\mapsto \mc{K}(\q_x)$.

By the continuity assumptions for the corner-completion function, this map $\theta$ is continuous. As its inverse is defined exactly the same way, we have that $\theta$ is an isomorphism of topological groups between $F$ and $F'$. Let $\ab_k$ denote the resulting topological group. Since $\ab_k$ is homeomorphic to each fiber $F$, its topology is \textsc{lch}.

The action of $\ab_k$ on each fiber is continuous. However, we need to establish continuity of the action of $\ab_k$ globally on $\ns$, that is, we need to show that the map $\ns\times \ab_k\to \ns$, $(x,z)\mapsto x+z$ is continuous. This follows since this map is a composition of the following form (illustrated for $k=2$; for $k>2$ it is similar): $(x,z)\mapsto \q_{x,z }:=\smallcorner{e}{e}{e}{e+z}{x}{x}{x}\mapsto \mc{K}(\q_{x,z})$.
Again, continuity of completion and the addition function on $F$ implies that $\mc{K}(\q_{x,z})$ is continuous as well.
\end{proof}
We obtain the following consequence concerning the quotient topology on $\ns_{k-1}$.
\begin{corollary}\label{cor:quotient-top-k-1-factor}
Let $\ns$ be a $k$-step nilspace endowed with an \textsc{lch} topology such that for every $n\ge 0$ the cube set $\cu^n(\ns)\subset \ns^{\db{n}}$ is closed and the completion of $(k+1)$-corners is continuous. Then the canonical quotient map $\pi_{k-1}:\ns\to\ns_{k-1}$ is an open map and $\ns_{k-1}$ equipped with the quotient topology is an \textsc{lch} space.\footnote{Note this claim is just about the quotient topology on $\ns_{k-1}$, we are not claiming yet that $\ns_{k-1}$ is an \textsc{lch} nilspace.}
\end{corollary}
\begin{proof}
By Proposition \ref{prop:k-th-group-action} the saturation of any open set in $\ns$ by the action of $\ab_k$ is also an open set, and this implies that $\pi_{k-1}$ is an open map (in other words, the relation  $\sim_{k-1}$ is an open relation, in the sense of \cite[p.\ 52, Definition 2]{BourbakiGT1}). By Lemma \ref{lem:closure} the graph of $\sim_{k-1}$ is closed in $\ns\times\ns$. Hence $\pi_{k-1}(\ns)=\ns_{k-1}$ is Hausdorff, by \cite[p.\ 79, Proposition 8]{BourbakiGT1}. By \cite[p.\ 107, Proposition 10]{BourbakiGT1}, it follows that $\ns_{k-1}$ is locally compact.
Finally, since the continuous surjective open image of a second-countable space is second-countable, so is $\ns_{k-1}$.
\end{proof}
\begin{corollary}\label{cor:cpct-sets-are-im-of-cpct-sets}
Let $\ns$ be a $k$-step nilspace satisfying the hypothesis of Corollary \ref{cor:quotient-top-k-1-factor}. Then for every compact set $K'\subset \ns_{k-1}$ there exists a compact $K\subset \ns$ such that $\pi_{k-1}(K)=K'$.
\end{corollary}
\begin{proof}
This follows from \cite[p.\ 107, Proposition 10]{BourbakiGT1}. 
\end{proof}
\begin{lemma}\label{lem:closure-cube-in-k-1-factor}
Let $\ns$ be a $k$-step nilspace endowed with an \textsc{lch} topology such that for every $n\ge 0$ the cube set $\cu^n(\ns)\subset \ns^{\db{n}}$ is closed and the completion function on $(k+1)$-corners is continuous. Then $\cu^n(\ns_{k-1})$ is a closed subset of $\ns_{k-1}^{\db{n}}$ in the product topology.
\end{lemma}
\begin{proof}
Since $\ns_{k-1}$ is a $(k-1)$-step nilspace, we have that if $\cu^k(\ns_{k-1})$ is closed, then arguing as in the proof of \cite[Lemma 2.1.1]{Cand:Notes2} we deduce that so is $\cu^n(\ns_{k-1})$ for every $n>k$. Thus it suffices to show that $\cu^n(\ns_{k-1})$ is closed for every $n\leq k$. 
But for such $n$, any lift of any $n$-cube on $\ns_{k-1}$ under $\pi_{k-1}^{\db{n}}$ is an $n$-cube on $\ns$ (recall \cite[Remark 3.2.12]{Cand:Notes1}), so ${\pi_{k-1}^{\db{n}}}^{-1}(\cu^n(\ns_{k-1})) = \cu^n(\ns)$. 
Since $\cu^n(\ns_{k-1})$ is closed, we deduce that $\cu^n(\ns_{k-1})$ is closed in the quotient topology on $\ns_{k-1}^{\db{n}}$ induced by $\pi_{k-1}^{\db{n}}$. But this quotient topology is homeomorphic to the $\db{n}$-power of the quotient topology on $\ns_{k-1}$ induced by $\pi_{k-1}$ (by \cite[p.\ 55, Corollary]{BourbakiGT1}, so we conclude that $\cu^n(\ns_{k-1})$ is indeed closed in the product topology on $\ns_{k-1}^{\db{n}}$.\end{proof}
\begin{theorem}\label{thm:factor-of-lch-nil}
Let $\ns$ be an \textsc{lch} $k$-step nilspace. Then $\ns_{k-1}$, equipped with the cube sets $(\pi_{k-1}^{\db{n}}(\cu^n(\ns)))_{n\ge 0}$, is a $(k-1)$-step \textsc{lch} nilspace with respect to the quotient topology.
\end{theorem}
\begin{proof}
Note that the algebraic part of this result follows from \cite[Lemma 3.2.10]{Cand:Notes1}. Combining Corollary \ref{cor:quotient-top-k-1-factor} and Lemma \ref{lem:closure-cube-in-k-1-factor} we have that the topology on $\ns_{k-1}$ is also \textsc{lch}, and part $(i)$ of Definition \ref{def:top-nil-open-maps-version} follows. Hence, we just have to prove part $(ii)$, i.e.,  that the projection maps $p_{k-1}^{\db{n}}:\cu^n(\ns_{k-1})\to \cor^{n}(\ns_{k-1})$ are open. 

First let us assume that $n\le k$. In this case, the sets $\cu^n(\ns)$ are closed and saturated with respect to $\pi_{k-1}^{\db{n}}$. That is, $\cu^n(\ns)$ is closed and $(\pi_{k-1}^{\db{n}})^{-1}(\pi_{k-1}^{\db{n}}(\cu^n(\ns)))=\cu^n(\ns)$. By \cite[Theorem 22.1]{Mu} the restricted map $\pi_{k-1}^{\db{n}}|_{\cu^n(\ns)}:\cu^n(\ns)\to \cu^n(\ns_{k-1})$ is a quotient map. Similarly, the restriction $\pi_{k-1}^{\db{n}\setminus\{1^n\}}|_{\cor^n(\ns)}:\cor^n(\ns)\to \cor^n(\ns_{k-1})$ is a quotient map. Given any open set $U\subset \cu^n(\ns_{k-1})$, we claim that $p_{k-1}^{\db{n}}(U)\subset \cor^n(\ns_{k-1})$ is open. By definition of the quotient topology, $p_{k-1}^{\db{n}}(U)$ is open if and only if $(\pi_{k-1}^{\db{n}\setminus\{1^n\}}|_{\cor^n(\ns)})^{-1}(p_{k-1}^{\db{n}}(U))$ is. But this set is precisely $p^{\db{n}}((\pi_{k-1}^{\db{n}}|_{\cu^n(\ns)})^{-1}(U))$ where $p^{\db{n}}:\cu^n(\ns)\to\cor^n(\ns)$. As $\pi_{k-1}^{\db{n}}|_{\cu^n(\ns)}:\cu^n(\ns)\to\cu^n(\ns_{k-1})$ is continuous, we have that $(\pi_{k-1}^{\db{n}}|_{\cu^n(\ns)})^{-1}(U)$ is open, and then the openness of $p^{\db{n}}$ gives us the desired conclusion.
In particular, if $n=k$ we get the result that the completion function $\mc{K}:\cor^k(\ns_{k-1})\to \cu^k(\ns_{k-1})$ is continuous (as it is the inverse of the open map $p_{k-1}^{\db{k}}:\cu^k(\ns_{k-1})\to \cor^k(\ns_{k-1})$). 

For $n>k$ note that we have a bijection $\cor^n(\ns_{k-1})\to \cu^n(\ns_{k-1})$, which is also continuous (as it just consists in applying $\mc{K}$ on some upper face of $\db{n}$). But by definition, the inverse of this function is $p_{k-1}^{\db{n}}:\cu^n(\ns_{k-1})\to \cor^n(\ns_{k-1})$. Hence $p_{k-1}^{\db{n}}$ is open and the result follows.
\end{proof}
\noindent We can now establish the equivalence between the completion function being continuous, and the projection to the sets of corners being an open map.
\begin{lemma}\label{lem:eq-def-lch-nil}
Let $\ns$ be a $k$-step nilspace endowed with an \textsc{lch} topology. Then $\ns$ is an \textsc{lch} nilspace if and only if the following conditions are satisfied:
\begin{enumerate}
    \item For every $n\ge 0$, the cube set $\cu^n(\ns)$ is a closed subset of $\ns^{\db{n}}$.
    \item For every $n\in[k]$, the corner completion function $\mc{K}_{n+1}:\cor^{n+1}(\ns_n)\to \ns_n$ is continuous, where the canonical factor $\ns_n$ is endowed with the quotient topology.
\end{enumerate}
\end{lemma}
\begin{proof}
For the forward implication, note that $(i)$ follows from Definition \ref{def:top-nil-open-maps-version}, and $(ii)$ follows from Lemma \ref{lem:cont-if-open} applied to each factor of $\ns$, using Theorem \ref{thm:factor-of-lch-nil} to pass to lower-step  factors. 

To prove the converse, first note that we can apply Proposition \ref{prop:k-th-group-action}. Hence, the  structure group $\ab_k$ is an \textsc{lch} topological group that acts on $\ns$ continuously. By Corollary \ref{cor:quotient-top-k-1-factor}, we get that for every $n\in[k]$ the projection map $\pi_n:\ns\to\ns_n$ is an open quotient map. 

First, note that for $m>k$ the openness of $p^{\db{m}}:\cu^m(\ns)\to \cor^m(\ns)$ is trivial as its inverse can be recovered using $\mc{K}_{k+1}$ (which is continuous by hypothesis).

For $m\le k$, note that it is enough to prove that $p^{\db{m}}(\cu^m(\ns)\cap \prod_{v\in\db{m}} A_v)$ is open for open sets $A_v\subset \ns$, $v\in \db{m}$. Moreover, we can restrict our attention to sets where $A_v=\ns$ for all $v\not=1^m$, since we have
\[
p^{\db{m}}\big(\cu^m(\ns)\cap \prod_{v\in\db{m}} A_v\big)=\big(\prod_{v\in\db{m}\setminus\{1^m\}} A_v\big)\cap p^{\db{m}}\big(\cu^m(\ns)\cap [\ns^{\db{m}\setminus\{1^m\}}\times A_{1^m}]\big).
\]
But for those sets we know how to compute such images. Let
\setlength{\leftmargini}{0.4cm}
\begin{itemize}
    \item $\pi_{m-1}:\ns\to\ns_{m-1}$ be the quotient map, which is also open.
    \item $\pi_{m-1}^{\db{m}\setminus\{1^m\}}:\cor^m(\ns)\to\cor^m(\ns_{m-1})$ be the coordinate-wise application of $\pi_{m-1}$ (which is continuous).
    \item $\mc{K}_m:\cor^m(\ns_{m-1})\to \ns_{m-1}$ the completion function (which is continuous by hypothesis).
\end{itemize}
Then $p^{\db{m}}(\cu^m(\ns)\cap [\ns^{\db{m}\setminus\{1^m\}}\times A_{1^m}]) =(\pi_{m-1}^{\db{m}\setminus\{1^m\}})^{-1}(\mc{K}_{m}^{-1}(\pi_{m-1}(A_{1^m})))$, which is open, and the result follows.
\end{proof}
\begin{remark}
Note that compact nilspaces are in particular \textsc{lch} nilspaces. To see this it suffices to check the continuity of the completion function. This follows by \cite[Lemma 2.1.12]{Cand:Notes2}.
\end{remark}

\noindent Recall from \cite[Definition 3.1.4]{Cand:Notes1} the definition of a \emph{simplicial set} $S\subset \db{n}$, and from \cite[Definition 3.3.10]{Cand:Notes1} the definition of the set of \emph{restricted morphisms} $\hom_f(S,\ns)$. For our purposes here, we will just need the case where $S\subset \db{n}$ is a simplicial set, $\ns$ is a nilspace and $f:\emptyset\to \ns$ (so $f$ is just the trivial function). Thus, we define $\hom(S,\ns)\subset \ns^S$ as the set of maps $g:S\to \ns$ such that for every discrete-cube morphism $\phi:\db{m}\to \db{n}$ with image lying in $S$, we have $g\co\phi\in\cu^m(\ns)$ (no reference made to $f$, since in this case it is trivial). In particular $\cor^n(\ns)=\hom(\db{n}\setminus\{1^n\},\ns)$. When $\ns$ is an \textsc{lch} nilspace, we always consider $\hom(S,\ns)$ to be equipped with the subspace topology induced by the product topology on $\ns^S$.
\begin{lemma}\label{lem:simplicial-proy}
Let $\ns$ be a $k$-step \textsc{lch} nilspace and let $S\subset \db{n}$ be a simplicial set. Then the projection map $p_S:\cu^n(\ns)\to \hom(S,\ns)$ is open.
\end{lemma}
\begin{proof}
For any simplicial set $S\subset \db{n}$, let us say that $v\in S$ is \emph{maximal} if there is no $w\in S$ with $w\not=v$ such that $w\sbr{j}\ge v\sbr{j}$ for all $j\in[n]$. We claim that the projection map $p':\hom(S,\ns)\to\hom({S\setminus\{v\}},\ns)$ is open. It suffices to check this for open sets of the form $\prod_{w\in S} U_w\cap \hom(S,\ns)$ where $U_w\subset \ns$ are open for all $w\in S$. Letting $|v|:=\sum_{j=1}^n v\sbr{j}$, and defining the sets $V_{w'}:=\ns$ if $w'\not=1^{|v|}$ and $V_{1^{|v|}}=U_v$, we claim that
\[
p'(\prod_{w\in S} U_w\cap \hom(S,\ns))= \prod_{w\in S\setminus\{v\}} U_w\cap \hom({S\setminus\{v\}},\ns) \cap r^{-1}\big(p^{\db{|v|}}(\prod_{w'\in \db{|v|}} V_{w'}\cap \cu^{|v|}(\ns))\big)
\]
where $r$ is the map $\hom({S\setminus\{v\}},\ns)\to \cor^{|v|}(\ns)$ that projects to the coordinates $w\in S$ such that $w\sbr{j}\le v\sbr{j}$ for all $j\in[n]$,  relabeling these coordinates to the set $\db{|v|}$. Clearly $r$ is continuous, and as $p^{\db{|v|}}$ is open by hypothesis the result follows.

Then, for any simplicial $S\subset \db{n}$ we can find a sequence $v_1,v_2,\ldots,v_t\in \db{n}$ such that $S=\db{n}\setminus\{v_1,\ldots,v_t\}$, and for every $\ell=0,\ldots,t$ we have that $\db{n}\setminus\{v_1,\ldots,v_\ell\}$ is simplicial and $v_{\ell+1}$ is maximal in $\db{n}\setminus\{v_1,\ldots,v_\ell\}$. We can then write $p^S:\cu^n(\ns)\to\hom(S,\ns)$ as a composition of $t$ open maps $\cu^n(\ns)\to \hom(\db{n}\setminus\{v_1\},\ns)\to \hom(\db{n}\setminus\{v_1,v_2\},\ns)\to \cdots \to\hom(S,\ns)$ which is open in particular.
\end{proof}
\begin{lemma}\label{lem:open-projection-of-cubes}
Let $\ns$ be a $k$-step \textsc{lch} nilspace. Then for every $n\ge 0$ the projection function $\pi_i^{\db{n}}|_{\cu^n(\ns)}:\cu^n(\ns)\to \cu^n(\ns_i)$ is an open map. In particular $\pi_i^{\db{n}}|_{\cu^n(\ns)}$ is a quotient map.
\end{lemma}
\begin{proof}
It suffices to prove the result for $i=k-1$, and then apply this case $i-1$ times.
For $i=k-1$ there are two cases. Let us deal with $n\le k$ first. Note that it is enough to prove the result for open sets of the form $\prod_{v\in\db{n}} U_v\cap \cu^n(\ns)$ where $U_v\subset \ns$ are open. We claim that in this case
\[
\pi_{k-1}^{\db{n}}(\prod_{v\in\db{n}} U_v\cap \cu^n(\ns))=\prod_{v\in\db{n}} \pi_{k-1}(U_v)\cap \cu^n(\ns_{k-1}).
\]
If this is true then the result will follow for $n\le k$, $\pi_{k-1}$ is an open map. In order to prove that, note that clearly the LHS is contained in the RHS. For the other inclusion, let $\pi_{k-1}\co\q$ be such that $(\pi_{k-1}\co\q)(v)\in \pi(U_v)$ for all $v\in \db{n}$. For each $v\in \db{n}$ there exists $d(v)\in U_v$ such that $(\pi_{k-1}\co\q)(v)=(\pi_{k-1}\co d)(v)$. By \cite[Corollary 3.2.8]{Cand:Notes1} we have that in particular $d\in \cu^n(\ns)$ and the result follows in this case.
For $n>k$, note that we can write the map $\pi_{k-1}^{\db{n}}$ as the composition of the following maps:
\[
\begin{tikzpicture}
  \matrix (m) [matrix of math nodes,row sep=3em,column sep=4em,minimum width=2em]
  {\cu^n(\ns) & \hom({\db{n}_{\le k}},\ns) & \hom({\db{n}_{\le k}},\ns_{k-1}) & \cu^n(\ns_{k-1}) \\};
  \path[-stealth]
    (m-1-1) edge node [above] {$p_{\db{n}_{\le k}}$} (m-1-2)
    (m-1-2) edge node [above] {$\pi_{k-1}^{\db{n}_{\le k}}$} (m-1-3)
    (m-1-3) edge node [above] {$\mc{K}'$} (m-1-4);
\end{tikzpicture}
\]
Here $p_{\db{n}_{\le k}}$ is defined as in Lemma \ref{lem:simplicial-proy}, $\db{n}_{\le k}:=\{v\in\db{n}:\sum_{j=1}^n v\sbr{j}\le k\}$, $\pi_{k-1}^{\db{n}_{\le k}}$ is the coordinate-wise application of $\pi_{k-1}$ and $\mc{K}'$ is the composition of several applications of the completion function $\mc{K}:\cor^k(\ns_{k-1})\to\ns_{k-1}$ in order to complete an element of $\hom({\db{n}_{\le k}},\ns_{k-1})$. By Lemma \ref{lem:simplicial-proy} we have that $p_{\db{n}_{\le k}}$ is open. By an argument very similar to the one in the previous paragraph we have that $\pi_{k-1}^{\db{n}_{\le k}}$ is open as well, and $\mc{K}'$ is open as its inverse is continuous (because it is just a projection to some coordinates).
\end{proof}
\noindent Our next step is to prove that several important constructions associated with nilspaces inherit the \textsc{lch} property from the original spaces. Recall the definition of fiber product of nilspaces.
\begin{defn}\label{def:fib-prod}
Let $\ns, \ns'$ and $\nss$ be nilspaces and let $\varphi:\ns\to\nss$ and $\varphi':\ns'\to\nss$ be fibrations. We define the \emph{fiber product} of $\ns$ and $\ns'$ as the nilspace consisting of the set $\ns\times_{\nss}\ns':=\{(x,x')\in \ns\times\ns': \varphi(x)=\varphi'(x')\}$ endowed with the cubes $\cu^n(\ns \times_{\nss} \ns'):=\{(\q,\q')\in \cu^n(\ns)\times \cu^n(\ns'):\varphi\co\q=\varphi'\co\q'\}$.
\end{defn}
It is proved in \cite[Lemma 4.2]{CGSS} that such fiber products are indeed nilspaces.
\begin{remark}
Note that in this definition we have not assumed anything about the topology on $\ns,\ns'$ or $\nss'$. We will do so in the following lemmas. Note also that in the expression $\ns\times_{\nss}\ns'$ there is no reference to $\varphi$ or $\varphi'$. These maps will be clear from the context.
\end{remark}

\noindent An observation which will often be useful in what follows is that some structures that we will have to handle are in fact \emph{Cartan principal bundles}, recalled here from \cite[1.1.2 Definition]{P61}.
\begin{defn}[Cartan principal bundle]\label{def:Cartan}
Let $B$ be a $G$-principal bundle. Let the action of $G$ over $B$ be given by $(g,b)\in G\times B\mapsto  gb\in B$. We say that the bundle is a \emph{Cartan principal bundle} if given the map $T:B\times G\to B\times B$ defined by $(b,g)\mapsto(b,gb)$ there exists a continuous map\footnote{Here the image $\im(T)$ is understood to have the subspace topology.} $R:\im(T)\to B\times G$ such that $R\co T=\id_{B\times G}$.
\end{defn}

\noindent In our context, the group $G$ will usually be abelian. Thus, if the action of $G$ on $B$ is given by $(g,b)\mapsto g+b$, the principal $G$-bundle $B$ is Cartan if and only if the \emph{difference map}, i.e.\ the map $\{(b,b')\in B\times B:\exists g\in G \text{ s.t. }g+b=b'\}\to G$, $(b,b')\mapsto b'-b$, is continuous.
\begin{lemma}\label{lem:diff-cont}
Let $\ns$ be a $k$-step \textsc{lch} nilspace. For each $i\in[k]$, let $\ab_i$ be the $i$-th structure group of $\ns$. Then for every $i\in[k]$, the $\ab_i$-principal bundle $\ns_i$ is a Cartan principal bundle.
\end{lemma}
\begin{proof}
By induction on $k$ it suffices to prove this for $i=k$. Recall that $\ab_k$ can be identified with any fixed equivalence class (or fiber) in $\ns$ for the relation $\sim_k$, in which we have fixed a particular element $e\in F$ to be the identity for addition. Then, given $x,y\in \ns\times_{\ns_{k-1}}\ns$, let $\q_{x,y}\in\cor^{k+1}(\ns)$ be the corner defined by  $\q_{x,y}(v,0):=y$ for $v\in \db{k}\setminus\{1^k\}$, $\q_{x,y}(1^k,0):=x$ and $\q_{x,y}(v,1)=e$ for $v\in\db{k}\setminus\{1^k\}$. The difference $x-y$ can be then expressed as $\mc{K}(\q_{x,y})$ where $\mc{K}:\cor^{k+1}(\ns)\to\ns$ is the unique-completion function. As an example, for $k=2$ this composition can be viewed as follows: $(x,y)\mapsto \q_{x,y}:=\smallcorner{y}{y}{y}{x}{e}{e}{e}\mapsto \mc{K}(\q_{x,y})$. As the map $(x,y)\mapsto \q_{x,y}$ is continuous, so is $\mc{K}\co\q_{x,y}$ and the result follows.
\end{proof}
\begin{lemma}\label{lem:fib-prod-top-nil}
Let $\ns,\ns'$ and $\nss$ be $k$-step \textsc{lch} nilspaces, and let $\varphi:\ns\to\nss$,   $\varphi':\ns'\to\nss$ be (continuous) fibrations. Then the corresponding fiber product $\ns \times_{\nss} \ns'$ (recall Definition \ref{def:fib-prod}) is also an \textsc{lch} nilspace. Moreover, if $\ab_i(\ns),\ab_i(\ns')$ and $\ab_i(\nss)$ are the $i$th structure groups of $\ns,\ns'$ and $\nss$ respectively, then the $i$th structure group of $\ns \times_{\nss} \ns'$ is $G_i:=\{(z,z')\in \ab_i(\ns)\times\ab_i(\ns'): \phi_i(z)=\phi'_i(z')\}$ where $\phi_i:\ab_i(\ns)\to\ab_i(\nss)$ and $\phi'_i:\ab_i(\ns')\to\ab_i(\nss)$ are the structure homomorphisms of $\varphi$ and $\varphi'$ respectively.
\end{lemma}
\begin{proof}
By \cite[Lemma 4.2]{CGSS} we know that this fiber product is a
$k$-step nilspace (algebraically speaking). Hence, we just have to check that the conditions of Definition \ref{def:top-nil-open-maps-version} are satisfied.

From the continuity of $\varphi$ and $\varphi'$ it is easy to check that the topology on $\ns \times_{\nss} \ns'$ is also \textsc{lch}. Indeed, this set is a closed subset of $\ns\times\ns'$, as it is the preimage of the diagonal of $\nss\times \nss$ under the map $(x,x')\mapsto(\varphi(x),\varphi'(x'))$. Similarly, for each $n\geq 0$ the cube set $\cu^n(\ns \times_{\nss} \ns')$ is closed, since it is $\big[\cu^n(\ns\times\ns')\big]\cap (\ns \times_{\nss} \ns')^{\db{n}}$, which is a closed subset of $(\ns\times\ns')^{\db{n}}$, so in particular it is a closed subset of $(\ns \times_{\nss} \ns')^{\db{n}}$.

Instead of proving the second part of Definition \ref{def:top-nil-open-maps-version} we use Lemma \ref{lem:eq-def-lch-nil} and just check that $\ns \times_{\nss} \ns'$ satisfies property $(ii)$ in this lemma. By \cite[Proposition A.20]{CGSS-p-hom}, for each $i\in [k]$ the $i$-step characteristic factor of this nilspace is $\ns_i\times_{\nss_i}\ns'_i$. Also, by the proof of \cite[Proposition A.20]{CGSS-p-hom}, the $i$-th structure group of $\ns \times_{\nss} \ns'$ is $G_i:=\{(z,z')\in \ab_i(\ns)\times\ab_i(\ns'): \phi_i(z)=\phi'_i(z')\}$ where $\phi_i:\ab_i(\ns)\to\ab_i(\nss)$ and $\phi'_i:\ab_i(\ns')\to\ab_i(\nss)$ are the structure homomorphisms of $\varphi$ and $\varphi'$ respectively (see \cite[Definition 3.3.1]{Cand:Notes1}). To see that part $(ii)$ of Lemma \ref{lem:eq-def-lch-nil} holds for $i=k$, note that the completion function on $\cor^{k+1}(\ns \times_{\nss} \ns')$ is continuous, as it is just the coordinate-wise application of the $(k+1)$-completion functions on $\ns$ and $\ns'$. For $i< k$, we already have a purely algebraic nilspace isomorphism $(\ns \times_{\nss} \ns')_i \to \ns_i \times_{\nss_i} \ns_i'$. If we show that this is also a homeomorphism then we will be done, as we can then repeat the preceding argument to obtain the continuity of the completion function. To obtain the homeomorphism property it suffices to show that the quotient topology on $(\ns \times_{\nss} \ns')_i$ is equal to the subspace topology on $\ns_i \times_{\nss_i} \ns_i'$ induced by the product topology on $\ns_i\times\ns_i$. Note that it suffices to show this for $i=k-1$, as the rest of the cases follow by induction on $k$. Hence, let us consider the following diagram:
\begin{equation}\label{diag:fib-prod-homoeom}
\begin{aligned}[c]
\begin{tikzpicture}
  \matrix (m) [matrix of math nodes,row sep=2em,column sep=4em,minimum width=2em]
  {\ns \times_{\nss} \ns' & \\
     (\ns \times_{\nss} \ns')_{k-1} & \ns_{k-1} \times_{\nss_{k-1}} \ns_{k-1}', \\};
  \path[-stealth]
    (m-1-1) edge node [above] {$T$} (m-2-2)
    (m-1-1) edge node [right] {$\pi$} (m-2-1)
    (m-2-1) edge node [above] {$\iota$} (m-2-2);
\end{tikzpicture}
\end{aligned}
\end{equation}

\vspace{-0.3cm}

\noindent where $T(x,x'):=(\pi_{k-1}(x),\pi_{k-1}(x'))$, $\pi$ is the quotient map for the relation $\sim_k$ (see \cite[Definition 3.2.3]{Cand:Notes1}) and $\iota(\pi(x,x'))=(\pi_{k-1}(x),\pi_{k-1}(x'))$. Recall that by \cite[Proposition A.20]{CGSS-p-hom}, $\iota$ is a well-defined bijective morphism. We want to prove that it is a homeomorphism.

We already know that the relation $\sim_k$ on $\ns \times_{\nss} \ns'$ comes from the continuous action of the group $G_{k}$. Also, using the closure of $\cu^n(\ns \times_{\nss} \ns')$ it follows that the set $\{((x_0,x_0'),(x_1,x_1'))\in (\ns \times_{\nss} \ns')^2: (x_0,x_0')\sim_k(x_1,x_1')\}$ is closed. Hence we can apply Lemma \ref{lem:quo-by-cont-gr-act} to conclude that the topology on $(\ns \times_{\nss} \ns')_{k-1}$ is \textsc{lch} and, in particular, it is metrizable. Let us now check that $\iota$ is continuous. Indeed, if $U\subset \ns_{k-1} \times_{\nss_{k-1}} \ns_{k-1}'$ is open, then since $\iota^{-1}(U)=\pi(T^{-1}(U))$, the continuity of $T$ and the openness of $\pi$ imply that $\iota^{-1}(U)$ is open as required. 

To prove that $\iota^{-1}$ is continuous, it suffices to prove that if $(\pi_{k-1}(x_n),\pi_{k-1}(x_n'))$ converges to $(\pi_{k-1}(x),\pi_{k-1}(x'))$ in $\ns_{k-1}\times_{\nss_{k-1}}\ns_{k-1}'$ as $n\to\infty$ then $\pi(x_n,x_n')\to \pi(x,x')$ as $n\to \infty$. By Lemma \ref{lem:conv-quot}, it suffices to prove that there exists a sequence $(g_n\in G_{k})_{n\in\mb{N}}$ such that $(x_n,x_n')+g_n\to (x,x')$, which we do in the remainder of this proof. 

Note that, without loss of generality, we can assume that $\varphi(x_n)=\varphi'(x_n')$ for all $n\in\mb{N}$ (so far we only know that $\varphi_{k-1}(\pi_{k-1}(x_n))=\varphi'_{k-1}(\pi_{k-1}(x_n'))$) and similarly that $\varphi(x)=\varphi'(x')$. Indeed, if this is not the case, say if $\varphi(x)\not=\varphi'(x')$, then since $(\pi_{k-1}(x),\pi_{k-1}(x'))$ is in $\ns_{k-1}\times_{\nss_{k-1}}\ns_{k-1}'$, we have $\pi_{k-1}(\varphi(x))=\varphi_{k-1}(\pi_{k-1}(x)) = \varphi'_{k-1}(\pi_{k-1}(x'))=\pi_{k-1}(\varphi'(x'))$. In particular, there exists $s\in \ab_k(\nss)$ such that $\varphi(x)=\varphi'(x')+s$. Using the surjectivity of $\phi'_k$ we can find some $r\in \ab_k(\ns')$ such that $\phi'_k(r)=s$, whence  $\varphi(x)=\varphi'(x'+r)$. With similarly chosen $r_n$ for each $n$, we obtain the claimed additional assumption for $(x_n,x_n'+r_n)$ and $(x,x'+r)$ (and note that $\pi(x_n,x_n'+r_n)=\pi(x_n,x_n')$ and similarly for $(x,x'+r)$, by the invariance of $\pi_{k-1}$ under the action of $\ab_k(\ns')$).

By Lemma \ref{lem:conv-quot} in $\ns$, there is a sequence $(z_n)_n$ in $\ab_k(\ns)$ such that $x_n+z_n\to x$ as $n\to\infty$ in $\ns$. For every $n\in \mb{N}$, by the surjectivity of $\phi_k'$ there is $z_n'\in \ab_k(\ns')$ such that $(z_n,z_n')\in G_k$. Note that  $\varphi(x+z_n)\to\varphi(x)$ as $n\to\infty$ and that $\varphi(x_n+z_n)=\varphi'(x_n'+z_n')$ for all $n\in \mb{N}$. Next, by Lemma \ref{lem:conv-quot} applied in $\ns'$, for every $n\in \mb{N}$ there exists $z_n''\in \ab_k(\ns')$ such that $x_n'+z_n'+z_n''\to x'$ as $n\to\infty$. We claim that $\phi'_k(z_n'')\to 0$ as $n\to \infty$. Indeed $\varphi'(x_n'+z_n'+z_n'')\to \varphi'(x')=\varphi(x)$ and also $\varphi'(x_n'+z_n'+z_n'')=\varphi'(x_n'+z_n')+\phi_k'(z_n'')=\varphi(x_n+z_n)+\phi_k'(z_n'')$. As $\varphi(x_n+z_n)\to \varphi(x)$ we get that $\phi'_k(z_n'')\to 0$ (here we are using Lemma \ref{lem:diff-cont}). Again by Lemma \ref{lem:conv-quot} we get that there exists $z_n^*\in \ker(\phi'_k)$ such that $z_n''-z_n^*\to 0$ in $\ab_k(\ns')$. We claim now that the sequence $(g_n:=(z_n,z_n'+z_n^*)\in G_k)_{n\in\mb{N}}$ works, i.e., that $(x_n,x'_n)+g_n\to (x,x')$. Indeed  $x_n'+z_n'+z_n^*=(x_n'+z_n'+z_n'')+(z_n^*-z_n'')$, where the first summand converges to $x'$ and the second to $0$.
\end{proof}
We now focus on \textsc{lch} nilspaces whose structure groups are Lie groups.
\begin{defn}[Lie-fibered nilpaces]\label{def:lcfr-nil}
Let $\ns$ be a $k$-step \textsc{lch} nilspace. We say that $\ns$ is a \emph{locally-compact and finite-rank} nilspace, or \emph{Lie-fibered nilspace}, if all its structure groups are compactly generated (abelian) Lie groups.
\end{defn}
\noindent For any such nilspace $\ns$, the factor projection $\pi_{k-1}:\ns\to\ns_{k-1}$ is locally trivial as a bundle projection. This can be deduced from a result of Palais \cite[Theorem 2.3.3]{P61} as follows.
\begin{theorem}\label{thm:local-triviality-bundle}
Let $\ns$ be a Lie-fibered $k$-step nilspace. Then there is a covering of $\ns_{k-1}$ by open sets $U$ with the property that for each $U$ there is a homeomorphism $\phi_U:\pi_{k-1}^{-1}(U)\to U\times \ab_k$ which is $\ab_k$-equivariant. 
\end{theorem}
The equivariance here is meant relatively to the natural actions of $\ab_k$ on the two spaces in question, namely its action on $\ns$  as the $k$-th structure group, and its action  on $U\times \ab_k$ by addition in the second component.
\begin{proof}
By Proposition \ref{prop:k-th-group-action} and Lemma \ref{lem:diff-cont}, the nilspace $\ns$ is a Cartan principal $\ab_k$-bundle, and therefore by \cite[Theorem in \S 4]{P61}) this bundle is locally trivial.
\end{proof}

\subsection{The open mapping theorem for Lie-fibered nilspaces}\hfill\\
A well-known version of the open mapping theorem for topological groups states that if $G$ and $H$ are \textsc{lch} topological groups, with $G$ being $\sigma$-compact, and $\varphi:G\to H$ is a continuous surjective homomorphism, then $\varphi$ is an open map \cite[p.\ 669]{H&M-Cpct}. In this section we prove the following analogue.
\begin{theorem}\label{thm:open-mapping-thm}
Let $\ns$,$\nss$ be Lie-fibered $k$-step nilspaces, and let $\varphi:\ns\to\nss$ be a continuous fibration. Then $\varphi$ is an open map.
\end{theorem}
\noindent As with classical open mapping theorems, this is useful to establish that certain invertible continuous maps are homeomorphisms. In particular we have the following consequence.
\begin{corollary}\label{cor:inv-cont-lcfr}
Let $\ns$ be a Lie-fibered $k$-step nilspace, and let $\alpha$ be a continuous translation on $\ns$. Then $\alpha$ is a homeomorphism. In particular $\alpha^{-1}$ is continuous and the set of continuous translations on $\ns$ is a group under composition.
\end{corollary}
\noindent Hence, for a Lie-fibered nilspace $\ns$, the notation $\tran(\ns)$ will always refer to the group of \emph{continuous} translations on $\ns$.
\begin{proof}
Any translation on $\ns$ is a fibration. Hence, by Theorem \ref{thm:open-mapping-thm} a continuous translation is an open map, so its inverse is continuous.
\end{proof}
\begin{proof}[Proof of Theorem \ref{thm:open-mapping-thm}]
We argue by induction on $k$. The case $k=1$ holds by the open mapping theorem for topological groups. For $k>1$, we shall use that the bundle map $\pi:\ns\to\ns_{k-1}$ is locally trivial, which is given by Theorem \ref{thm:local-triviality-bundle}. That is, we shall use that there exists a covering of $\ns_{k-1}$ by open sets $\{U_i\}_{i\in I}$ such for each $i$ there is a map $\pi_{k-1}^{-1}(U_i)\to  U_i \times\ab_k(\ns)$ that is a homeomorphism which is also $\ab_k$-equivariant.

Fix any $x\in\ns$. We now show that for small enough neighborhoods of $x$, their images under $\varphi$ are open. Let $V\subset \nss_{k-1}$ be any open neighborhood of $\pi_{k-1}(\varphi(x))$ such that there is a $\ab_k$-equivariant homeomorphism $\tau:V\times \ab_k(\nss)\to \pi_{k-1}^{-1}(V)$.  Note that we can express $\tau$ as $(\pi_{k-1}(y),z')\mapsto s'(\pi_{k-1}(y))+z'$ for some continuous map $s':V\to \pi_{k-1}^{-1}(V)$.  The set $\varphi^{-1}(\pi_{k-1}^{-1}(V))$ is an open neighborhood of $x$. Thus $\pi_{k-1}(\varphi^{-1}(\pi^{-1}(V)))$ is an open neighborhood of $\pi_{k-1}(x)$ (recall that $\pi_{k-1}$ is an open map by Corollary \ref{cor:quotient-top-k-1-factor}). Let $U^*$ be an open neighborhood of $\pi_{k-1}(x)$ such that there is an equivariant homeomorphism $U^*\times \ab_k(\ns)\to \pi_{k-1}^{-1}(U^*)$, $(\pi_{k-1}(x),z)\mapsto s(\pi_{k-1}(x))+z$. Take $U:=U^*\cap \pi_{k-1}(\varphi^{-1}(\pi^{-1}(V)))$ and note that the same homemorphism works replacing $U^*$ with $U$.

Identifying $\pi_{k-1}^{-1}(U)$ with $U\times \ab_k(\ns)$, we can then define the base of neighborhoods of $x$ consisting of sets of the form $A\times H\subset U\times \ab_k(\ns)$, where $A\subset U$ and $H\subset \ab_k(\ns)$ are open. The image of any such neighbourhodd $A\times H$ under $\varphi$ is
\[
\varphi(A\times H)=\tau^{-1}\Big(\,\{(\varphi_{k-1}(u),\phi_k(z)+(\varphi\co s-s'\co \varphi_{k-1})(a)): (a,z)\in A\times H\}\Big),
\]
where $\phi_k:\ab_k(\ns)\to\ab_k(\nss)$ is the $k$-th structure homomorphism of $\varphi$ and $\varphi_{k-1}:\ns_{k-1}\to\nss_{k-1}$ is the morphism between the $k-1$ factors of $\ns$ and $\nss$ (note that $\phi_k$ is continuous and surjective since $\varphi$ is a fibration). To see that this image is open, it suffices to prove the claim that the composition $U\times \ab_k\to U\times \ab_k'\to V\times \ab_k'$, $(u,z)\mapsto (u,\phi_k(z))\mapsto (\varphi_{k-1}(u),\phi_k(z)+(\varphi\co s-s'\co \varphi_{k-1})(u))$ is open. Since the first of these maps is clearly open, we focus on showing that the second map, namely 
$T:U\times \ab_k'\to V\times \ab_k'$, $(u,z')\mapsto (\varphi_{k-1}(u),z'+(\varphi\co s-s'\co \varphi_{k-1})(u))$, is open. Let  $I:U\times Z_k'\to U\times Z_k'$,  $(u,z')\mapsto (u,z'-(\varphi\co s-s'\co \varphi_{k-1})(u))$. Since $I$ is clearly a  homeomorphism, it suffices to show that $T\co I$ is open. But $T\co I(u,z')=(\varphi_{k-1}(u),z')$, which is clearly open. This proves our claim.

By induction $\varphi_{k-1}(A)$ is open and by the open mapping theorem for topological groups so is $\phi_k(H)$. It follows that $\varphi(A\times H)$ is open as required

We have thus proved that images under $\varphi$ of the above basic open neighbourhoods are open, so the openness of $\varphi$ follows by taking unions.
\end{proof}
\noindent In fact, a similar local-triviality argument  proves the following more general result.
\begin{theorem}\label{thm:open-mapping-general}
Let $\ns$ be a $k$-step \textsc{lch} nilspace, and let $\nss$ be a $k$-step Lie-fibered nilspace. Then any continuous fibration $\varphi:\ns\to\nss$ is an open map.
\end{theorem}

\noindent To prove this we first prove the following auxiliary result extending \cite[Proposition A.19]{CGSS-p-hom}.
\begin{proposition}\label{prop:quot-last-str-gr}
Let $\ns$ be a $k$-step, \textsc{lch} nilspace and $\ab$ a closed subgroup of its last structure group. Then $\ns/\ab$ with the cubes $\cu^n(\ns/\ab):=\{\pi_{\ab}\co \q: \q\in \cu^n(\ns) \}$ is an \textsc{lch} nilspace.\footnote{Here as usual $\ns/\ab$ is the quotient of $\ns$ by the action of $\ab$ with the quotient topology and $\pi_{\ab}$ is the quotient map.} Moreover, $(\ns/\ab)_{k-1}\cong \ns_{k-1}$ and the $k$-th structure group of $\ns/\ab$ is isomorphic to $\ab_k/\ab$ as topological groups where $\ab_k$ is the $k$-th structure group of $\ns$.
\end{proposition}
\begin{proof}
The algebraic part of the proof follows from \cite[Proposition A.19]{CGSS-p-hom}, so here we just prove the topological part. We need to prove first that the topology of $\ns/\ab$ is \textsc{lch} and later we will use Lemma \ref{lem:eq-def-lch-nil} to prove that it is indeed an \textsc{lch} nilspace.

Let $\pi_{k-1}:\ns\to \ns_{k-1}$ be the usual quotient map in $\ns$. First we prove that the set $C=\{(x,x')\in \ns^2: x=gx' \text{ for some } g\in \ab\}$ is closed. Since $\ns$ is a Cartan bundle over $\ns_{k-1}$, on the closed set $C'=\{(x,x')\in\ns^2:\pi_{k-1}(x)=\pi_{k-1}(x')\}$ we know that the difference function $C'\to \ab_k$, $(x,x')\mapsto x-x'$ is well-defined and continuous; since $C$ is the preimage under this function of the closed set $\ab$, it is indeed closed. By Lemma \ref{lem:quo-by-cont-gr-act} the topology on $\ns/\ab$ is \textsc{lch}.

Next, we want to see that $\cu^n(\ns/\ab)$ is closed for all $n\in \mb{N}$. As the map $\pi_{\ab}:\ns\to \ns/\ab$ is open and continuous, the map $\ns^{\db{n}}\to (\ns/\ab)^{\db{n}}$ is a quotient map. Hence, it suffices to check that $\cu^n(\ns)+\ab^{\db{n}}$ is a closed subset of $\ns^{\db{n}}$. For $n\le k$ the result is trivial and the case $n>k+1$ follows from the case $n=k+1$, so let us just prove this latter case. Let $\q_n+d_n\in \cu^{k+1}(\ns)+\ab^{\db{k+1}}$ be a convergent sequence. Let $r_n':=(\q_n+d_n)|_{\db{k+1}\setminus{1^{k+1}}}\in \cor^{k+1}(\ns)$. This is clearly convergent so let us define $r_n\in \cu^{k+1}(\ns)$ its unique completion for all $n\in \mb{N}$ and note that then $r_n\in \cu^{k+1}(\ns)$ is a convergent sequence. It follows that $\q_n-r_n\in \cu^{k+1}(\mc{D}_k(\ab_k))$ and $(\q_n-r_n)(v) \in \ab$ for all $v\not=1^{k+1}$. Thus, applying the $\sigma_{k+1}$ operator (see \cite[Definition 2.2.22]{Cand:Notes1}) we get that $\sigma_{k+1}(\q_n-r_n)=0$ (by \cite[Proposition 2.2.28 $(iii)$]{Cand:Notes1}) so in particular $(\q_n-r_n)(1^{k+1})\in \ab$ as well. Thus, if we let $e_n:=\q_n-r_n$ we have that $\q_n+d_n = r_n+(e_n+d_n)$. As both $\q_n+d_n$ and $r_n$ are convergent sequences so is $e_n+d_n\in \ab^{\db{k+1}}$. Using that $\cu^{k+1}(\ns)$ and $\ab^{\db{k+1}}$ are closed we have that $\q_n+d_n$ converges to a limit in $\cu^{k+1}(\ns)+\ab^{\db{k+1}}$. Lemma \ref{lem:eq-def-lch-nil} $(i)$ follows.

Next, let us prove first that the completion function $\cor^{k+1}(\ns/\ab)\to \ns/\ab$ is continuous. If $\pi_{\ab}\co \q_n\in \cor^{k+1}(\ns/\ab)$ is convergent, by Lemma \ref{lem:conv-quot} applied for each $v\in \db{k+1}\setminus\{1^{k+1}\}$, we can assume that $\q_n\in \ns^{\db{k+1}\setminus\{1^{k+1}\}}$ is a convergent sequence as well. And moreover, note that if $\pi:\ns/\ab\to \ns_{k-1}$ is the map such that $\pi_{k-1}=\pi\co\pi_{\ab}$, by \cite[Proposition A.19 $(iii)$]{CGSS-p-hom} we have that $\pi$ is a morphism. If we let $F\subset \db{k+1}\setminus\{1^{k+1}\}$ be any $k$-dimensional face containing $0^{k+1}$, $\pi_{k-1}\co (\q_n|_F) = \pi\co \pi_{\ab}\co (\q_n|_F)$. But $\pi_{\ab}\co (\q_n|_F)\in \cu^k(\ns/\ab)$ so in particular, $\pi_{k-1}\co (\q_n|_F)\in \cu^k(\ns_{k-1})$. Thus by \cite[Remark 3.2.12]{Cand:Notes1} we conclude that $\q_n\in \cor^{k+1}(\ns)$. Hence, the continuity of completion of $(k+1)$-corners for $\ns/\ab$ follows from that for  $\ns$.

By \cite[Proposition A.19 $(iii)$]{CGSS-p-hom} we have a nilspace isomorphism $(\ns/\ab)_{k-1}\to \ns_{k-1}$ (purely algebraically). We want to prove that this map is also a homeomorphism. Let $\pi'_{k-1}:\ns/\ab\to (\ns/\ab)_{k-1}$ be the quotient map to the $(k-1)$-step factor of $\ns/\ab$. We want to prove that the map $\pi'_{k-1}(\pi_{\ab}(x)) \mapsto \pi_{k-1}(x)$ is a homeomorphism (we already know that it is an algebraic nilspace isomorphism). By Corollary \ref{cor:quotient-top-k-1-factor} we have that $\pi'_{k-1}$ is open and continuous. As both $\pi_{k-1}$ and $\pi_{\ab}$ are also open and continuous, it follows that the previous map is a homeomorphism. Hence, $(ii)$ of Lemma \ref{lem:eq-def-lch-nil} for $n$-corners, $n\le k$, follows from the fact that $\ns$ is an \textsc{lch} nilspace.

The $k$-th structure group of $\ns/\ab$ is isomorphic to $\ab_k/\ab$. To prove this, recall that by \cite[Corollary 3.2.16]{Cand:Notes1} the $k$-th structure group of $\ns/\ab$ is isomorphic (as a degree-$k$ abelian group) to any fiber of the factor map $\pi_{k-1}$ on $\ns/\ab$. Moreover, on this fiber we consider the subspace topology. By \cite[Proposition A.19]{CGSS-p-hom} we know that algebraically the $k$-th structure group of $\ns/\ab$ is isomorphic to $\ab_k/\ab$. To see that the isomorphism is also topological, fix any $\pi_{k-1}$-fiber $\pi_{\ab}(e)+\ab_k/\ab$ of $\ns/\ab$. Note that $\pi_{\ab}|_{e+\ab_k}: e+\ab_k\to \pi_{\ab}(e)+\ab_k/\ab$ is a continuous map (being the restriction of a continuous map). Moreover, it is also a homomorphism of groups when we identify $e+\ab_k$ with $\ab_k$ and $\pi_{\ab}(e)+\ab_k/\ab$ with the $k$-th structure group of $\ns/\ab$. By the open mapping theorem for Polish groups this map is open and thus the $k$-th structre group of $\ns/\ab$ is topologically isomorphic to $\ab_k$ modulo the kernel of $\pi_{\ab}|_{e+\ab_k}$, which is precisely $\ab$. \end{proof}

\begin{proof}[Proof of Theorem \ref{thm:open-mapping-general}]
Let $\phi_k:\ab_k\to\ab_k'$ be the last structure morphism of $\varphi$ (where $\ab_k$ and $\ab_k'$ are the $k$-th structure groups of $\ns$ and $\nss$ respectively). By \cite[Proposition A.19]{CGSS-p-hom} and Proposition \ref{prop:quot-last-str-gr} we have that $\ns/\ker(\phi_k)$ is a $k$-step \textsc{lch} nilspace. The canonical factor map $\pi_{\ker(\phi_k)}:\ns\to\ns/\ker(\phi_k)$ is continuous and open, and $\varphi$ factors through a continuous fibration $\varphi':\ns/\ker(\phi_k)\to\nss$, $x+\ker(\phi_k)\mapsto \varphi(x)$, i.e., $\varphi = \varphi'\co \pi_{\ker(\phi_k)}$. Note that since $\phi_k$ is a continuous  surjective homomorphism, by Lemma \ref{lem:quot-polish-gr-lie-gr} we have that $\ab_k/\ker(\phi_k)\cong \ab_k'$ is a Lie group. By \cite[Theorem 2.3.3]{P61}, $\ns/\ker(\phi_k)$ is a locally-trivial abelian bundle over $(\ns/\ker(\phi_k))_{k-1}\cong \ns_{k-1}$. By induction on $k$ we can suppose that $\varphi'_{k-1}$ is open. We can then argue as in the proof of Theorem \ref{thm:open-mapping-thm}, using the local triviality, to deduce that $\varphi'$ is open. Since $\pi_{\ker(\phi_k)}$ is also open, we deduce that $\varphi$ is open.
\end{proof}

\subsection{Topological aspects of nilspace extensions}\label{subsec:topaspextLCH}\hfill\\
Recall from \cite[\S 3.3.3]{Cand:Notes1} the purely algebraic definition of extensions of nilspaces. In this subsection we study several ways in which topology can be added to this construction in the present non-compact setting. To this end, let us first recall the following general notion of continuous abelian bundle (see e.g.\ \cite[Definition 2.1.6]{Cand:Notes2}), valid in non-compact settings.
\begin{defn}\label{def:CtsAbBund}
Let $\bnd$ be an abelian bundle with base $S$, structure group $\ab$ and projection $\pi$. We say that $\bnd$ is  \emph{continuous} if the following conditions hold:
\begin{enumerate}
\item $\bnd$ and $S$ are topological spaces.
\item $\ab$ is an abelian topological group.
\item The action $\alpha:\ab\times \bnd\to \bnd$ is continuous.
\item A set $U\subset S$ is open if and only if $\pi^{-1}(U)$ is open in $\bnd$.
\end{enumerate}
\end{defn}
\noindent We say that $\bnd$ is an \emph{\textsc{lch} abelian $\ab$-bundle} if it is a continuous abelian bundle and in addition the topologies on $\bnd$, $S$, $\ab$ are all locally-compact, second-countable Hausdorff. We also say that $\bnd$ is a \emph{Cartan continuous $\ab$-bundle} if, in addition to $\bnd$ being a continuous $\ab$-bundle, it is also a Cartan principal $\ab$-bundle (i.e.\ the difference map is continuous: see Definition \ref{def:Cartan}). Note that there are examples of continuous abelian bundles that are not Cartan. Take for example $\bnd=\mb{R}/\mb{Z}$ and consider the action of $\mb{Z}$ on $\bnd$ defined as $(x\mod 1,n)\mapsto x+n\alpha \mod 1$ for some irrational $\alpha\in\mb{R}$. If we let $S:=\bnd/\mb{Z}$ with the quotient topology we have that this is a continuous abelian bundle. However, it can be checked not to be Cartan.

With this we can define the following topologically enriched version of the notion of nilspace extension, for the case of \textsc{lch} nilspaces. In what follows, when we say that a nilspace is an \emph{algebraic} extension of another nilspace, we mean that it is an extension in the purely algebraic sense of \cite[\S 3.3.3]{Cand:Notes1} (before any topological aspects are added).

\begin{defn}[Continuous extensions of \textsc{lch} nilspaces]\label{def:ext-lch-nil}
Let $k,t\ge 0$ be integers. Let $\ns$, $\nss$ be $k$-step \textsc{lch} nilspaces such that $\nss$ is an algebraic nilspace extension of $\ns$ of degree $t$, by some \textsc{lch} abelian group $\ab$, with associated projection $p:\nss\to\ns$. We say that $\nss$ (or $p:\nss\to \ns$) is a \emph{continuous extension} of $\ns$ if $\nss$ is additionally an \textsc{lch} abelian $\ab$-bundle with base $\ns$.
\end{defn}
\begin{remark}\label{rem:topquot}
Equivalently $\nss$ is a continuous $\ab$-extension of $\ns$ if the action of $\ab$ on $\nss$ is  continuous and $p$ is a continuous open map. Note that, in this case, the topology on $\ns$ is equal to the quotient of the topology on $\nss$ under the action of $\ab$.
\end{remark}
\noindent The next result extends a well-known fact in the theory of Lie groups. Recall that if $0\to H\to G\to F\to 0$ is a proper short exact sequence of \textsc{lch} groups  where $H$ and $F$ are Lie, then so is $G$ (see \cite[Thm.\ 2.6]{Mosk}, \cite[Lem.\ A.3]{HK-non-conv}). For nilspaces we have the following analogous result.

\begin{proposition}\label{prop:ext-lie-is-lie-1}
Let $\ns$, $\nss$ be \textsc{lch} $k$-step nilspaces, and suppose that $\nss$ is a degree-$k$ algebraic nilspace extension of $\ns$ by an abelian \textsc{lch} group $\ab$. Let $p:\nss\to \ns$ be the associated fibration. If $p$ is continuous, $\ns$ is Lie-fibered, and $\ab$ is a Lie group, then $\nss$ is also Lie-fibered. Moreover $\nss$ is then a continuous extension of $\ns$.
\end{proposition}

\begin{proof}
Recall that $p$ is indeed a fibration (\cite[Proposition A.17]{CGSS-p-hom}), and note that $p_{k-1}:\nss_{k-1}\to \ns_{k-1}$ is then a continuous fibration as well. Moreover $p_{k-1}$ is bijective. Indeed, for any elements $\pi_{k-1}(x),\pi_{k-1}(x')\in\nss_{k-1}$ such that $p_{k-1}(\pi_{k-1}(x))=p_{k-1}(\pi_{k-1}(x'))$, we have $\pi_{k-1}(p(x))=\pi_{k-1}(p(x'))$ and thus there exists a cube $\q\in \cu^k(\ns)$ such that $\q(v)=p(x)$ for all $v\not=1^k$ and $\q(1^k)=p(x')$. Let $\q':\db{k}\to\nss$ be the map with $\q'(v)=x$ for all $v\not=1^k$ and $\q'(1^k)=x'$. Since $\q'$ restricted to $\db{k}\setminus\{1^k\}$ is a corner and $p$ is a fibration with $p\co \q'|_{\db{k}\setminus\{1^k\}}=\q|_{\db{k}\setminus\{1^k\}}$, there exists $z\in \ab$ such that the map $\q'+1_{\{1^k\}}z$ is in $\cu^k(\nss)$. But note that $\ab^{\db{k}} = \cu^k(\mc{D}_k(\ab))$, so $1_{\{1^k\}}z\in \cu^k(\mc{D}_k(\ab))$, and therefore $\q'\in \cu^k(\nss)$. Hence $\pi_{k-1}(x)=\pi_{k-1}(x')$. By Theorem \ref{thm:open-mapping-general} the map $p_{k-1}$ is open, and therefore $p_{k-1}$ is an isomorphism of \textsc{lch} nilspaces. Hence every structure group of $\nss_{k-1}$ is a Lie group.

It therefore remains only to prove that the last structure group $\ab_k(\nss)$ is also a Lie group. Recall from Proposition \ref{prop:k-th-group-action} that $\ab_k(\nss)$ is isomorphic to any fixed fiber $F$ of $\pi_{k-1,\nss}$ where we fix some (any) element $e\in F$ to be the identity. Similarly, in $\ns$ we fix the fiber $F':=p(F)$ and fix $p(e)$ as the zero element of $\ab_k(\ns)$. The free action of $\ab$ on $F$ induces an inclusion of $\ab$ into $\ab_k(\nss)$. The homormorphism $\phi_k:\ab_k(\nss)\to\ab_k(\ns)$ can then be identified with $p|_F:F\to F'$. 

First we claim that the kernel of this homomorphism  $p|_F$ is precisely $e+\ab\subset F$. To see this, let $x\in F$ be any element such that $p(x)=p(e)$. By hypothesis on $p$, this means that if $\q\in F^{\db{k+1}}$ is the map $\q(v)=e$ for all $v\not=1^{k+1}$ and $\q(1^{k+1})=x$, there exists $d\in \ab^{\db{k+1}}$ such that $\q+d\in \cu^{k+1}(\nss)$. Note that we are not assuming that either $\q$ or $d$ are cubes for any nilspace, only that the sum $\q+d$ is. Then, by definition of the cubes in $\cu^{k+1}(\mc{D}_k(\ab))$ in terms of the single Gray-code linear equation, we can find a cube $d'\in \cu^{k+1}(\mc{D}_k(\ab))$ such that $d'(v)=-d(v)$ for all $v\not=1^{k+1}$. Hence $c+d+d'\in \cu^{k+1}(\nss)$. Note that all values of this cube are $e$ except perhaps the value at  $1^{k+1}$. By uniqueness of completion, this value must also be $e$. Hence $x+z=e$ for some $z\in \ab$,  which proves our claim.

Next, note that using the definition of addition in the fiber $F$ we have that the action of $\ab$ is also well defined and continuous. Furthermore, as mentioned above this free action induces an inclusion of $\ab$ in $\ab_k(\nss)$, and $e+\ab$ is closed because it is precisely $p^{-1}(\{p(e)\})$.

Thus we have a short exact sequence $0\to \ab \to \ab_k(\nss) \to \ab_k(\ns)\to 0$ where all groups are \textsc{lch} and $\ab,\ab_k(\ns)$ are Lie. By \cite[Theorem 2.6]{Mosk} $\ab_k(\nss)$ is Lie. Finally, by Theorem \ref{thm:open-mapping-thm} the map $p$ is open, so $\nss$ is a continuous extension of $\ns$.
\end{proof}

The previous result dealt only with degree-$k$ extensions of $k$-step nilspaces. We can use it to prove the following generalization.

\begin{corollary}\label{cor:ext-of-lcfr-is-lcfr}
Let $k,t\ge 0$ be integers. Let $\ns$, $\nss$ be \textsc{lch} $k$-step nilspaces such that $\nss$ is a degree-$t$ algebraic nilspace extension of $\ns$ by an abelian \textsc{lch} group $\ab$, with associated projection $p:\nss\to \ns$. If $p$ is continuous, $\ns$ is Lie-fibered, and $\ab$ is a Lie group, then $\nss$ is also Lie-fibered. In particular, $\nss$ is a continuous extension of $\ns$.
\end{corollary}

\begin{proof}
If $t\ge k$ then we can directly apply Proposition \ref{prop:ext-lie-is-lie-1} with $k=t$ (since $\ns,\nss$ are in particular $t$-step nilspaces).

For $t<k$, note that by \cite[Proposition A.18]{CGSS-p-hom} the nilspace $\nss$ is algebraically isomorphic to the fiber-product $\nss_t\times_{\ns_t}\ns$, with projection maps $p_t:\nss_t\to\ns_t$ and $\pi_t:\ns\to\ns_t$. Note that the algebraic isomorphism is the map $\varphi(y)=(\pi_t(y),p(x))$, so here this is also a continuous map. Note that $p_t$ is also continuous, thus implying that $p_t$ is continuous as well because $p_t\co \pi_{t,\nss} = \pi_{t,\ns} \co p$. As $p$ is continuous and both $\pi_{t,\nss}$ and $\pi_{t,\ns}$ are open and continuous we conclude that $p_t$ is continuous. By Proposition \ref{prop:ext-lie-is-lie-1} we get that $\nss_t$ is Lie-fibered as well, and $p_t$ is a continuous extension of $\ns_t$. By Lemma \ref{lem:fib-prod-top-nil} we know that $\nss_t\times_{\ns_t}\ns$ is an \textsc{lch} nilspace. Hence, in order to prove that it is Lie-fibered, it suffices to prove that its structure groups are Lie groups.  The structure groups of $\nss_t\times_{\ns_t}\ns$ are easy to compute. Using the formula given in Lemma \ref{lem:fib-prod-top-nil} we deduce that for $i\le t$, the $i$-th structure group of $\nss_t\times_{\ns_t}\ns$ is $\ab_i(\nss_t)$, which is Lie. For $i>t$ the $i$-th structure group of $\nss_t\times_{\ns_t}\ns$ is $\ab_i(\ns)$ (which is Lie by assumption). Hence $\nss_t\times_{\ns_t}\ns$ is Lie-fibered. Then, Theorem \ref{thm:open-mapping-general} implies that the isomorphism  $\varphi$ is also a homeomorphism. Hence $\nss$ is Lie-fibered. The last sentence in the corollary follows from Theorem \ref{thm:open-mapping-thm}.
\end{proof}
\noindent Later in this paper (e.g.\ in the proof of Theorem \ref{thm:splitext}) it will be important to have a useful criterion, given a $k$-step \textsc{lch} nilspace $\ns$, for the possibility to lift a continuous translation on $\ns_{k-1}$ to a translation on $\ns$. In the setting of compact nilspaces, a useful such criterion was given in \cite{CamSzeg} (see also \cite[Proposition 3.3.39]{Cand:Notes1}). We shall now prove an analogous criterion applicable to \textsc{lch} nilspaces. This will involve the notion of translation bundles from \cite[Definition 3.3.34]{Cand:Notes1}.

\begin{defn}
Let $\ns$ be a $k$-step nilspace and let $\alpha\in \tran_i(\ns_{k-1})$. We define the associated \emph{translation bundle} $\mc{T}=\mc{T}(\alpha,\ns,i):=\{(x_0,x_1)\in \ns^2: \alpha(\pi_{k-1}(x_0))=\pi_{k-1}(x_1)\}$. We equip this set with the subcubespace structure induced by the cubes on $\ns\bowtie_i \ns$ (see \cite[\S 3.1.4]{Cand:Notes1}).
\end{defn}

\noindent Recall from \cite[\S 3.3.4]{Cand:Notes1} that $\mc{T}$ equipped with the above-mentioned cubes is a $k$-step nilspace (in particular $\mc{T}$ is thus a degree-($k-i$) algebraic nilspace extension of $\ns$). Moreover, its $k$-th structure group is the $k$-th structure group $\ab_k$ of $\ns$, albeit with an action on $\mc{T}$ that is diagonal, namely the action $((x_0,x_1),z)\in \mc{T}\times \ab_k \mapsto (x_0+z,x_1+z)\in \mc{T}$.

When $\ns$ has the additional topological structure of being an \textsc{lch} nilspace, we shall always equip $\mc{T}$ with the subspace topology induced by the product topology on $\ns\times \ns$, and the factor $\mc{T}_{k-1}$ will be equipped with the quotient of the topology on $\mc{T}$ by the action of $\ab_k$. It is not hard to show that $\mc{T}$ thus becomes an \textsc{lch} abelian $\ab_k$-bundle over $\mc{T}_{k-1}$. This will be useful to prove the aforementioned criterion for lifting translations. Let us record this.
\begin{lemma}\label{lem:transbnd-ctsbnd}
Let $\ns$ be an \textsc{lch} nilspace, and let $\alpha\in \tran_i(\ns_{k-1})$. The (algebraic) nilspace $\mc{T}=\mc{T}(\alpha,\ns,i)$, equipped with the relative topology induced by the \textsc{lch} topology on $\ns\times \ns$, is a continuous abelian $\ab_k$-bundle over $\mc{T}_{k-1}$, equipping the latter with the quotient of the topology on $\mc{T}$ under the action of $\ab_k$. Moreover, this quotient topology is also \textsc{lch}.
\end{lemma}

\begin{proof}
That $\mc{T}$ is \emph{algebraically} a $\ab_k$-bundle over $\mc{T}_{k-1}$  follows from the fact, recalled above, that $\mc{T}$ is a $k$-step nilspace with $(k-1)$-step factor $\mc{T}_{k-1}$ (see e.g.\  \cite[Lemma 3.3.35]{Cand:Notes1}). In particular, it follows from \cite[Proposition 3.3.36]{Cand:Notes1} that the $k$-th structure group of $\mc{T}$ is $\ab_k(\ns)$ (acting diagonally as recalled above). We check that $\mc{T}$ is also an \textsc{lch} \emph{continuous} $\ab_k$-bundle.

First note that the topology on $\mc{T}$ is \textsc{lch} because the product topology on $\ns\times \ns$ is \textsc{lch} and $\mc{T}$ is a closed subset (being the preimage, under the continuous map $(\alpha\co \pi_{k-1,\ns}) \times \pi_{k-1,\ns}$, of the diagonal in $\ns_{k-1}\times\ns_{k-1}$, which is closed by the Hausdorff property of $\ns_{k-1}$).  

Now we check that the properties defining a continuous bundle (\cite[Definition 2.1.6]{Cand:Notes2}) hold.

Property $(i)$ is clear: $\mc{T}$ has the above-described \textsc{lch} topology and $\mc{T}_{k-1}$ is given the quotient topology under the $\ab_k$-action. Property $(ii)$ holds in that $\ab_k$ comes naturally with an \textsc{lch} topology, i.e.\ the one given from $\ns$. To see property $(iii)$, note that the continuity of the diagonal $\ab_k$-action follows from the continuity of the action of $\ab_k$ on $\ns$ (looking at convergent sequences in $\mc{T}\times \ab_k$, for example).
Property $(iv)$ holds since we equip $\mc{T}_{k-1}$ with the quotient topology.

Finally, the fact that the quotient topology on $\mc{T}_{k-1}$ is \textsc{lch} follows from Lemma \ref{lem:quo-by-cont-gr-act}, provided that the set $C=\{(x,x')\in\mc{T}\times\mc{T}: \exists\,z\in \ab_k\textrm{ such that }x=x'+z\}$ is a closed subset of $\mc{T}\times\mc{T}$. To see this closure, suppose that a sequence $(x_n,x'_n)=((x_{0,n},x_{1,n}),(x_{0,n}',x_{1,n}'))$ in $C$ converges in $\mc{T}\times \mc{T}$ to some $((x_0,x_1),(x_0',x_1'))\in (\ns\times\ns)^2$. Then $x_{0,n}\to x_0$, $x_{1,n}\to x_1$, $x_{0,n}'\to x_0'$, $x_{1,n}'\to x_1'$ in $\ns$ (using the product topology). By the definition of $C$, for each $n$ there is $z_n\in \ab_k$ such that $x_{0,n}'=x_{0,n}+z_n$ and $x_{1,n}'=x_{1,n}+z_n$ for all $n$. Then, since $\ns$ is an \textsc{lch} nilspace, by Lemma \ref{lem:closure} the graph of $\sim_{k-1}$ on $\ns$ is closed, so there exist $z_0,z_1\in \ab_k$ such that $x_0'=x_0+z_0$ and $x_1'=x_1+z_1$. By continuity of the difference map (Lemma \ref{lem:diff-cont}) we have $z_n=x_{0,n}'-x_{0,n}\to x_0'-x_0=z_0$, and similarly $z_n\to z_1$, so $z_0=z_1$ and we are done.
\end{proof}

We can now obtain the aforementioned criterion for lifting continuous  translations.

\begin{proposition}\label{prop:cont-lift-of-trans}
Let $\ns$ be a $k$-step \textsc{lch} nilspace, let $i<k$, and let $\alpha\in\tran_i(\ns_{k-1})$ be a (continuous) translation. Suppose that there is a continuous nilspace morphism $m:\ns_{k-1}\to \mc{T}_{k-1}$ such that $\pi\co m=\id_{\ns_{k-1}}$ (where $\pi$ is the projection $\mc{T}_{k-1}\to\ns_{k-1}$). Then there exists a continuous translation $\beta\in\tran_i(\ns)$ such that $\pi_{k-1}\co \beta(x) = \alpha\co \pi_{k-1}(x)$ for all $x\in \ns$.
\end{proposition}

\begin{proof}
The algebraic part of this result is given by \cite[Proposition 3.3.39]{Cand:Notes1}. To avoid confusion, let us denote the projection $\ns\to \ns_{k-1}$ by $\pi_{k-1,\ns}$, and the bundle  projection $\mc{T}\to \mc{T}_{k-1}$ by $\pi_{k-1,\mc{T}}$. By Lemma \ref{lem:transbnd-ctsbnd} we know that $\mc{T}$ is an \textsc{lch} continuous $\ab_k$-bundle over $\mc{T}_{k-1}$.

Recall that the (purely algebraic) translation $\beta$ on $\ns$ lifting $\alpha$, given by \cite[Proposition 3.3.39]{Cand:Notes1}, is defined by $\beta(x)=m(\pi_{k-1,\ns}(x))\, (x)$, where $m(\pi_{k-1,\ns}(x))$ is a \emph{local translation} from $\pi_{k-1,\ns}^{-1}\pi_{k-1,\ns}(x)$ to $\pi_{k-1,\ns}^{-1}\alpha (\pi_{k-1,\ns}(x))$ \cite[Definition 3.2.26]{Cand:Notes1} which is applied to $x$. 

We now prove the continuity of the map $\beta$. Let $(x_n)_n$ be a convergent sequence in $\ns$ with limit $x$. Fix any sequence $((x_0^{(n)},x_1^{(n)}))_n$ in $\mc{T}$ such that $m(\pi_{k-1,\ns}(x_n))=\pi_{k-1,\mc{T}}(x_0^{(n)},x_1^{(n)})$ for every $n$, and similarly fix $(x_0',x_1')\in \mc{T}$ such that $m(\pi_{k-1,\ns}(x))=\pi_{k-1,\mc{T}}(x_0',x_1')$. The continuity of $m$ then implies that $\pi_{k-1,\mc{T}}(x_0^{(n)},x_1^{(n)})\to \pi_{k-1,\mc{T}}(x_0',x_1')$ as $n\to\infty$. By Lemma \ref{lem:conv-quot}, there is a sequence $(z_n)_n$ in $\ab_k(\ns)$ such that $(x_0^{(n)}+z_n,x_1^{(n)}+z_n)\to (x_0',x_1')$ as $n\to\infty$. 

Now, for any pair of points $x_0,x_1\in \ns$ and any $x\sim_{k-1} x_0$, let $\q'_{(x_0,x_1),x}$ denote the $(k+1)$-corner on $\ns$ defined by $\q'_{(x_0,x_1),x}(v,0)=x_0$ for $v\neq 1^k$, $\q'_{(x_0,x_1),x}(1^k,0)=x$, and $\q'_{(x_0,x_1),x}(v,1)=x_1$ for $v\neq 1^k$. For $k=2$ the corner $\q'_{(x_0,x_1),x}$ is
$
\smallcorner{x_0}{x_0}{x_0}{x}{x_1}{x_1}{x_1}.\\
$
By definition of $\beta$ and of local translations, for each $x\in \ns$ we have $\beta(x)=\mc{K}_{k+1}(\q'_{(x_{0,x},x_{1,x}),x})$, where $(x_{0,x},x_{1,x})$ is any element of $\mc{T}$ with $\pi_{k-1,\mc{T}}(x_{0,x},x_{1,x})=m(\pi_{k-1,\ns}(x))$ (in particular $x_{0,x}\sim_{k-1} x$). 
Now the continuity of $\mc{K}_{k+1}$ implies the following convergence, confirming the continuity of $\beta$:\quad $\beta(x_n)=\mc{K}_{k+1}(\q'_{(x_0^{(n)}+z_n,x_1^{(n)}+z_n),x_n})\to \mc{K}_{k+1}(\q'_{(x_0',x_1'),x})=\beta(x)$.
\end{proof}
\noindent The continuous $\ab_k$-bundle structure of $\mc{T}$ was enough for the above proof (i.e.\ enough to be able to apply Lemma \ref{lem:conv-quot}). However, in later sections there will be situations (e.g.\ the proof of Theorem \ref{thm:splitext}) in which it will be useful to know that $\mc{T}$ is not only such a bundle, but in fact is fully an \textsc{lch} nilspace, in particular the corner-completion functions on $\mc{T}$ are continuous. For such situations, we establish the following final main result of this subsection.

\begin{proposition}\label{prop:transbundleLCH}
Let $\ns$ be a $k$-step \textsc{lch} nilspace, let $i<k$, and let $\alpha\in\tran_i(\ns_{k-1})$ be a (continuous) translation. Then the translation bundle $\mc{T}=\mc{T}(\alpha,\ns,i)$, endowed with the sub-cubespace structure induced by $\ns\bowtie_i \ns$ and the subspace topology induced by $\ns\times\ns$, is an \textsc{lch} $k$-step nilspace that is a degree-($k-i$) continuous extension of $\ns$.
\end{proposition}
\begin{remark}
On the translation bundle $\mc{T}$ we will have two different actions of $\ab_k$. If we let $(x_0,x_1)\in \mc{T}$ then we have on the one hand the diagonal action of $\ab_k$, given by $((x_0,x_1),z)\mapsto (x_0+z,x_1+z)$. And on the other hand, we have the addition on the second component $((x_0,x_1),z)\mapsto (x_0,x_1+z)$. Consider now the  commutative diagram
\[
\begin{tikzpicture}
  \matrix (m) [matrix of math nodes,row sep=1.5em,column sep=4em,minimum width=2em]
  {\mc{T} & \ns  \\
     \mc{T}_{k-1} & \ns_{k-1}.\\};
  \path[-stealth]
    (m-1-1) edge node [above] {$q$} (m-1-2)
    (m-1-1) edge node [right] {$\pi$} (m-2-1)
    (m-1-2) edge node [right] {$\pi$} (m-2-2)
    (m-2-1) edge node [above] {$q_{k-1}$} (m-2-2);
\end{tikzpicture}
\]

\vspace{-0.3cm}

\noindent The diagonal action makes $\mc{T}$ a degree-$k$ extension over $\mc{T}_{k-1}$ whereas the addition on the second component makes $\mc{T}$ a degree-$(k-i)$ extension over $\ns$.
\end{remark}

\noindent The main difficulty here is to prove the continuity of the corner completion maps. For this, we begin by identifying topological properties which, when added to a degree-$k$ algebraic nilspace extension of a $k$-step \textsc{lch} nilspace, imply that the extension is also an \textsc{lch} nilspace.
\begin{lemma}\label{lem:top-degree-k-k-ext}
Let $\ns$ be a $k$-step \textsc{lch} nilspace. Let $\nss$ be a degree-$k$ algebraic nilspace extension of $\ns$ by an abelian group $\ab$, with corresponding fibration $q:\nss\to\ns$. Suppose that $\nss,\ab$ are equipped with \textsc{lch} topologies making $\nss$ a Cartan continuous $\ab$-bundle over $\ns$. Suppose also that the following  hold:
\setlength{\leftmargini}{0.8cm}
\begin{enumerate}
    \item For every $n\geq 0$ the cube set $\cu^n(\nss)$ is a closed subset of $\nss^{\db{n}}$.
    \item The completion function $\mc{K}:\cor^{k+1}(\nss)\to \nss$ is continuous.
\end{enumerate}
Then $\nss$ is an \textsc{lch} nilspace.
\end{lemma}
\noindent Note that the assumption that $\ab$ has a continuous free action on $\nss$ ensures that the topology of $\ab$ as an \textsc{lch} topological group is the same as the subspace topology that $\ab$ inherits when identified with any fiber of $q$ in $\nss$.
\begin{proof}
Most of the conditions in Definition \ref{def:top-nil-open-maps-version} are already satisfied. Indeed, the only missing one is that for each $n\le k$ the map $p^{\db{n}}:\cu^n(\nss)\to\cor^n(\nss)$ is open (for $n >k$ this property already follows from $(ii)$).

By Remark \ref{rem:topquot}, the topology on $\ns$ is precisely the quotient of the topology on $\nss$ under the action of $\ab$. In particular, the map $q^{\db{n}}:\nss^{\db{n}}\to \ns^{\db{n}}$ is also a quotient open map. Note that for $n\leq k$ the set $\cu^n(\nss)\subset \nss^{\db{n}}$ is saturated with respect to the action of $\ab^{\db{n}}$ on $\nss^{\db{n}}$. Thus, the restriction $q^{\db{n}}|_{\cu^n(\nss)}:\cu^n(\nss)\to \cu^n(\ns)$ is also open. Now we want to prove that $p^{\db{n}}_{\nss}:\cu^n(\nss)\to \cor^n(\nss)$ is open for $n\le k$. It suffices to prove that for every open $U\subset \nss$ the set $p_{\nss}^{\db{n}}(E_U)$ is open where $E_U:=\left(U\times \prod_{v\in\db{n}\setminus\{1^n\}} \nss\right)\cap\cu^{n}(\nss)$. By definition of cubes in an extension we have $p_{\nss}^{\db{n}}(E_U) = (q^{\db{n}\setminus\{1^n\}})^{-1}(p^{\db{n}}_{\ns}(q^{\db{n}}(E_U)))$, where $p^{\db{n}}_{\ns}:\cu^n(\ns)\to \cor^n(\ns)$ is the usual projection on $\ns$. Since $\ns$ is an \textsc{lch} nilspace, $p^{\db{n}}_{\nss}$ is open. The result follows.
\end{proof}
\noindent It will be useful to extend the previous lemma, allowing $\nss$ to be a degree-$t$ extension of $\ns$ with $t \neq k$. For this, it turns out that we need an additional condition to ensure that the extension is still an \textsc{lch} nilspace. We call this condition \emph{fiber-completion continuity}. Let us describe it informally. Suppose that $\nss$ is a degree-$t$ extension of a $k$-step nilspace $\ns$. Let $\q'\in\cor^{t+1}(\nss)$ and $\q\in \cu^{t+1}(\ns)$ be such that $q\co\q'(v)=\q(v)$ for all $v\not=1^{t+1}$. Then by definition of degree-$t$ extensions \cite[Definition 3.3.13]{Cand:Notes1}, there exists a unique completion of $\q'$ whose image under $q$ is $\q$ (the uniqueness of this completion follows from uniqueness of  completion in $\cor^{t+1}(\mc{D}_t(\ab))$). The following condition ensures that this \emph{fiber-completion} is continuous.

\begin{defn}[Fiber-completion continuity]\label{def:fib-comp-cont}
Let $\nss$ be a nilspace endowed with an \textsc{lch} topology and let $\ns$ be a $k$-step \textsc{lch} nilspace. Suppose that $\nss$ is a degree-$t$ algebraic nilspace extension of $\ns$ by an abelian group $\ab$. Let $\Delta_{q:\nss\to\ns}:= \{(\q',\q)\in \cor^{t+1}(\nss)\times \cu^{t+1}(\ns):q\co \q'(v)=\q(v)\text{ for all } v\not=1^{t+1}\}$. Let $\mc{K}_{(q)}:\Delta_{q:\nss\to\ns}\to \nss$ be the function that maps each $(\q',\q)\in \Delta_{q:\nss\to\ns}$ to the unique element $y\in\nss$ which completes the corner $\q'$ and satisfies $q(y)=\q(1^{t+1})$. We say that the extension is \emph{fiber-completion continuous} if $\mc{K}_{(q)}$ is continuous (endowing $\cor^{t+1}(\nss)$ with the subspace topology induced by the product topology on $\nss^{\db{t+1}\setminus\{1^{t+1}\}}$).
\end{defn}
\noindent When there is no risk of confusion, we write $\Delta$ instead of $\Delta_{q:\nss\to\ns}$. Note that if $t=k$ and the completion on $\cor^{k+1}(\nss)$ is continuous, then the extension is  automatically fiber-continuous. We are now ready to prove our main result about extensions of \textsc{lch} nilspaces.

\begin{theorem}\label{thm:ext-lch-nil}
Let $\ns$ be a $k$-step \textsc{lch} nilspace. Let $\nss$ be a degree-$t$ algebraic nilspace extension of $\ns$ by an abelian group $\ab$, with corresponding fibration $q:\nss\to\ns$. Suppose that $\nss$, $\ab$ are equipped with \textsc{lch} topologies making $\nss$ a Cartan continuous $\ab$-bundle over $\ns$. Suppose also that the following hold:
\begin{enumerate}
    \item For every $n\geq 0$ the cube set $\cu^n(\nss)$ is a closed subset of $\nss^{\db{n}}$.
    \item The completion function $\mc{K}:\cor^{k+1}(\nss)\to \nss$ is continuous.
    \item The extension is fiber-completion continuous.
\end{enumerate}
Then $\nss$ is an \textsc{lch} nilspace.
\end{theorem}

\begin{proof}
We argue by induction on $k$ (with $t$ fixed). For $k=1$ the result follows by Lemma \ref{lem:top-degree-k-k-ext}. 

For $k>1$, we prove that $\nss$ is an \textsc{lch} nilspace by proving that the conditions of Lemma \ref{lem:eq-def-lch-nil} are satisfied. Note that by $(i)$ the cube sets are closed, so we just need to prove that the completion functions are continuous. For $n>k$ this follows from $(ii)$. It then suffices to prove that $\nss_{k-1}$ is a degree-$t$ extension of $\ns_{k-1}$ that satisfies all the assumptions in the theorem. Indeed, if we prove this, then by induction the completion functions on the factors of $\nss$ will also be continuous, as required in Lemma \ref{lem:eq-def-lch-nil}-$(ii)$. By Lemma \ref{lem:top-degree-k-k-ext} we can assume that $t<k$.

Recall that by \cite[Proposition A.18]{CGSS-p-hom} the following diagram commutes:
\begin{equation}\label{diag:factor-of-ext-1}
\begin{aligned}[c]
\begin{tikzpicture}
  \matrix (m) [matrix of math nodes,row sep=1.5em,column sep=4em,minimum width=2em]
  {\nss & \ns  \\
     \nss_{k-1} & \ns_{k-1}, \\};
  \path[-stealth]
    (m-1-1) edge node [above] {$q$} (m-1-2)
    (m-1-1) edge node [right] {$\pi$} (m-2-1)
    (m-1-2) edge node [right] {$\pi$} (m-2-2)
    (m-2-1) edge node [above] {$q_{k-1}$} (m-2-2);
\end{tikzpicture}
\end{aligned}
\end{equation}

\vspace{-0.3cm}

\noindent where $q_{k-1}$ defines a degree-$t$ extension of $\ns_{k-1}$. (We just write $\pi$ instead of $\pi_{k-1}$ to avoid overloading the notation; the space where $\pi$ is defined will be clear from the context.)

Let us then prove that $\nss_{k-1},\ns_{k-1}$ satisfy all the  assumptions in the theorem.

Firstly, $\ns_{k-1}$ with the quotient topology is an \textsc{lch} nilspace, by Theorem \ref{thm:factor-of-lch-nil} applied to $\ns$.

Let us now analyze some properties of objects and maps in \eqref{diag:factor-of-ext-1}. The map $\pi:\ns\to\ns_{k-1}$ is open and continuous by Corollary \ref{cor:quotient-top-k-1-factor}. By Corollary \ref{cor:quotient-top-k-1-factor} applied to $\nss$,  the topology on $\nss_{k-1}$ is also \textsc{lch}. By Proposition \ref{prop:k-th-group-action} the $k$-th structure group $\ab_k(\nss)$ is an \textsc{lch} topological group acting continuously on $\nss$. The group $\ab$ is an \textsc{lch} group by assumption. 

We claim that the action of $\ab$ on $\nss_{k-1}$ is continuous. This action is defined by the formula $\pi(y)+z:=\pi(y+z)$ (see \cite[Proposition A.18]{CGSS-p-hom}). To see that it is continuous, fix any sequences $(y_n)$ in $\nss$ and $(z_n)$ in $\ab$ such that $\pi(y_n)\to \pi(y)$ and $z_n\to z$. By Lemma \ref{lem:conv-quot} there exists a sequence $g_n\in \ab_k(\nss)$ such that $y_n+g_n\to y$. The continuity of the action of $\ab$ on $\nss$ implies  that $(y_n+g_n)+z_n\to y+z$. As $\pi$ is continuous, we know that $\pi((y_n+g_n)+z_n)\to \pi(y+z)=:\pi(y)+z$. We leave it as an exercise for the reader (using the uniqueness of completion on $\cor^{k+1}(\nss)$) to see that $(y_n+g_n)+z_n=(y_n+z_n)+g_n$.  Then $\pi((y_n+g_n)+z_n)=\pi(y_n+z_n)=\pi(y_n)+z_n$, and the claimed continuity follows. This shows that Definition \ref{def:CtsAbBund}-$(iii)$ holds.

Next, note that since $q:\nss\to\ns$, $\pi:\ns\to \ns_{k-1}$ and $\pi:\nss\to\nss_{k-1}$ are open and continuous, we have that $q_{k-1}$ is open and continuous as well, by the commutativity of the diagram \eqref{diag:factor-of-ext-1}. This shows that Definition \ref{def:CtsAbBund}-$(iv)$ holds.

So far, we have proved that $q_{k-1}:\nss_{k-1}\to\ns_{k-1}$ is a continuous $\ab$-bundle. We now prove that this bundle is also Cartan (i.e.\ the difference map is continuous). Suppose that $(\pi(y_n),\pi(y_n'))\to (\pi(y),\pi(y'))$ for some $(\pi(y_n),\pi(y_n'))$ and $ (\pi(y),\pi(y'))\in \nss_{k-1}\times_{\ns_{k-1}}\nss_{k-1}$. By Lemma \ref{lem:conv-quot} there are sequences $(g_n),(g_n')$ in $\ab_k(\nss)$ such that $y_n+g_n\to y$ and $y_n'+g_n'\to y'$ in $\nss$. Hence, without loss of generality $y_n\to y$ and $y'_n\to y'$ as $n\to\infty$. As $q_{k-1}(\pi(y_n))= q_{k-1}(\pi(y'_n))$, there exists $z_n\in \ab_k(\ns)$ such that $q(y_n)+z_n=q(y_n')$. As $\ns$ is an \textsc{lch} nilspace, and $q$ is continuous, $z_n\to q(y)-q(y')=z\in \ab_k(\ns)$ as $n\to\infty$. We leave it as an exercise for the reader to check that with the given assumptions (i.e.\ that $t<k$) the $k$-th structure homomorphism of $q$, $\phi_k:\ab_k(\nss)\to \ab_k(\ns)$, is an (algebraic) group isomorphism (this follows from \cite[Definition 3.3.13]{Cand:Notes1} using the definition of cubes in an extension). To prove that $\phi_k$ is continuous, let $z_n'\in \ab_k(\nss)$ such that $z_n'\to z'$ as $n\to \infty$. Fix any $y_0\in \nss$. Then by definition $\phi_k(z_n')=q(y_0+z_n')-q(y_0)$. Since $\nss$ is Cartan, this is the composition of two continuous maps. Thus the limit is $q(y_0+z')-q(y_0)=\phi_k(z')$ as required. 
Hence, we have that $y_n+\phi_k^{-1}(z_n)\to y+\phi_k^{-1}(z)$ and $y_n'\to y'$. As $(y_n+\phi_k^{-1}(z_n),y_n')\in \nss\times_{\ns}\nss$, by continuity of the difference map in $\nss$ we have $(y_n+\phi_k^{-1}(z_n))-y_n'\to (y+\phi_k^{-1}(z))-y'$.\footnote{In this equation the ``$+$'' signs denote the action of $\ab_k(\nss)$ and the ``$-$'' signs denote the $\ab$-valued difference.} This confirms  that the difference map $\nss_{k-1}\times_{\ns_{k-1}}\nss_{k-1}$ is continuous. 

It remains to check that $\nss_{k-1}$ satisfies properties $(i)$, $(ii)$, $(iii)$ in the theorem. Property $(i)$ follows from Lemma \ref{lem:closure-cube-in-k-1-factor} applied to $\nss$. Let us now prove $(ii)$, i.e.\ that completion on $\cor^k(\nss_{k-1})$ is continuous. Let $\pi\co \q_n'\to \pi\co \q'$ be a convergent sequence in $\cor^k(\nss_{k-1})$ (for some $\q_n',\q'\in \cor^k(\nss)$). We want to prove that their completions also converge. By Lemma \ref{lem:conv-quot} there exists $d_n\in \ab_k(\nss)^{\db{k}\setminus\{1^k\}}$ such that $\q_n'+d_n\to \q'$ in $\cor^k(\nss)$. By continuity of $q$, we have $q\co (\q_n'+d_n)\to q\co \q'$ in $\cor^k(\ns)$. Now we argue by a diagonalization, as follows. Let $\q\in \cu^k(\nss)$ be such that $p^{\db{k}}(q\co \q)=q\co \q'$ where $p^{\db{k}}:\cu^k(\ns)\to \cor^k(\ns)$. For every $m\in \mb{N}$, let $B_{1/m}(q\co \q)\subset \cu^k(\ns)$ be the open ball of radius $1/m$ and center $q\co \q$. As $p^{\db{k}}$ is an open map, the sets $p^{\db{k}}(B_{1/m}(q\co \q))$ are open subsets of $\cor^k(\ns)$ that contain $q\co \q'$. As $q\co (\q_n'+d_n)\to q\co \q'$, there exists $N_m\ge 0$ such that if $n\ge N_m$ then $q\co (\q_n'+d_n)\in p^{\db{k}}(B_{1/m}(q\co \q))$. Without loss of generality we can suppose that the sequence $N_m$ is increasing. Now let $\tau(n):=\sup\{m:n\ge N_m\}$. Then for each $n$ let $\q_n\in \cu^k(\nss)$ be any cube such that $p^{\db{k}}(q\co \q_n)=q\co(\q_n'+d_n)$ and $q\co \q_n\in B_{1/\tau(n)}(q\co\q)$. In particular $q\co \q_n\to q\co \q$ in $\cu^k(\ns)$ as $n\to\infty$ (this is why we chose $N_m$ to be increasing, so that $\tau(n)\to\infty$ as $n\to\infty$). This is the point where fiber-completion continuity becomes useful. Indeed, by property $(iii)$ for $q:\nss\to \ns$, we can complete continuously the pair $(\q_n'+d_n,q\co\q_n)\in \cor^k(\nss)\times \cu^k(\ns)$, that is, we can apply $\mc{K}_{(q)}$ to obtain a cube $\q_n+f_n$ for some $f_n\in \cu^k(\mc{D}_t(\ab))$ in such a way that $(\q_n+f_n)(v)=(\q_n'+d_n)(v)$ for all $v\not=1^k$. Note that for this step to work it is crucial that $t<k$, as we are precisely applying the fiber-completion function to faces of dimension $t+1$ inside $\db{k}$. The process is as follows: first, for every $n\ge 1$ choose an element $f_n\in\cu^k(\mc{D}_t(\ab))$ such that $(\q_n+f_n)(v) =(\q_n'+d_n)(v)$ for $v\in \db{k}_{\le t}$. But then, by uniqueness of fiber-completion $\q_n+f_n$ automatically equals $\q_n'+d_n$ for all $v\not=1^k$. Then, the continuity of the fiber-completion implies that $(\q_n+f_n)(1^k)$ converges to the unique completion of $(\q',q\co\q)$. Finally, letting $\mc{K}_k:\cor^k(\nss_{k-1})\to \nss_{k-1}$ be the unique completion function, note that $\mc{K}_k(\pi\co \q_n')=\mc{K}_k(\pi\co (\q_n'+d_n))=\mc{K}_k(\pi\co (\q_n+f_n)|_{\db{k}\setminus\{1^k\}}) = \pi((\q_n+f_n)(1^k))$, which converges by the continuity of the fiber-completion function. The limit of $(\q_n+f_n)(1^k)$ is precisely $\mc{K}_{(q)}(\q',q\co\q)$ (recall Definition \ref{def:fib-comp-cont}). Hence $\mc{K}_k(\pi\co\q_n')\to \pi(\mc{K}_{(q)}(\q',q\co\q))$. Since $\pi(\mc{K}_{(q)}(\q',q\co\q)) = \mc{K}_k(\pi\co\q')$, this confirms the desired continuity property $(ii)$. Finally, let us prove that $(iii)$ holds. Let $(\pi\co \q_n',\pi\co \q_n)\in \Delta_{q_{k-1}:\nss_{k-1}\to\ns_{k-1}}$ be a sequence converging to some limit $(\pi\co\q',\pi\co\q)\in \Delta_{q_{k-1}:\nss_{k-1}\to\ns_{k-1}}$. We want to prove that the completions of the corners $\pi\co\q_n'$ converge as well. By Lemma \ref{lem:conv-quot} there exists $d_n\in\ab_k(\nss)^{\db{t+1}\setminus\{1^{t+1}\}}=\cor^{t+1}(\mc{D}_k(\ab_k(\nss)))$ such that $\q_n'+d_n\to \q'$ in $\cor^{t+1}(\nss)$. Then, clearly $q\co (\q_n'+d_n)$ is a lift of $\pi\co\q_n|_{\db{t+1}\setminus\{1^{t+1}\}}$. Applying again Lemma \ref{lem:conv-quot}, this time to $\q_n(1^{t+1})$, we get a sequence of elements $z_n\in \ab_k(\ns)$ such that $\q_n(1^{t+1})+z_n\to \q(1^{t+1})$. Now for each $n$ let $\q^*_n: 
v\mapsto \begin{cases}
 q\co(\q_n'+d_n)(v) & \text{ if }v\not=1^{t+1}\\
 \q_n(1^{t+1})+z_n & \text{otherwise}
\end{cases}$. By \cite[Remark 3.2.12]{Cand:Notes1} (using that $t<k$) we have $\q^*_n\in \cu^{t+1}(\ns)$ for every $n$. Furthermore, these cubes converge to some limit $\q^*\in \cu^{t+1}(\ns)$ such that $\pi\co\q^*=\pi\co\q$. Indeed, for $v\not=1^{t+1}$ we have $\q^*_n(v)=q(\q_n'+d_n)(v)$ which converges because $(\q_n'+d_n)(v)$ converges and $q$ is continuous, and $\q^*_n(1^{t+1})=\q_n(1^{t+1})+z_n$, which converges by assumption (the resulting $\q^*$ is a cube again by \cite[Remark 3.2.12]{Cand:Notes1}). Thus, we can apply the fiber-completion continuity to the sequence $(\q_n'+d_n,\q_n^*)$. The limit of the completions is the completion of their limit, i.e.\ $(\q',\q^*)\in \Delta_{q:\nss\to\ns}$. Letting  $\mc{K}_{(q_{k-1})}$ be the completion function in $\Delta_{q_{k-1}:\nss_{k-1}\to\ns_{k-1}}$, note that  $\mc{K}_{(q_{k-1})}(\pi\co\q_n',\pi\co\q_n) = \pi(\mc{K}_{(q)}(\q_n'+d_n,\q_n^*))$, and the latter, by continuity of $\pi$ and $\mc{K}_{(q)}$, converges to $\pi(\mc{K}_{(q)}(\q',\q^*))$. Arguing as in the previous sentence, we have $\pi(\mc{K}_{(q)}(\q',\q^*)) = \mc{K}_{(q_{k-1})}(\pi\co\q',\pi\co\q^*) = \mc{K}_{(q_{k-1})}(\pi\co\q',\pi\co\q)$, and $(iii)$ follows.
\end{proof}
\begin{proof}[Proof of Proposition \ref{prop:transbundleLCH}]
The proof of the algebraic part of the statement follows similarly to \cite[\S 3.3.4]{Cand:Notes1}. Note that the action of $\ab_k$ on $\mc{T}$ is $((x_0,x_1),z)\mapsto (x_0,x_1+z)$. The projection $q:\mc{T}\to \ns$ is $(x_0,x_1)\mapsto x_0$. We  prove that the conditions in Theorem \ref{thm:ext-lch-nil} hold for $\ns$, $\nss=\mc{T}$.

By assumption $\ns$ is an \textsc{lch} nilspace. We first claim that $\mc{T}$ is a Cartan continuous $\ab_k$-bundle over $\ns$. We are endowing $\mc{T}$ with the subspace topology induced by the product topology on $\ns\times \ns$, in which the set $\mc{T}$ is easily seen to be closed. Hence (as noted previously, e.g.\ in the proof of Proposition \ref{prop:k-th-group-action}) the subspace topology on $\mc{T}$ is also \textsc{lch}. The group $\ab_k$ is equipped with a structure of \textsc{lch} topological group, being the $k$-th structure group of $\ns$. The continuity of the action of $\ab_k$ on $\mc{T}$ follows from that of the action of $\ab_k$ on $\ns$ (considering convergent sequences). Similarly the difference continuity (Cartan property) follows from that of $\ns$. The map $q:\mc{T}\to \ns$ is clearly continuous. It is also open, since if $(U\times V)\cap \mc{T}$ is an open set of $\mc{T}$, then $q((U\times V)\cap \mc{T})=U\cap \pi^{-1}(\alpha^{-1}(\pi(V)))$, which is open. This confirms our claim about $\mc{T}$.

Now, to see that property $(i)$ in Theorem \ref{thm:ext-lch-nil} holds, note that the map $\iota:\mc{T}^{\db{n}}\to \ns^{\db{n+i}}$, $(\q_0,\q_1)\mapsto \langle \q_0,\q_1\rangle_i$ is continuous (see \cite[\S 3.1.4]{Cand:Notes1} for the definition of arrows $\langle \q_0,\q_1\rangle_i$), and $\cu^n(\mc{T})=\iota^{-1}(\cu^{n+i}(\ns))$. Thus, those sets are all closed and $(i)$ follows. To see that $(ii)$ holds, note that for any corner $(\q_0',\q_1')\in \cor^{k+1}(\mc{T})$, its completion is just $(\mc{K}_{k+1}(\q_0'),\mc{K}_{k+1}(\q_1'))$, where $\mc{K}_{k+1}$ is the completion function on $\cor^{k+1}(\ns)$. Finally, to see that $(iii)$ holds, we prove the continuity of the map $\mc{K}_{(q)}$ on $\Delta_{q:\mc{T}\to \ns}$. For every $(\q_0'\times \q_1',\q)\in \Delta_{q:\mc{T}\to \ns}$ we have $\mc{K}_{(q)}(\q_0'\times \q_1',\q) = \mc{K}_{k+1}(\langle \q,\q_1'\rangle_i)$, where $\langle \q,\q_1'\rangle_i$ is defined as in \cite[\S 3.1.4]{Cand:Notes1}.\footnote{The map $\langle \q,\q_1'\rangle_i:\db{k+1}\setminus\{1^{k+1}\}\to \ns$ is defined as follows: for every vertex in $\db{k+1}\setminus\{1^{k+1}\}$ written as $(v,w)$ where $w\in \db{i}$, we have $\langle \q,\q_1'\rangle_i(v,w)=\q(v)$ if $w\not=1^i$ and $\langle \q,\q_1'\rangle_i(v,1^i)=\q_1'(v)$ for $v\not=1^{k+1-i}$.} (Note that here $\q|_{\db{k+1}\setminus\{1^{k+1}\}} = \q_0'$.) As  $(\q_0'\times \q_1',\q)\mapsto \langle \q,\q_1'\rangle_i$ and $\mc{K}_{k+1}$ are continuous, the result follows.
\end{proof}

\subsection{On translation groups of \textsc{lch} nilspaces}\label{sec:trans-group}\hfill\\
Given any Lie-fibered nilspace $\ns$, we want to endow the translation group $\tran(\ns)$ with a useful topology compatible with the group structure. This group is  included in the set $C(\ns,\ns)$ of continuous maps from $\ns$ to itself, so a natural choice is to equip $\tran(\ns)$ with the subspace topology induced by the well-known \emph{compact-open} topology on $C(\ns,\ns)$. 
Recall (e.g.\ from \cite[\S 46]{Mu}) that  for arbitrary topological spaces $X,Y$, the compact-open topology on $C(X,Y)$ is the smallest topology containing all sets of the form $\{f\in C(X,Y): f(K)\subset U\}$, for all choices of a compact set $K\subset X$ and an open set $U\subset Y$.
The many useful features of this topology include the fact that it reduces to the topology of uniform convergence on compact sets (or \emph{topology of compact convergence}) when $Y$ is a metric space \cite[Theorem 46.8]{Mu}. In what follows, when considering a set of continuous functions $C(X,Y)$, we shall always assume that it is equipped with the compact-open topology.

Our first aim now is to prove that $\tran(\ns)$ equipped with this topology becomes a Polish group. We begin with the following fact.

\begin{lemma}\label{lem:tran-i-closed-subset}
Let $\ns$ be a $k$-step \textsc{lch} nilspace. Then, for each $i\in[k]$, the set $\tran_i(\ns)$ is a closed subset of $C(\ns,\ns)$. In particular $\tran_i(\ns)$ is a Polish space.
\end{lemma}

\begin{proof}
By Lemma \ref{prop:cont-LCH-polish} we know that $C(\ns,\ns)$ is Polish. The closure of the cube sets $\cu^n(\ns)$ is readily seen to imply that $\tran_i(\ns)$ is a closed subset of $C(\ns,\ns)$ (considering convergent sequences of translations). Hence $\tran_i(\ns)$ is Polish \cite[(3.3) Proposition]{Kechris}.
\end{proof}

To confirm that this topology makes $\tran_i(\ns)$ a Polish group, we need to prove that the group operation and inversion are continuous relative to this topology. To this end, the assumption that $\ns$ is Lie-fibered will be useful, for the following reason.
\begin{proposition}
Let $\ns$ be a $k$-step Lie-fibered nilspace. Then $\ns$ is locally connected.
\end{proposition}

\begin{proof}
We argue by induction on $k$. Since a 0-step nilspace is a singleton, the case $k=0$ is trivial. For $k>0$, as the structure group $\ab_k$ is an abelian Lie group, it is locally connected (being a manifold), and then, from Theorem \ref{thm:local-triviality-bundle} and the inductive assumption that $\ns_{k-1}$ is locally-connected, it follows that every point of $\ns$ has a basis of connected neighborhoods.
\end{proof}

\begin{theorem}\label{thm:trans-group-polish}
Let $\ns$ be a $k$-step Lie-fibered nilspace. Then for every $i\in [k]$ the translation group $\tran_i(\ns)$ is a Polish group.
\end{theorem}

\begin{proof}
We have already proved that the topology induced on $\tran_i(\ns)$ by the compact-open topology is Polish, so it suffices to prove that the group operation and inversion are both continuous. For the group operation (composition of translations), this follows directly from \cite[\S 46, Ex. 7]{Mu}. For the inversion, first note that $\tran_i(\ns)$ is a subset of the set of homeomorphisms of $\ns$. This is because the inverse of a continuous translation is continuous by Corollary \ref{cor:inv-cont-lcfr}. Moreover, by \cite[Theorem 4]{Ar46} the set of homeomorphisms with the compact-open topology is a topological group under composition, provided that the space is \textsc{lch} and locally connected.\footnote{Since Lie-fibered nilspaces are manifolds, one could also use \cite{Ka}, where a uniqueness property of the compact-open topology is also proved.} Thus, the set $\tran_i(\ns)$ inherits the continuity of the inversion from that of the set of homeomorphisms.
\end{proof}
\begin{remark}
Without local connectedness, it can happen that in the space's homeomorphism group the inversion is not continuous relative to the compact-open topology \cite{Di}.
\end{remark}

Let us prove here an additional result about Lie-fibered nilspaces that will be useful later.

\begin{proposition}\label{prop:hom-ns-z-polish}
Let $\ns$ be a $k$-step Lie-fibered nilspace, let $\ab$ be an abelian Lie group and let $t\in \mb{Z}_{\ge 0}$. Then the set $\hom(\ns,\mc{D}_t(\ab))$ equipped with pointwise addition and the compact-open topology is a Polish group.
\end{proposition}

\begin{proof}
By Lemma \ref{prop:cont-LCH-polish} we have that $C(\ns,\ab)$ is a Polish space. We claim that $\hom(\ns,\mc{D}_t(\ab))$ is a closed subset of $C(\ns,\ab)$. To prove this, let $(f_n)$ be a sequence in $\hom(\ns,\mc{D}_t(\ab))$ converging to $f\in C(\ns,\ab)$. We need to prove that for every $\q\in \cu^m(\ns)$ we have $f\co \q\in \cu^m(\mc{D}_t(\ab))$. Since $f_n\co\q\in \cu^m(\mc{D}_t(\ab))$ and $\cu^m(\mc{D}_t(\ab))$ is closed, our claim follows.

Now we need to prove that the group operations are also continuous. We will just do it for addition, as the argument is almost identical for the inverse. Fix any sequence of elements $(f_n,g_n)\in \hom(\ns,\mc{D}_t(\ab))^2$ converging to some $(f,g)\in \hom(\ns,\mc{D}_t(\ab))^2$. We need to prove that $f_n+g_n\to f+g$. Suppose for a contradiction that this fails. Let $d$ be a  metric on $\ab$ compatible with the topology. By \cite[Theorems 46.8 and 46.2]{Mu}, there is then some compact set $K\subset \ns$ and some $\epsilon_0>0$ such that $\sup_{x\in K} d((f_n+g_n)(x),(f+g)(x))>\epsilon_0$ for infinitely many $n$. Passing to a subsequence if necessary, we may assume that the previous lower bound holds for all $n$. Passing to a further subsequence, we obtain a sequence $(x_n)$ in $K$ such that $x_n\to x$ and $ d((f_n+g_n)(x_n),(f+g)(x_n))>\epsilon_0$ for all $n$. This yields a contradiction, as by \cite[Theorem 46.10]{Mu} the evaluation map is continuous and thus $f_n(x_n)+g_n(x_n)\to f(x)+g(x)$.
\end{proof}

\section{Free nilspaces}\label{sec:freenilspaces}
\noindent This section treats the following class of \textsc{lch} nilspaces, which plays a key role in this paper.

\begin{defn}\label{def:free-nil}
A \emph{free nilspace} is a direct product (in the nilspace category) of finitely many components of the form $\mc{D}_i(\mb{Z})$ and $\mc{D}_i(\mb{R})$ where $i\in \mb{N}$. If all these components are of the form $\mc{D}_i(\mb{Z})$ (resp.\ $\mc{D}_i(\mb{R})$) then the free nilspace is said to be \emph{discrete} (resp.\ \emph{continuous}).
\end{defn}
\noindent Thus, if $F$ is a free nilspace of step $k$ then we can write it in the form $F=\prod_{i=1}^k \mc{D}_i(\mb{Z}^{a_i}\times \mb{R}^{b_i})$ for some integers $a_i,b_i\ge 0$. We shall often use the following alternative expression as well.

\begin{defn}
Let $F$ be a $k$-step free nilspace $\prod_{i=1}^k \mc{D}_i(\mb{Z}^{a_i}\times \mb{R}^{b_i})$, where  $a_i,b_i\in\mb{Z}_{\ge 0}$ for each $i\in [k]$. Thus an element $g$ of $F$ is a sequence of pairs $\left((x_i,y_i)\in \mc{D}_i(\mb{Z}^{a_i}\times \mb{R}^{b_i})\right)_{i\in[k]}$ where $x_i\in \mc{D}_i(\mb{Z}^{a_i})$, $y_i\in \mc{D}_i(\mb{R}^{b_i})$. By a permutation of coordinates (which will often be used tacitly), we can write any $g\in F$ as a pair $g=(x,y)$ where $x:=(x_i=(x_{i,j})_{j\in [a_i]})_{i\in[k]}$ lies in the product nilspace $\prod_{i=1}^k \mc{D}_i(\mb{Z}^{a_i})$, nilspace which we call the \emph{discrete part} of $F$, and $y:=(y_i=(y_{i,j})_{j\in [b_i]})_{i\in[k]}$ lies in the product nilspace $\prod_{i=1}^k \mc{D}_i(\mb{R}^{b_i})$, the \emph{continuous part} of $F$. We shall often refer to $x,y$ as the \emph{discrete and continuous parts} of $g$ respectively.
\end{defn}

\noindent From a purely algebraic viewpoint, free nilspaces are a specific type of \emph{group nilspaces}, i.e.\ nilspaces consisting of filtered groups where the cube structure is given by the Host--Kra cube groups associated with the filtration, see \cite[Ch. 6]{HKbook}. More precisely, for a free $k$-step nilspace the underlying group is an \emph{abelian} group of the form $\mb{Z}^m\times \mb{R}^n$, and the filtration is the product of filtrations associated with the nilspaces $\mc{D}_i(\mb{Z}),\mc{D}_j(\mb{R})$, $i,j\in [k]$. It follows that morphisms between free nilspaces are polynomial maps between the corresponding filtered groups (recall e.g.\ \cite[Theorem 2.2.14]{Cand:Notes1}), and can be expressed in terms of multivariate polynomials. We shall now give precise formulations of this, in the form of certain Taylor expansions for such morphisms. We will then use these expressions to prove results on lifting morphisms from free nilspaces, which will play an important role in the next section.

\subsection{Morphisms between free nilspaces as polynomials}\hfill\\
To formulate the Taylor expansions, we use the following notation. 

\begin{defn}\label{def:height-lattice}
Let $k\in \mb{N}$, and for each $i\in [k]$ let $a_i,b_i\in\mb{Z}_{\ge 0}$. Then, for any element $(m,n)= \big(((m_{i,j})_{j\in [a_i]})_{i\in[k]},((n_{i,\ell})_{\ell\in [b_i]})_{i\in[k]}\big) \in (\prod_{i=1}^k \mb{Z}_{\ge 0}^{a_i})\times (\prod_{i=1}^k \mb{Z}_{\ge 0}^{b_i})$, we define \[
|(m,n)|_{k,(a_i)_i,(b_i)_i}:=\sum_{i=1}^k i\Big(\sum_{j=1}^{a_i} m_{i,j}+\sum_{\ell=1}^{b_i} n_{i,\ell}\Big).
\]
If there is no risk of confusion, we abbreviate this to $|(m,n)|$. Given variables $x_{i,j},y_{i,\ell}$ for $i\in[k], j\in [a_i],\ell\in[b_i]$, we write $\binom{(x,y)}{(m,n)}$ for the monomial\footnote{Here $\binom{x}{n}$ denotes the standard binomial coefficient. It turns out to be technically convenient here to consider monomials as products of binomial coefficients rather than the more usual products of variables.} $\prod_{i=1}^k\prod_{j=1}^{a_i} \binom{x_{i,j}}{m_{i,j}}\prod_{\ell=1}^{b_i} \binom{y_{i,\ell}}{n_{i,\ell}}$.
\end{defn}

\begin{remark}\label{rem:monomial-degree}
Thus, when $(x,y)$ is a point in a free nilspace $\prod_{i=1}^k\mc{D}_i(\mb{Z}^{a_i})\times\prod_{i=1}^k\mc{D}_i(\mb{R}^{b_i})$, the number $|(m,n)|$ yields a notion of \emph{degree} of the monomial $\binom{(x,y)}{(m,n)}$, which takes into account the powers to which the variables are raised in the monomial, but also the degree $i$ of the nilspace structure $\mc{D}_i(\mb{Z})$ or $\mc{D}_i(\mb{R})$ to which $x_{i,j}$ or $y_{i,\ell}$ belongs. Thus we shall refer to $|(m,n)|$ as the \emph{filtered degree} of the monomial $\binom{(x,y)}{(m,n)}$.
\end{remark}
\noindent We now state the Taylor expansions in two separate results. The first one addresses the discrete case (in which the image of the morphism on $F$ is discrete), as follows.
\begin{lemma}\label{lem:Taylor-discrete}
Let $k\in \mb{N}$. For each $i\in [k]$ let $a_i,b_i\in\mb{Z}_{\ge 0}$, and let $F$ be the free nilspace $\big[\prod_{i=1}^k \mc{D}_i(\mb{Z}^{a_i})\big] \times \big[\prod_{i=1}^k \mc{D}_i( \mb{R}^{b_i})\big]$. Let $A$ be a discrete finitely generated abelian group, and let $\phi:F\to \mc{D}_k(A)$ be a continuous nilspace morphism. Then for each $m\in \prod_{i=1}^k \mb{Z}_{\ge 0}^{a_i}$ with $|(m,0)|_{k,(a_i),(0)}\le k$ there exists a coefficient $g_m\in A$, such that for every $x\in \prod_{i=1}^k \mc{D}_i(\mb{Z}^{a_i})$ and $y\in \prod_{i=1}^k \mc{D}_i( \mb{R}^{b_i})$ we have
\begin{equation}\label{eq:Taylor-discrete}
\phi(x,y)=\sum_{m:|(m,0)|_{k,(a_i),(0)}\le k} g_m \binom{(x,y)}{(m,0)}.
\end{equation}
Conversely, every map $\phi$ of the form \eqref{eq:Taylor-discrete} is a continuous morphism $F\to\mc{D}_k(A)$.
\end{lemma}
\noindent We are thus expressing $\phi$ as a linear combination of monomials $(x,y)^{(m,0)}$ which \emph{involve only the discrete variables} $x_{i,j}$. The proof of Lemma \ref{lem:Taylor-discrete} is given in Appendix \ref{app:Taylor} and uses mainly standard arguments from polynomial algebra (see Theorem \ref{thm:Taylor-discrete-app} and Remark \ref{rem:Taylor-discrete-app}).

 We now state the result addressing the continuous case.

\begin{lemma}\label{lem:poly-free-to-R}
Let $k\in \mb{N}$. For each $i\in [k]$ let $a_i,b_i\in\mb{Z}_{\ge 0}$, and let $F$ be the free nilspace $\big[\prod_{i=1}^k \mc{D}_i(\mb{Z}^{a_i})\big] \times \big[\prod_{i=1}^k \mc{D}_i( \mb{R}^{b_i})\big]$. Then for any continuous morphism $\phi:F\to \mc{D}_k(\mb{R})$, for each $(m,n)\in (\prod_{i=1}^k \mb{Z}_{\ge 0}^{a_i})\times (\prod_{i=1}^k \mb{Z}_{\ge 0}^{b_i})$ with $|(m,n)|\leq k$ there is a coefficient $\lambda_{m,n}\in \mb{R}$ such that for every $x\in \prod_{i=1}^k \mc{D}_i(\mb{Z}^{a_i})$ and $y\in \prod_{i=1}^k \mc{D}_i( \mb{R}^{b_i})$ we have
\begin{equation}\label{eq:poly-free-to-R}
\phi(x,y)=\sum_{(m,n):\,|(m,n)|\le k} \lambda_{m,n} \binom{(x,y)}{(m,n)},
\end{equation}
Conversely, every map $\phi$ of the form \eqref{eq:poly-free-to-R} is a continuous morphism $F\to\mc{D}_k(\mb{R})$.
\end{lemma}
\noindent The proof is given in Appendix \ref{app:Taylor}; a simple modification of it yields the following expression of continuous morphisms from $F$ to the circle group $\mb{T}$, which is used in the next subsection.
\begin{lemma}\label{lem:Taylor-cont}
Let $F=\big[\prod_{i=1}^k \mc{D}_i(\mb{Z}^{a_i})\big] \times \big[\prod_{i=1}^k \mc{D}_i( \mb{R}^{b_i})\big]$ be a free nilspace and let $\pi:\mb{R}\to\mb{T}$ be the canonical quotient homomorphism. A map $\psi:F\to \mc{D}_k(\mb{T})$ is a continuous morphism if and only if there is a morphism $\phi:F\to \mc{D}_k(\mb{R})$ of the form \eqref{eq:poly-free-to-R} such that
\begin{equation}\label{eq:Taylor-cont}
\psi(x,y)=\pi\co\phi(x,y).
\end{equation}
\end{lemma}

\begin{remark}\label{rem:poly-degree}
As mentioned in Remark \ref{rem:monomial-degree}, the number $|(m,n)|$ gives a notion of filtered degree of a monomial $\binom{(x,y)}{(m,n)}$. Lemmas \ref{lem:Taylor-discrete} and \ref{lem:poly-free-to-R} together tell us that a morphism $\phi$ from a free nilspace $F$ to $\mb{R}$ or $\mb{Z}$ is a multivariate polynomial, more precisely a linear combination of monomials of the form $\binom{(x,y)}{(m,n)}$. Therefore we can now define the \emph{filtered degree} of such a polynomial to be the maximum of $|(m,n)|$ over all the monomials forming the polynomial. We can then summarize the results so far in this section as follows.
\end{remark}

\begin{lemma}\label{lem:mor-free-nil-as-poly} 
Let $F$, $F'$ be free nilspaces. A function $P:F\to F'$ is a continuous morphism if and only if the coordinate functions $F\to \mc{D}_i(G)$ of $P$ (where $G$ is $\mb{Z}$ or $\mb{R}$ and $i\in [k]$) are all multivariate polynomials of the form \eqref{eq:Taylor-discrete} or \eqref{eq:poly-free-to-R}, all of filtered degree at most $i$.
\end{lemma}

\begin{proof}
By composing $P$ with projections to the different coordinates, the lemma follows from Lemmas \ref{lem:Taylor-discrete} and  \ref{lem:poly-free-to-R}.
\end{proof}

\begin{remark}
Despite the various existing treatments of polynomial maps between filtered groups (e.g.\ \cite[Appendix B]{GTZ}, \cite[\S 2.2]{Cand:Notes1} and many more), we did not find a quick deduction of the Taylor expansions in this subsection from the existing literature. Nevertheless, we chose to leave the proofs for Appendix \ref{app:Taylor} as the arguments are basically standard. 
\end{remark}

\subsection{Lifting a morphism defined on a free nilspace}\label{subsec:lift-inv-free-nil}\hfill\\
Let $F$ be a free nilspace, let $A$ be a compact abelian Lie group, and suppose that for some $k\in \mb{N}$ we have a nilspace morphism $\phi:F\to\mc{D}_k(A)$. We know that $A\cong Z\times \mb{T}^s$ for some finite abelian group $Z$ and some integer $s\geq 0$. Let $r$ be the minimal cardinality of a set of generators for $Z$, so that $Z$ is isomorphic to a direct sum of $r$ finite cyclic groups (we call $r$  the \emph{rank} of $Z$). We then consider the abelian Lie group $B=\mb{Z}^r\times \mb{R}^s$ and note that there is a natural continuous surjective homomorphism $\pi:B\to A$ (which simply applies a quotient homomorphism in each coordinate, either from $\mb{Z}$ to a finite cyclic group, or from $\mb{R}$ to $\mb{T}$). We call $B$ the \emph{covering group} of $A$, and $\pi$ the \emph{covering homomorphism}. The question treated in this subsection is whether we can lift $\phi$ through $\pi$ to a morphism $\psi:F\to \mc{D}_k(B)$, that is, whether there is a nilspace morphism $\psi:F\to\mc{D}_k(B)$ such that $\pi\co\psi=\phi$.

Using the Taylor expansions from the previous subsection, this question can be settled quickly in the affirmative, obtaining the following main result of this subsection.

\begin{theorem}\label{thm:lift-mor-cfr-ab-gr}
Let $A=Z\times \mb{T}^s$ be a compact\footnote{The assumption of compactness can be relaxed to that of being compactly-generated, using similar arguments.} abelian Lie group, let $B=\mb{Z}^r\times \mb{R}^s$ be the covering group and $\pi:B\to A$ the covering homomorphism. Then for any free nilspace $F$, any $k\in \mb{N}$ and any $\phi\in \hom(F,\mc{D}_k(A))$, there exists $\psi\in\hom(F,\mc{D}_k(B))$ such that $\phi = \pi\co \psi$. 
\end{theorem}
\noindent There are more general versions of this result (see for instance Remark \ref{rem:moregenlift}). 
\begin{proof}
First we address the discrete part, in which $A$ is just a finite abelian group. In this case, letting $r$ be the rank of $A$, the conclusion is the following:
\begin{equation}\label{prop:lift-disc}
 \textrm{there exists a morphism $\psi:F\to \mc{D}_k(\mb{Z}^r)$ such that $\phi=\pi\co\psi$.}   
\end{equation}
Indeed, it suffices to lift the expansion \eqref{eq:Taylor-discrete} for $\phi$ to a Taylor expansion of a morphism $\psi:F\to \mc{D}_k(B=\mb{Z}^r)$ by replacing each coefficient $g_m\in A$ by some $g_m'\in B$ such that $\pi(g_m')=g_m$.

Next we treat the continuous part of the problem, where the image group $A$ is a torus $\mb{T}^s$. Letting $\pi:\mb{R}\to\mb{T}$ be the canonical quotient homomorphism, the main claim here is that
\begin{equation}\label{prop:lift-cont}
\textrm{there exists a morphism $\psi:F\to \mc{D}_k(\mb{R}^s)$ such that $\phi=\pi^s\co\psi$.}
\end{equation}
This follows readily by applying Lemma \ref{lem:Taylor-cont} to each of the $s$ coordinate maps of $\phi$ and taking the product of the resulting maps.

The theorem now follows. Indeed,  we can first lift separately the morphisms $\pi_d\co\phi$ and $\pi_c\co\phi$, using \eqref{prop:lift-disc} and \eqref{prop:lift-cont} respectively, where $\pi_d$ is the coordinate projection from $A$ to its discrete (finite) part $Z$ and $\pi_c$ is the coordinate projection from $A$ to its continuous part $\mb{T}^s$. Then we put together the two lifts to obtain $\psi$ (taking the product of the two maps).
\end{proof}

\begin{remark}\label{rem:moregenlift}
Theorem \ref{thm:lift-mor-cfr-ab-gr} can be extended to more general situations. In fact, if $\gamma:B\to A$ is a surjective homomorphism where $A,B$ are abelian Lie groups it can be proved that for any $f\in \hom(F,\mc{D}_k(A))$ there exists $g\in \hom(F,\mc{D}_k(B))$ such that $\gamma\co g = f$. As we will not use this result, we omit its proof.
\end{remark}

\subsection{On the translation groups and structure groups of free nilspaces}\hfill\\
The main goal of this final subsection on free nilspaces is to obtain the following lifting result for translations, which will be used in the proof of Theorem \ref{thm:gpcongrep}.

\begin{theorem}\label{thm:lift-trans}
Let $F=\prod_{i=1}^k \mc{D}_i(\mb{Z}^{a_i}\times\mb{R}^{b_i})$ be a free nilspace \textup{(}$a_i,b_i\in\mb{Z}_{\ge 0}$, for each $i\in [k]$\textup{)}. Let $A=Z\times \mb{T}^s$ be a compact abelian Lie group, let $B=\mb{Z}^r\times \mb{R}^s$ be the covering group for $A$, and let $\pi:B\to A$ be the covering homomorphism. Then for every $i\in [k]$ and $\alpha\in \tran_i(F\times \mc{D}_k(A))$ there exists $\beta\in \tran_i(F\times \mc{D}_k(B))$ such that, letting $\phi:F\times\mc{D}_k(B)\to F\times \mc{D}_k(A)$ be the morphism $(f,b)\mapsto (f,\pi(b))$, we have $\alpha\co \phi= \phi\co \beta$.
\end{theorem}
\noindent To prove this we shall use the following description of the translation groups of a free nilspace, which extends \cite[Theorem 5.9]{CGSS-p-hom} and seems to be of independent interest. In particular, this description will also help to prove that these translation groups are Lie groups.

\begin{remark}\label{rem:D0}
In the following result, we use the basic degree-$t$ nilspace structure on an abelian group $\ab$, denoted by $\mc{D}_t(\ab)$; see \cite[equation $(2.9)$]{Cand:Notes1}. However, we shall use this even for $t\leq 0$. Note that for $t=0$ the resulting nilspace is non-ergodic, and the $n$-cubes on this nilspace are the constant maps $\db{n}\to\ab$.  It follows that for any (ergodic) nilspace $\ns$, the set of morphisms $\hom(\ns,\mc{D}_0(\ab))$ is the set of constant maps $\ns \to \ab$.  For $t<0$ we take $\mc{D}_t(\ab)$ to be $\{0_{\ab}\}$.
\end{remark}

\begin{theorem}\label{thm:decrip-trans-group}
Let $(A_i)_{i=1}^k$ be a sequence of abelian groups, and let $\ns$ be the product nilspace $\prod_{i=1}^k\mc{D}_i(A_i)$. Then for each $s\in [k]$ we have
\begin{equation}\label{eq:str-trans-ab-nil}
\tran_s(\ns) = \prod_{i=1}^k \hom\left(\prod_{j=1}^{i-s} \mc{D}_j(A_j),\mc{D}_{i-s}(A_i)\right),
\end{equation}
where the action of an element $(0,\ldots,0,T_s,T_{s+1},\ldots,T_k)$ in this product, as a translation $\alpha\in \tran_s(\ns)$, is defined by $\alpha(x)=(x_1,\ldots,x_k)+(0,\ldots,0,T_s,T_{s+1}(x_1),\ldots,T_k(x_1,\ldots,x_{k-s}))$.
\end{theorem}

\begin{remark}
In \eqref{eq:str-trans-ab-nil} the product signs outside and inside the bracket indicate a Cartesian product and a nilspace product respectively. The factor for $i=s$ is the set $\hom(\{\emptyset\},\mc{D}_0(A_s))$ (taking the empty product $\prod_{j=1}^{0} \mc{D}_j(A_j)$ to be a singleton, as usual), so this is the set of constants $T_s\in A_s$. For each $i<s$ the $i$-th factor can be viewed as consisting only of $0_{A_i}$.
\end{remark}

\begin{proof}
First we prove that any such function $\alpha$ is in $\tran_s(\ns)$. We argue by induction on $k$. For $k=1$, since  $\alpha$ then just adds a constant, it is  indeed a translation on $\mc{D}_1(A_1)$. For $k\geq 2$ and $s\in [k]$, by induction the map $(x_1,\ldots,x_{k-1})+(0,\ldots,0,T_s,T_{s+1}(x_1),\ldots,T_{k-1}(x_1,\ldots,x_{k-1-s}))$ is in $\tran_s(\ns_{k-1})$. Hence it suffices to prove that the map $\alpha': x=(x_1,,\ldots,x_k)\mapsto x+\big(0,\ldots,0,T_k(x_1,\ldots,x_{k-s})\big)$ is  in $\tran_s(\ns)$.
By \cite[Lemma 3.2.32]{Cand:Notes1} it suffices to prove that for every $\q=(\q_1,\ldots,\q_k)\in \cu^n\big(\prod_{i=1}^{k}\mc{D}_i(A_i)\big)$ the $s$-arrow $\langle \q,\alpha'\co\q\rangle_s$ is in $\cu^{n+s}(\ns)$. Now $\langle \q,\alpha'\co\q\rangle_s=\langle \q,\q\rangle_s + g$, for $g:\db{n+s}\to \prod_{i=1}^k A_i$, $v\mapsto (0,\ldots,0,\langle 0,T_k\co (\q_1,\ldots,\q_{k-s})\rangle_s(v))$. Thus it suffices to prove that $\langle 0,T_k\co (\q_1,\ldots,\q_{k-s})\rangle_s\in \cu^n(\mc{D}_{k-s}(A_k))$. By \cite[Lemma 2.2.19]{Cand:Notes1}, this follows if $T_k\co (\q_1,\ldots,\q_{k-s})\in \cu^n(\mc{D}_{k-s}(A_k))$, which holds by our assumption on $T_k$. 

Now we prove that every translation $\alpha\in \tran_s(\ns)$ has the form claimed in the theorem. Note that by the previous paragraph, adding a constant $x_i$ in the $i$-th component of $\ns$ is a translation in $\tran_i(\ns)$. Moreover, for $i\geq k-s+1$ such translations are central in $\tran_s(\ns)$. Hence for every $x\in \ns$ we have $\alpha(x)=\alpha(x_1,\ldots,x_{k-s},0^s)+(0^{k-s},x_{k-s+1},\ldots,x_k)$. By induction, the first summand here equals $(x_1,\ldots,x_{k-s},0^s)+(0^{k-s},T_s,T_{s+1}(x_1),\ldots,T_{k-1}(x_1,\ldots,x_{k-1-s}),g(x))$ for some map $g:\ns\to A_k$ which factors through $\pi_{k-s}$, i.e.\ $g(x)=g'(x_1,\ldots,x_{k-s})$ for some map $g':\ns_{k-s}\to A_k$. It now suffices to prove that $g'\in \hom(\prod_{i=1}^{k-s}\mc{D}_i(A_i),\mc{D}_{k-s}(A_k))$. Fix any $\q\in \cu^n(\prod_{i=1}^{k-s}\mc{D}_i(A_i))$. Let $\q^*\in \cu^n(\ns)$ be defined by $\q^*(v)=(\q(v),0^s)$ for $v\in\db{n}$, and consider the map $\langle\q^*,\alpha\co\q^*\rangle_s$. On one hand, by \cite[Lemma 3.2.32]{Cand:Notes1} this map is a cube. On the other hand, by the above expression of $\alpha$, this map equals $\langle\q^*,\q^*\rangle_s+\langle0,\q'\rangle_s+\langle 0,\q''\rangle_s$ for some $\q'=(\q'_1,\ldots,\q'_{k-s},0^s)\in \cu^n(\ns)$, and with $\q''(v)=(0^{k-s},g'\co \q(v))$. Then $\langle 0,\q''\rangle_s\in \cu^{n+s}(\ns)$, since it is the combination of cubes $\langle\q^*,\alpha\co\q^*\rangle_s-\langle\q^*, \q^*\rangle_s - \langle0,\q'\rangle_s$. Hence $\langle 0,g'\co\q\rangle_s\in \cu^{n+s}(\mc{D}_k(A_k))$, so $g'\co\q\in \cu^n(\mc{D}_{k-s}(A_k))$ by \cite[Lemma 2.2.19]{Cand:Notes1}, as required.
\end{proof}

\begin{proof}[Proof of Theorem \ref{thm:lift-trans}]
Given any translation $\alpha\in \tran_s(F\times \mc{D}_k(A))$, by Theorem \ref{thm:decrip-trans-group}, we can lift $\alpha$ in the desired way just by lifting the polynomial $T_k\in \hom(\prod_{j=1}^{k-s}\mc{D}_j(\mb{Z}^{a_i}\times \mb{R}^{b_i}),\mc{D}_{k-s}(A))$. By Theorem \ref{thm:lift-mor-cfr-ab-gr} there exists $T_k'\in \hom(\prod_{j=1}^{k-s}\mc{D}_j(\mb{Z}^{a_i}\times \mb{R}^{b_i}),\mc{D}_{k-s}(B))$ such that $\pi\co T_k'=T_k$. Using $T_k'$ instead of $T_k$ in the expression \eqref{eq:str-trans-ab-nil} for $\alpha$, we obtain a translation $\beta$ as claimed.
\end{proof}

\noindent To close this section we shall prove that the translation group and structure groups of a free nilspace are always Lie groups. This can be established using an alternative description of the translation group of free nilspaces given by the next theorem (which actually concerns more general nilspaces consisting in products of higher-degree abelian Lie groups).

Given a $k$-step nilspace $\ns$, recall from \cite[Definition 3.3.1 (ii)]{Cand:Notes1} that for every $j\in [k]$ we have the following group homomorphism (e.g.\ by \cite[Lemma 1.5]{CGSS} with $\psi=\pi_{k-1}$): 
\begin{equation}\label{eq:etaj}
\eta_j:\tran(\ns)\to \tran(\ns_j),\quad \alpha\mapsto \alpha_j,\textrm{ where } \alpha_j\co \pi_j = \pi_j\co \alpha,
\end{equation}

\begin{theorem}\label{thm:freetransLie}
Let $(A_i)_{i=1}^k$ be a sequence of (compactly generated) abelian Lie groups, and let $\ns$ be the $k$-step Lie-fibered nilspace $\prod_{i=1}^k \mc{D}_i(A_i)$. Then $\tran(\ns)$ is a topological group extension of  $\tran(\ns_{k-1})$ by $\ker(\eta_{k-1})$.
\end{theorem}
\noindent When $\ns$ is a Lie-fibered nilspace, the map $\eta_{k-1}$ is seen to be a continuous homomorphism similarly as in \cite[Lemma 2.9.3]{Cand:Notes2} (where $\eta_{k-1}$ was denoted simply by $h$). Let us record this.

\begin{lemma}\label{lem:h-is-cont}
Let $\ns$ be a Lie-fibered $k$-step nilspace. Then the homomorphism $\eta_{k-1}:\tran(\ns)\to \tran(\ns_{k-1})$ is continuous.
\end{lemma}

\begin{proof}
Recall that we equip $\tran(\ns)$ with the compact-open topology, which in this case equals the topology of convergence in compact sets (see \S \ref{sec:trans-group}). Let $(\alpha_n)_n$ be a sequence in $\tran(\ns)$ converging to some $\alpha\in \tran(\ns)$. We want to prove that $\eta_{k-1}(\alpha_n)\to \eta_{k-1}(\alpha)$. Let $K\subset \ns_{k-1}$ be any compact set. By Corollary \ref{cor:cpct-sets-are-im-of-cpct-sets} there is a compact set $K'\subset \ns$ such that $\pi_{k-1}(K')=K$. Fix any metric $d$ metrizing the topology on $\ns_{k-1}$. We claim  that $\sup_{y\in K}d\big(\eta_{k-1}(\alpha_n)(y),\eta_{k-1}(\alpha)(y)\big)\to 0$ as $n\to \infty$. Since $\pi_{k-1}$ is surjective from $K'$ onto $K$, this last supremum is at most $\sup_{x\in K'}d(\pi_{k-1}(\alpha_n(x)),\pi_{k-1}(\alpha(x)))$ for every $n$. Suppose this supremum does not converge to 0 as $n\to \infty$. Then there exists $\epsilon>0$ and a sequence $(x_n)$ in $K'$ such that for every $n$ we have $d(\pi_{k-1}(\alpha_n(x_n)),\pi_{k-1}(\alpha(x_n)))>\epsilon$. As $K'$ is compact, passing to a subsequence we can assume that $x_n$ converges to some $x'\in K'$. This yields a contradiction because $\lim_{n\to\infty}\alpha_n(x_n)=\alpha(x')=\lim_{n\to\infty} \alpha(x_n)$.
\end{proof}

\begin{proof}[Proof of Theorem \ref{thm:freetransLie}]
By Lemma \ref{lem:h-is-cont} and Theorem \ref{thm:trans-group-polish}, the map $\eta_{k-1}$ is a continuous homomorphism between the Polish groups $\tran(\ns)$ and $\tran(\ns_{k-1})$. Moreover, the map $\eta_{k-1}$ is seen to be surjective using the description of $\tran(\ns)$ and $\tran(\ns_{k-1})$ given in \eqref{eq:str-trans-ab-nil}. Indeed, since $\ns_{k-1}$ is isomorphic as an Lie-fibered nilspace to $\prod_{i=1}^{k-1}\mc{D}_i(A_i)$, by \eqref{eq:str-trans-ab-nil} applied to $\tran(\ns_{k-1})$ any $\alpha'\in \tran(\ns_{k-1})$ corresponds uniquely to a tuple $(T_1,\ldots,T_{k-1})$, and then by \eqref{eq:str-trans-ab-nil} applied to $\tran(\ns)$ the tuple $(T_1,\ldots,T_{k-1},0)$ corresponds to a translation $\alpha\in \tran(\ns)$ such that $\eta_{k-1}(\alpha)=\alpha'$. Finally, letting $\iota$ denote the inclusion homomorphism $\ker(\eta_{k-1}) \xrightarrow{\iota} \tran(\ns)$, note that $\iota$ and $\eta_{k-1}$ are open maps onto their images, by the open mapping theorem for Polish groups (see \cite[Chapter 1]{BeKe}). We have thus confirmed that we have the short exact sequence of topological groups $0\to \ker(\eta_{k-1}) \xrightarrow{\iota} \tran(\ns) \xrightarrow{\eta_{k-1}}\tran(\ns_{k-1})\to 0$. 
\end{proof}
We need a couple of additional  lemmas to deduce that the translation group of a free nilspace is a Lie group. The first of these will also be used several times later, and concerns the set $\hom(F,\mc{D}_k(\ab))$ for a free nilspace $F$ and an abelian Lie group $\ab$. the group $\hom(F,\mc{D}_k(\ab))$ will be always considered to be equipped with pointwise addition and the compact-open topology. We then have the following fact.
\begin{lemma}\label{lem:hom-Lie}
Let $F$ be a free nilspace, let $\ab$ be an abelian \textup{(}compactly generated\textup{)} Lie group, and let $k$ be any positive integer. Then $\hom(F,\mc{D}_k(\ab))$ is an abelian Lie group.
\end{lemma}
\begin{proof}
Let $B$ be the covering group of $\ab$ as in \S \ref{subsec:lift-inv-free-nil}, thus $B=\mb{Z}^r\times \mb{R}^s$ with a continuous surjective homomorphism $\pi:B\to \ab$. It is readily seen from Lemmas  \ref{lem:Taylor-discrete} and \ref{lem:poly-free-to-R} that $\hom(F,\mc{D}_k(B))\cong \mb{Z}^m\times \mb{R}^n$ for some $m,n\in\mb{Z}_{\ge 0}$, so this is a Lie group. Let $\phi:\hom(F,\mc{D}_k(B))\to \hom(F,\mc{D}_k(\ab))$ be the homomorphism $f\mapsto \pi\co f$. Arguing as in the proof of Lemma \ref{lem:h-is-cont}, we see that $\phi$ is continuous. Hence $\ker(\phi)$ is a closed subgroup of $\hom(F,\mc{D}_k(B))$, hence a Lie group. By Theorem \ref{thm:lift-mor-cfr-ab-gr} the map $\phi$ is also surjective. It follows as in the proof of Theorem \ref{thm:freetransLie} that $\hom(F,\mc{D}_k(\ab))\cong \hom(F,\mc{D}_k(B))/\ker(\phi)$, whence $\hom(F,\mc{D}_k(\ab))$ is a Lie group by \cite[Theorem 2.6]{Mosk}.
\end{proof}

\begin{lemma}\label{lem:ker-h-Lie}
Let $F$ be a $k$-step free nilspace, and let $\eta_{k-1}:\tran(F)\to\tran(F_{k-1})$ be the surjective continuous homomorphism from \eqref{eq:etaj}. Then $\ker(\eta_{k-1})$ is a Lie group.
\end{lemma}
\begin{proof}
By Lemma \ref{lem:hom-Lie} it suffices to prove that the group $\ker(\eta_{k-1})$ is topologically isomorphic to $\hom(F_{k-1},\mc{D}_k(\ab_k))$. The algebraic part is given by \cite[Lemma 2.9.5]{Cand:Notes2}, which gives us an (purely algebraic so far) isomorphism $\varphi:\ker(\eta_{k-1})\to \hom(F_{k-1},\mc{D}_k(\ab_k))$. Hence, it suffices to check that $\varphi$ is continuous, as then it is open (by the open mapping theorem for Polish groups) and therefore a topological-group isomorphism. To check this continuity, recall that we are using the compact-open topology on both $\ker(\eta_{k-1})$ and $\hom(F_{k-1},\mc{D}_k(\ab_k))$. Suppose that $\varphi$ is not continuous. Then there is a convergent sequence $(\alpha_n)$ in $\ker(\eta_{k-1})$ with limit $\alpha\in \ker(\eta_{k-1})$ such that, for some compact $K\subset F_{k-1}$ and $\epsilon>0$, for infinitely many $n$ we have $\sup_{f\in K}d_{\ab_k}(\varphi(\alpha_n)(f),\varphi(\alpha)(f))>\epsilon$. Passing to a subsequence we can assume that for every $n\in \mb{N}$ there exists $f_n\in K$ such that $d_{\ab}(\varphi(\alpha_n)(f_n),\varphi(\alpha)(f_n))>\epsilon$. By Corollary \ref{cor:cpct-sets-are-im-of-cpct-sets} we know that there exists a compact set $K'\subset F$ such that $\pi_{k-1}(K')=K$. For each $n$ let $x_n\in K'$ be such that $\pi(x_n)=f_n$. By compactness, passing to a subsequence we can assume that $x_n\to x$ as $n\to\infty$. Hence $d_{\ab_k}(\alpha_n(x_n)-x_n,\alpha(x_n)-x_n)>\epsilon$ for all $n$. Now letting $n\to \infty$, we know that $\alpha_n(x_n)\to \alpha(x)$, and by continuity of subtraction in an \textsc{lch} nilspace we deduce $0=d_{\ab_k}(\alpha(x)-x,\alpha(x)-x)>\epsilon$, a contradiction.
\end{proof}

\begin{corollary}\label{cor:trans&structFree}
If $F$ is a free nilspace, then $\tran(F)$ and all structure groups of $F$ are Lie groups.
\end{corollary}
\begin{proof}
We argue by induction on $k$. For $k=1$ the translation group is isomorphic to the abelian group underlying $F$, which is Lie. For $k>1$, by induction $\tran(\ns_{k-1})$ is  Lie, and $\ker(h\eta_{k-1})$ is Lie by Lemma \ref{lem:ker-h-Lie}. Hence $\tran(\ns)$ is Lie by Theorem \ref{thm:freetransLie} and \cite[Theorem 2.6]{Mosk}.
\end{proof}

\section{Lie-fibered extensions of free nilspaces are split extensions}\label{sec:split-ext}
We shall now combine results from the previous sections to establish the following theorem, which is a central ingredient for the proof of our main result (Theorem \ref{thm:gpcongrep}). 
\begin{theorem}\label{thm:splitext}
Let $F$ be a $k$-step free nilspace and let $\nss$ be a degree-$k$ extension of $F$ by an abelian Lie group $\ab$, with corresponding fibration $q:\nss\to F$. Then this  extension splits, i.e., there exists a continuous morphism $s:F\to\nss$ such that $q\co s=\id$. In particular $\nss$ is isomorphic as an \textsc{lch} nilspace to the product-nilspace $F\times \mc{D}_k(\ab)$.
\end{theorem}
\noindent This theorem is a generalization for free nilspaces of a splitting-extension result which was proved for finite (discrete) nilspaces in \cite[Theorem 5]{SzegFin}. We first record the case $k=1$ of the theorem, which is a known result and will also be used in the main proof.
\begin{lemma}\label{lem:abcasesplit}
Let $F$ be a torsion-free abelian Lie group, let $\ab$, $G$ be locally-compact abelian groups, and suppose that $G$ is a topological-group extension of $F$ by $\ab$, i.e.\ there is a proper exact sequence of topological groups $0\to \ab \to G \to F\to 0$. Then this extension splits.
\end{lemma}
\begin{proof}
We have that $F$ is a direct product $\mb{R}^n \oplus \mb{Z}^m$, so by \cite[Theorem 3.3]{Mosk} the group $F$ is projective for \textsc{lch} groups, and then the result follows from \cite[Theorem 3.5]{Mosk}.
\end{proof}

A consequence of Lemma \ref{lem:abcasesplit} is that, under its hypotheses, there exists a continuous homomorphism $s:F\to G$ such that, letting $p:G\to F$ be the projection homomorphism, we have $p\co s=\id$. Let us record the following definition for nilspaces, which extends this phenomenon.

\begin{defn}[Cross-section]
Let $\ns,\nss$ be nilspaces and let $\varphi:\nss\to \ns$ be a fibration. A \emph{cross-section} for $\varphi$ is a nilspace morphism $s:\ns\to \nss$ such that $\varphi\co s =\id_{\ns}$.
\end{defn}

\begin{proof}[Proof of Theorem \ref{thm:splitext}] 
We argue by induction on $k$. The case $k=1$ is given by Lemma \ref{lem:abcasesplit}. 

For $k>1$, we assume that the result holds for step at most $k-1$. The crux of the proof consists in specifying a continuous cross-section $\cs:F\to \nss$. Fixing some $y\in q^{-1}(0)$, we will obtain from each element $f\in F$ a translation on $\nss$ in a consistent way, defining then $\cs(f)$ as the image of $y$ under this translation. 

Before detailing the above construction of $\cs$, we make the following observation that uses the inductive hypothesis: for any $t\le k-1$, and any free nilspace $F'$ (of arbitrary finite step), any degree-$t$ extension $\psi:\nss'\to F'$ splits. This observation is essentially \cite[Lemma 4.4]{CGSS-p-hom}; for completeness we detail the proof here. Consider the following commutative diagram:
\begin{equation}\label{diag:splt-1}
\begin{aligned}[c]
\begin{tikzpicture}
  \matrix (m) [matrix of math nodes,row sep=2em,column sep=4em,minimum width=2em]
  {\nss' & F'  \\
     \nss'_{t} & F'_{t}. \\};
  \path[-stealth]
    (m-1-1) edge node [above] {$\psi$} (m-1-2)
    (m-1-1) edge node [right] {$\pi_t$} (m-2-1)
    (m-1-2) edge node [right] {$\pi_t$} (m-2-2)
    (m-2-1) edge node [above] {$\psi_t$} (m-2-2);
\end{tikzpicture}
\end{aligned}
\end{equation}

\vspace{-0.3cm}

\noindent By \cite[Proposition A.18]{CGSS-p-hom} we already know that $\nss'$ is isomorphic (algebraically) to the nilspace $\nss'_t\times_{F'_t} F'$, the isomorphism being the map $y\mapsto (\pi_t(y),\psi(y))$. As this map is a continuous fibration between Lie-fibered nilspaces, it is an open map by Theorem \ref{thm:open-mapping-thm}. Hence it is a homeomorphism. By the induction hypothesis, the extension $\psi_t:\nss'_t\to F'_t$ splits, so there exists a continuous cross-section $s_t:F'_t\to \nss'_t$ such that $\psi_t\co s_t=\id$. Then, the map $F'\to \nss'_t\times_{F'_t} F'$, $x\mapsto(s_t(\pi_t(x)),x)$ is a continuous cross-section from $F'$ to $\nss'$.

Now let $q:\nss\to F$ be the continuous fibration given in the theorem. Note that $q_{k-1}$ is then an isomorphism $\nss_{k-1}\to F_{k-1}$ (by similar arguments as in the proof of Proposition \ref{prop:ext-lie-is-lie-1}). This implies that for every translation $\alpha\in\tran(\nss)$ the map $q$ is \emph{$\alpha$-consistent}, meaning that  $q(y)=q(y')\Rightarrow q(\alpha y)=q(\alpha y')$. Indeed, if $q(y)=q(y')$, then from $q_{k-1}$ being an isomorphism it follows that $y,y'$ are in the same fiber of $\pi_{k-1,\nss}$, so $y'=y+z$ for some $z\in\ab_k(\nss)$, and then $q(y)=q(y')$ also implies that $z$ is in the kernel of the $k$-th structure homomorphism $\phi_{k,q}$ of $q$ (see \cite[Definition 3.3.1]{Cand:Notes1}). Then $\alpha (y')=\alpha(y)+ z$, so $q(\alpha y')=q(\alpha(y))+ \phi_{k,q}(z)=q(\alpha(y))$ as claimed. We can then define (as in \cite[Lemma 1.5]{CGSS}) a homomorphism $\wh{q}:\tran(\nss) \to \tran(F)$, which is a \emph{filtered} group homomorphism, i.e.\ such that for each $i\in [k]$, $\wh{q}$ sends $\alpha\in \tran_i(\nss)$ to the translation $\wh{q}(\alpha): x\in F\mapsto q(\alpha(y))$ in $\tran_i(F)$ (for any $y\in \nss$ with $q(y)=x$).

To obtain the desired cross-section $\cs$ below, we shall use the following properties of $\wh{q}$. 

First note that $\wh{q}$ is continuous (relative to the compact-open topology on $\tran_i(\nss)$ and $\tran_i(F)$; recall \S \ref{sec:trans-group}). Indeed, suppose for a contradiction that there was a sequence $(\alpha_n)_{n\in \mb{N}}$ converging to $\alpha$ in $\tran_i(\nss)$ such that $\wh{q}(\alpha_n)\not\to \wh{q}(\alpha)$. Then there would be a compact set $K\subset F$ and some $\epsilon>0$ such that $\sup_{x\in K}d_F(\wh{q}(\alpha_n)(x),\wh{q}(\alpha)(x))>\epsilon$ for all $n$. Thus there would exist a sequence $(x_n)_{n\in \mb{N}}$ in $K$ such that $d_F(\wh{q}(\alpha_n)(x_n),\wh{q}(\alpha)(x_n))>\epsilon$ for all $n$. By compactness of $K$ we can assume that $x_n$ converges to some $x^*\in K$ as $n\to\infty$. Letting $y_n,y\in\nss$ be such that $x_n=q(y_n)$ for all $n$ and $x^*=q(y)$, by Lemma \ref{lem:conv-quot} there would exist a sequence $z_n\in \ab$ such that $y_n+z_n\to y$ as $n\to\infty$. Renaming this sequence, we could assume that $y_n\to y$. By \cite[Theorem 46.10]{Mu} we would then have  $\alpha_n(y_n)\to \alpha(y)$ as $n\to \infty$. Since $\wh{q}(\alpha_n)(x_n)=q(\alpha_n(y_n))$ and $\wh{q}(\alpha)(x^*)=q(\alpha(y))$, by continuity of $q$ we would have $\wh{q}(\alpha_n)(x_n)\to \wh{q}(\alpha)(x^*)$ as $n\to\infty$, contradicting $d_F(\wh{q}(\alpha_n)(x_n),\wh{q}(\alpha)(x_n))>\epsilon$ for $n$ large enough.

Next we prove the key fact that $\wh{q}$ is surjective from $\tran_i(\nss)$ onto $\tran_i(F)$, for each $i\in [k]$. Note that for $i=k$ this is clear by virtue of $q$ being a fibration (its $k$-th structure homomorphism is then surjective), so we may assume that $i<k$. Let $\eta_{k-1}:\tran_i(F)\to\tran_i(F_{k-1})$ be the homomorphism from \eqref{eq:etaj}. Since $q_{k-1}:\nss_{k-1}\to F_{k-1}$ is an \textsc{lch}-nilspace isomorphism, for every $\gamma\in \tran_i(F)$ we can treat $\eta_{k-1}(\gamma)$ as a translation on $\nss_{k-1}$. By Corollary \ref{cor:ext-of-lcfr-is-lcfr}, $\nss$ is also a Lie-fibered $k$-step nilspace. Now, we want to apply Proposition \ref{prop:cont-lift-of-trans} to prove the existence of a (continuous) translation in $\tran_i(\nss)$ which is a lift of $\eta_{k-1}(\gamma)$ (viewing the latter as a translation in $\tran_i(\nss_{k-1})$). To be able to apply Proposition \ref{prop:cont-lift-of-trans}, we need the associated translation bundle to be a \emph{split} extension of $\nss_{k-1}$. But since $\nss_{k-1}$ is isomorphic to $F_{k-1}$, this translation bundle is thus a degree-$(k-i)$ extension of $F_{k-1}$, so by the above observation related to \eqref{diag:splt-1}, this extension does split. Hence we can indeed apply Proposition \ref{prop:cont-lift-of-trans} and thus we obtain a continuous translation $\beta\in \tran_i(\nss)$ that lifts $\eta_{k-1}(\gamma)$ through $\pi_{k-1}$. Now a remaining issue is that a priori $\wh{q}(\beta)$ may not equal $\gamma$. However, the images of these translations through $\eta_{k-1}$ are both equal to $\eta_{k-1}(\gamma)$. Hence $\wh{q}(\beta)-\gamma:F\to \mc{D}_k(\ab_k(F))$ is a well-defined continuous morphism. Let $\phi_k$ be the $k$-th structure morphism of $q$, thus $\phi_k$ is a surjective continuous group homomorphism $\ab_k(\nss)\to \ab_k(F)$. These last two groups are Lie groups (by assumption for $\ab_k(\nss)$, and by Corollary \ref{cor:trans&structFree} for $\ab_k(F)$). By Lemma \ref{lem:abcasesplit} we have that $\ab_k(\nss)$ is a \emph{split} topological-group extension of $\ab_k(F)$, so there is a continuous homomorphism $s_k:\ab_k(F)\to  \ab_k(\nss)$ such that $\phi_k\co s_k$ is the identity map on $\ab_k(F)$. We can now define the map $\alpha:=\beta+s_k\co(\gamma-\wh{q}(\beta))\co q:\nss\to\nss$. This is a continuous translation on $\nss$, and a simple calculation shows that  $\wh{q}(\alpha)=\gamma$, as required.

Having proved that $\wh{q}$ is a continuous surjective homomorphism, we are nearly ready to define the desired cross-section, by lifting from $\tran_i(F)$ to $\tran_i(\nss)$ using the surjectivity. But to do so consistently on the continuous part of $F$, we shall use that $\tran(\nss)$ is a Lie group (which enables us to lift 1-parameter subgroups), so let us now establish this Lie property.

First we claim that $\ker(\wh{q})$ is a Lie group. For each $\alpha\in \ker(\wh{q})=\ker(\wh{q})\cap\tran_1(\nss)$  we have a well-defined map $\xi_\alpha:F\to \ab_k(\nss)$, $x\mapsto y-\alpha(y)$ for any $y\in q^{-1}(x)$, and by \cite[Lemma 2.9.5]{Cand:Notes2} this map $\xi_\alpha$ is a nilspace morphism $F\to \mc{D}_{k-1}(\ab_k(\nss))$. Let $\varphi:\ker(\wh{q})\to \hom(F,\mc{D}_{k-1}(\ab_k(\nss)))$ be the map $\alpha\mapsto \xi_\alpha$. We claim that $\varphi$ is a continuous group homomorphism (where, as usual, $\hom\big(F,\mc{D}_{k-1}(\ab_k(\nss))\big)$ is equipped with pointwise addition and the compact-open topology). Let $\alpha,\beta\in \ker(\wh{q})$. The value of  $\varphi(\alpha)+\varphi(\beta)$ at $x\in F$ can be calculated as follows. Let $y\in q^{-1}(x)$ and note that by the assumption $\alpha\in \ker(\wh{q})$ we have $\alpha(y)\in q^{-1}(x)$. Hence $\varphi(\alpha)+\varphi(\beta)=(y-\alpha(y))+(\alpha(y)-\beta(\alpha(y)) = y-\beta(\alpha(y)) = \varphi(\beta\alpha)$. The calculation for the inverse follows similarly. Continuity of $\varphi$ follows by a similar argument as in previous proofs: let $K\subset F$ be compact and $(\alpha_n)$ be a convergent sequence in $\ker(\wh{q})$ with limit $\alpha$. Suppose for a contradiction that there is $\epsilon >0$ such that $\sup_{x\in K} d(\varphi(\alpha_n)(x),\varphi(\alpha)(x))>\epsilon$ for all $n$. Then passing to a subsequence we may assume that there exists $x_n\to x\in K$ such that $d(\varphi(\alpha_n)(x_n),\varphi(\alpha)(x_n))>\epsilon$ for all $n$. By Lemma \ref{lem:conv-quot} we can further assume that there exists $y_n\in \nss$ with $y_n\to y$ as $n\to \infty$ such that $q(y_n)=x_n$ for all $n$. We thus obtain a contradiction, as then $\epsilon<d(y_n-\alpha_n(y_n),y_n,\alpha(y_n))\to 0$ as $n\to \infty$. 
Now to deduce that $\ker(\wh{q})$ is a Lie group, note that the continuous homomorphism $\varphi:\ker(\wh{q})\to \hom(F,\mc{D}_{k-1}(\ab_k(\nss)))$ is easily checked to be an isomorphism, with its inverse $\varphi^{-1}:\hom(F,\mc{D}_{k-1}(\ab_k(\nss)))\to \ker(\wh{q})$ being $\varphi^{-1}(f)(x)=x+f(q(x))$. Indeed, for any $f\in \hom(F,\mc{D}_{k-1}(\ab_k(\nss)))$ the map $x\mapsto x+f(q(x))$ is continuous (since $f$, $q$, and the $\ab_k$-action are all continuous), so this map is indeed a \emph{continuous} translation in $\ker(\wh{q})$. Hence $\varphi$ is a bijective continuous homomorphism, hence an isomorphism (by the open mapping theorem). By Lemma \ref{lem:hom-Lie} the group $\hom(F,\mc{D}_{k-1}(\ab_k(\nss)))$ is Lie, and our claim follows.

To see that $\tran_i(\nss)$ is a Lie group for each $i$, note first for $i=1$ that by the continuity and surjectivity of $\wh{q}$ we have the short exact sequence 
$0\to \ker(\wh{q})\to \tran_1(\nss)\to \tran_1(F)\to 0$. 
We know that $\tran_1(\nss)$ is a Polish group (by Theorem \ref{thm:trans-group-polish}), that its closed normal subgroup $\ker(\wh{q})$ is a Lie group, and that $\tran_1(F)$ is also a Lie group (by Corollary \ref{cor:trans&structFree}). Therefore by \cite[Corollary A.2]{HK-non-conv} we have that $\tran_1(\nss)$ is also locally compact, and then it is a Lie group by \cite[Lemma A.3]{HK-non-conv}. The Lie property of the closed subgroups $\tran_i(\nss)$, $i>1$, now follows by Cartan's theorem.

We can now complete the proof by specifying the cross-section $F\to \nss$. Recall that $F=\prod_{i=1}^k \mc{D}_i(\mb{Z}^{a_i}\times \mb{R}^{b_i})$. For $i\in[k]$ and $j\in[a_i]$ we denote by $\gamma_{i,j}\in \tran_i(F)$ the translation that adds 1 to the $j$-th coordinate of elements of $F$ in the discrete part of $\mc{D}_i(\mb{Z}^{a_i}\times \mb{R}^{b_i})$. Similarly, for every $i\in[k]$ and every $s\in [b_i]$ we denote by $\delta_{i,s}:\mb{R}\to \tran_i(F)$ the one-parameter subgroup defined by mapping $t\in \mb{R}$ to the translation in $\tran_i(F)$ that adds $t$ to the $s$-th coordinate of elements of $F$ in the continuous part of  $\mc{D}_i(\mb{Z}^{a_i}\times \mb{R}^{b_i})$. Since $\wh{q}$ is surjective, for every $i\in[k]$ and $j\in[a_i]$ there exists  $\alpha_{i,j}\in \tran_i(\nss)$ such that $\wh{q}(\alpha_{i,j})= \gamma_{i,j}$. Similarly, for each $i\in[k]$ and $s\in [b_i]$, the surjectivity of $\wh{q}$ makes it a quotient morphism in the sense of \cite[p.\ 169]{H&M-ProLie} (by the open mapping theorem), so by the one-parameter subgroup lifting lemma \cite[Lemma 4.19]{H&M-ProLie} there exists a homomorphism $\beta_{i,s}:\mb{R}\to \tran_i(\nss)$ such that for all $t\in \mb{R}$ we have $\wh{q}(\beta_{i,s}(t))=\delta_{i,s}(t)$. Now, writing the elements of $F$ as tuples $((x_{i,j})_{j\in[a_i]},(x'_{i,\ell})_{\ell\in[b_i]})_{i\in [k]}$, and fixing any $y\in q^{-1}(0)$, let
\begin{eqnarray}\label{eq:maincrosssec}
\cs:\hspace{2cm}  F & \to & \nss \nonumber\\
 ((x_{i,j}),(x'_{i,\ell})) & \mapsto & \Big[\prod_{i=1}^k\Big(\prod_{j=1}^{a_i} \alpha_{i,j}^{x_{i,j}}\Big)\Big(\prod_{\ell=1}^{b_i} \beta_{i,\ell}(x'_{i,\ell})\Big)\Big] (y).
\end{eqnarray}
This is clearly a continuous morphism, and by construction  it is a cross-section for $q$.\end{proof}
\subsection{Lie-fibered nilspaces as fibration images of free nilspaces}\hfill\smallskip\\
Proving Theorem \ref{thm:splitext} is a key step in this paper on the way to the main theorems \ref{thm:gpcongrep} and \ref{thm:cfr=double-coset}. In this subsection we pause to deduce some first consequences of Theorem \ref{thm:splitext}, the first of  which is the following general description of Lie-fibered nilspaces in terms of free nilspaces.
\begin{theorem}\label{thm:lcfr-factor-of-free}
Let $\ns$ be a $k$-step Lie-fibered nilspace. Then there exists a $k$-step free nilspace $F$ and a continuous fibration $\varphi:F\to\ns$.
\end{theorem}
\begin{proof}
We argue by induction on the step $k$. For $k=1$, the nilspace $\ns$ is topologically isomorphic to $\mc{D}_1(\ab)$ where $\ab$ is an abelian Lie group. By known results (see e.g.\ \cite[Theorem 2.4]{Mosk}) we have $\ab\cong \mb{R}^m\times \mb{T}^n\times D$ for some discrete abelian group $D$. By \cite[Corollary 12]{Fu&Shak}, if $\ab$ is compactly generated then $D$ is finitely generated, hence of the form $\mb{Z}^r\times G$ for some finite abelian group $G$. The conclusion of the theorem then clearly holds in this case.

Let $k>1$ and suppose by induction that $\ns_{k-1}$ satisfies the desired conclusion, so there is a $(k-1)$-step free nilspace $F_{k-1}$, and a fibration $\varphi_{k-1}: F_{k-1}\to \ns_{k-1}$. We then define the fiber-product
$\nss:=F_{k-1}\times_{\ns_{k-1}}\ns = \{(f,x)\in F_{k-1}\times \ns: \varphi_{k-1}(f)=\pi_{k-1}(x)\}$. Algebraically this is a $k$-step nilspace that is a degree-$k$ extension of $F$ by the structure group $\ab_k=\ab_k(\ns)$. In particular (using for instance \cite[Proposition A.20]{CGSS-p-hom}) for each $i\in [k-1]$ the $i$-th nilspace factor $\pi_i(\nss)$ is isomorphic to the $i$-th nilspace factor $F_i$ of $F_{k-1}$. Recall also that $\pi_{k-1,\nss}$ is just projection to the $F_{k-1}$ component. By Lemma \ref{lem:fib-prod-top-nil} we also know that $\nss$ is a Lie-fibered nilspace. It is straightforwardly seen that the component projections $p_1:\nss\to F_{k-1}$, $(f,x)\mapsto f$ and $p_2:\nss\to \ns$, $(f,x)\mapsto x$ are continuous fibrations. By Theorem \ref{thm:splitext}, the nilspace $\nss$ is a  split extension, so $\nss\cong F_{k-1}\times \mc{D}_k(\ab_k)$. Let $\ab_k'$ be a covering group of $\ab_k$ as per Theorem \ref{thm:lift-mor-cfr-ab-gr}. In particular $\ab_k'= \mb{Z}^b\times\mb{R}^c$ for some integers $b,c\ge 0$.  Let $\phi$ be the natural continuous surjective homomorphism $\ab_k'\to \ab_k$.  We now define the free nilspace $F:= F_{k-1}\times \mc{D}_k(\ab_k')$. The map $\pi_\phi:F\to\nss$, $(f,f_k)\to (f,\phi(f_k))$ is easily checked to be a fibration, and then the composition $p_2\co \pi_{\phi}$ is a fibration $\varphi:F\to \ns$ as required. 
\end{proof}
\noindent The following commutative diagram summarizes the above proof and will also be useful later:
\begin{equation}\label{diag:lift-1}
\begin{aligned}[c]
\begin{tikzpicture}
  \matrix (m) [matrix of math nodes,row sep=2em,column sep=4em,minimum width=2em]
  {F & F_{k-1}\times \mc{D}_k(\ab_k)\simeq F_{k-1}\times_{\ns_{k-1}} \ns & \ns \\
     & F_{k-1} & \ns_{k-1}. \\};
  \path[-stealth]
    (m-1-1) edge node [above] {$\pi_\phi$} (m-1-2)
    (m-1-2) edge node [above] {$p_2$} (m-1-3)
    (m-1-3) edge node [right] {$\pi_{k-1}$} (m-2-3)
    (m-2-2) edge node [above] {$\varphi_{k-1}$} (m-2-3)
    (m-1-2) edge node [right] {$p_1$} (m-2-2);
\end{tikzpicture}
\end{aligned}
\end{equation}

\vspace{-0.3cm}

\noindent We complete this subsection by recording some consequences of Theorem \ref{thm:lcfr-factor-of-free} that will be useful in what follows. We begin by strengthening Proposition \ref{prop:hom-ns-z-polish}.

\begin{corollary}\label{cor:hom-is-lie}
Let $\ns$ be a $k$-step Lie-fibered nilspace, let $\ab$ be an abelian Lie group and let $t\in \mb{Z}_{\ge 0}$. Then the set $\hom(\ns,\mc{D}_t(\ab))$ equipped with pointwise addition and the compact-open topology is a Lie group.
\end{corollary}

\begin{proof}
By Theorem \ref{thm:lcfr-factor-of-free} there is a $k$-step free nilspace $F$ and a continuous fibration $\varphi:F\to\ns$. By Lemma \ref{lem:hom-Lie} the group $\hom(F,\mc{D}_t(\ab))$ is Lie. Let 
\[
C:=\{g\in \hom(F,\mc{D}_t(\ab)): \forall x,y\in F, \text{ if } \varphi(x)=\varphi(y) \text{ then } g(x)=g(y)\}.
\]
Note that this is a closed subset of $\hom(F,\mc{D}_t(\ab))$ (as can be seen straightforwardly, considering converging sequences in $C$). Hence, as a closed subgroup of $\hom(F,\mc{D}_t(\ab))$, in particular $C$ is a Lie group. By Proposition \ref{prop:hom-ns-z-polish}, $\hom(\ns,\mc{D}_t(\ab))$ is a Polish group. Consider the map $\iota:\hom(\ns,\mc{D}_t(\ab))\to C\subset \hom(F,\mc{D}_t(\ab))$, $\gamma\mapsto \gamma\co \varphi$. This map is readily seen to be continuous and injective. Moreover, given any $g\in C$, defining  $\gamma^*\in \hom(\ns,\mc{D}_t(\ab))$  by the formula $\gamma^*(s):=g(x)$ for any $x\in \varphi^{-1}(s)$, using that $\varphi$ is open by Theorem \ref{thm:open-mapping-thm} we get that $\gamma^*$ is in fact continuous (using that $\varphi$ is 1-sequence covering as in \cite[\S 4]{Li}). Hence, $C$ is topologically isomorphic to $\hom(\ns,\mc{D}_t(\ab))$ and so the latter is a Lie group.
\end{proof}
If the group $\ab$ is discrete then we can strengthen the previous result as follows.
\begin{corollary}\label{cor:hom-is-discrete}
Let $\ns$ be a $k$-step Lie-fibered nilspace, let $\ab$ be a discrete abelian Lie group and let $t\in \mb{Z}_{\geq 0}$. Then the set $\hom(\ns,\mc{D}_t(\ab))$ with addition as operation is a discrete group.
\end{corollary}

\begin{proof}
As in the previous proof, let $F$ be a $k$-step free nilspace such that there exists a fibration $\varphi:F\to\ns$. Arguing similarly, we have that $\hom(\ns,\mc{D}_t(\ab))$ is topologically isomorphic to a subgroup of $\hom(F,\mc{D}_t(\ab))$. Lemma \ref{lem:Taylor-discrete} implies that the latter group is isomorphic to $\ab^n$ for some $n\ge 0$, so it is discrete. Hence so are all its subgroups, in particular $\hom(\ns,\mc{D}_t(\ab))$.
\end{proof}

\noindent The last consequence of Theorem \ref{thm:lcfr-factor-of-free} in this subsection is that if $\ns$ is a $k$-step, Lie-fibered nilspace then $\tran(\ns)$ is a Lie group. To prove this we use the following lemma.

\begin{lemma}\label{lem:lift-tran-to-free}
Let $\ns$ be a $k$-step Lie-fibered nilspace. Let $F$ be the free nilspace provided by Theorem \ref{thm:lcfr-factor-of-free}, with fibration $\varphi:F\to \ns$. Then for every $i\in[k]$ and $\alpha\in \tran_i(\ns)$ there exists $\beta\in \tran_i(F)$ such   $\alpha\co\varphi = \varphi\co\beta$.
\end{lemma}

\begin{proof}
We argue by induction on $k$, with the case $k=0$ being trivial. For $k>0$, we use the notation in diagram \eqref{diag:lift-1}. Thus $\varphi := p_2\co \pi_\phi$. Let $\alpha\in \tran_i(\ns)$, and let $\alpha_{k-1}\in\tran_i(\ns_{k-1})$ be such that $\pi_{k-1}\co\alpha = \alpha_{k-1}\co \pi_{k-1}$. By induction there exists $\beta_{k-1}'\in\tran_i(F_{k-1})$ such that $\varphi_{k-1}\co \beta'_{k-1} = \alpha_{k-1}\co \varphi_{k-1}$. Note that  $(\beta_{k-1}',\alpha)$ is an element of $\tran_i( F_{k-1}\times_{\ns_{k-1}} \ns)$ acting coordinate-wise. By Theorem \ref{thm:lift-trans} we can lift this to some $\beta\in \tran_i(F)$, and the result follows.
\end{proof}
\noindent We can now prove the desired consequence of Theorem \ref{thm:lcfr-factor-of-free}. This extends \cite[Theorem 2.9.10]{Cand:Notes2}.
\begin{theorem}\label{thm:LFnstransLie}
Let $\ns$ be a $k$-step Lie-fibered nilspace. For each $i\in [k]$, $\tran_i(\ns)$ is a Lie group.
\end{theorem}

\begin{proof}
We know that $\tran_i(\ns)$ is a closed subgroup of $\tran(\ns)$ by Lemma \ref{lem:tran-i-closed-subset}, so by Cartan's theorem it suffices to prove that $\tran(\ns)$ is a Lie group. Let $\varphi:F\to \ns$ be the fibration in Theorem \ref{thm:lcfr-factor-of-free}, and let
$C:=\{\alpha\in \tran(F):\forall x,y\in \ns, \text{ if } \varphi(x)=\varphi(y) \text{ then } \varphi(\alpha(x))=\varphi(\alpha(y))\}$. Note that $C$ is a closed subgroup of $\tran(F)$. Hence $C$ is a Lie group (by Theorem \ref{thm:freetransLie} and Cartan's theorem). Let $\wh{\varphi}:C\to \tran(\ns)$ be the homomorphism induced by $\varphi$ as in \cite[Lemma 1.5]{CGSS}. 
 To see that $\wh{\varphi}$ is continuous, let $\alpha_n\to\alpha$ in $\tran(F)$, that is, for any compact set $K\subset F$ we have $\sup_{s\in K}d_{F}(\alpha_n(s),\alpha(s))\to 0$ as $n\to\infty$. We need to check that for every compact set $K'\subset \ns$ we have $\sup_{x\in K'}d_{\ns}(\widehat{\varphi}(\alpha_n)(x),\widehat{\varphi}(\alpha)(x))\to 0$ as $n\to\infty$. By Corollary \ref{cor:cpct-sets-are-im-of-cpct-sets}, there exists a compact set $K\subset F$ such that $\varphi(K)=K'$. Suppose for a contradiction that $\sup_{x\in K'}d_{\ns}(\widehat{\varphi}(\alpha_n)(x),\widehat{\varphi}(\alpha)(x))\not\to 0$ as $n\to\infty$. Then, passing to a subsequence if necessary, we can assume that for some $\epsilon>0$ we have  $\sup_{x\in K'}d_{\ns}(\widehat{\varphi}(\alpha_n)(x),\widehat{\varphi}(\alpha)(x))>\epsilon$ for all $n$. In particular, there exists a sequence $x_n\in K'$ such that $d_{\ns}(\widehat{\varphi}(\alpha_n)(x_n),\widehat{\varphi}(\alpha)(x_n))>\epsilon$ for all $n$. For each $n$ let $s_n\in K$ be any element such that $\varphi(s_n)=x_n$. By compactness of $K'$ we can assume that $s_n\to s\in K$ as $n\to \infty$. This yields a contradiction, because $d_{\ns}(\varphi(\alpha_n(s_n)),\varphi(\alpha(s_n)))=d_{\ns}(\widehat{\varphi}(\alpha_n)(x_n),\widehat{\varphi}(\alpha)(x_n))>\epsilon$, whereas $\alpha_n(s_n))\to \alpha(s)$ and $\alpha(s_n))\to\alpha(s)$ as $n\to \infty$. Thus we have a continuous homomorphism $\wh{\varphi}:C\to \tran(\ns)$, which is also surjective by Lemma \ref{lem:lift-tran-to-free}. As $C$ and $\tran(\ns)$ are Polish, the map $\wh{\varphi}$ is also open (by the open mapping theorem) and $\tran(\ns)$ is topologically isomorphic to $C/\ker(\widehat{\varphi})$. But $\ker(\widehat{\varphi})$ is Lie (being a closed subgroup of a Lie group). Hence $\tran(\ns)$ is Lie by \cite[Theorem  2.6]{Mosk}.
\end{proof}

\section{Groupable congruences and fiber-transitive filtrations}\label{sec:gpcongs}
\noindent Given an equivalence relation $\sim$ on a set $X$, let us denote by $\pi_\sim$ the natural map sending $x\in X$ to the equivalence class $\{y\in X: y\sim x\}$. In universal algebra, a \emph{congruence} (or \emph{congruence relation}) is generally speaking an equivalence relation $\sim$ on an algebraic object that is compatible with the algebraic structure, in the sense that the structure is preserved by the map $\pi_\sim$. The precise definition can be formulated for a general algebra, thus encompassing the well-known special cases of congruences on groups, rings, modules, and so on; see \cite[Definition 1.19]{Bergman}. In this spirit, it is natural to define congruences on nilspaces as follows.

\begin{defn}\label{def:cong} 
Let $\ns$ be a nilspace. A \emph{congruence} on $\ns$ is an equivalence relation $\sim$ on $\ns$ such that  $\pi_{\sim}(\ns)$ equipped with the image cube sets $\pi_{\sim}^{\db{n}}\big(\! \cu^n(\ns)\big)$, $n\in\mb{Z}_{\geq 0}$, is a nilspace.
\end{defn}
\noindent Recall that a nilspace morphism $\varphi:\ns\to\nss$ is \emph{cube-surjective} if $\varphi^{\db{n}}$ maps $\cu^n(\ns)$ onto $\cu^n(\nss)$ for every $n\geq 0$. Such morphisms are closely related to congruences.
\begin{lemma}
Let $\ns$ be a nilspace. If $\sim$ is a congruence on $\ns$, then $\pi_\sim$ is a cube-surjective morphism onto $\pi_{\sim}(\ns)$ with the image cube sets. Conversely, if $\nss$ is a nilspace and $\varphi:\ns\to\nss$ is a cube-surjective morphism, then the  relation $x\sim y \Leftrightarrow \varphi(x)=\varphi(y)$ is a congruence on $\ns$.
\end{lemma}
\noindent Fibrations are a very useful class of cube-surjective morphisms. Therefore, nilspace congruences $\sim$ whose quotient maps $\pi_\sim$ are guaranteed to be fibrations will be particularly convenient. A nilspace congruence does not necessarily have this property, because there exist cube-surjective morphisms that are not fibrations (see Remark \ref{rem:congfibr}). However, there is a more specific type of congruence on nilspaces, which does have the above property and which plays a major role in this paper. To formulate the definition we need a couple of additional notions.
\begin{defn}
Let $X$ be a set and let $\sim$ be an equivalence relation on $X$. We say that a map $g:X\to X$ is \emph{$\sim$-vertical} (or \emph{$\sim$-fiber-preserving over the identity map}) if we have $x\sim g(x)$ for every $x\in \ns$. We say that $g$ is \emph{$\sim$-consistent} (or that $\sim$ is \emph{$g$-consistent}) if for every $x,y\in \ns$ we have $x\sim y \Rightarrow g(x)\sim g(y)$.
\end{defn}
\noindent If $g$ is $\sim$-vertical then it is $\sim$-consistent; the converse  fails (e.g.\ congruences on abelian groups).

The notion of congruence that we are about to define uses an important algebraic feature of nilspaces, namely that every nilspace $\ns$ is naturally equipped with the action of its translation group $\tran(\ns)$. This enables us to use the following basic fact.
\begin{lemma}
Let $X$ be a set equipped with an action by a group $G$, let $\sim$ be an equivalence relation on $X$, and let 
\begin{equation}
G^{(\sim)}:=\{g\in G: g\textrm{ is $\sim$-vertical}\}.
\end{equation}
Then $G^{(\sim)}$ is a subgroup of $G$.
\end{lemma}
\begin{proof}
For any $g_1,g_2\in G^{(\sim)}$, for any $x\in X$ we have $x\sim g_2 x\sim g_1g_2x$, so $x\sim g_1g_2x$, so $g_1g_2\in G^{(\sim)}$, and $x=g_1 (g_1^{-1} x)\sim g_1^{-1} x$, so $g_1^{-1}\in G^{(\sim)}$.
\end{proof}
Let us say that a group $H$ acting on $X$ is \emph{finer} (or that its action is finer) than an equivalence relation $\sim$ on $X$ if the orbit relation of $H$ (i.e.\ the equivalence relation whose classes are the orbits of $H$) is finer than $\sim$ (i.e.\ every orbit of $H$ is a subset of some equivalence class of $\sim$). Note that $H$ is finer than $\sim$ if and only if every $g\in H$ acts as a $\sim$-vertical map on $X$. Thus, the group $G^{(\sim)}$ is the largest subgroup of $G$ (relative to inclusion) that is finer than $\sim$.

We can now define the announced special type of congruence.
\begin{defn}[Groupable congruence]\label{def:groupable}
Let $\ns$ be a nilspace. A \emph{groupable congruence} on $\ns$ is an equivalence relation $\sim$ on $\ns$ that satisfies the following property:
\begin{eqnarray}\label{eq:groupable}
&\textrm{for every $i\geq 0$, for all $x,y\in \ns$ such that $x\sim y$ and $\pi_i(x)=\pi_i(y)$,}&\nonumber \\
&\textrm{there exists $g\in \tran(\ns)^{(\sim)}\cap \tran_{i+1}(\ns)$ such that $g\,x=y$.}&
\end{eqnarray}
\end{defn}
\noindent Note that if $\ns$ is $k$-step then \eqref{eq:groupable} holds trivially for $i\geq k$ (as then $\pi_i(x)=\pi_i(y)$ implies that $x=y$, so the required conclusion holds with $g=\id$, the identity map). Note also that the case $i=0$ of  \eqref{eq:groupable} implies that $\tran(\ns)^{(\sim)}$ acts transitively on each equivalence class of $\sim$, whence $\sim$ is equal to the orbit relation of the group $\tran(\ns)^{(\sim)}$ (this motivates the term ``groupable").

\begin{remark}\label{rem:noncongorbit}
An orbit relation of a subgroup of $\tran(\ns)$ is not necessarily a congruence. For instance, let $\ns:=\mc{D}_1(\mb{Z}_2)\times \mc{D}_2(\mb{Z}_2)$ and $\alpha\in \tran(\ns)$ be the map $(x,y)\mapsto (x,x+y)$ (this is a translation by Theorem \ref{thm:decrip-trans-group}). Then letting $G=\langle \alpha \rangle$, it turns out that the relation  $\sim$ induced by $G$ (i.e., $x\sim y$ iff there exists $\beta\in G$ such that $\beta(x)=y$) is not a congruence. To prove this, consider the cubes $\q_1,\q_2\in \cu^3(\ns)$ defined as $\q_1(v_1,v_2,v_3)=(1-v_3,0)$ and $\q_2(v_1,v_2,v_3)= (v_3,v_1v_2)$. Note that the cubes $\pi_\sim \co \q_1 $ and $\pi_\sim \co \q_2$ are adjacent (see \cite[Definition 3.1.6]{Cand:Notes1}) in the sense that $\pi_\sim \co \q_1(v_1,v_2,0)= \pi_\sim \co \q_2(v_1,v_2,1)$. Thus, by \cite[Lemma 3.1.7]{Cand:Notes1} we have that the function $\q_3\in (\pi_\sim(\ns))^{\db{3}}$ defined as $\q_3(v_1,v_2,v_3):=\pi_\sim(0,0)$ for $(v_1,v_2,v_3)\not=(1,1,0)$ and $\q_3(1,1,0)=\pi_\sim(0,1)$ should be in $\cu^3(\pi_\sim(\ns)):=\{\pi_\sim \co \q:\q\in \cu^3(\ns)\}$. It can be checked that $\q_3$ does not equal $\pi_\sim\co \q$ for any $\q\in \cu^3(\ns)$. Moreover, it is easy to check that $\tran(\ns)^{(\sim)} = G$. This shows that some condition on a subgroup $G$ of $\tran(\ns)$ (or the congruence that it generates) is needed to ensure that $\sim_G$ is a congruence.
\end{remark}

\noindent The next lemma is a fundamental fact about groupable congruences, and its proof is elementary. 

\begin{lemma}\label{lem:gpequivcong}
Let $\sim$ be a groupable congruence on a nilspace $\ns$. Then $\sim$ is a congruence and $\pi_{\sim}$ is a fibration from $\ns$ to the nilspace $\pi_{\sim}(\ns)$.
\end{lemma}

\begin{proof}
To prove that $\pi_{\sim}$ is a congruence, we need to show that $\pi_{\sim}(\ns)$, equipped with the cube sets $\{\pi_{\sim}\co\q : \q\in \cu^n(\ns)\}$, $n\in\mb{N}$, is a nilspace. The ergodicity and composition axioms are clear. To prove the completion axiom, let $\q'$ be an $n$-corner on $\pi_{\sim}(\ns)$, and for each $i\in [n]$ let $F_i$ denote the lower face $\{v\in \db{n}:v\sbr{i}=0\}$. Then for each $i$, by definition there is a cube $\q_{F_i}\in \cu^{n-1}(\ns)$ such that $\pi_{\sim}\co \q_{F_i} = \q'|_{F_i}$ (where by an abuse of notation we identify $F_i$ and $\db{n-1}$). We now argue by induction on the height $|v|$ of $v\in \db{n}$ to show that the cubes $\q_{F_i}$ can be modified, by applying $\sim$-vertical maps to their vertices, in order to obtain cubes $\q_i$ that agree at every $w\leq v$ (where by ``agree at $w$" we mean that for any $i,j\in [n]$ such that $w\in F_i\cap F_j$, we have $\q_i(w)=\q_j(w)$). For $v=0^n$, by \eqref{eq:groupable}, for each $i\in [2,n]$ there is $g_i\in \tran(\ns)$ such that the $(n-1)$-cube $g_i^{F_i}\q_{F_i}:= v\mapsto g\cdot \q_{F_i}(v)$ satisfies $g^{F_i}\q_{F_i}(0^n)=\q_{F_1}(0^n)$. Thus, relabeling these cubes $g_i^{F_i}\q_{F_i}$ as our new $\q_i$ for $i\in [2,n]$, and relabeling $\q_{F_1}$ as $\q_1$, we now have that all $\q_i$ agree at $0^n$. Now let $v\in \db{n}$ with  $|v|>0$ and suppose by induction that the $\q_i$ all agree at any $w < v$, i.e.\ at any $w\in \db{n}$ with $w\sbr{j}\leq v\sbr{j}$ for all $j\in [n]$ and $w\neq v$ (i.e.\ for any such $w$, for all $i\in [n]$ such that  $w\in F_i$ the value $\q_i(w)$ is the same). Let $I_v\subset [n]$ be the set of all $i$ such that $v\in F_i$ (note that $I_v=\{i\in [n]: v\sbr{i}=0\}$). Let $i_0=\min I_v$. Then for every $i>i_0$ in $I_v$, since $\q_i(w)=\q_{i_0}(w)$ for every $w<v$, we have $\pi_{|v|-1}\co\q_i(v)=\pi_{|v|-1}\co\q_{i_0}(v)$, so by \eqref{eq:groupable} there exists a $\sim$-vertical $g_{i,v}\in \tran_{|v|}(\ns)$ such that $g_{i,v} \q_i(v) = \q_{i_0}(v)$. Also, since $g_{i,v}\in \tran_{|v|}(\ns)$, we can multiply every value $\q_i(u)$ at any vertex $u\in F_i$ with $v\leq u$, and still thus get an $(n-1)$-cube (because we are multiplying on an upper face of codimension $|v|$ by an element of $\tran_{|v|}(\ns)$). Doing this for each $i\in I_v$ greater than $i_0$, we upgrade our set of cubes $\q_i$ to agree now also at $v$, with value $\q_{i_0}(v)$ at this vertex, and note that this upgrade preserves the agreement of the $\q_i$ at every other vertex $w$ with $|w|\leq |v|$ previously upgraded. This completes the inductive step. 

The above induction eventually ``glues" adequately all the initial cubes $\q_{F_i}$ to yield an $n$-corner $\q''$ on $\ns$ with $\pi_{\sim}\co \q'' = \q'$. Then, by completing $\q''$ in $\ns$ and projecting this completion to $\pi_\sim(\ns)$, the completion axiom follows.

To complete the proof we establish that $\pi_\sim$ is a fibration. Let $\q'$ be an $n$-corner on $\ns$ and let $\q\in \cu^n(\ns/\!\sim)$ be a completion of $\pi_\sim\co \q'$. We know by definition of the cube structure on $\ns/\!\sim$ that there is a cube $\q_0\in\cu^n(\ns)$ such that $\pi_\sim\co \q_0=\q$. We now want to move progressively the values $\q_0(v)$, $v\neq 1^n$, by $\sim$-vertical transformations, to get a cube $\tilde \q\in \cu^n(\ns)$ completing $\q'$; then $\pi_\sim \co \tilde \q$ will still be equal to $\q$, and this will complete the proof that $\pi_{\sim}$ is a fibration. 

To get $\tilde \q$ we can argue as we did above. First, since $\q_0(0^n)\sim \q'(0^n)$, there is a translation $g$ such that if we multiply every value of $\q_0$ by $g$ then we obtain a new cube $\q_1$ that is pointwise $\sim$equivalent to $\q_0$ and such that  $\q_1(0^n)=\q'(0^n)$. Relabel $\q_1$ as $\q_0$. Now let $|v|>0$ and suppose by induction that $\q_0(w)=\q'(w)$ for all $w$ with $|w|<|v|$, thus in particular for all $w<v$. Then we deduce again that $\pi_{|v|-1}\co \q_0(v)=\pi_{|v|-1}\co\q'(v)$ and so by \eqref{eq:groupable} there is a $\sim$-vertical $g\in \tran_{|v|}(\ns)$ such that $g\q_0(v)= \q'(v)$. Applying $g$ to $\q_0(u)$ for every $u\in\db{n}$ with $v\leq u$, we upgrade $\q_0$ to a new cube $\q_1$ that now also agrees with $\q'$ at $v$ (and still agrees at any other $w$ with $|w|\leq |v|$ previously upgraded). The result follows.
\end{proof}
\noindent Recall from \cite[Definition 3.2.27]{Cand:Notes1} that given a nilspace $\ns$, the group of translations $\tran(\ns)$ is naturally endowed with a filtration $\tran_\bullet(\ns)=(\tran_i(\ns))_{i\ge 0}$, which is of degree $k$ if $\ns$ is $k$-step. Let us also say that a filtered group $(G,G_\bullet)$ is a \emph{filtered subgroup} of another filtered group $(H,H_\bullet)$ if $G_i\leq H_i$ for all $i\geq 0$. We shall also say that $G_\bullet$ is a \emph{subfiltration} of $H_\bullet$. When we only write ``$G_\bullet$",  the associated filtered group $(G,G_\bullet)$ is always meant to be the one with $G$ being the first term of this filtration $G_\bullet$, i.e.\ $G=G_0=G_1$. Finally, whenever we have an action of a group $G$ on a set $X$, we denote the orbit relation of $G$ on $X$ by $\sim_G$.

We have seen that a groupable congruence $\sim$ on $\ns$ is equal to the orbit relation of the subgroup $G=\tran(\ns)^{(\sim)}$ of $\tran(\ns)$, where the filtration $G_\bullet$ induced by $\tran_\bullet(\ns)$ on $G$ (i.e.\ the filtration with $i$-th term $G_i=G\cap \tran_i(\ns)$) satisfies the following property, implied by \eqref{eq:groupable}:  
\begin{equation}\label{eq:ft-fil}
\textrm{if $x,y\in \ns$ satisfy $x \sim_G y$ and $\pi_i(x)=\pi_i(y)$ for some $i\geq 0$,  then $x\sim_{G_{i+1}} y$.}
\end{equation}
It can be useful to treat groupable congruences from the viewpoint of subfiltrations   $G_\bullet$ of $\tran_\bullet(\ns)$ satisfying property \eqref{eq:ft-fil}. This property tells us that each subgroup $G_i$ in the filtration acts transitively on each fiber of $\pi_{i-1}$ intersected with a $G$-orbit. Let us capture this as follows. 
\begin{defn}[Fiber-transitive filtrations and groups]\label{def:ft-fil}
Let $\ns$ be a nilspace and let $G_\bullet$ be a subfiltration of $\tran_\bullet(\ns)$. We say that $G_\bullet$ is a \emph{fiber-transitive filtration on} $\ns$ (or that $(G,G_\bullet)$ is a \emph{fiber-transitive} filtered group on $\ns$) 
if it satisfies property \eqref{eq:ft-fil}. A subgroup $\Gamma$ of $\tran(\ns)$ will be said to be \emph{fiber-transitive on} $\ns$ if the filtration $(\Gamma\cap \tran_i(\ns))_{i\geq 0}$ is fiber-transitive on $\ns$.
\end{defn}
\noindent Clearly, if $G_\bullet$ is fiber-transitive on $\ns$ then $\sim_G$ is a groupable congruence on $\ns$ (where, as mentioned earlier, we always take $G=G_0$). Conversely, given a groupable congruence $\sim$ on $\ns$, the group $G:=\tran(\ns)^{(\sim)}$ equipped with the filtration induced by $\tran_\bullet(\ns)$ is fiber-transitive. However, it may be the case that $\Gamma$ induces a groupable congruence $\sim_\Gamma$ but $\Gamma$ (with filtration $\Gamma_\bullet:=(\Gamma\cap \tran_i(\ns))_{i\in[k]}$) is not fiber-transitive. For example let $F:=\mc{D}_1(\mb{R})\times \mc{D}_3(\mb{R})$ and $\Gamma:=\{ \alpha_r \}_{r\in \mb{R}}$ where $\alpha_r(x,y)=(x,y+r(x^2+1))$. This group is not fiber-transitive. Indeed we have $(0,0)\sim_\Gamma (0,1)$ (using $\alpha_1$) and $\pi_2(0,0)=\pi_2(0,1)$, but the only translation of height 3 that could take $(0,0)$ to $(0,1)$, namely the translation $(x,y)\mapsto (x,y+1)$, is not in $\Gamma$. On the other hand, the relation $\sim_\Gamma$ is simply given by $(x,y)\sim_\Gamma (x',y')$ if and only if $x=x'$ (to see this, fix any $(x_0,y_0)\in F$ and note that $\alpha_r(x_0,y_0)=(x_0,y_0+r(x_0^2+1))$ and thus, as $r\in \mb{R}$ the class of $(x_0,y_0)$ is $(x_0,\mb{R})$). But this relation is a groupable congruence. Indeed, the group $\Gamma':=\{(x,y)\mapsto (x,y+r):r\in \mb{R}\}$ is fiber-transitive and $\sim_\Gamma = \sim_{\Gamma'}$.

Given a groupable congruence $\sim$, the group $\tran(\ns)^{(\sim)}$ is the largest subfiltration of $\tran_\bullet(\ns)$  (relative to inclusion of each subgroup) that yields the given congruence. However, there may be other fiber-transitive filtrations on $\ns$ that generate the same groupable congruence. In some situations, a desired property of this congruence can be established by working with some such filtration rather than some other; this will be discussed and illustrated in Remark \ref{rem:different-filtr}.

Our next aim is to detail some algebraic and topological properties of groupable congruences and fiber-transitive filtrations. Before this, we record some remarks on these notions.

\begin{remark}\label{rem:cosnilex}
Important (and particularly simple) examples of nilspaces that arise as quotients by groupable congruences are \emph{coset nilspaces}. Recall (e.g.\ from  \cite[Proposition 2.3.1]{Cand:Notes1}) that a $k$-step coset nilspace $\ns$ is constructed starting with a filtered group $(G,G_\bullet)$ of degree $k$ and a subgroup $\Gamma\subset G$, and letting $\ns$ be the set $G/\Gamma = \pi_\Gamma(G)$ of left cosets of $\Gamma$ in $G$ equipped with the cube sets $\pi_\Gamma^{\db{n}}(\cu^n(G_\bullet))$, $n\geq 0$. Then the orbit relation $\sim_\Gamma$ (where elements of $\Gamma$ act by right-multiplication on $G$) is a groupable congruence on the \emph{group nilspace}  $\wt{\ns}$ associated with $(G,G_\bullet)$. Indeed if $g,g'\in G$ satisfy $g'= g\gamma$ for some $\gamma\in \Gamma$ and $g G_{i+1} = g' G_{i+1}$ (i.e.\ $\pi_i(g)=\pi_i(g')$ in the group nilspace), then clearly $\gamma\in G_{i+1}$, so the translation on $\wt{\ns}$ consisting in right-multiplication by $\gamma$ is already in $\tran_{i+1}(\wt{\ns})$, and therefore \eqref{eq:groupable} holds. In Section \ref{sec:DC} we shall see a more general example of groupable congruences, involving \emph{double}-coset spaces.
\end{remark}

\begin{remark}\label{rem:congfibr}
The following example shows that not all cube-surjective morphisms are fibrations, thus confirming the fact mentioned earlier, that not every congruence induces a fibration.
Let $\ns=\mc{D}_1(\mb{Z}_2)\times \mc{D}_2(\mb{Z}_2)$, and let $\mb{F}_2^\omega=\bigoplus_{n=1}^\infty \mb{F}_2$. It can be proved that there exists a map $f:\mc{D}_1(\mb{F}_2^\omega)\to \ns$ that is arbitrarily well-balanced, meaning that for every $b>0$ there exists $N_b>0$ such that for any $n\ge N_b$ the map $f_n:\mc{D}_1(\mb{F}_2^n)\to \ns$, $v\mapsto f(v\times 0^{\mb{N}\setminus [n]})$ is $b$-balanced in the sense of  \cite[Definition 5.1]{CSinverse} (this is a non-trivial fact about nilspaces, but we omit the details in this paper as we will avoid such examples).

Let $\sim$ be the equivalence induced by $f$, i.e.\ for $x,y\in \mc{D}_1(\mb{F}_2^\omega)$ we have $x\sim y$ $\Leftrightarrow$ $f(x)=f(y)$. Then $\sim$ is a nilspace congruence and $\mc{D}_1(\mb{F}_2^\omega)/\!\!\sim$ is isomorphic to the nilspace $\ns$. Indeed, first note that the map $f_\sim:\mc{D}_1(\mb{F}_2^\omega)/\!\!\sim\, \to \ns$, $\pi_\sim(x) \mapsto f(x)$ is well-defined and injective. We claim that for all $n\ge 0$ the map $f^{\db{n}}:\cu^n(\mc{D}_1(\mb{F}_2^\omega))\to \cu^n(\ns)$ is surjective (hence in particular $f_\sim$ is surjective, hence bijective). Let us denote by $i_n:\mc{D}_1(\mb{F}_2^n)\to \mc{D}_1(\mb{F}_2^\omega)$ the usual inclusion. As $\ns$ is finite, so is $\cu^m(\ns)$ for any $m\ge 0$. The Haar probability measure $\mu_{\cu^m(\ns)}$ is then the normalized counting measure. Fixing any $\q\in\cu^m(\ns)$, we therefore have $\mu_{\cu^m(\ns)}(\{\q\})\gg_{m} 1$. Then, by definition of balance, there is $N_{\q}>0$ such that for $n\geq N_{\q}$ the pushforward measure $\nu:=\mu_{\cu^n(\mc{D}_1(\mb{F}_2^n))}\co (f\co i_n)^{-1}$ is sufficiently close to $\mu_{\cu^m(\ns)}$ in the vague topology to be able to conclude that $\nu(\{\q\})>0$, which implies the claimed surjectivity of $f^{\db{n}}$. Hence $f$ is  a cube-surjective morphism. Now note that $f_\sim:\pi_\sim(x)\mapsto f(x)$ is a nilspace a morphism, and its inverse is a morphism as well, because for any $\q\in \cu^m(\ns)$, by the cube-surjectivity we know that there exists $\q^*\in \cu^m(\mc{D}_1(\mb{F}_2^\omega))$ such that $f\co \q^*=\q$. But $f$ factors through $\mc{D}_1(\mb{F}_2^\omega)/\sim$, $f=f_\sim \co \pi_\sim$ and as $f_\sim$ is a bijection we have that $f_\sim^{-1}\co \q = \pi_\sim \co \q^*\in \cu^m(\mc{D}_1(\mb{F}_2^\omega)/\sim)$. We have thus proved that $f$ is a cube-surjective morphism. However $f$ cannot be a fibration, because the image nilspace of a fibration is always of step at most as large as the step of its domain nilspace, and here  $\mc{D}_1(\mb{F}_2^\omega)$ is 1-step whereas $\ns$ is 2-step.
\end{remark}

\begin{remark}\label{rem:addgpcongprops}
There are several additional pleasant features of groupable congruences, though we shall not directly use them in this paper. For example, it can be shown that if $\sim$ is a groupable congruence on a nilspace $\ns$, then every equivalence class of $\sim$ is a coset nilspace. More precisely, since $\pi_{\sim}$ is a fibration, it already follows from basic results (e.g.\ \cite[Lemma 3.2]{CGSS}) that every fiber of $\pi_{\sim}$ is a sub-nilspace  $\nss$ of $\ns$, and it can then be shown using \eqref{eq:groupable} that $G:=\tran(\ns)^{(\sim)}$ has $\cu^n(G_\bullet)$ acting transitively on every cube set $\cu^n(\nss)$ (where $G_\bullet$ is the filtration on $G$ induced by $\tran_\bullet(\ns)$). 
\end{remark}

\subsection{Algebraic aspects of groupable congruences and fiber-transitive filtrations}\hfill\smallskip\\
Given a $k$-step nilspace $\ns$, recall the homomorphisms $\eta_j:\tran(\ns)\to\tran(\ns_j)$, $j\in [k]$, from \eqref{eq:etaj}.
\begin{lemma}\label{lem:group-cong-to-factors}
If $H_\bullet$ is a fiber-transitive filtration on a $k$-step nilspace $\ns$, then for every $j\ge 0$ the filtration $\eta_j(H_\bullet):=(\eta_j(H_i))_{i\ge 0}$ is fiber-transitive on $\ns_j$. Equivalently, every groupable congruence $\sim$ on $\ns$ induces a groupable congruence $\sim_{(j)}$ on $\ns_j$ defined by $\pi_j(x)\sim_{(j)}\pi_j(y) \Leftrightarrow \exists x'\sim x$, $\pi_j(x')=\pi_j(y)$.
\end{lemma}
\begin{proof}
By induction it suffices to prove this for $j=k-1$. If $\pi_i(x)=\pi_i(y)$ and $g_{k-1}\pi_{k-1}(x)=\pi_{k-1}(y)$ we have $gx+z=y$ for some $z$ in the $k$-th structure group $\ab_k$ of $\ns$. The translation $g$ commutes with the action of $\ab_k$ (by \cite[Lemma 3.2.37]{Cand:Notes1}). Hence, $g(x+z)=y$. Applying the hypothesis,  we get that there exists $g'\in H_{i+1}$ such that $g'(x+z)=y$ and thus $g'_{k-1}\pi_{k-1}(x)=\pi_{k-1}(y)$ where $g'_{k-1}\in \eta_{k-1}(H_{i+1})$. 
\end{proof}
\noindent Given a fiber-transitive filtration $H_\bullet$ on $\ns$, we denote by $\pi_H$ the quotient map for the groupable congruence $\sim_H$ induced by $H=H_0$.

We mentioned in Remark \ref{rem:addgpcongprops} that the fibers of a groupable congruence are coset nilspaces. The following lemma details this for fiber-transitive filtrations, and describes the structure groups of the quotient nilspace.
\begin{lemma}\label{lem:cong-equi}
Let $\ns$ be a $k$-step nilspace, with structure groups $\ab_i$, $i\in [k]$, and let $H_\bullet$ be a fiber-transitive filtration on $\ns$. Then for each $i\in[k]$ the $i$-th structure group of $\pi_H(\ns)$ is $\ab_i/\eta_i(H_i)$. Moreover, for every $n\geq 0$, for each $\q\in\cu^n(\ns)$ the group $\cu^n(H_\bullet)$ acts transitively on $\{\q'\in\cu^n(\ns): \pi_H \co \q' = \pi_H \co \q\}$.
\end{lemma}

\begin{proof}
 We already know by Lemma \ref{lem:gpequivcong} that $\pi_H$ is a fibration. 
 
 For the claim about the structure groups, note that by induction on $k$ it suffices to prove it for $i=k$. Indeed, for $i<k$, by the previous lemma $\eta_{k-1}(H_\bullet)$ is fiber-transitive on $\ns_{k-1}$, so by induction the $i$-th structure group of $\pi_{\eta_{k-1}(H)}(\ns_{k-1})$ is $\ab_i(\ns_{k-1})/\eta_i(\eta_{k-1}(H_i))=\ab_i(\ns)/\eta_i(H_i)$ as required. Now for $i=k$, the defining property \eqref{eq:ft-fil} clearly implies that when we quotient by $\sim_H$, in each $\pi_{k-1}$-fiber (which is an orbit of $\ab_k$) we quotient by the action of $H_k=\eta_k(H_k)$ (note that $\eta_k$ is the identity). Therefore the $k$-th structure group of $\ns$ is $\ab_k / H_k$ as required. 
 
For the last claim in the lemma it suffices to prove that if two cubes $\q,\q'\in \cu^n(\ns)$ satisfy $\pi_H\co\q'=\pi_H\co\q$ then there exists $\tilde\q\in \cu^n(H_\bullet)$ such that $\q'=\tilde\q\cdot \q$ (i.e.\ for every $v\in \db{n}$, the value $\q'(v)$ is the image of $\q(v)$ under the action of $\tilde\q(v)$). This can again be proved by induction on $k$: the $\pi_{k-1}$-images of $\q,\q'$ are pointwise equivalent under $\pi_{\eta_{jk-1}(H)}$, so by induction there is a cube $\q_0\in \cu^n(H_\bullet/H_k)$ such that $\pi_{k-1}\co \q' = \q_0\cdot \pi_{k-1}\co \q$. Then any lift $\tilde \q_0$ of $\q_0$ in $ \cu^n(H_\bullet)$ satisfies $\pi_{k-1}\co \q' = \pi_{k-1}\co (\tilde \q_0\cdot \q)$, and then the difference $ \q' -\tilde \q_0\cdot \q$ is a cube $\q''\in \cu^n(\mc{D}_k(\ab_k))$ which, by the assumption of point-wise $\pi_H$-equivalence of $\q,\q'$, must in fact be in $\cu^n(\mc{D}_k(H_k))$,  by \eqref{eq:ft-fil}. Setting $\tilde \q=\q''\cdot\tilde\q_0$, the result follows.
\end{proof}
\noindent Let us detail the relation between groupable congruences and fiber-transitive filtrations.
\begin{defn}\label{def:closure-of-fib-tran-group}
Given a fiber-transitive filtration $H_\bullet$ on a nilspace $\ns$, we denote by $\compl{H}$ the group $\tran(\ns)^{(\sim_H)}$ of all translations that are $\sim_H$-vertical, and call this the \emph{fiber-transitive closure} of $H$. We shall always take the filtration $\compl{H}_\bullet$ on this group to be the filtration induced by $\tran_\bullet(\ns)$ (i.e.\ with $i$-term ${\compl{H}}\cap \tran_i(\ns)$ for $i\ge 0$). Note that $H\subset {\compl{H}}$ and that $\compl{H}_\bullet$ is fiber-transitive on $\ns$.
\end{defn}
\noindent The group $\compl{H}$ may be larger than $H$. For example, let $\ns=\mc{D}_1(\mb{Z})\times \mc{D}_2(\mb{Z})$ and $H=\langle \alpha,\beta\rangle$ where $\alpha(x,y)=(x+2,y)$ and $\beta(x,y)=(x,y+2)$. The filtration $\langle \alpha,\beta\rangle= \langle \alpha,\beta\rangle \ge \langle \beta\rangle \ge \langle \id\rangle$ is a fiber-transitive filtration. Indeed, it is easy to check that $\pi_H(\ns)$ is isomorphic as a nilspace to $\mc{D}_1(\mb{Z}_2)\times \mc{D}_2(\mb{Z}_2)$. However, note that for example $\gamma(x,y):=(x,y+2x)$ is $\sim_H$-vertical but it is not in $H$, i.e.\ $\gamma\in \compl{H}\setminus H$.  

While ${\compl{H}}$ may be larger than $H$, it is natural to expect that these two groups yield isomorphic nilspace quotients of $\ns$. We formalize this as follows.
\begin{defn}[Equivalent filtrations on a nilspace]
Let $\ns$ be a $k$-step nilspace. We say that two subfiltrations $H^0_\bullet, H^1_\bullet$ of $\tran_\bullet(\ns)$ are \emph{equivalent} if for each $i\in\{0,1\}$, for every $\q\in \cu^n(\ns)$ and every $d\in \cu^n(H^i_\bullet)$, there exists $d^*=d^*_{\q,d}\in \cu^n(H^{1-i}_\bullet)$ such that $d\cdot \q = d^*\cdot \q$. 
\end{defn}
Note the key point here that $d^*$ may depend on both $\q$ and $d$.  For example, the filtration $\compl{H}_\bullet$ is equivalent to $H_\bullet$ (as we shall prove below). As mentioned earlier for groupable congruences, the filtration ${\compl{H}}_\bullet$ is the largest subfiltration of $\tran_\bullet(\ns)$ that is equivalent to $H_\bullet$.
\begin{lemma}\label{lem:ft-iff}
Let $\ns$ be a $k$-step nilspace and let $H^1_\bullet, H^2_\bullet$ be equivalent subfiltrations on $\ns$. Then $H^1_\bullet$ is fiber-transitive if and only if $H^2_\bullet$ is fiber-transitive.
\end{lemma}

\begin{proof}
Without loss of generality assume that $H^2_\bullet$ is fiber-transitive. We have to prove that if $\alpha(x)=y$ for some $\alpha\in H^1$ and $\pi_i(x)=\pi_i(y)$ then there exists $\gamma\in H^1_{i+1}$ such that $\gamma(x)=y$. Let us see first that by induction on $k$ we can assume that $i=k-1$. Note that $\eta_{k-1}(H_\bullet^1)$ is equivalent to $\eta_{k-1}(H_\bullet^2)$, so by induction $\eta_{k-1}(H_\bullet^1)$ is  fiber-transitive on $\ns_{k-1}$. Since $\eta_{k-1}(\alpha)(\pi_{k-1}(x)) = \pi_{k-1}(y)$ and $\pi_i(\pi_{k-1}(x))=\pi_i(\pi_{k-1}(y))$, it follows that there exists $\beta_{k-1}\in \eta_{k-1}(H^1_{i+1})$ such that $\beta_{k-1}(\pi_{k-1}(x)) = \pi_{k-1}(y)$. If $\beta_{k-1}=\eta_{k-1}(\beta)$ we have that for some $t\in \ab_k(\ns)$, $\beta(x)=y+t=\alpha(x)+t$. Since $\ab_k$ is central in $\tran(\ns)$, we have $x-t=\beta^{-1}\alpha(x)$. If now we define $y':=x-t$ note that $\pi_{k-1}(y')=\pi_{k-1}(x)$ and that $y'=\beta^{-1}\alpha(x)$ where $\beta^{-1}\alpha\in H^1$. Hence, if we knew that the result was true for $i=k-1$ we could conclude from here that there exists $h\in H^1_k$ such that $x+h=y'$. But this would imply that $x+h = \beta^{-1}\alpha(x)$ and so $\beta(x+h)=\alpha(x)=y$. As the function $\beta(\cdot+t)\in H^1_{i+1}$, this would complete the proof.

Thus assume that $x,y\in\ns$ satisfy $\alpha(x)=y$ for some $\alpha\in H^1$ and $\pi_{k-1}(x)=\pi_{k-1}(y)$. We need to prove that $y-x\in H^1_k$. By the equivalence of $H^1_\bullet$ and $H^2_\bullet$, there exists $\beta\in H^2$ such that $\beta(x)=y$. As $H^2_\bullet$ satisfies \eqref{eq:ft-fil} we have $z:=y-x\in H^2_k$. Let $\q_x\in\cu^k(\ns)$ be the constant cube with value $x\in \ns$ and let $d\in\cu^k(H^2_\bullet)$ be the cube such that $d(v)=0$ for all $v\not=1^k$ and $d(1^k)=z$. By the assumed equivalence, there exists $s\in\cu^k(H^1_\bullet)$ such that $s\co \q_x = d+\q_x$. By \cite[Lemma 2.2.5]{Cand:Notes1} we can write $s$ as $g_{2^k-1}^{F_{2^k-1}}\cdots g_1^{F_1}g_0^{F_0}$ where $g_i\in H^1_{\codim(F_i)}$ (here the order is the inverse as the one in \cite[Lemma 2.2.5]{Cand:Notes1}, but as $\cu^k(H^1_\bullet)$ is a group, we can just take the inverse). From this expression it is  readily seen that $g_0$ stabilizes $x$. In fact, evaluating at $v=0^k$ we see that $s(\q_x(0^k))=g_0(x)=x=(d+\q_x)(0^k)$. Then we repeat evaluating at points $v\in \db{k}$ which are all $0$s and just one element equal to $1$. This shows that $g_1,\ldots,g_k$ also stabilizes $x$. We repeat this process for $v\in \db{k}$ with two non-zero elements, then three, etc..

In the last step of this process we have the fact that $g_{2^k-1}(x)=x+z$ (where $F_{2^k-1}=\{1^k\}$). Since $H^1$ is of degree $k$, we have that $g_{2^k-1}$ is just addition by an element of the last structure group, so it has to be the element $z$. Hence $z\in H^1_k$.
\end{proof}

\begin{corollary}\label{cor:same-last-str-gr}
Let $\ns$ be a $k$-step nilspace, and let $H^1_\bullet,H^2_\bullet$ be equivalent fiber-transitive filtrations on $\ns$. Then for every $j\in [k]$, the $j$-th groups in $\eta_j(H^1_\bullet)$ and $\eta_j(H^2_\bullet)$ are equal.
\end{corollary}

\begin{proof}
Note that it is enough to assume that $\ns$ is $k$-step and prove that $H^0_k=H^1_k$. Let $z\in H^0_k$ and any $x\in \ns$. Let $\q\in \cu^k(\ns)$ be the constant map equal to $x$, and let $d\in \cu^k(H^0_k)$ be the map whose values equal $\id$ in every point except in $v=1^k$, where it equals $z$ (recall that $\tran_k(\ns)$ equals the $k$th structure group of $\ns$). By definition of equivalence, there exists $d^*\in \cu^k(H^1_\bullet)$ such that $d\cdot \q=d^*\cdot \q$. Thus $d^*(1^k)(x)=d(1^k)(x)=x+z$ and $\pi_{k-1}(x)=\pi_{k-1}(x+z)$. As $H^1_\bullet$ is fiber-transitive, there exists $z'\in H_k^1$ such that $x+z'=x+z$. Hence $z=z'$ and we have that $H_k^0\subset H_k^1$. The opposite inclusion is proved similarly and the result follows.
\end{proof}

\begin{lemma}\label{lem:eq-sub-implies-eq-nil}
Let $\ns$ be a $k$-step nilspace, and let $H^1_\bullet,H^2_\bullet$ be equivalent filtrations on $\ns$.  If any of these filtrations is fiber-transitive, then so is the other and $\pi_{H^1}(\ns)\cong \pi_{H^2}(\ns)$.
\end{lemma}

\begin{proof}
By Lemma \ref{lem:ft-iff} we have that if either $H^1_\bullet$ or $H^2_\bullet$ is fiber-transitive so is the other. Thus both quotients $\pi_{H^1}(\ns)$ and $\pi_{H^2}(\ns)$ define nilspaces. We prove that the map $\pi_{H^1}(\ns) \to \pi_{H^2}(\ns)$, $\pi_{H^1}(x)\mapsto \pi_{H^2}(x)$ is a well-defined nilspace isomorphism. To see that it is well-defined, note that if $\pi_{H^1}(x)=\pi_{H^1}(y)$ then there exists $\alpha\in H^1$ such that $\alpha(x)=y$. By hypothesis, there exists $\beta\in H^2$ such that $y=\alpha(x)=\beta(x)$. Hence $\pi_{H^2}(x)=\pi_{H^2}(y)$. This map is a morphism because given a cube $\pi_{H^1}\co \q$ for some $\q\in\cu^n(\ns)$ we have that its image is $\pi_{H^2}\co \q\in\cu^n(\pi_{H^2}(\ns))$. Taking inverses it follows that this map is a nilspace isomorphism.
\end{proof}

\begin{proposition}
Let $\ns$ be a $k$-step nilspace, let $H_\bullet$ be a fiber-transitive filtration on $\ns$, and let ${\compl{H}}$ be the fiber-transitive closure of $H$. Then $\pi_H(\ns) \cong \pi_{{\compl{H}}}(\ns)$.
\end{proposition}

\begin{proof}
By Lemmas \ref{lem:ft-iff} and \ref{lem:eq-sub-implies-eq-nil} it suffices to see that the filtrations are equivalent. The only non-trivial part is to prove that if $\q\in\cu^n(\ns)$ and $d\in\cu^n({\compl{H}})$ then there exists $d^*\in \cu^n(H)$ such that $d\co \q = d^*\co \q$.

We prove this by induction on $k$, with the case $k=0$ being trivial. Note that $\pi_{k-1}\co \q\in \cu^n(\ns_{k-1})$ and $\eta_{k-1}(d)\in\cu^n(\eta_{k-1}({\compl{H}}))$. It is easy to see that $\eta_{k-1}({\compl{H}})\subset \compl{\eta_{k-1}(H)}$. Hence, by induction we have that there exists $d_{k-1}^*\in \cu^n(\eta_{k-1}(H))$ such that $\eta_{k-1}(d) \co (\pi_{k-1}\co \q) = d_{k-1}^*\co \pi_{k-1}\co \q$. Therefore, if $d_{k-1}^* = \eta_{k-1}\co d^*$ for some $d^*\in \cu^n(H)$ we get that $\pi_{k-1}(d\co \q)=\pi_{k-1}(d^*\co \q)$. But by definition of ${\compl{H}}$ we have that for every $v\in \db{n}$, $d(\q(v))\sim_H d^*(\q(v))$. Since $H_\bullet$ is fiber-transitive, this implies that $d(\q(v))-d^*(\q(v))\in H_k$ for all $v\in \db{n}$ and so $d\co\q-d^*\co \q\in \cu^n(\mc{D}_k(H_k))$. Therefore $d\co \q = d^*\co \q+(d\co\q-d^*\co \q)$ which is the composition of $\q$ with the cube $d^*\co (\cdot)+(d\co\q-d^*\co \q)\in \cu^n(H_\bullet)$.
\end{proof}
\begin{corollary}\label{cor:equivalent-filtrations}
Let $\ns$ be a $k$-step nilspace and let $H_\bullet$ be a fiber-transitive filtration on $\ns$. Let $H'_\bullet$ be the filtration with $i$-th term $H\cap \tran_i(\ns)$ for each $i\geq 0$. Then $\pi_H(\ns)\cong \pi_{H'}(\ns)$.
\end{corollary}

\subsection{Topological aspects of groupable congruences and fiber-transitive filtrations}\hfill\smallskip\\
The main goal in this subsection is to find convenient topological properties that can be required of a groupable congruence (equivalently, a fiber-transitive filtration $H_\bullet$) on a Lie-fibered nilspace $\ns$, in order to ensure that the corresponding quotient of $\ns$ is also a Lie-fibered nilspace. Recall that by Theorem \ref{thm:LFnstransLie}, for all $i\in[k]$ the group $\tran_i(\ns)$ is Lie.
\begin{proposition}
Let $\ns$ be a $k$-step Lie-fibered nilspace. Then for every $i\in [k]$ the homomorphism $\eta_i:\tran(\ns)\to \tran(\ns_i)$ from \eqref{eq:etaj} is continuous.\footnote{Here it is convenient to assume that $\ns$ is Lie-fibered, so that by Theorem \ref{thm:trans-group-polish} we know that $\tran(\ns)$ is a topological (Polish) group. This is not clear if we just assume $\ns$ to be an \textsc{lch} nilspace.}
\end{proposition}

\begin{proof}
It suffices to prove this for $i=k-1$, because for $i<k-1$ the homomorphism $\eta_i$ is the composition of $\eta_{k-1}$ on $\tran(\ns)$ with the map $\eta_i$ on $\tran(\ns_{k-1})$, and the latter map can be assumed to be continuous by induction on $k$. Let $(\alpha_n)_{n\in \mb{N}}$ be a convergent sequence in $\tran(\ns)$ with limit $\alpha$ (which is in $\tran(\ns)$, since $\tran(\ns)$ is closed by Lemma \ref{lem:tran-i-closed-subset}). Suppose for a contradiction that $\eta_{k-1}(\alpha_n)\not\to\eta_{k-1}(\alpha)$, so there exists a compact set $K\subset \ns_{k-1}$ and $\epsilon_0>0$ such that $\sup_{\pi_{k-1}(x)\in K}d_{\ns_{k-1}}(\eta_{k-1}(\alpha_n)(\pi_{k-1}(x)),\eta_{k-1}(\alpha)(\pi_{k-1}(x))) > \epsilon$ for infinitely many $n$. Passing to a subsequence if necessary, we can assume that this inequality holds for all $n$. In particular, for each $n$ there exists $\pi_{k-1}(x_n)\in K$ such that
\begin{equation}\label{eq:convfail}
d_{\ns_{k-1}}(\eta_{k-1}(\alpha_n)(\pi_{k-1}(x_n)),\eta_{k-1}(\alpha)(\pi_{k-1}(x_n))) > \epsilon.
\end{equation} 
By compactness of $K$, and passing to a subsequence if necessary, we can assume that there exists $x\in \ns$ such that $\pi_{k-1}(x_n)\to \pi_{k-1}(x)$ as $n\to\infty$. By Lemma \ref{lem:conv-quot} there exists a sequence $z_n\in \ab_k(\ns)$ such that $x_n+z_n\to x$ in $\ns$. Since $\ns$ is \textsc{lch}, the evaluation map is continuous (see e.g.\ \cite[Theorem 46.10]{Mu}), so $\alpha_n(x_n+z_n)\to \alpha(x)$ as $n\to\infty$. By continuity of $\pi_{k-1}$, it follows that $\eta_{k-1}(\alpha_n)(\pi_{k-1}(x_n)) = \eta_{k-1}(\alpha_n)(\pi_{k-1}(x_n+z_n)) = \pi_{k-1} (\alpha_n (x_n+z_n))    \to   \pi_{k-1} (\alpha x) = \eta_{k-1}(\alpha)(\pi_{k-1}(x))$ as $n\to\infty$. Similarly,  by continuity of $\alpha$ and $\pi_{k-1}$ we have $\eta_{k-1}(\alpha)(\pi_{k-1}(x_n))\to \eta_{k-1}(\alpha)(\pi_{k-1}(x))$. This implies that $d_{\ns_{k-1}}(\eta_{k-1}(\alpha_n)(\pi_{k-1}(x_n)),\eta_{k-1}(\alpha)(\pi_{k-1}(x_n)))\to 0$ as $n\to\infty$, contradicting \eqref{eq:convfail}.
\end{proof}

\noindent The topological condition that we will use is the following, which requires the congruence to be closed ``at all levels". The associated natural formulation of this condition for fiber-transitive filtrations can be viewed as an analogue of properties used for quotients of topological groups. 

\begin{defn}[Closed fiber-transitive filtrations]\label{def:oldclosed}
Let $\ns$ be a $k$-step \textsc{lch} nilspace, and let $H_\bullet$ be a fiber-transitive filtration on $\ns$. We say that $H_\bullet$ is \emph{closed} if for every $j\in [k]$ the group $\eta_j(H)$ is a closed subgroup of $\tran(\ns_j)$. We say that a groupable congruence $\sim$ on $\ns$ is \emph{closed} if the associated fiber-transitive filtration $\big(\tran(\ns)^{(\sim)}\cap \tran_i(\ns)\big)_{i\geq 0}$ is closed.
\end{defn}

\begin{remark}
Recall that for a $k$-step \textsc{lch} nilspace $\ns$, its $k$-th structure group $\ab_k(\ns)$ is  homeomorphic to any fiber $\pi_{k-1}^{-1}(\pi_{k-1}(x_0))$ (for any $x_0\in \ns$) equipped with the subspace topology as a subset of $\ns$. Also, it is known (\cite[Lemma 3.2.37]{Cand:Notes1}) that $\tran_k(\ns)$ is (algebraically) isomorphic as a group to $\ab_k(\ns)$ where the isomorphism is given by $\ab_k(\ns)\to \tran_k(\ns)$, $z\mapsto g_z$ where $g_z(x):=x+z$. In particular, $\tran_k(\ns)$ is naturally endowed with the subspace topology as a subset of $\tran(\ns)$ (which is endowed with the compact-open topology by \S \ref{sec:trans-group}). We leave it as an exercise for the reader to check that the map $\ab_k(\ns)\to \tran_k(\ns)$ is a topological group isomorphism with respect to the aforementioned topologies.
\end{remark}
\noindent Our first observation about Definition \ref{def:oldclosed} is the following.
\begin{lemma}\label{lem:closedconseq}
Let $\ns$ be a $k$-step Lie-fibered nilspace and let $H_\bullet$ be a fiber-transitive filtration on $\ns$ such that for every $i\in [k]$ the group $\eta_i(H_i)$ is a discrete subgroup of $\ab_i(\ns)$. Then for all $i,j\in [k]$ the group $\eta_j(H_i)$ is a discrete subgroup of $\tran_i(\ns_j)$, whence closed. In particular $H_\bullet$ is then a closed fiber-transitive filtration on $\ns$.
\end{lemma}

\begin{proof}
We argue by induction on $k$, with the case $k=1$ being trivial. By induction it suffices to prove that $H_i$ is a discrete subgroup of $\tran(\ns)$ for every $i\in[k]$. Note that the case $i=k$ holds by assumption ($\eta_k$ is the identity map). For every $i\in [k-1]$ consider the continuous map $\eta^{(i)}:\tran_i(\ns)\to \tran_i(\ns_{k-1})$, $\alpha\mapsto \eta_{k-1}(\alpha)$. This homomorphism factors through $\tran_i(\ns)/\ker(\eta^{(i)})$. Thus, there is a continuous homomorphism $\overline{\eta^{(i)}}:\tran_i(\ns)/\ker(\eta^{(i)})\to \tran_i(\ns_{k-1})$ such that, letting $\pi:\tran_i(\ns)\to\tran_i(\ns)/\ker(\eta^{(i)})$ be the canonical quotient map, we have $\eta^{(i)} = \overline{\eta^{(i)}}\co \pi$ (here, as usual $\tran_i(\ns)/\ker(\eta^{(i)})$ is endowed with the quotient topology).

Let $(\alpha_n)$ be a sequence in $H_i$ converging to some $\alpha\in \tran_i(\ns)$. We want to prove that $\alpha_n=\alpha$ for $n$ large enough. By continuity the sequence $(\eta^{(i)}(\alpha_n))$ is convergent in $\tran_i(\ns_{k-1})$. By induction, this sequence is constant for $n$ large enough,  so by continuity of $\eta^{(i)}$ we have  $\eta^{(i)}(\alpha_n) = \eta^{(i)}(\alpha)$ for $n\ge N$. In particular, for $n,m\ge N$ we have  $\alpha_n^{-1}\alpha_m\in \ker(\eta^{(i)})$. By \cite[Lemma 2.9.5]{Cand:Notes2} we know that $\ker(\eta^{(i)})=\hom\big(\ns_{k-1},\mc{D}_{k-i}(\ab_k(\ns))\big)$. That is, the elements $f\in \ker(\eta^{(i)})$ are precisely the maps of the form $f(x)=x+\tilde{f}(\pi_{k-1}(x))$ for some $\tilde{f}\in \hom\big(\ns_{k-1},\mc{D}_{k-i}(\ab_k(\ns))\big)$. Moreover, for any $x\in \ns$ we have  $\alpha_n^{-1}\alpha_m(x)=x+z$ for some unique $z\in \ab_k(\ns)$, and since $H_\bullet$ is fiber-transitive, by \eqref{eq:ft-fil} we must have $z\in H_k$. Hence $\alpha_n^{-1}\alpha_m \in \hom\big(\ns_{k-1},\mc{D}_{k-i}(H_k)\big) $ for $n,m\ge N$. In particular, note that $(\alpha_n^{-1}\alpha_N)_n$ is a convergent sequence in $\hom(\ns_{k-1},\mc{D}_{k-i}(H_k))$, the latter being discrete by Corollary \ref{cor:hom-is-discrete}, so $\alpha_n^{-1}\alpha_N = \alpha_M^{-1}\alpha_N$ for $n\ge \max(N,M)$. Thus $\alpha_n=\alpha_N (\alpha_M^{-1}\alpha_N)^{-1} $ for $n$ large enough.
\end{proof}
\noindent The main goal in the rest of this subsection is to prove that a quotient of a Lie-fibered nilspace by a closed groupable congruence is a Lie-fibered nilspace. The proof relies on the following.
\begin{lemma}\label{lem:re-param-conv-seq}
Let $\ns$ be a $k$-step Lie-fibered nilspace and let $H_\bullet$ be a closed fiber-transitive filtration on $\ns$. Let $i\in [k]$, let $(x_n)_{n\in\mb{N}}$ be a sequence in $\ns$ and $(h_n)_{n\in\mb{N}}$ be a sequence in $H_i$ such that for some $x,x'\in\ns$ we have $x_n\to x$ and $h_n(x_n)\to x'$ as $n\to \infty$. Then there is a sequence $(h_n')_{n\in \mb{N}}$ in $H_i$ and $h\in H_i$ with $h_n'\to h$ as $n\to\infty$ and $h_n'(x_n)=h_n(x_n)$ for all $n$.
\end{lemma}

\begin{proof}
We prove a stronger version of the result: instead of assuming that $h_n\in H_i$ for a fixed $i$, we assume that for a fixed $i\in[k]$ and for all $n\in\mb{N}$ we have $\eta_{i-1}(h_n)\in \stab_{\ns_{i-1}}(\pi_{i-1}(x_n))$. Without loss of generality, we may assume that $H_i=H\cap\tran_i(\ns)$ for all $i\in[k]$. Moreover, recall that $H$ being closed fiber-transitive implies that $\eta_i(H_i)$ is a closed subgroup of $\tran_i(\ns_i)$ for all $i\in[k]$. Fix $k\in\mb{N}$. The proof is a reverse induction in $i\in[k]$.

If $i=k$, note that $h_n(x_n)=x_n+z_n$ where $z_n\in H_k$ by the fiber-transitive property. Hence $z_n=h_n(x_n)-x_n\to x'-x$ as $n\to\infty$. As $H_k$ is closed we have that $x'-x\in H_k$ and the proof follows just by letting $h_n'$ be be map that adds $z_n$ for all $n\in\mb{N}$.

For $i<k$ note that $\eta_i(\pi_i(x_n))=\pi_i(x_n)+z_n$ for some $z_n\in \ab_i(\ns)$. Moreover, by the fiber-transitive property we have that for all $n\in \mb{N}$, $z_n\in \eta_i(H_i)$. As $h_n(x_n)$ and $x_n$ both converge as $n\to\infty$, we deduce that $z_n\to z$ for some $z\in\eta_i(H_i)$ (using that the latter group is closed). Since $\tran(\ns)$ and $\tran(\ns_{i})$ are Polish groups (by Theorem \ref{thm:trans-group-polish}) and by assumption $H_i$ and $\eta_i(H_i)$ are closed subgroups of these respectively, it follows from \cite[Proposition 1.2.1]{BeKe} that $H_i$ and $\eta_i(H_i)$ are Polish groups. By Lemma \ref{lem:conv-polish-group} applied to $\eta_i:H_i\to \eta_i(H_i)$ we can assume that there exists $h_n'\in H_i\to h'\in H_i$ as $n\to\infty$ such that $\eta_i(h_n')$ is the map that adds $z_n$. Hence, we have that $\eta_i(h_n)(\pi_i(x_n))=\eta_i(h_n')(\pi_i(x_n))$. In particular, note that $(h_n')^{-1}h_n(x_n)$ is a sequence that converges to $(h')^{-1}(x')$ and that for all $n\in \mb{N}$ we have $\eta_{i}((h_n')^{-1}h_n)\in \stab_{\ns_i}(\pi_i(x_n))$. Thus, by induction we may choose some $h_n''\in H_{i+1}$ such that $(h_n')^{-1}h_n(x_n)=h_n''(x_n)$ and $h_n''\to h''$ as $n\to \infty$. Therefore, $h_n'h_n''\in H_i$ is a convergent sequence such that $h_n'h_n''(x_n)=h_n(x_n)$ and thus the result follows.\end{proof}

\begin{example}
It may be tempting to try to strengthen Lemma \ref{lem:closedconseq} by weakening the \emph{discreteness} assumption to just a \emph{closure} assumption on the subgroups $\eta_i(H_i)$. However, the following example shows that even if $\eta_i(H_i)$ is closed for each $i\in [k]$, it may happen that $H$ is not closed, even assuming that $H_\bullet$ is fiber-transitive. Let $F:=\mc{D}_1(\mb{R})\times \mc{D}_2(\mb{R})$, and let $f:\mb{R}\to \mb{R}$ be a linear map that is not continuous (it is known that such maps can be defined using a Hamel basis for $\mb{R}$ as a vector space over $\mb{Q}$). For each $\lambda\in \mb{R}$ let $\alpha_\lambda$ be the translation  $(x,y)\mapsto (x+\lambda,y+f(\lambda))$ on $F$, and let $H$ denote the group $\{\alpha_\lambda:\lambda\in\mb{R}\}$. The filtration $H=H_0=H_1\ge \{\id\}$ is fiber-transitive on $F$. Indeed, if $\alpha_\lambda$ fixes the $\pi_1$ factor then clearly $\lambda=0$ so there is nothing to prove in this case. Also, $\eta_1(H_1)=\mb{R}$ and $\eta_2(H_2)=\{\id\}$, so both of these groups are closed. But clearly $H$ is not a closed subgroup of $\tran(F)$, by discontinuity of $f$.
\end{example}

\begin{example}
Another seemingly natural condition on $H_\bullet$ that one could have tried to use instead of Definition \ref{def:oldclosed} is to require just that each subgroup $H_i$ is closed in $\tran(\ns)$. However, this leaves open the possibility that at the level of some lower-step factor of $\ns$ a failure of closure occurs, making the quotient of $\ns$ by $\sim_H$ non-viable topologically because some of its factors are not Hausdorff. Indeed, for example, let $\ns$ be the group nilspace $\mb{R}$ with the 2-step filtration $\mb{R}=\mb{R}\ge \mb{Z}\ge \{0\}$. Let $H\subset \tran(\ns)$ be the subgroup of translations $\{ x\mapsto x+\sqrt{2}n: n\in\mb{Z}\}$. Clearly $H$ is a discrete subgroup, hence closed. Moreover, no translation in $H$ is in $\tran_2(\ns)\cong\mb{Z}$, so $H_2=\{\id\}$ and is therefore also closed. However, the 1-step factor of $\ns$ is $\ns_1 = \mb{R}/\mb{Z}$, and $\eta_1(H)$ acts on $\ns_1$ as the group $\{\sqrt{2}n\!\!\mod 1:n\in \mb{Z}\}$, a non-closed subgroup of $\tran_1(\ns_1) = \mb{R}/\mb{Z}$. Thus, the quotient $\ns/H$ (which algebraically is a perfectly valid quotient nilspace) has 1-step factor the group $\sqrt{2}\mb{Z}\backslash (\mb{R}/\mb{Z})\cong  \mb{R}/(\mb{Z}+\sqrt{2}\mb{Z})$, which is not Hausdorff.
\end{example}

\begin{proposition}\label{prop:quot-lcfr-nil}
Let $\ns$ be a $k$-step Lie-fibered nilspace and let $H_\bullet$ be a closed fiber-transitive filtration on $\ns$. Then $\pi_H(\ns)$ equipped with the quotient topology is an \textsc{lch} space.
\end{proposition}

\begin{proof}
Since $\pi_H(\ns)$ is the quotient of an \textsc{lch} space by the action of a group of homeomorphisms, the map $\pi_H$ is continuous and open. Hence $\pi_H(\ns)$ is second-countable, and it is also locally compact \cite[Ex.\ 29.3]{Mu}. To prove that $\pi_H(\ns)$ is Hausdorff it suffices to show that the set $C=\{(x,y)\in\ns^2:x\sim_H y\}$ is closed in the product topology \cite[p.\ 79, Prop.\ 8]{BourbakiGT1}. Since $\ns$ is \textsc{lch}, it is metrizable \cite[Ex.\ 32.3 and Thm.\ 34.1]{Mu}, so closure in $\ns$ is equivalent to sequential closure. Let $((x_n,h_n x_n))_n$ be a sequence in $C$ convergent to $(x,y)\in \ns^2$. By Lemma \ref{lem:re-param-conv-seq} there is a convergent sequence $(h_n')_n$ in $H$ with limit $h\in H$ such that $h_nx_n=h_n'x_n$ for all $n$. Since $x_n\to x$, the continuity of the action implies $h_n' x_n\to hx$, so $y=hx$ and $C$ is closed.
\end{proof}
\noindent Next, we have to prove the two conditions for being an \textsc{lch} nilspace from Lemma \ref{lem:eq-def-lch-nil}, namely, that the cube sets are closed and that the completion function is continuous. 

To check the first of these conditions we will use the following lemma.

\begin{lemma}\label{lem:closure-sum-of-cu-and-h}
Let $\ab$ be a Polish abelian group and let $H$ be a closed subgroup of $\ab$. Then for all integers $n\ge 0$ and $k\ge 1$ the following holds: for every convergent sequence $(g_m)_{m\in \mb{N}}$ in $\cu^n(\mc{D}_k(\ab))+H^{\db{n}}$ (in the product topology on $\ab^{\db{n}}$), there exist convergent sequences $(\q_m)_m$ in $\cu^n(\mc{D}_k(\ab))$ and $(d_m)_m$ in $H^{\db{n}}$ such that $\q_m+d_m=g_m$ for all $m$. In particular, the group $\cu^n(\mc{D}_k(\ab))+H^{\db{n}}$ is a closed subgroup of $\ab^{\db{n}}$.
\end{lemma}

\begin{proof}
Note first that the result is trivial if $k\le n$, as then $\cu^n(\mc{D}_k(\ab))=\ab^{\db{n}}$. Hence, we assume that $n\ge k+1$. We prove the result for $n=k+1$, as the idea of the proof is clearer in this case and other cases are similar. Let $(g_m)_m$ be a convergent sequence in $\cu^{k+1}(\mc{D}_k(\ab))+H^{\db{k+1}}$. Let us define $\q_m'\in \cor^{k+1}(\mc{D}_k(\ab))$ as $\q_m':=g_m|_{\db{k+1}\setminus\{1^{k+1}\}}$ (note that $\q_m'$ is indeed in $\cor^{k+1}(\mc{D}_k(\ab))$ since any map $\db{k}\to \ab$ is in $\cu^k(\mc{D}_k(\ab))$). As $(g_m)$ is convergent, so is $(\q_m')$. For every $m$ let $\q_m^*\in \cu^{k+1}(\mc{D}_k(\ab))$ be the unique completion of $\q_m'$. By continuity of the completion function,  the sequence $(\q_m^*)$ converges to a limit in $\cu^{k+1}(\mc{D}_k(\ab))$ (equal to the completion of the limit of $(\q_m')$). Finally, let us define $f_m:=g_m-\q_m^*\in \ab^{\db{k+1}}$. As a pointwise difference of two convergent sequences, the sequence $(f_m)$ converges. We now prove that $f_m\in H^{\db{k+1}}$ for every $m$, which will complete the proof in this case (since $H$ is closed). First note that clearly $f_m\sbr{v}=0$ for all $v\not=1^{k+1}$ so it suffices to prove that $f_m(1^{k+1})\in H$. Recall from  \cite[Definition 2.2.22]{Cand:Notes1} the Gray-code map $\sigma_{k+1}:\ab^{\db{k+1}}\to\ab$, $r\mapsto \sum_{v\in \db{k+1}} (-1)^{|v|}r\sbr{v}$, where $|v|:=v\sbr{1}+\cdots+v\sbr{k+1}$. We have $(-1)^{k+1}f_m(1^{k+1}) = \sigma_{k+1}(f_m) = \sigma_{k+1}(g_m-\q_m^*)=\sigma_{k+1}(g_m)-\sigma_{k+1}(\q_m^*)$, and this lies in $H$ because both of $\sigma_{k+1}(g_m), \sigma_{k+1}(\q_m^*)$ do (the former term does because by assumption $g_m\in \cu^{k+1}(\mc{D}_k(\ab))+H^{\db{k+1}}$ and $\sigma_{k+1}(\q)=0$ for any $\q\in \cu^{k+1}(\mc{D}_k(\ab))$). This completes the proof in this case $n=k+1$.

For $n\ge k+2$ let us sketch the argument. First, given any such converging sequence $(g_m)_m$ in $\cu^n(\mc{D}_k(\ab))+H^{\db{n}}$ we define\footnote{Recall that given $v\in \db{n}$ we define $|v|:=v\sbr{1}+\cdots+v\sbr{n}$.} $\q_m':=g_m|_{\{v\in \db{n}:|v|\le k\}}\in \ab^{\{v\in \db{n}:|v|\le k\}}$. Note that there exists a unique cube $\q_m^*\in \cu^n(\mc{D}_k(\ab))$ such that $\q_m^*|_{\{v\in \db{n}:|v|\le k\}} = \q_m'$. We then let  $f_m:=g_m-\q_m^*$. Applying the map $\sigma_{k+1}$ on faces of dimension $k+1$ on the function $f_m\in \ab^{\db{n}}$ it follows similarly as before that $f_m\in H^{\db{n}}$. The result follows.
\end{proof}

\begin{proposition}\label{prop:closure-cubes-coset}
Let $\ns$ be a $k$-step Lie-fibered nilspace and let $H_\bullet$ be a closed fiber-transitive filtration on $\ns$. Then for every $n\ge 0$ the cube set $\cu^n(\pi_H(\ns))$ is a closed subset of $\pi_H(\ns)^{\db{n}}$.
\end{proposition}

\begin{proof}
It suffices to prove that $\cu^n(\pi_H(\ns))$ is closed for $n\leq k+1$ (arguing as for  \cite[Lemma 2.1.1]{Cand:Notes2}).
The map $\pi_H^{\db{n}}:\ns^{\db{n}}\to (\pi_H(\ns))^{\db{n}}$ is open, so to prove the closure of $\cu^n(\pi_H(\ns))$ it suffices to prove that $(\pi_H^{\db{n}})^{-1}\cu^n(\pi_H(\ns))=H^{\db{n}} \cdot \cu^n(\ns)$ is closed. Note that it suffices to prove that if $(g_m)$ is a convergent sequence in $H^{\db{n}}\cdot \cu^n(\ns)$ then there exist convergent sequences $(d_m')$ in $H^{\db{n}}$ and $(\q_m')$ in $\cu^n(\ns)$ such that $d_m'\cdot \q_m' = g_m$ for all $m$. We prove this by induction on $k$, the case $k=0$ being trivial. 

For $k>0$, we first prove that $H^{\db{n}} \cdot \cu^n(\ns)$ is (algebraically) an abelian bundle over $\eta_{k-1}(H)^{\db{n}} \cdot \cu^n(\ns_{k-1})$ with structure group $\cu^n(\mc{D}_k(\ab_k))+H_k^{\db{n}}$, where $\ab_k=\ab_k(\ns)$.

For $n\le k$ note that the claimed structure group $\cu^n(\mc{D}_k(\ab_k))+H_k^{\db{n}}$ is simply $\ab_k^{\db{n}}$. But if $d\cdot \q, d'\cdot \q'\in  H^{\db{n}} \cdot\cu^n(\ns)$ have the same image under $\pi_{k-1}^{\db{n}}$ then they indeed differ in an element $f\in \ab_k^{\db{n}}$ by definition, so our claim holds.

For $n= k+1$, suppose again that $d\cdot \q, d'\cdot \q'\in H^{\db{n}}\cdot \cu^n(\ns)$ are equal modulo $\pi_{k-1}^{\db{n}}$. For each $v\in \db{k+1}\setminus\{1^{k+1}\}$, let $f(v)\in \ab_k$ be the unique element such that $(d\cdot \q)(v)+f(v)=(d'\cdot\q')(v)$. By uniqueness of $(k+1)$-corner completion on $\mc{D}_k(\ab_k)$, the map $f$ extends uniquely to an element $f\in \cu^n(\mc{D}_k(\ab_k))$. Since the action of $\ab_k$ commutes with translations, we can assume  (renaming $\q+f$ as $\q$) that $d\cdot \q(v)=d' \cdot \q'(v)$ for every $v\neq 1^{k+1}$, and $\pi_{k-1}(d\cdot \q(1^{k+1}))=\pi_{k-1}(d' \cdot \q'(1^{k+1}))$. Hence there is $z\in \ab_k$ such that $(d\cdot \q)(1^{k+1})=(d'\cdot \q')(1^{k+1})+z$. We claim that $z\in H_k$. Let us denote by $\q_z\in \ab_k^{\db{k+1}}$ the map that equals $z$ at  $1^{k+1}$ and 0 elsewhere. Then $\q+\q_z=d^{-1}d'\cdot \q'$, and we want to deduce that $z\in H_k$. Note that $(d^{-1}d'\cdot \q')|_{\db{k+1}\setminus\{1^{k+1}\}} = \q|_{\db{k+1}\setminus\{1^{k+1}\}}$, so this is in $\cor^{k+1}(\ns)$. By the last sentence of Lemma \ref{lem:cong-equi}, there exists $d^*\in \cu^{k+1}(H_\bullet)$ such that $(d^{-1}d'\cdot \q')|_{\db{k+1}\setminus\{1^{k+1}\}} = (d^*\cdot \q')|_{\db{k+1}\setminus\{1^{k+1}\}}$. Let $\gamma:=(d^{-1}d'(d^*)^{-1})(1^{k+1})\in H$, and let $d_\gamma\in H^{\db{k+1}}$ be the map with $d_\gamma(1^{k+1})=\gamma$ and $d_\gamma(v)=\id_H$ otherwise. Then $d^{-1}d'\cdot \q' = d_\gamma d^*\cdot \q'$.  Hence $\q+\q_z = d_\gamma d^*\cdot \q'$. But $\q$ and $d^*\cdot \q'$ are elements of $\cu^{k+1}(\ns)$ that are equal when restricted to $\db{k+1}\setminus\{1^{k+1}\}$. By uniqueness of completion we then have $\q(1^{k+1})+z= \gamma(d^*\cdot\q'(1^{k+1})) = \gamma(\q(1^{k+1}))$. By property \eqref{eq:ft-fil}, this implies that $z\in H_k$.

Now let $(d_m\cdot \q_m)$ be a convergent sequence in $H^{\db{n}}\cdot \cu^n(\ns)$. This implies that the sequence $(\eta_{k-1}(d_m)\cdot \pi_{k-1}(\q_m))$ converges in $\eta_{k-1}(H)^{\db{n}}\cdot \cu^n(\ns_{k-1})$. By induction, there exist $d_m'\in H^{\db{n}}$ and $\q_m'\in \cu^n(\ns_{k-1})$ such that the sequence  $\eta_{k-1}(d_m')$ converges in $\eta_{k-1}(H)^{\db{n}}$, $(\pi_{k-1}(\q_m'))$ converges in $\cu^n(\ns_{k-1})$, and $\eta_{k-1}(d_m')\pi_{k-1}(\q_m')=\eta_{k-1}(d_m)\cdot \pi_{k-1}(\q_m)$. By Lemma \ref{lem:conv-polish-group} we can assume that $d_m'$ converges in $H^{\db{n}}$ and by Lemma \ref{lem:open-projection-of-cubes} we can assume that $\q_m'$ converges in $\cu^n(\ns)$ as well. Since $d_m'\cdot \q_m'$ is an element of $H^{\db{n}}\cdot \cu^n(\ns)$ equal to $d_m\cdot \q_m$ modulo $\pi_{k-1}$, by the previous paragraph  we have $d_m'\cdot \q_m'-d_m\cdot \q_m\in \cu^n(\mc{D}_k(\ab_k))+H_k^{\db{n}}$. As both $(d_m'\cdot \q_m')$ and $(d_m\cdot \q_m)$ converge, so does $(d_m'\cdot \q_m'-d_m\cdot \q_m)$. Moreover, by Lemma \ref{lem:closure-sum-of-cu-and-h} there exist convergent sequences $(f_m)$ in $\cu^n(\mc{D}_k(\ab_k))$ and $(\gamma_m)$ in $H_k^{\db{n}}$ such that $d_m'\cdot \q_m'-d_m\cdot \q_m=f_m+\gamma_m$. In particular $d_m\cdot \q_m = d_m'\cdot\q_m'-f_m-d_m = (d_m'-\gamma_m)\cdot (\q_m'-f_m)$. But $(\q_m'-f_m)$ is a convergent sequence in the closed set $\cu^n(\ns)$, with limit $\q$ say, and so is  $(d_m'-\gamma_m)$ in $H^{\db{n}}$, with limit $d$, so $d_m\cdot \q_m$ converges to $d\cdot \q\in H^{\db{n}}\cdot \cu^n(\ns)$, which proves the claimed closure.
\end{proof}

Before continuing, note that we have the following commutative diagram for any $n\le k$:
\begin{equation}\label{diag:quots}
\begin{aligned}[c]
\begin{tikzpicture}
  \matrix (m) [matrix of math nodes,row sep=3em,column sep=4em,minimum width=2em]
  {\ns & \ns_{n}  \\
     \pi_H(\ns) & \pi_{\eta_{n}(H)}(\ns_{n}). \\};
  \path[-stealth]
    (m-1-1) edge node [above] {$\pi_n$} (m-1-2)
    (m-1-1) edge node [right] {$\pi_H$} (m-2-1)
    (m-1-2) edge node [right] {$\pi_{\eta_{n}(H)}$} (m-2-2)
    (m-2-1) edge node [above] {$\pi$} (m-2-2);
\end{tikzpicture}
\end{aligned}
\end{equation}

\vspace{-0.3cm}

\noindent We know from the algebraic part of the theory that this diagram commutes. Moreover $\pi_H, \pi_n$ and $\pi_{\eta_{n}(H)}$ are continuous open maps. Hence so is $\pi$. As this map is also the $n$-step factor map of the (for now only algebraic) nilspace $\pi_H(\ns)$, we will denote it also by $\pi_{n}$ (it will be clear from the context which domain space a map $\pi_{n}$ is considered on, e.g.\ $\ns$ or $\pi_H(\ns)$).

We now turn to proving the continuity of corner completion on $\pi_H(\ns)$. According to Lemma \ref{lem:eq-def-lch-nil} we need to prove this continuity for the completion function $\mc{K}_{n+1}$ for each $n\in [k]$. But we can make the following observation straighataway.
\begin{remark}\label{rem:comdiag}
It suffices to prove continuity $\mc{K}_{k+1}$ on $\pi_H(\ns)$. Indeed, then the same result applied with $k=n$ will yield the continuity of completion of $\mc{K}_{n+1}$ on $\pi_{\eta_{n}(H)}(\ns_{n})$, for any $n\le k$, when considering $\pi_{\eta_{n}(H)}(\ns_{n})$ as the quotient of $\ns_n$ modulo the congruence  induced by $\eta_n(H)$ (with the quotient topology). Now note that what we actually need is continuity of $\mc{K}_{n+1}$ on $\pi_{\eta_n(H)}(\ns_n)$ relative to the topology on the latter space obtained as the quotient of the $\pi_H(\ns)$ topology under the map $\pi_n:\pi_H(\ns)\to \pi_{\eta_n(H)}(\ns_n)$. But by the above-mentioned commutativity of \eqref{diag:quots}, the two quotient topologies on $\pi_{\eta_{n}(H)}(\ns_{n})$ (coming from $\pi_H(\ns)$ and $\ns_n$) are equal.
\end{remark}

To prove continuity of $\mc{K}_{k+1}$ on $\pi_H(\ns)$ we will use the following lemma.

\begin{lemma}\label{lem:cont-diff-map-coset}
Let $\ns$ be a $k$-step Lie-fibered nilspace, with $k$-th structure group $\ab_k$, and let $H_\bullet$ be a closed fiber-transitive filtration on $\ns$. Then $\pi_H(\ns)$ is a Cartan principal $\ab_k/H_k$-bundle over $\pi_{k-1}(\pi_H(\ns))$.
\end{lemma}

\begin{proof}
By Lemma \ref{lem:cong-equi}, algebraically $\pi_H(\ns)$ is a $\ab_k/H_k$-bundle over $\pi_{k-1}(\pi_H(\ns))$. By the paragraph below \eqref{diag:quots} we know that $\pi_{k-1}:\pi_H(\ns)\to \pi_{\eta_{k-1}(H)}(\ns_{k-1})$ is an open continuous map. In order to prove that $\pi_H(\ns)$ is a principal $\ab_k/H_k$-bundle it suffices then to prove that the addition map $A_H:\pi_H(\ns)\times (\ab_k/H_k)\to \pi_H(\ns)$ defined as $(\pi_H(x),z+H_k)\mapsto \pi_H(x+z)$ is continuous. We have the following commutative diagram:
\begin{equation}\label{diag:quot-sum}
\begin{aligned}[c]
\begin{tikzpicture}
  \matrix (m) [matrix of math nodes,row sep=3em,column sep=4em,minimum width=2em]
  {\ns\times \ab_k & \ns  \\
     \pi_H(\ns)\times \ab_k/H_k & \pi_H(\ns), \\};
  \path[-stealth]
    (m-1-1) edge node [above] {$A$} (m-1-2)
    (m-1-1) edge node [right] {$\pi_H\times \pi_{H_k}$} (m-2-1)
    (m-1-2) edge node [right] {$\pi_{H}$} (m-2-2)
    (m-2-1) edge node [above] {$A_{H}$} (m-2-2);
\end{tikzpicture}
\end{aligned}
\end{equation}

\vspace{-0.3cm}

\noindent where $A:\ns\times \ab_k\to \ns$ is the addition action of $\ab_k$ on $\ns$ and $\pi_{H_k}:\ab_k\to \ab_k/H_k$ is the quotient map. As $A$ and $\pi_H$ are continuous and $\pi_H\times \pi_{H_k}$ is open it follows that $A_H$ is continuous.

It remains to check continuity of the difference map. As $\pi_H(\ns)$ is equipped with an \textsc{lch} topology by Proposition \ref{prop:quot-lcfr-nil}, let $(\pi_H(x_n),\pi_H(x_n')) \to (\pi_H(x),\pi_H(x'))$ for some elements $x_n,x_n',x,x'\in \ns$ such that $(\pi_H(x_n),\pi_H(x_n')),(\pi_H(x),\pi_H(x'))$ are in the fiber-product $(\pi_H(\ns))\times_{\ns_{k-1}}(\pi_H(\ns))$. By Lemma \ref{lem:conv-quot} we can assume without loss of generality that $x_n\to x$ and $x_n'\to x'$ in $\ns$ as $n\to\infty$. Since $\pi_{k-1}(\pi_H(x_n)) = \pi_{k-1}(\pi_H(x'_n))$ there exists $z_n\in \ab_k$ such that $\pi_H(x_n) = \pi_H(x_n')+z_n+H_k= \pi_H(x_n'+z_n)$. Hence, for every $n\in \mb{N}$ there exists $\gamma_n\in H$  such that $\gamma_n(x_n)=x_n'+z_n$. In particular, applying $\pi_{k-1}$ on both sides of this equation we have $\eta_{k-1}(\gamma_n)(\pi_{k-1}(x_n)) = \pi_{k-1}(x_n')$. By Lemma \ref{lem:re-param-conv-seq} there exists $\gamma_n'\in H$ such that $\eta_{k-1}(\gamma_n)(\pi_{k-1}(x_n)) = \eta_{k-1}(\gamma_n')(\pi_{k-1}(x_n))$ and $\eta_{k-1}(\gamma_n')\to \eta_{k-1}(\gamma)$ for some $\gamma\in H$ as $n\to\infty$. By Lemma \ref{lem:conv-polish-group} we can assume without loss of generality that indeed $\gamma_n'\to \gamma$ in $\tran(\ns)$ as $n\to\infty$. Hence, since now $\pi_{k-1}(\gamma_n(x_n)) = \pi_{k-1}(\gamma_n'(x_n))$, there exists $g_n\in \ab_k$ such that $\gamma_n(x_n) = \gamma_n'(x_n)+g_n$. By \eqref{eq:ft-fil} we have  $g_n\in H_k$. Therefore $\gamma_n'(x_n)+g_n = \gamma_n(x_n)=x_n'+z_n$. This implies that $\gamma_n'(x_n)-x_n' = z_n-g_n$. By construction, and continuity of difference on $\ns$, the left hand side of this equation converges (as $\gamma_n'\to \gamma$, $x_n\to x$ and $x_n'\to x'$). Hence $z_n-g_n\to \gamma(x)-x'$. Since $\pi_H(x_n)-\pi_H(x_n') = z_n+H_k = z_n-g_n+H_k$, we deduce that this converges to $(\gamma(x)-x')+H_k=\pi_H(x)-\pi_H(x')$.
\end{proof}
\noindent We shall also need a result telling us that convergent sequences of cubes in $\pi_H^{\db{n}}(\cu^n(\ns))$ can be lifted to convergent sequences in $\cu^n(\ns)$. For technical reasons we establish the following lifting result that holds for more general objects than cubes. Recall from \cite{Cand:Notes1} the notions of discrete-cube morphism $\phi:\db{m}\to\db{n}$ and of simplicial subsets of $\db{n}$. Recall also from the paragraph preceding Lemma \ref{lem:simplicial-proy} the definition of $\hom(S,\ns)$ for a simplicial set $S$ and a nilspace $\ns$. As usual, when $\ns$ is \textsc{lch} we equip $\hom(S,\ns)$ with the subspace topology induced by the product topology on $\ns^S$.
\begin{proposition}\label{prop:lift-of-simplicial}
Let $\ns$ be a $k$-step Lie-fibered nilspace, let $H_\bullet$ be a closed fiber-transitive filtration on $\ns$, and let $S$ be a simplicial subset of $\db{n}$. Then for every convergent sequence $(\pi_H\co\q_n)_{n\in\mb{N}}$ in $\hom(S,\pi_H(\ns))$,  there exists $\q_n'\in \hom(S,\ns)$ such that $\pi_H\co\q_n=\pi_H\co \q_n'$ and $(\q_n')_n$ converges in $\hom(S,\ns)$.
\end{proposition}

\begin{proof}
For $v\in \db{n}$ recall that $|v|:=v\sbr{1}+\cdots+v\sbr{n}$. It suffices to prove the result for any $S\subset \db{n}$ such that $|v|\le k$ for any $v\in S$. Indeed, if this case holds, then given a simplicial set $S'\subset \db{n}$, let $S^*:=\{v\in S:|v|\le k\}$, and note that if a  sequence $(\pi_H\co \q_n)_n$ in $\hom({S'},\pi_H(\ns))$ converges then $(\pi_H\co\q_n|_{S^*})_n$ in $\hom({S^*},\pi_H(\ns))$  converges as well. Then by assumption there is a sequence $(\q_n')_n$ in $\hom({S^*},\ns)$ such that $\pi_H\co\q_n' = \pi_H\co\q_n|_{S^*}$ and such that $(\q_n')_n$ converges. It is then not hard to see that $\q_n'$ can be extended uniquely and continuously to an element of $\hom({S'},\ns)$ in such a way that these extensions also converge as $n\to\infty$. Abusing the notation let us consider now $\q_n'\in \hom({S'},\ns)$ such that $(\q_n')_n$ converges and $\pi_H\co\q_n'|_{S^*} = \pi_H\co \q_n|_{S^*}$. By uniqueness of completion we have  $\pi_H\co\q_n' = \pi_H\co \q_n$, so the sequence $(\q_n')_n$ satisfies the desired conclusion.

We have thus indeed reduced the task to proving that, given a $k$-step Lie-fibered nilspace $\ns$, the proposition's conclusion holds for simplicial sets $S\subset \db{n}$ such that $|v|\le k$ for all $v\in S$. We prove this by induction on $k$. The case $k=0$ is trivial. For $k>0$, assuming the case $k-1$, let $\q_n\in \hom(S,\ns)$ be such that $\pi_H\co\q_n$ converges as $n\to\infty$. Then $\pi_{k-1}\co\pi_H\co\q_n = \pi_{\eta_{k-1}(H)}\co \pi_{k-1}\co \q_n$ is a convergent sequence in $\hom\big(S,\pi_{\eta_{k-1}(H)}(\ns_{k-1})\big)$. By induction there exists $\pi_{k-1}\co \q_n'\in \hom(S,\ns_{k-1})$ such that $(\pi_{k-1}\co\q_n')_n$ converges and $\pi_{\eta_{k-1}(H)} \co \pi_{k-1}\co \q_n' = \pi_{\eta_{k-1}(H)}\co \pi_{k-1}\co \q_n$. As $S$ does not contain elements $v$ with $|v|>k$, note that we can correct pointwise $\q_n'\in \hom(S,\ns)$ using Lemma \ref{lem:conv-quot} by adding an element $d_n\in \ab_k^S$ in such a way that the sequence $(\q_n'+d_n)_n$ still lies in $\hom(S,\ns)$ (this is because $\hom(S,\mc{D}_k(\ab_k))=\ab_k^S$, since $|v|\le k$ for all $v\in S$) but now converges as well. Thus, without loss of generality we can assume that $\q_n'$ itself converges in $\hom(S,\ns)$.

The cube $\pi_H\co\q_n'$ is almost what we need. Indeed, we know that $\pi_H\co\q_n' = \pi_H\co \q_n+\overline{z_n}$ for some $\overline{z_n}\in (\ab_k/H_k)^{S}$. This follows from the fact that $\pi_H\co\q_n', \pi_H\co \q_n$ are equal modulo the factor map $\pi_{k-1}:\pi_H(\ns)\to \pi_{\eta_{k-1}(H)}(\ns_{k-1})$. But then note that $\overline{z_n} = \pi_H\co\q_n'-\pi_H\co \q_n$ and both sequences $(\pi_H\co \q_n)_n$ and $(\pi_H\co \q_n')_n$ converge by hypothesis. Hence, by Lemma \ref{lem:cont-diff-map-coset} we have that $\overline{z_n}$ converges as well. By Lemma \ref{lem:conv-quot} there is a convergent sequence $(z_n)_n$ in $\ab_k^S$ such that $z_n+H_k^S = \overline{z_n}$ for every $n$. Then $\q_n'-z_n$ is in $\hom(S,\ns)$ (again using that $\hom(S,\mc{D}_k(\ab_k))=\ab_k^S$), converges, and satisfies $\pi_H\co (\q_n'-z_n)=\pi_H\co \q_n$, as required.
\end{proof}

\begin{proposition}\label{prop:cont-completion-coset}
Let $\ns$ be a $k$-step Lie-fibered nilspace and let $H_\bullet$ be a closed fiber-transitive filtration on $\ns$. Then the completion function $\mc{K}:\cor^{k+1}(\pi_H(\ns))\to \pi_H(\ns)$ is continuous.
\end{proposition}

\begin{proof}
Let $\pi_H\co\q_n'$ be a convergent sequence in $\cor^{k+1}(\pi_H(\ns))$. By Proposition \ref{prop:lift-of-simplicial} we can assume that $\q_n'\in \cor^{k+1}(\ns)$ and  converges to $\q'\in \cor^{k+1}(\ns)$. For each $n\in \mb{N}$ let $\q_n\in \cu^{k+1}(\ns)$ be the unique completion of $\q_n'$. Then by continuity of completion on $\ns$, the cubes $\q_n$ converge to the cube $\q\in\cu^{k+1}(\ns)$ which completes $\q'$. But then $\mc{K}(\pi_H\co\q_n')=\pi_H(\q_n(1^{k+1}))$, which converges (by continuity of $\pi_H$) to $\pi_H(\q(1^{k+1}))=\mc{K}(\pi_H\co \q')$.
\end{proof}

\begin{theorem}\label{thm:coset-quot-closed}
Let $\ns$ be a $k$-step Lie-fibered nilspace, and let $H_\bullet$ be a closed fiber-transitive filtration on $\ns$. Then $\pi_H(\ns)$ equipped with the quotient topology is a $k$-step  Lie-fibered nilspace.
\end{theorem}

\begin{proof}
Since $\sim_H$ is a groupable congruence, the algebraic part (i.e.\ that $\pi_H(\ns)$ is a $k$-step nilspace) follows from Lemma \ref{lem:gpequivcong} (as used in Lemma \ref{lem:cong-equi}). 

The topology on $\pi_H(\ns)$ is \textsc{lch} by Proposition \ref{prop:quot-lcfr-nil}. Now we want to check that the hypotheses of Lemma \ref{lem:eq-def-lch-nil} are satisfied. The cube sets on $\pi_H(\ns)$ are closed, by Proposition \ref{prop:closure-cubes-coset}. The continuity of the completion of $(k+1)$-corners in $\pi_H(\ns)$ follows from Proposition \ref{prop:cont-completion-coset}. As explained in Remark \ref{rem:comdiag}, this implies that  $\mc{K}_{n+1}$ is continuous on $\pi_{\eta_n(H)}(\ns_n)$. This shows that $\pi_H(\ns)$ is an \textsc{lch} $k$-step nilspace.

Finally, we need to prove that the structure groups of $\pi_H(\ns)$ are Lie groups. Let us prove that $\ab_k(\pi_H(\ns)) \cong \ab_k(\ns)/H_k$ (as topological groups) and note that for any $i\in[k]$ the proof of $\ab_i(\pi_{\eta_i(H)}(\ns_i)) \cong \ab_i(\ns)/\eta_i(H_i)$ follows similarly. The map $\pi_H:\ns\to \pi_H(\ns)$ is a continuous and open fibration. Restricting to a fiber\footnote{Recall from Proposition \ref{prop:k-th-group-action} the definition of the $k$-th structure group of a $k$-step  \textsc{lch} nilspace.} of $\pi_{k-1}$, we deduce that the $k$-th structure morphism $\phi_k:\ab_k(\ns)\to \ab_k(\pi_H(\ns))$ is a continuous surjective homomorphism. The kernel of this map is precisely $H_k$. Indeed, for any $z\in \ab_k(\ns)$ we have $\phi_k(z)=0$ if and only if for any fixed $e\in \ns$ we have that $\pi_H(e+z)=\pi_H(e)$. Then, by \eqref{eq:ft-fil} we have that $z\in H_k$. By the Open mapping theorem for Polish groups we have that $\ab_k(\ns)/H_k \cong \ab_k(\pi_H(\ns))$ as topological groups. But now both $\ab_k(\ns)$ and $H_k$ are Lie groups (the first by hypothesis and the second because by Cartan's theorem it is a closed subgroup of a Lie group). Thus $\ab_k(\pi_H(\ns))$ is a Lie group.
\end{proof}

With the same assumptions as in Theorem \ref{thm:coset-quot-closed}, if in addition $\eta_i(H_i)$ is discrete and cocompact for all $i\in[k]$ then the quotient Lie-fibered nilspace is in fact a \textsc{cfr} nilspace. To detail this let us extend the notions of fiber discreteness and cocompactness, from free nilspaces (as in Definition \ref{def:FDCA}) to all Lie-fibered nilspaces.

\begin{defn}\label{def:FDCA-general}
Let $\ns$ be a $k$-step Lie-fibered nilspace and let $\Gamma$ be a fiber-transitive group on $\ns$. We say that $\Gamma$ is \emph{fiber-discrete} if for every $j\in[k]$, the group $\eta_j(\Gamma)\cap \tran_j(\ns_j)$ is a discrete subgroup of $\tran_j(\ns_j)$. We say that $\Gamma$ is \emph{fiber-cocompact} if for every $j\in[k]$, the group $\eta_j(\Gamma)\cap \tran_j(\ns_j)$ is cocompact in $\tran_j(\ns_j)$.
\end{defn}

\begin{corollary}\label{cor:quo-cfr-nil}
Let $\ns$ be a $k$-step Lie-fibered nilspace, and let $H_\bullet$ be a closed fiber-transitive and fiber-cocompact filtration on $\ns$. Then $\pi_H(\ns)$ with the quotient topology is a $k$-step \textsc{cfr} nilspace.
\end{corollary}

\begin{proof}
By Theorem \ref{thm:coset-quot-closed} we have that $\pi_H(\ns)$ is a $k$-step Lie-fibered nilspace. To prove that $\pi_H(\ns)$ is a compact nilspace, by \cite[Definition 1.0.2]{Cand:Notes2} it suffices to prove that the topology on $\pi_H(\ns)$ is compact (we already know that $\cu^n(\pi_H(\ns))$ is closed for each $n$, by Theorem \ref{thm:coset-quot-closed}).

We argue by induction on $k$, with the case $k=0$ being trivial. Assuming the case $k-1$, by \cite[Theorem 3.3]{Gleason} we have that $\ns$ is a $Z_k/H_k$-bundle over $\ns_{k-1}$ where $Z_k/H_k$ is a compact abelian Lie group. Thus for every $\pi_{k-1}(x)\in \ns_{k-1}$ there exists an open set $U\subset \ns_{k-1} $ such that $\pi_{k-1}^{-1}(U)$ is homeomorphic to $U\times Z_k/H_k$. For every $\pi_{k-1}(x)$ let $U_{\pi_{k-1}(x)}$ be such an open set. By compactness of $\ns_{k-1}$ there is a finite cover $(U_{\pi_{k-1}(x_j)})_{j=1}^b$. Then, given any open covering $\cup_{i\in I} V_i$ of $\ns$, by taking intersections we can assume that each $V_i$ is included in one of the open sets $\pi_{k-1}^{-1}(U_{\pi_{k-1}(x_j)})$.

By Lebesgue's number lemma, there exists $\delta>0$ such that if a subset $C\subset \ns_{k-1}$ has diameter $\diam(C)<\delta$, then $C$ is included in $U_{\pi_{k-1}(x_j)}$ for some $j\in [b]$. By compactness of $\ns_{k-1}$, there exists a finite covering of $\ns_{k-1}$ by balls of the form $B(\pi_{k-1}(x'_s),\delta/2)$, $s\in [m]$, for some $m\in \mb{N}$. The closed (hence compact)  balls $\overline{B(\pi_{k-1}(x'_s),\delta/2)}$ clearly cover $\ns_{k-1}$, and each of these compact sets is included in some $U_{\pi_{k-1}(x_{j_s})}$ for some $j_s\in[b]$. Thus, for each $s\in[m]$, since $\overline{B(\pi_{k-1}(x'_s),\delta/2)}\times (\ab_k/H_k)\subset U_{\pi_{k-1}(x_{j_s})}\times (\ab_k/H_k)$ and $U_{\pi_{k-1}(x_{j_s})}\times (\ab_k/H_k)$ is homeomorphic to $\pi_{k-1}^{-1}(U_{\pi_{k-1}(x_{j_s})})$, we have that $\cup_{i\in I} [V_i\cap \pi_{k-1}^{-1}(U_{\pi_{k-1}(x_{j_s})})]$ is an open covering of the compact set $\overline{B(\pi_{k-1}(x'_s),\delta/2)}\times (\ab_k/H_k)$. Thus, we can find a finite sub-cover, with corresponding finite index set $I_s\subset I$. Repeating this for every $s\in [m]$ and collecting, we deduce that $\cup_{s\in[m]}\cup_{i\in I_s} V_i$ is a finite sub-cover of $\cup_{i\in I} V_i$.
\end{proof}

\begin{example}\label{ex:fiber-co-comp-not-co-comp}
Given a nilspace $\ns$ and a closed fiber-transitive group $H\leq \tran(\ns)$ that is also fiber-cocompact, in general $H$ need not be a cocompact subgroup of $\tran(\ns)$. For example, let $F=\mc{D}_1(\mb{R})\times \mc{D}_2(\mb{R})$ and consider $H=\langle \alpha,\beta \rangle$ where $\alpha(x,y)=(x+1,y)$ and $\beta(x,y)=(x,y+1)$. This group $H$ is straightforwardly checked to be fiber-cocompact in $\tran(F)$. On the other hand we have $\tran(F)=\{(x,y)\mapsto(x+a,y+b+cx):a,b,c\in \mb{R}\}$. Note that this group is isomorphic to the Heisenberg group $\mc{H}= \begin{psmallmatrix} 1 & \mb{R} & \mb{R}\\[0.1em]0  & 1 & \mb{R}\\[0.1em] 0 & 0 & 1 \end{psmallmatrix}$. Letting $\gamma_{a,b,c}\in \tran(F)$ be the map $(x,y)\mapsto (x+a,y+b+cx)$,  it is straightforwardly checked that the map $\varphi:\tran(F)\to \mc{H}$, $\gamma_{a,b,c}\mapsto \begin{psmallmatrix} 1 & c & b\\[0.1em]0  & 1 & a\\[0.1em] 0 & 0 & 1 \end{psmallmatrix}$ is a group isomorphism. Thus, in order to show that $H\subset \tran(F)$ is not co-compact, it suffices to check that $\varphi(H)\subset \mc{H}$ is not co-compact. A simple calculation shows that $\varphi(H)=\langle \varphi(\alpha),\varphi(\beta) \rangle=\begin{psmallmatrix} 1 & 0 & \mb{Z}\\[0.1em]0  & 1 & \mb{Z}\\[0.1em] 0 & 0 & 1 \end{psmallmatrix}$ is a subgroup of $R:= \begin{psmallmatrix} 1 & 0 & \mb{R}\\[0.1em]0  & 1 & \mb{R}\\[0.1em] 0 & 0 & 1 \end{psmallmatrix}$. Thus, if $\varphi(H)$ is co-compact then  $R$ should also be co-compact. But $R$ is a normal subgroup of $\mc{H}$ and a direct computation shows that $\mc{H}/R\cong \mb{R}$. Hence $\varphi(H)$ cannot be co-compact.
\end{example}

\begin{remark}
One may naturally wonder to what extent the closure property in Definition \ref{def:oldclosed} is \emph{necessary} to ensure that the quotient of $\ns$ by $\sim_H$ is a Lie-fibered nilspace. It turns out that for $\ns$ a \emph{free} nilspace, a variant  condition can be proved to be necessary and sufficient. More precisely, if $F$ is a $k$-step free nilspace and $H$ is a fiber-transitive group on $F$, it can be proved that $\pi_H(F)$ is a Lie-fibered nilspace if and only if for every $i\in [k]$ the group $\eta_i(\compl{H})$ is a closed subgroup of $\tran(F_i)$. We shall not need this fact in this paper, so let us just outline the argument here. First one proves that more generally if $H$ is a fiber-transitive group on a Lie-fibered nilspace, then $\compl{H}=\ker(\wh{\pi_H})$ and so $\compl{H}$ must be a closed subgroup of $\tran(F)$. Then, the same fact applied on $F_i$ implies that $\compl{\eta_i(H)}$ is closed in $\tran(F_i)$. This yields the necessity of the above condition. Then, using that $F$ is a free nilspace, the additional property that for every $i\in [k]$ we obtain the additional fact that $\eta_i(\compl{H})=\compl{\eta_i(H)}$, and this implies the sufficiency of the above condition (indeed we then need not ask for all the terms in the filtration of ${\compl{H}}$ to be closed, because the filtration is $({\compl{H}}\cap \tran_i(F))_{i\in[k]}$, and similarly for lower steps. Hence, as $\tran_i(F)$ is always closed, it suffices to require ${\compl{H}}$ to be closed, and for lower-step factors it suffices to require $\eta_i(\compl{H})$ to be closed).
\end{remark}

\subsection{Pure groupable congruences and embeddings}\hfill

\noindent Recall that for any nilspace $\ns$ we have the canonical homomorphisms $\eta_i:\tran(\ns)\to \tran(\ns_i)$, and note that $
\tran(\ns)=\ker(\eta_0)\geq \ker(\eta_1) \geq \cdots \geq \ker(\eta_{k-1})\geq \ker(\eta_k=\id)=\{\id\}$.
Thus we have the partition
\begin{equation}\label{eq:transparti}
\tran(\ns)=\{\id\}\cup\bigsqcup_{i\in [k]} \ker(\eta_{i-1})\setminus \ker(\eta_i).
\end{equation}
We also have in general $\ker(\eta_{i-1})\supset \tran_i(\ns)$ and this inclusion can be strict, which can often complicate the analysis of the structure of $\ns$. It will therefore be useful to focus on groups of translations which avoid this potential problem.
\begin{defn}[Pure translations]\label{def:pure} Let $\ns$ be a nilspace. We say that a translation $\alpha\in\tran(\ns)$ is \emph{pure} if for every $i\in [k]$ we have $\alpha\in \ker(\eta_{i-1})\Rightarrow \alpha\in\tran_i(\ns)$. Equivalently, $\alpha$ is pure if, for the maximal $i\in [k]$ such that $\alpha\in\ker(\eta_{i-1})$, we have $\alpha\in \tran_i(\ns)$. Thus (using \eqref{eq:transparti}) the set of pure translations is $\{\id\}\cup\bigsqcup_{i\in [k]} \tran_i(\ns)\setminus \ker(\eta_i)$. We say that a group $G\leq\tran(\ns)$ is pure if every translation in $G$ is pure, which is equivalent to the following property: for every $i\in [k]$, we have $G\cap \ker(\eta_{i-1})=G\cap \tran_i(\ns)$. Finally, we say that a groupable congruence $\sim$ on $\ns$ is \emph{pure} if there exists a subfiltration $H_\bullet$ of $\tran(\ns)^{(\sim)}_\bullet$ such that $\sim=\sim_H$ and such that $H$ is pure.
\end{defn}
\noindent Note that, by Proposition \ref{prop:eq-free-action} and Remark \ref{rem:pure-implies-fib-tan} below, pureness of $H$ implies fiber-transitivity.

\begin{remark}\label{rem:different-filtr}
Note that for $\sim$ to be pure we indeed require just \emph{some} fiber-transitive filtration realizing $\sim$ to be pure. Indeed, there can be more than one fiber-transitive filtration yielding the same quotient nilspace. For example, let $F=\mc{D}_1(\mb{Z})\times \mc{D}_2(\mb{Z})$ and let $\alpha(x,y):=(x+2,y)$, $\beta(x,y)=(x,y+2)$ and $\gamma(x,y):=(x,y+2x)$. It is easy to see that $H:=\langle \alpha,\beta\rangle$ and $H':=\langle \alpha,\beta,\gamma\rangle$ generate the same congruence and so $\pi_H(F) \cong \pi_{H'}(F)$ (putting on both $H$ and $H'$ the filtration induced by $\tran_\bullet (F)$). However  $H$ is pure, whereas $H'$ is not because $\gamma$ is not a pure translation (it lies in $\ker(\eta_1)\setminus \tran_2(F))$.
\end{remark}

\noindent The following result establishes equivalences between pureness and other previous properties.

\begin{proposition}\label{prop:eq-free-action}
Let $F=\prod_{i=1}^k\mc{D}_i(\mb{Z}^{a_i}\times\mb{R}^{b_i})$ be a $k$-step free nilspace and let $\Gamma$ be a subgroup of $\tran(F)$. The following properties are equivalent.
\setlength{\leftmargini}{0.7cm}
\begin{enumerate}
\item\label{it:pure-3} $\Gamma$ is fiber-transitive and for every $\gamma\in \Gamma$, letting $(\underline{x_i})_{i\in[k]}\mapsto (\underline{x_1}+T_1,\underline{x_2}+T_2(\underline{x_1}),\ldots)$ be the expression of $\gamma$ given by Theorem \ref{thm:decrip-trans-group}, and $j\in [k]$  minimal such that $T_j\not=0$, the map $T_j$ is constant.
\item\label{it:pure-2} $\Gamma$ is pure, i.e.\ for every $i\in[k]$ we have $\Gamma\cap \ker(\eta_{i-1})=\Gamma\cap\tran_i(F)$.
\item\label{it:pure-1} The action of $\Gamma$ is fiber-transitive and free.
\end{enumerate}
\end{proposition}
Here by the action of $\Gamma$ being \emph{free} we mean the standard notion that for every $\gamma\in\Gamma$, if there exists $x\in F$ such that $\gamma(x)=x$, then $\gamma=\id$.
\begin{proof}
We prove the implications \eqref{it:pure-3} $\implies$ \eqref{it:pure-2} $\implies$ \eqref{it:pure-1} $\implies$ \eqref{it:pure-3}.

To prove \eqref{it:pure-3} $\implies$ \eqref{it:pure-2}, first note that the case $i=k$ follows directly from \eqref{it:pure-3}. Note that the inclussion $\Gamma\cap\tran_{i+1}(F)\subset \Gamma\cap\ker(\eta_i)$ always holds, so it suffices to prove the reverse inclusion. We are going to prove that \eqref{it:pure-2} holds for any $i$ by (reversed) induction.  That is, assume that it holds for $i\ge i_0$ and we want to prove it for $i_0-1$. Let $\gamma\in \Gamma\cap\ker(\eta_{i_0-1})$. By \eqref{it:pure-3}, we have that the action of $\gamma$ in the $i_0$-th degree term is given by a constant. Thus, by the fiber-transitive property, there exists $\gamma'\in \tran_{i_0}(F)\cap\Gamma$ such that $\eta_{i_0-1}(\gamma)=\eta_{i_0-1}(\gamma')$. But then clearly $\gamma'\gamma^{-1}\in\ker(\eta_{i_0})$. By induction, $\gamma'\gamma^{-1}\in \Gamma\cap\tran_{i_0+1}(F)$ and thus $\gamma\in \tran_{i_0}(F)$ and \eqref{it:pure-2} follows.

Let us now show that \eqref{it:pure-2} $\implies$ \eqref{it:pure-1}. Let us start by showing that $\Gamma$ acts freely. We are going to prove this by induction on the step $k$ of the free nilspace. Note that if $k=1$ the result is trivial. Assuming it holds for $k-1$, let $\gamma\in \Gamma$ and $x\in F$ be such that $\gamma(x)=x$. Then clearly $\eta_{k-1}(\gamma)(\pi_{k-1}(x))=\pi_{k-1}(x)$. By induction, we have that $\eta_{k-1}(\gamma)=\id$. But this means that $\gamma\in \ker(\eta_{k-1})$. By \eqref{it:pure-2}, we have that $\gamma\in \Gamma\cap\tran_k(F)$. But $\tran_k(F)$ are just the shifts by a constant by \cite[Lemma 3.2.37]{Cand:Notes1}. As the action of the last structure group is free we have that the only possibility for $\gamma$ to fix $x$ is if $\gamma=\id$.

To prove that $\Gamma$ is fiber-transitive, let $x,y\in F$, $\gamma\in \Gamma$, and $i\in[k]$ be such that $\pi_i(x)=\pi_i(y)$ and $\gamma(x)=y$. It is direct to see that if $\Gamma$ satisfies \eqref{it:pure-2}, then so does $\eta_i(\Gamma)$ for any $i\in[k]$. Note also that $\eta_i(\gamma)(\pi_i(x))=\pi_i(y)=\pi_i(x)$ by assumption. As $\eta_i(\Gamma)$ satisfies \eqref{it:pure-2}, we know by the previous paragraph that $\eta_i(\Gamma)$ acts freely on $F_i$ and therefore $\eta_i(\gamma)=\id$. But this means in turn that $\gamma\in \ker(\eta_i)$, and using \eqref{it:pure-2} we deduce that $\gamma\in \tran_{i+1}(F)$. Hence, $\Gamma$ is fiber-transitive.

To prove \eqref{it:pure-1} $\implies$ \eqref{it:pure-3}, let $\gamma\in \Gamma$ and $j\in[k]$ be such that $T_j\not=0$. Note that without loss of generality we can assume that $j=k$, as otherwise we may simply project to the $j$-th factor. For any fixed $x\in F$ note that $\pi_{k-1}(\gamma(x))=\pi_{k-1}(x)$ by definition of the action of $\gamma$. Hence, by the fiber-transitive property, there must exists $\gamma'\in \Gamma_k$ such that $\gamma(x)=x+\gamma'$. But by \eqref{it:pure-1} the action of $\Gamma$ is free and so $\gamma=\gamma'$ and \eqref{it:pure-3} follows.\end{proof}

\begin{remark}\label{rem:pure-implies-fib-tan}
Properties \eqref{it:pure-2} and \eqref{it:pure-1} can be seen to be equivalent for any $\Gamma\le \tran(\ns)$ where $\ns$ is any $k$-step nilspace. Indeed note that the proof of \eqref{it:pure-2} $\implies$ \eqref{it:pure-1} works for any $\ns$ and $\Gamma$.
\end{remark}

\begin{remark}
Recall from \eqref{eq:ft-fil} that if $\sim$ is a groupable congruence on $\ns$ then, letting $G_\bullet$ be the corresponding filtration on $\big(\tran(\ns)^{(\sim)}\cap \tran_i(\ns)\big)_{i\geq 0}$, the action of $G_i$ on each intersection of a $\sim$-fiber and a $\pi_{i-1}$-fiber is transitive. Hence, if $\sim$ is also pure, then each intersection of a $\sim$-fiber and a $\pi_{i-1}$-fiber is a \emph{principal homogeneous space} of $G_i$.
\end{remark}

\begin{remark}
If $H_\bullet$ is a pure filtration (meaning just that $H$ is pure, as then all subgroups in $H_\bullet$ are also pure) on $\ns$, then for each $i\in[k]$ the filtration  $\eta_j(H_\bullet)$ is pure on $\ns_i$. To confirm this note that by induction it suffices to prove the case $i=k-1$. Furthermore, it suffices to check that if $\eta_{k-1}(g)\pi_{k-1}(x)=\pi_{k-1}(x)$ for some $x\in \ns$ and $g\in H$ then $\eta_{k-1}(g)=\id$. But the element $y:=gx$ then satisfies $x\sim_H y$ and $\pi_{k-1}(x)=\pi_{k-1}(y)$. By \eqref{eq:ft-fil} there is therefore some $g'\in H_k$ such that $y=g'x$. Since the action of $H$ is free by \eqref{it:pure-1}, we conclude that $g=g'\in H_k$, so $\eta_{k-1}(g)=\id$ as required.
\end{remark}

\noindent Pure filtrations are especially useful on free nilspaces because they enable us to embed the corresponding quotient nilspace into a \emph{coset nilspace}. In particular this provides a connection with nilmanifolds that is relevant to the Jamneshan-Tao conjecture (see Proposition \ref{prop:embeb-strong-good} below). To explain this we shall use the following notion.
\begin{defn}\label{def:ctsclosure}
The \emph{continuous closure} of a free nilspace $F=\prod_{i=1}^k\mc{D}_i(\mb{Z}^{a_i})\times \mc{D}_i(\mb{R}^{b_i})$ is the nilspace $F^*=\prod_{i=1}^k\mc{D}_i(\mb{R}^{a_i})\times\mc{D}_i( \mb{R}^{b_i})$. The inclusion map $\iota:F\to F^*$ is a nilspace morphism.
\end{defn}

By Theorem \ref{thm:decrip-trans-group} and \eqref{prop:lift-disc}, \eqref{prop:lift-cont}, it follows that there is an injective filtered homomorphism $\wh{\iota}:\tran(F)\to \tran(F^*)$, consising simply in extending each polynomial involved in the expression \eqref{eq:str-trans-ab-nil} of a translation, and defined on a power of $\mb{Z}$, to the same polynomial defined on the same power of $\mb{R}$. Note that clearly for any $\alpha\in \tran(F)$ and any $x\in F$ we have $\iota(\alpha(x))=\wh{\iota}(\alpha)(\iota(x))$.

\begin{proposition}\label{prop:embeb-strong-good}
Let $F$ be a $k$-step free nilspace, and let $H_\bullet$ be a pure filtration on $F$ that is fiber-discrete\footnote{Hence $H$ is discrete, and thus finitely generated, as a discrete subgroup of a Lie group.}, and such that $\pi_H(F)$ is a compact nilspace. Then there exists a toral nilspace $\ns$ and a continuous injective morphism from $\pi_H(F)$ into $\ns$.
\end{proposition}
\begin{remark}
Recall from \cite[Theorem 2.9.17]{Cand:Notes2} that a toral nilspace is isomorphic to a compact connected filtered nilmanifold $(G/\Gamma,G_\bullet)$ equipped with the coset nilspace structure associated with $G_\bullet$. Thus, under the assumptions of Proposition \ref{prop:embeb-strong-good}, we can embed $\pi_H(F)$ as a \textsc{cfr} nilspace into a connected nilmanifold. See also Remark \ref{rem:2-step-embed} concerning the  2-step case.
\end{remark}

\begin{proof}
We want to prove that the natural inclusions $\iota:F\to F^*$ and $\wh{\iota}:\tran(F)\to \tran(F^*)$ induce a well-defined continuous injective morphism $\varphi:\pi_H(F) \to \pi_{\wh{\iota}(H)}(F^*)$, $\pi_H(x)\mapsto \pi_{\wh{\iota}(H)}(\iota(x))$. Let us first note that $H_\bullet$ being pure implies that $(\wh{\iota}(H_i))_{i\geq 0}$ is also pure, since the latter filtration satisfies the criterion \eqref{it:pure-2} in Proposition \ref{prop:eq-free-action}.

The map $\varphi$ is well-defined, indeed for any $x\in F$ and $\alpha\in H$ (thus $\alpha(x)\sim_H x$) we have  $\varphi(\pi_H(\alpha(x)))=\pi_{\wh{\iota}(H)}(\iota(\alpha(x)))=\pi_{\wh{\iota}(H)}(\wh{\iota}(\alpha)(\iota(x)))=\pi_{\wh{\iota}(H)}(\iota(x))=\varphi(\pi_H(x))$. Moreover, Lemma \ref{lem:cong-equi} implies that $\varphi$ is a morphism, as cubes in $\pi_H(F)$ are of the form $\pi_H\co  \q$ for $\q\in \cu^n(F)$, whence $\varphi\co \pi_H\co \q = \pi_{\wh{\iota}(H)}\co \iota\co \q\in \cu^n(F^*/\sim_{\wh{\iota}(H)})$. The injectivity of the map $\varphi$ is proved as follows: supposing that $\pi_{\wh{\iota}(H)}(\iota(x)) = \pi_{\wh{\iota}(H)}(\iota(y))$, there exists $\alpha\in H$ such that $\iota(y)=\wh{\iota}(\alpha)(\iota(x))= \iota(\alpha(x))$, and since $\iota$ is injective we have $\alpha(x)=y$, so $\pi_H(x)=\pi_H(y)$. Hence $\varphi$ is indeed a morphism  embedding $\pi_H(F)$ into the nilspace $\pi_{\wh{\iota}(H)}(F^*)$.

Finally we have to prove that $\pi_{\wh{\iota}(H)}(F^*)$ is a toral nilspace. Since by assumption $\pi_H(F)$ is compact, its structure groups are the compact abelian groups $\ab_i(F)/\eta_i(H_i)$ for $i\in[k]$ (by Lemma \ref{lem:cong-equi}). By  assumption $\eta_i(H_i)$ is a discrete cocompact subgroup pf $\ab_i(F)$. It is then easy to see that $\eta_i(\wh{\iota}(H_i))$ is also a discrete cocompact subgroup of $\ab_i(F^*)$, so $\ab_i(F^*)/\eta_i(\wh{\iota}(H_i))$ is a torus. Hence $\pi_{\wh{\iota}(H)}(F^*)$ is a toral nilspace.
\end{proof}

\begin{remark}\label{rem:isometric-embedding}
The toral nilspace $\ns$ in which we embed $\pi_H(F)$ can be endowed with an arbitrary metric that generates its topology. Given any such metric $d$ on $\ns$, we can define a metric $d'$ on $\pi_H(F)$ simply by declaring $d'(x,y):=d(\varphi(x),\varphi(y))$. As $\varphi$ is a homeomorphism onto its image this is indeed a metric on $\pi_H(F)$. Relative to $d'$ and $d$, the map $\varphi$ is an isometry.
\end{remark}

\section{\textsc{cfr} nilspaces as quotients of free nilspaces by groupable congruences}\label{sec:groupequivrep}
\noindent In this subsection we prove the following theorem, a central result of this paper, which will also be used in Section \ref{sec:DC} to prove the double-coset representation result, Theorem \ref{thm:cfr=double-coset}.

\begin{theorem}\label{thm:gpcongrep}
Let $\ns$ be a $k$-step \textsc{cfr} nilspace. Then there exists a $k$-step free nilspace $F$ and a closed groupable congruence $\sim$ on $F$ such that $\pi_\sim(F)$, $\ns$ are isomorphic as \textsc{cfr} nilspaces.
\end{theorem}

We shall in fact prove the following version of the result,  which yields additional information on parts of the structure.

\begin{theorem}\label{thm:groupequivrep}
Let $\ns$ be a $k$-step \textsc{cfr} nilspace. Then there exists a $k$-step free nilspace $F$, and a fiber-transitive filtration $H_\bullet$ on $F$ which is  finitely-generated, fiber-discrete and fiber-cocompact, such that $\pi_H(\ns)$ and  $\ns$ are isomorphic as \textsc{cfr} nilspaces.
\end{theorem}

In the proof we use the following fact about finitely generated nilpotent groups.
\begin{lemma}\label{lem:noeth-group}
Let $H$ be a $k$-step nilpotent finitely-generated group. Then for every filtration $H_\bullet=(H_i)_{i\ge 0}$ on $H$ there exist $h_1\ldots,h_n\in H$, and a subset $I_j\subset [n]$ for each $j\in[k]$, such that for every $j\in[k]$ we have $H_j=\langle h_{i}\rangle_{i\in I_j}$ .
\end{lemma}

\begin{proof}
The group $H$ is Noetherian. Thus, every subgroup is finitely generated, and we can take as our set $h_1,\ldots,h_n\in H$ the union of all those generators for $j\in[k]$.
\end{proof}

\begin{proof}[Proof of Theorem \ref{thm:groupequivrep}]
We argue by induction on the step $k$. For $k=1$ the result follows from the standard theory of compact abelian Lie groups.

Let $k>1$ and suppose by induction that $\ns_{k-1}$ satisfies the theorem's conclusion, so there is a $(k-1)$-step free nilspace $F_{k-1}$, and a finitely-generated fiber-transitive filtration $H_\bullet$ on $F_{k-1}$, fiber-discrete and fiber-cocompact, with corresponding quotient map the fibration $\varphi_{k-1}:F_{k-1}\to\ns_{k-1}$, such that the induced map $\overline{\varphi_{k-1}}:\pi_H(F_{k-1})\to \ns_{k-1}$ is an isomorphism of compact nilspaces. Recall from \eqref{eq:etaj} the definition of the maps $\eta_j$. By Corollary \ref{cor:equivalent-filtrations}  the filtration $(H\cap \tran_j(F_{k-1}))_{j\ge 0}$ defines the same quotient nilspace as $H_\bullet$.  By Corollary \ref{cor:same-last-str-gr} we then have  $\eta_j(H_j)=\eta_j(H\cap \tran_j(F_{k-1}))$ for every $j\in [k-1]$, so in particular for every such $j$ the group $\eta_j(H\cap \tran_j(F_{k-1}))$ is discrete and cocompact in $\ab_j(F_{k-1})$. Hence, we can assume that the group $H$ is finitely-generated, fiber-discrete and fiber-cocompact when considered with filtration $\big(H\cap \tran_j(F_{k-1})\big)_{j\geq 0}$, and that $\pi_H(F_{k-1})\cong \ns_{k-1}$. By Lemma \ref{lem:noeth-group} there is a set $\{h_1,\ldots,h_n\}$ generating $H$ and such that for every $j\in[k]$ the subgroup $H\cap \tran_j(F_{k-1})$ is generated by a subset of $\{h_1\ldots,h_n\}$. Hence, from now on we will assume that $H_j=H\cap \tran_j(F_{k-1})$ for all $j\in[k]$ and that $H_j$ is generated by some elements in $\{h_1,\ldots,h_n\}$.

Let us now consider the fiber product of $F_{k-1}$ and $\ns$, i.e., the nilspace $\nss:= F_{k-1}\times_{\ns_{k-1}}\ns$ $=\{(f,x)\in F_{k-1}\times \ns:\varphi_{k-1}(f)=\pi_{k-1}(x)\}$. This is a $k$-step Lie-fibered nilspace, by Lemma \ref{lem:fib-prod-top-nil}, and the following diagram commutes:
\begin{equation}
\begin{aligned}[c]
\begin{tikzpicture}
  \matrix (m) [matrix of math nodes,row sep=1.5em,column sep=4em,minimum width=2em]
  {\nss & \ns  \\
     F_{k-1} & \ns_{k-1}, \\};
  \path[-stealth]
    (m-1-1) edge node [above] {$p_2$} (m-1-2)
    (m-1-1) edge node [right] {$p_1$} (m-2-1)
    (m-1-2) edge node [right] {$\pi_{k-1}$} (m-2-2)
    (m-2-1) edge node [above] {$\varphi_{k-1}$} (m-2-2);
\end{tikzpicture}
\end{aligned}
\end{equation}

\vspace{-0.3cm}

\noindent where $p_1:\nss\to F_{k-1}$, $(f,x)\mapsto f$ and $p_2:\nss\to \ns$, $(f,x)\mapsto x$. It is checked straightforwardly that these maps are continuous fibrations.

The proof of Lemma \ref{lem:fib-prod-top-nil}  yields that $\nss$ is a degree-$k$ extension of $F_{k-1}$ by $\ab_k=\ab_k(\ns)$. By Theorem \ref{thm:splitext}, this extension splits. It follows that there is an isomorphism of \textsc{lch}  nilspaces $\iota:\nss\to F_{k-1}\times \mc{D}_k(\ab_k)$. Let $\ab_k'$ be the covering group of $\ab_k$ involved in Theorem \ref{thm:lift-mor-cfr-ab-gr}, with corresponding covering homomorphism $\pi:\ab_k'\to \ab_k$. In particular $\ab_k'=\mb{R}^c\times \mb{Z}^b$ for some integers $c,b\ge 0$, and there is a lattice $\Gamma_k=\ker(\pi)\leq \ab_k'$ such that $\ab_k'/\Gamma_k\cong \ab_k$. We can now define the free nilspace $F$ in the theorem's conclusion:
\begin{equation}
F:=  F_{k-1}\times \mc{D}_k(\ab_k'). 
\end{equation}
We define also the nilspace fibration
$\phi:F \to F_{k-1}\times \mc{D}_k(\ab_k), \; (f,z')\mapsto \big(f,\pi(z')\big)$.

Next, we observe that every translation $h\in \compl{H}\cap \tran_i(F_{k-1})$ (in particular every $h\in H_i$) yields a translation $\tilde{h}\in \tran_i(\nss)$ defined by
$\tilde{h}(f,x) = (h(f),x)$. Indeed, we have firstly that $\tilde{h}$ is $\nss$-valued, because since $\pi_H(f)=\pi_{k-1}(x)$ and $h$ is $\sim_{H}$-vertical, we also have $\pi_{H}(h(f))=\pi_{k-1}(x)$, so $ (h(f),x)\in \nss$. Secondly $\tilde{h}\in \tran_i(\nss)$, indeed for any $n$-cube $\q_1\times\q_2$ on $\nss$, and any $i$-codimensional face $E \subset \db{n}$, we have $h^{E}\cdot \q_1\in \cu^n(F)$ and $\pi_H\co (h^{E}\cdot \q_1) =\pi_H\co \q_1=\pi_{k-1}\co \q_2$, so  $(h^{E}\cdot \q_1)\times\q_2\in \cu^n(\nss)$. Recalling the notation from \cite[Lemma 1.5]{CGSS} we let $\widehat{\iota}:\tran(\nss)\to \tran(F_{k-1}\times \mc{D}_k(\ab_k))$ be the group isomorphism (in particular filtered homomorphism) induced by the nilspace isomorphism $\iota$. It is then checked straightforwardly that the map $\psi:H\to \tran\big(F_{k-1}\times \mc{D}_k(\ab_k)\big)$, $h\mapsto \widehat{\iota}(\tilde h)$ is an injective filtered homomorphism mapping $H_\bullet$ onto a fiber-transitive filtration on $F_{k-1}\times \mc{D}_k(\ab_k)$, whose corresponding quotient is isomorphic to $\ns$. 

Our aim now is to lift the filtration $\psi(H_\bullet)$ through $\phi$ to a fiber-transitive filtration on $F$ satisfying the theorem's conclusion. We shall do this by lifting the generators $\psi(h_i)$. By Theorem \ref{thm:lift-trans}, for every $\alpha\in \tran_j(F_{k-1}\times \mc{D}_k(\ab_k))$ there exists a lift $\alpha'$ of $\alpha$ through $\phi$, i.e.\ a (possibly non-unique) $\alpha'\in \tran_j(F)$ such that $\phi\co \alpha' = \alpha\co\phi$. For each $i\in[n]$, we thus choose some lift $h'_i\in \tran(F)$ of $\psi(h_i)$ through $\phi$, and we do so using the filtered-homomorphism property of $\psi$ in such a way that, if $j$ is the greatest index in $[k]$ such that $h_i\in H_j$, then $h_i'\in \tran_j(F)$ (this will be used below to establish the fiber-transitivity property  for the group generated by these lifts).

The lattice $\Gamma_k=\ker(\pi)$ is finitely generated, so let $z_1,\ldots,z_t\in \ab_k'$ form a generating set for $\Gamma_k$. Then for every $s\in[t]$ let $\gamma_s\in \tran_k(F)$ be the map $\gamma_s(f,a'):=(f,a'+z_s)$  (recall that $(f,a')\in F=F_{k-1}\times \mc{D}_k(\ab_k')$).

We claim that the filtered group $\big(H':=\langle h_1'\ldots,h_n',\gamma_1,\ldots,\gamma_t\rangle,\; H'_\bullet:=(H'\cap\tran_i(F))_{i\geq 0}\big)$, together with the map $\varphi:=p_2\co \iota^{-1}\co \phi: F\to \ns$, satisfy the conclusion of the theorem. 

To prove this, we first note that the map $\phi:F\to F_{k-1}\times \mc{D}_k(\ab_k)$ is $\tran(F)$-consistent in the sense of \cite[Definition 1.2]{CGSS}, i.e.\ for every $\alpha\in \tran(F)$, if $\phi((f,z))=\phi((f',z'))$ then $\phi(\alpha(f,z))=\phi(\alpha(f',z'))$. 
Indeed $\phi((f,z))=\phi((f',z'))$ implies that $f'=f$ and that $\pi(z)=\pi(z')$ (i.e.\ $z,z'$ are equal modulo $\Gamma_k$); the equality $f'=f$ implies (using the description \eqref{eq:str-trans-ab-nil} and the fact that the factor map $\pi_{k-1,F}:F\to F_{k-1}$ is simply the projection $(f,a')\mapsto f$) that  $\alpha(f,z)$, $\alpha(f',z')$ are equal modulo $\pi_{k-1,F}$ and that their $k$-th step components are respectively $z+t,z'+t$ for some $t\in \ab_k'$ depending on $\alpha$ and $f=f'$. Hence $\phi(\alpha(f',z'))$, $\phi(\alpha(f,z))$ have equal $\pi_{k-1}$ projection and $k$-th component respectively $\pi(z'+t),\pi(z+t)$, also equal, so $\phi(\alpha(f,z))=\phi(\alpha(f',z'))$ as claimed. Thus by \cite[Lemma 1.5]{CGSS} we have a well-defined filtered homomorphism $\widehat{\phi}:\tran(F)\to \tran(F_{k-1}\times \mc{D}_k(\ab_k))$, $\alpha\mapsto \big((f,\pi(z))\mapsto \phi (\alpha(f,z))\big)$.

Now a simple computation shows that $\pi_{k-1,F}$ equals the map $p_1\co \iota^{-1}\co \phi$. Let $\eta_{k-1}:\tran(F)\to \tran(F_{k-1})$ be the corresponding homomorphism as defined in \eqref{eq:etaj}. We then have the following fact, which we will use to obtain part of the fiber-transitive property by induction below, as well as the required topological properties thereafter:
\begin{equation}\label{eq:filtcorresp}
\textrm{$\eta_{k-1}(H'_j)=H_j$ for every $j\in [k-1]$.}
\end{equation}
 To see this note that by construction $\eta_{k-1}(H'_j)$ contains all the generators of $H_j$, so $\eta_{k-1}(H'_j)\supset H_j$, and since clearly $\eta_{k-1}(\gamma_i)=\id$ for all $i\in[t]$ (as $\gamma_i\in\tran_k(F)$), we deduce \eqref{eq:filtcorresp}. 

In addition to \eqref{eq:filtcorresp}, we also need the following description of $H_k'$ (where 
$\Gamma_k=\ker(\pi)$ is identified with the subgroup of $\tran(F)$ consisting of translations $f\mapsto f+z$ for $z\in \Gamma_k$):
\begin{equation}\label{eq:filtcorresp-2}
H'_k:=H'\cap \tran_k(F) = \Gamma_k.
\end{equation}
The inclusion $\Gamma_k\subset H'\cap \tran_k(F)$ is clear since the elements $\gamma_s$ generating $\Gamma_k$ are all in $\tran_k(F)\cap H'$ by construction, so we prove the opposite inclusion. Every element $\alpha\in \tran_k(F)$ is of the form $(f,a')\mapsto(f,a'+z)$ for some $z\in \ab_k'$, so we have to prove that if $\alpha=h'\in H'$ then $z\in \Gamma_k$. As an element of $\tran(F_{k-1})$, the translation $\eta_{k-1}(h')$ acts on $\nss=F_{k-1}\times_{\ns_{k-1}}\ns$ via the first component, so we have on one hand that $\widehat{\iota^{-1}}(\widehat{\phi}(h'))$  maps every $(f,x)\in  F_{k-1}\times_{\ns_{k-1}}\ns$ to $(\eta_{k-1}(h')f,x)$. On the other hand, since $h':(f,a')\mapsto (f,a'+z)$, we see by developing the definitions that the map $\widehat{\iota^{-1}}(\widehat{\phi}(h'))$ must take every $(f,x)\in \nss$ to $(f,x+\pi(z))$. Thus for all $(f,x)\in \nss$ we have $(\eta_{k-1}(h')f,x)=(f,x+\pi(z))$. Composing with $p_2$ we conclude that $z\in \ker(\pi)= \Gamma_k$, yielding the required inclusion.

Now let us prove that the filtration $H'_\bullet$ is fiber-transitive on $F$. First let us see that if $\pi_{k-1}(f_1,a_1')=\pi_{k-1}(f_2,a'_2)$ and $\alpha(f_1,a_1')=(f_2,a'_2)$ for some $\alpha\in H$ then we can find $z\in \Gamma_k=H'\cap\tran_k(F)$ such that $(f_1,a'_1)+z=(f_2,a'_2)$. Recall that  $\pi_{k-1}:F\to F_{k-1}$ is the map $(f,a')\mapsto p_1(\iota^{-1}(\phi(f,a')))=f$. Thus $f_1=f_2$ and there exists $z\in \ab_k'$ such that $a_1'+z=a_2'$. But also, as $\alpha=\alpha^{d_1}_1\cdots \alpha^{d_\ell}_\ell$ we have that $\iota^{-1}(\phi(f_2,a_2')) = \iota^{-1}(\phi(\alpha(f_1,a_1'))) = (\beta^{d_1}_1\cdots \beta^{d_\ell}_\ell(f_1),x(f_1,a_1')) $ where $\beta_u$ are again either some $h_i$ or the identity and $x(f_1,a_1')\in \ns$ is such that $\iota^{-1}(\phi(f_1,a_1'))=(f_1,x(f_1,a_1'))$. As computed before, we have $\beta^{d_1}_1\cdots \beta^{d_\ell}_\ell(f_1) = f_1$. Composing with $\iota$ in the equality $\iota^{-1}(\phi(f_2,a'_2))=(f_1,x(f_1,a_1'))$ we have that $(f_2,a_2'\mod \Gamma_k)=(f_1,a_1'\mod\Gamma_k)$. So in particular $z\in \Gamma_k$ and we have found that $(f_1,a_1')+z=(f_2,a_2')$ for some $z\in \Gamma_k=H'\cap \tran_k(F)$.

Now, let $\alpha(f_1,a_1')=(f_2,a_2')$ and $\pi_j(f_1,a_1')=\pi_j(f_2,a_2')$ for $j<k-1$. In particular, $\pi_j(f_1)=\pi_j(f_2)$ and $h(f_1)=f_2$ for some $h\in H$ (here, abusing the notation we indicate by $\pi_j$ the projection to the $j$-th factor, either from $F$ or from $F_{k-1}$). By hypothesis on $H$ there exists an element $h^*\in H_{j+1}$ such that $h^*(f_1)=f_2$. By definition, there exists then some $\tau\in H'_{j+1}$ such that $\pi_{k-1}(\tau(f_1,a_1)) = h^*(f_1)$ (just write $h^*$ in terms of the fixed generators $h_i$ and lift these one by one, using the above-mentioned guarantee that if $h_i\in H_{j+1}$, then $h_i'\in H'_{j+1}$). Then, we have that $\pi_{k-1}(\tau(f_1,a_1))=\pi_{k-1}(f_2,a'_2)$ and clearly they differ in an element of $H'$. By the previous paragraph we have that there exists $\gamma\in \Gamma_k$ such that $\gamma(\tau(f_1,a'_1))=(f_2,a_2')$, which concludes the proof.

For the topological aspects, note that the map $\pi_{H'}(f,a')\mapsto \varphi(f,a')=p_2(\iota^{-1}(\phi(f,a')))$ factors through $\pi_H$, and as $\varphi$ is continuous and $\pi_H$ is an open quotient map, the map $\overline{\varphi}\co \pi_H:=\varphi$ is continuous. By \eqref{eq:filtcorresp}, \eqref{eq:filtcorresp-2}, and induction on the step it follows that $H'$ is also fiber-discrete and fiber-cocompact. By Corollary \ref{cor:quo-cfr-nil} we know that $\pi_{H'}(F)$ is a compact nilspace. A simple computation yields that indeed $\overline{\varphi}$ is bijective. Hence, as a bijection between compact Hausdorff spaces, it is a homeomorphism. Both $\overline{\varphi}$ and $\overline{\varphi}^{-1}$ are easily checked to be morphisms.
\end{proof}

\noindent We close this section with the following corollary.

\begin{corollary}\label{cor:k=2-str-good}
In the case $k=2$ of Theorem \ref{thm:groupequivrep}, the filtration $H_\bullet$ is pure.
\end{corollary}

\begin{proof}
Using the same notation as for Theorem \ref{thm:groupequivrep}, by construction we can take $F=\mc{D}_1(\ab_1')\times \mc{D}_2(\ab_2')$ and take $H'=\langle h_1',\ldots,h_n',\gamma_1,\ldots,\gamma_t\rangle$ where  $h_1',\ldots,h_n'$ are lifts of the translations that freely generate the lattice $\eta_1(H')\subset \ab_1'$ and $\gamma_1,\ldots,\gamma_t$ freely generate the group $H_2'$. Note that in particular, any commutator of elements in $h_1',\ldots,h_n'$ lie necessarily in $H_2'$. As the elements $\gamma_i$ generate $H_2'$, we deduce that every element of $H'$ can be (uniquely) written as ${h_1'}^{a_1}\cdots{h_n'}^{a_n}\gamma_1^{b_1}\cdots\gamma^{b_t}_t$ for integers $a_i,b_j\in \mb{Z}$, $i\in[n]$ and $j\in[t]$. By Theorem \ref{thm:decrip-trans-group} any translation on $F$ has the form $(x,y)\in \ab_1'\times \ab_2'\mapsto (x+c,y+d+T(x))$ where $T:\mc{D}_1(\ab_1')\to \mc{D}_1(\ab_2')$ is a morphism, $c\in \ab_1'$ and $d\in \ab_2'$. Moreover, the elements $\gamma_j$ only add a fixed constant to $\ab_2'$. Now we check that the condition \eqref{it:pure-3} for free-transitivity given in Proposition \ref{prop:eq-free-action} hold. Given any $\alpha={h_1'}^{a_1}\cdots{h_n'}^{a_n}\gamma_1^{b_1}\cdots\gamma^{b_t}_t\in H'$, viewed as the map $\alpha(x,y)=(x+c,y+d+T(x))$, there are two cases. Firstly, if  $c\not=0$ then the first condition in Proposition \ref{prop:eq-free-action} clearly holds. If $c=0$ then $a_1=\cdots=a_n=0$, as $\eta_1(h_1'),\ldots,\eta_n(h_n')$ freely generate the lattice $\eta_1(H')$. Thus, condition \eqref{it:pure-3} in Proposition \ref{prop:eq-free-action} holds, as the $\gamma_i$ only add constants in the $\mc{D}_2$ component.
\end{proof}

\begin{remark}
The filtered group $(H,H_\bullet)$ in Theorem \ref{thm:groupequivrep} for $k>2$ is not necessarily pure. We leave as an exercise for the reader to check that the fiber-transitive group $H$ on $\mc{D}_1(\mb{Z}^3)\times\mc{D}_3(\mb{R})$ generated by $(x,y,z)\mapsto (x+3,y,z+xy)$, $(x,y,z)\mapsto (x,y+3,z+x^2+y^2)$ and $(x,y,z)\mapsto(x,y,z+3)$ is fiber-transitive, can appear as $H$ in Theorem \ref{thm:groupequivrep} and yet it is not pure.
\end{remark}

\begin{remark}\label{rem:2-step-embed}
Corollary \ref{cor:k=2-str-good}  and Proposition \ref{prop:embeb-strong-good} together imply that  any 2-step \textsc{cfr} nilspace can be embedded into a coset nilspace corresponding to a 2-step connected nilmanifold. This will enable us to recover and extend a recent inverse theorem of Jamneshan and Tao \cite{J&T} for the Gowers $U^3$-norm on finite abelian groups; see Theorem \ref{thm:jam-tao}. Note that a version of the above embedding result was proved  for finite 2-step nilspaces by Host and Kra in \cite[Theorem 3]{HK-par}.
\end{remark}

\begin{remark}
From Theorem \ref{thm:groupequivrep} it can be deduced that $k$-step Lie-fibered nilspaces are in fact (real) $C^\infty$ manifolds, and more generally that for every $n\in \mb{Z}{\geq 0}$ the cube-set $\cu^n(\ns)$ is a $C^\infty$ manifold. Moreover if $\varphi:\ns\to \nss$ is a continuous morphism between $k$-step Lie-fibered nilspaces, it can also be proved that $\varphi$ is automatically $C^\infty$. These facts are not needed in this paper but turn out to be useful in subsequent work; proofs can be found in \cite[Appendix A]{CGSS-spec}.
\end{remark}

\section{Double-coset nilspaces}\label{sec:DC}

\noindent Given a filtered group $(G,G_\bullet)$ and two subgroups $K,\Gamma$ of $G$, the \emph{double-coset space} $K\backslash G /\Gamma$ is the set of double cosets $K g \Gamma= \{kg\gamma:k\in K,\gamma\in \Gamma\}$, $g\in G$. The following is a natural generalization of the notion of coset nilspace for double-coset spaces.
\begin{defn}\label{def:dcn}
Let $(G,G_\bullet)$ be a filtered group let $K,\Gamma$ be subgroups of $G$. The associated \emph{double-coset nilspace} is the double-coset space $K\backslash G /\Gamma$ together with the cubes of the form $K\q \Gamma: \db{n}\to K\backslash G /\Gamma$, $v \mapsto K\q(v)\Gamma$ (for $\q\in\cu^n(G_\bullet)$).
\end{defn}
Recall from Remark \ref{rem:cosnilex} the basic example of groupable congruences constituted by coset nilspaces. Given a filtered group $(G,G_\bullet)$ and two subgroups $K,\Gamma$ of $G$, we can define an equivalence relation on $G$ as follows: we say that $g_1,g_2\in G$ are equivalent if the double cosets $Kg_1\Gamma$, $Kg_2\Gamma$ are equal. Naturally we may wonder when this defines a nilspace congruence (as per Definition \ref{def:cong}), and when the quotient map $G\to K\backslash G /\Gamma$ is a fibration (from the group nilspace associated with $G_\bullet$ to the double-coset nilspace $K\backslash G /\Gamma$).

\begin{defn}[Nilpair] 
Let $(G,G_\bullet)$ be a degree-$k$ filtered group, let $\ns$ be the associated group nilspace, and let  $K,\Gamma$ be subgroups of $G$. We say that $(K,\Gamma)$ is a \emph{nilpair}  in $(G,G_\bullet)$ if the equivalence relation on $G$ defined by $g_1\sim g_2 \iff Kg_1\Gamma=Kg_2\Gamma$ is a nilspace congruence on $\ns$ and the quotient map $G\to K\backslash G /\Gamma$ is a nilspace fibration. 
\end{defn}
\noindent To see if $(K,\Gamma)$ is a nilpair in $G$, a natural approach is to start by constructing the right-coset nilspace $K\backslash G$, and then study the equivalence induced on this coset space by \emph{right}-multiplication by elements of $\Gamma$, noting that this yields precisely the equivalence relation on $G$ that we are interested in. In terms of nilspace structures, multiplying on the right by elements of $\Gamma$ induces a  homomorphism $\Gamma\to \tran(K\backslash G)$. By Lemma \ref{lem:gpequivcong}, if this action of $\Gamma$ defines a groupable congruence on $K\backslash G$ then  $(K,\Gamma)$ is a nilpair. Let us record this result equivalently in terms of fiber-transitive filtrations (Definition \ref{def:ft-fil}) with the following lemma, which tells us exactly when the filtration $\Gamma_\bullet:=(\Gamma\cap G_i)_{i\ge 0}$ is fiber-transitive on the coset nilspace $K\backslash G$.
\begin{lemma}\label{lem:nil-pair-good-fil}
Let $(G,G_\bullet)$ be a filtered group of finite degree $k$. Let $K,\Gamma$ be subgroups of $G$. The following statements are equivalent:
\begin{enumerate}[leftmargin=0.7cm]
    \item The filtration $\Gamma_\bullet:=(\Gamma\cap G_i)_{i\ge 0}$ is fiber-transitive on the coset nilspace $K\backslash G$.
    \item For any $x\in G$ and all $i\ge 0$ we have $(Kx\Gamma)\cap(KxG_i)=Kx(\Gamma\cap G_i)$.
\end{enumerate}
\end{lemma}

\begin{proof}
Let us recall a useful fact about coset nilspaces. If $(G,G_\bullet)$ is of degree $k$ and $K$ is a subgroup of $G$, the $(i-1)$-step nilspace factor of $K\backslash G$ is isomorphic to the coset nilspace associated with the filtered group $(G/G_i,(G_j/G_i)_{j\ge 0})$ and the subgroup $(K G_i)/G_i$.

Assume now that $(ii)$ holds. Suppose that for some $\gamma\in \Gamma$, $x,y\in G$ we have $Kx\gamma=Ky$ and $\pi_{i-1}(Kx)=\pi_{i-1}(Ky)$. The last equality means that 
\[
(K G_i/G_i) (xG_i) = (K G_i/G_i) (y G_i).
\]
Hence, for some $s\in K$ we have $(sG_i)(xG_i)=yG_i$. As $G_i$ is normal, this means that $sxG_i = yG_i$, so there exists $g\in G_i$ such that $xsg=y$. Since also $Kx\gamma=Ky$, for some $s'\in K$ we have $s'x\gamma=y$. Hence $s'x\gamma=xsg$. By $(ii)$ there exists some $t\in (\Gamma\cap G_i)$ such that $Ky=Kxh=K x t$ and $(i)$ follows.

The converse is similar. Assuming $(i)$, let us just prove the inclusion $(Kx \Gamma)\cap(KxG_i)\subset K x(\Gamma\cap G_i)$ (the opposite inclusion is clear). Let $k,k'\in K$, $\gamma\in \Gamma$, $g\in G_i$ and $x\in G$ be such that $kx\gamma = k'x g$. Then  $Kx \gamma = Kx g$ and $\pi_{i-1}(Kx)=\pi_{i-1}(Kxg)$, so by $(i)$ there exists $t\in \Gamma\cap G_i$ such that $Kxt=Kxg$. Thus for some $k''\in K$ we have $k''xt = k'xg$ and therefore $k'xg\in K x (\Gamma\cap G_i)$, as required.
\end{proof}
\noindent An analogous result holds if we start by taking the \emph{left}-coset nilspace $G/\Gamma$ and we let $K$ act on $G/\Gamma$ by left multiplication. Hence, it is natural to ask whether the order in which we do the two quotienting operations matters, i.e., whether  if $K_\bullet$ is a fiber-transitive filtration on $G/\Gamma$ then $\Gamma_\bullet$ is a fiber-transitive filtration on $K\backslash G$. This is confirmed by the next result.

\begin{lemma}\label{lem:sym-nil-pair}
Let $(G,G_\bullet)$ be a degree-$k$ filtered group and let ${\Gamma},K\subset G$ be subgroups. Then the following are equivalent:
\begin{enumerate}
    \item For all $x\in G$ and all $i\ge 0$ we have $(Kx{\Gamma})\cap(G_ix{\Gamma})=(K\cap G_i)x{\Gamma}$.
    \item For all $x\in G$ and all $i\ge 0$ we have $(Kx{\Gamma})\cap(KxG_i)=Kx({\Gamma}\cap G_i)$.
\end{enumerate}
\end{lemma}

\begin{proof}
We prove the implication $(i)\Rightarrow (ii)$, the converse is proved similarly. Assuming that $(i)$ holds, it then suffices to prove for every $i\ge 0$ and $x\in G$ the inclusion $(Kx{\Gamma})\cap(KxG_i)\subset Kx({\Gamma}\cap G_i)$ (the opposite inclusion is clear). We argue by induction on the degree $k$. For $k=1$ all groups are abelian and the result is clear. For $i>k$ the result is trivial.

First we prove the required inclusion for $i=k$. Suppose that $gxs = g'xh$ is an element of $(Kx{\Gamma})\cap (KxG_k)$ for some $g,g'\in K$, $x\in G$, $s\in {\Gamma}$ and $h\in G_k$. Then, as the elements of $G_k$ commute with every other element of $G$ we have that $g'^{-1}gxs=hx$. Applying $(i)$ we get that this equals $txs'$ for some $t\in K\cap G_k$ and $s'\in {\Gamma}$. Again, as the elements of $G_k$ commute with everything, we get that $ht^{-1}=s'\in {\Gamma}$ and by definition  $ht^{-1}\in G_k$ as well. Hence $Kx({\Gamma}\cap G_k)\supset Kxht^{-1} = Kxh$ (again using that $t\in G_k$ commutes with everything).

Now we prove the inclusion of any $i<k$ by induction on $k$. Let $gxs = g'xh$ be an element of $(Kx{\Gamma})\cap (KxG_i)$. Passing to the quotient $G/G_k$, note that on the filtered group $(G/G_k,(G_i/G_k)_{i\ge 0})$ we have the action of  $KG_k/G_k$ on the left and of ${\Gamma}G_k/G_k$ on the right. Hence, by induction on $k$ we know that there exists $g''\in K$ and $t\in G_i\cap ({\Gamma}G_k)$ such that $gxsG_k = g'xhG_k=g''xtG_k$. We can write $t=s_tg_t$ for some elements $s_t\in {\Gamma}$ and $g_t\in G_k$. Thus, for some $z\in G_k$ we have that $gxss_t^{-1} =g'' xg_tz$. But $ss_t^{-1}\in {\Gamma}$ and $g_tz\in G_k$ so we can apply the case $i=k$ proved previously to conclude that for some $r\in {\Gamma}\cap G_k$ and $g'''\in K$ we have $gxss_t^{-1} =g'' xg_tz= g'''xr$. Hence $gxs = g''' x rs_t$, where $rs_t\in \Gamma$ and $rs_t = r s_tg_tg_t^{-1} = r  tg_t^{-1}\in G_i$ as $r,g_t^{-1}\in G_k$ and $t\in G_i$, so $rs_t\in \Gamma\cap G_i$ as required.
\end{proof}
\noindent Thus, the left-action of $K$ on $G/{\Gamma}$ induces a groupable congruence if and only if the right-action of ${\Gamma}$ on $K\backslash G$ yields a groupable congruence. This justifies the following notion.
\begin{defn}[Groupable nilpair]
Let $(G,G_\bullet)$ be a filtered nilpotent group of degree $k$ and let $K,\Gamma$ be  subgroups of $G$. We say that $(K,\Gamma)$ is a \emph{groupable nilpair} in $(G,G_\bullet)$ if any of the conditions in Lemma \ref{lem:sym-nil-pair} hold. We then denote the associated double-coset nilspace as $K\backslash G/\Gamma$.
\end{defn}

\noindent Thus, by Lemmas \ref{lem:nil-pair-good-fil} and \ref{lem:cong-equi}, if $(K,\Gamma)$ is a groupable nilpair in $(G,G_\bullet)$ then $K\backslash G/{\Gamma}$ is indeed a double-coset nilspace as per Definition \ref{def:dcn}, and the quotient map $G\to K\backslash G/\Gamma$ is a fibration (to see this, note that we can write it as a composition $G\to G/{\Gamma}\to K\backslash G/{\Gamma}$ where both maps are fibrations by Lemma \ref{lem:cong-equi}). We now proceed towards the  main result of this paper (Theorem \ref{thm:cfr=double-coset} below), showing that every $k$-step \textsc{cfr} nilspace $\ns$ is obtained this way, that is, that $\ns$ is a double-coset nilspace $G\to K\backslash G/{\Gamma}$ for some degree-$k$ filtered Lie group $(G,G_\bullet)$ and a groupable nilpair $(K,\Gamma)$ in $G$. This will require some topological preliminaries concerning double-coset nilspaces.

\subsection{Topological aspects of double-coset nilspaces}\hfill\smallskip\\
Suppose that $(G,G_\bullet)$ is a degree-$k$ filtered Lie group and that $(K,\Gamma)$ is a groupable nilpair in $G$. For our main result in the next subsection (Theorem \ref{thm:cfr=double-coset}), we will need the double-coset space $K\backslash G/{\Gamma}$ to be not only a nilspace \emph{algebraically}, but also a \emph{Lie-fibered} nilspace, which involves important topological requirements as per Definition \ref{def:lcfr-nil}. The following definition provides a \emph{sufficient} condition for this to hold, as we will prove in Theorem \ref{thm:top-double-coset}. The definition is arrived at by straightforwardly considering the property from Definition \ref{def:oldclosed} in the case of $H=\Gamma$ acting on a coset nilspace $\ns=K\backslash G$.

\begin{defn}[Right-closed groupable nilpair]\label{def:rcc}
Let $(G,G_\bullet)$ be a degree-$k$ filtered Lie group such that for every $i\in [k]$ the subgroup $G_i$ is closed, and let $(K,\Gamma)$ be a groupable nilpair in $G$. We say that  $(K,\Gamma)$ is a \emph{right-closed}  nilpair if for all $i\in[k]$, $KG_{i+1}/G_{i+1}$ is a closed subgroup of $G/G_{i+1}$ and the filtration $(\Gamma_j:=\Gamma\cap G_j)_{j\in[k]}$ is closed as a fiber-transitive filtration (acting by right multiplication) on the coset nilspace $K \backslash G$. 
\end{defn}

\begin{theorem}\label{thm:top-double-coset}
Let $(G,G_\bullet)$ be a degree-$k$ filtered Lie group such that for every $i\in [k]$ the subgroup $G_i$ is closed, and let $(K,{\Gamma})$ be a groupable right-closed nilpair. Then  $K\backslash G/{\Gamma}$ is a $k$-step Lie-fibered nilspace, and for each $i\in [k]$ the $i$-th structure group of $K\backslash G/ \Gamma$ is $G_i/(K_i \Gamma_i G_{i+1})$.
\end{theorem}
\begin{remark}
The normality of each $G_i$ in $G$ implies that $KG_i$ is a subgroup of $G$. Similarly $K_iG_{i+1}$ is a  subgroup of $G_i$ ($\forall\, k_1,k_2\in K_i, g_1,g_2\in G_{i+1}$ we have $k_1g_1k_2g_2=k_1k_2(k_2^{-1}g_1k_2)g_2\in K_i G_{i+1}$). Moreover $K_iG_{i+1}$ is a normal subgroup of $G_i$, indeed for any $g\in G_{i+1}$, $h\in G_i$, $k\in K_i$, we have  $hkgh^{-1} = kh[h,k]gh^{-1}$, and as $[h,k]g\in  G_{i+1}$ (which is normal in $G_i$), we have $h([h,k]g)h^{-1}\in G_{i+1}$, so $hkgh^{-1} \in K_iG_{i+1}$ as required. It is seen similarly that $K_i \Gamma_i G_{i+1}$ is a subgroup of $G$. Moreover $G_i/(K_iG_{i+1})$ is an abelian group (as a quotient of  $G_i/G_{i+1}$). 
\end{remark}
To prove Theorem \ref{thm:top-double-coset} we use the following basic fact.
\begin{lemma}\label{lem:closed-group-as-trans}
Let $\ns$ be a group nilspace associated with a filtered Lie group of finite degree $(G,G_\bullet)$. Let $\tau:G\to \tran(\ns)$ be the injective homomorphism $g\mapsto \tau_g$ where $\tau_g(s):=gs$. Then for any closed subgroup $H$ of $G$, we have that $\tau(H)$ is a closed subgroup of $\tran(\ns)$.
\end{lemma}
A similar version holds with $\tau_g(s):=sg^{-1}$.
\begin{proof}
Let $(g_n)_{n\in \mb{N}}$ be a sequence in $H$ such that $(\tau_{g_n})_n$ is Cauchy in $\tran(\ns)$. Then for any compact set $C\subset G$ we have $\lim_{n,m\to\infty} \sup_{x\in C}d_G(\tau_{g_n}x,\tau_{g_m}x)=0$ (where $d_G$ is any compatible metric on $G$). The special case of this for $C=\{\id_G\}$ implies that $\lim_{n,m\to\infty} d_G(g_n,g_m)=0$, so $(g_n)_{n\in\mb{N}}$ is Cauchy in $G$. Thus $(g_n)$ has a limit $g\in H$ and a simple argument then shows that $\tau_{g_n}\to \tau_g$, so $\tau(H)$ is complete and therefore closed.
\end{proof}

\begin{proof}[Proof of Theorem \ref{thm:top-double-coset}]
That $K\backslash G/{\Gamma}$ is algebraically a (double-coset) nilspace follows from the assumption that $(K,\Gamma)$ is a groupable nilpair.

We claim that the fiber-transitive filtration $(K_i:=K\cap G_i)_{i\geq 0}$ is closed on $G$ as per Definition \ref{def:oldclosed}. Indeed, by Lemma \ref{lem:closed-group-as-trans}, for each $i$ the group 
$\eta_i(K)\cong KG_{i+1}/G_{i+1}$, which acts by left-multiplication on $G/G_{i+1}$,  is a closed subgroup of $\tran(G/G_{i+1})$ (since $KG_{i+1}/G_{i+1}$ is closed in $G/G_{i+1}$).

Applying Theorem \ref{thm:coset-quot-closed} to the group nilspace $(G,G_\bullet)$ and the closed fiber-transitive filtration $(K_i)_{i\geq 0}$ we get that $K\backslash G$ is a $k$-step Lie-fibered nilspace. Applying Theorem \ref{thm:coset-quot-closed} again but now to the action of $\Gamma$ by right multiplication on $K\backslash G$ (noting that the assumptions imply that this defines a closed groupable congruence), the result follows.
\end{proof}
\noindent In Definition \ref{def:rcc} we specify \emph{right}-closed to emphasize that the order of the quotienting operations matters: the first quotienting is $G\to K\backslash G$, and on the resulting Lie-fibered coset nilspace $K\backslash G$, we require the \emph{right}-action of $\Gamma$ to induce a groupable congruence which is \emph{closed} as per Definition \ref{def:oldclosed}. The relevance of the order of quotienting is seen with simple examples.
\begin{example}
Let $G=\mc{D}_1(\mb{R})$ and consider the subgroups $K:=\mb{R}$ and $\Gamma:=\mb{Q}$. It is trivial that $(K,\Gamma)$  is a right-closed nilpair, since the quotienting by $K$ yields a 1-point nilspace. However $(K,\Gamma)$ is not ``left-closed" (in the obvious sense) since $\Gamma= \mb{Q}$ is not closed in $\mb{R}$. 
\end{example}

\begin{example}
Let $G:=\mc{D}_1(\mb{R})$, $K:=\mb{Z}$ and ${\Gamma}:=\sqrt{2}\mb{Z}$. Clearly the latter two subgroups are closed (even discrete), but $\Gamma+K$ is dense (and not closed) in $G$. Thus  $K\backslash G/ \Gamma$ is not even a Hausdorff space. This does not satisfy Definition \ref{def:rcc}, however. 
\end{example}

\noindent While Theorem \ref{thm:top-double-coset} shows that the property in Definition \ref{def:rcc} suffices for the double-coset space to be a Lie-fibered nilspace, in our main proof in the next subsection we shall be able to guarantee  a stronger condition, namely Definition \ref{def:CRDCnilpair}, which we recall here.
\begin{defn}
Let $(G,G_\bullet)$ be a degree-$k$ filtered Lie group such that for each $i\in[k]$ the subgroup $G_i$ is closed in $G$. Let $(K,\Gamma)$ be a groupable nilpair in $G$. For each $i\in [k]$ let $K_i=K\cap G_i$, $\Gamma_i=\Gamma\cap G_i$. We say the nilpair $(K,\Gamma)$ is \emph{closed right-discrete} if for every $i\in[k]$ the group $(KG_{i+1})/G_{i+1}$ is a closed subgroup of $G/G_{i+1}$ and $K_i \Gamma_i G_{i+1} /(K_iG_{i+1})$ is a discrete subgroup of $G_i/(K_iG_{i+1})$. We say the nilpair $(K,\Gamma)$ is \emph{fiberwise cocompact} if for every $i\in[k]$, the group $K_i\Gamma_iG_{i+1}/(K_iG_{i+1})$ is a cocompact subgroup of $G_i/(K_iG_{i+1})$.
\end{defn}

\begin{remark}\label{rem:closed-right-discrete-implies-right-closed}
By Lemma \ref{lem:closedconseq} any closed right-discrete nilpair is in particular right-closed.
\end{remark}

\subsection{The double-coset representation theorem}\label{subsec:DoubleCosetRep}\hfill\smallskip\\
We can now prove one of the main results in this paper, Theorem \ref{thm:cfr=double-coset-intro}, which we recall here.

\begin{theorem}\label{thm:cfr=double-coset}
Let $\ns$ be a $k$-step \textsc{cfr} nilspace. Then there exists a degree-$k$ filtered Lie group $(G,G_\bullet)$ and two closed subgroups $K,\Gamma\subset G$ such that $(K,\Gamma)$ is a closed right-discrete groupable nilpair in $(G,G_\bullet)$, and $\ns$ and $K\backslash G/ \Gamma$ are isomorphic \textsc{cfr} nilspaces.
\end{theorem}
\noindent Before proving Theorem \ref{thm:cfr=double-coset}, it is convenient to prove the following auxiliary result.

\begin{lemma}\label{lem_aux-for-dou-coset}
Let $F$ be a free nilspace, let $G=\tran(F)$ and fix any point $f_0\in F$. Let $K=\stab_G(f_0):=\{g\in G:g(f_0)=f_0\}$. Then the following properties hold:
\begin{enumerate}
    \item The right-coset nilspace $K\backslash G$ is a Lie-fibered nilspace.
    \item The map $\psi:K\backslash G\to F$, $Kg\mapsto g^{-1}(f_0)$ is an isomorphism of \textsc{lch} nilspaces.
    \item The map $\psi$ is  equivariant relative to the action of $G$ on $F$ by translations and the action of $G$ on $K\backslash G$ defined by $(Kg,h)\mapsto Kgh^{-1}$ for any $g,h\in G$. 
\end{enumerate}
\end{lemma}
\noindent In other words, property $(iii)$
 states that for any $g,h\in G$ we have $\psi(Kgh^{-1})=h(\psi(Kg))$.
 \begin{proof}
First let us prove $(i)$. By Theorem \ref{thm:coset-quot-closed} it suffices to show that the fiber-transitive filtration induced by $K$ on the group nilspace associated with $(G,G_\bullet)$ is closed. Arguing as in the proof of Theorem \ref{thm:top-double-coset}, the only non-trivial thing to check is that for all $i\in[k]$, the subgroup $KG_{i+1}/G_{i+1}$ of $G/G_{i+1}$ is closed. Consider the continuous homomorphism $\eta_i:G=\tran(F)\to \tran(F_i)$ (Lemma \ref{lem:h-is-cont}). We claim that $\eta_i^{-1}(\stab_{\tran(F_i)}(\pi_i(f_0))) = KG_{i+1}$. The latter group is clearly included in the former. To prove the other inclusion, note that for any $g\in \eta_i^{-1}(\stab_{\tran(F_i)}(\pi_i(f_0)))$ we have $\pi_i(g(f_0))=\pi_i(f_0)$ and thus, since $f_0, g(f_0)$ are elements of the free nilspace $F$ (which is an abelian group nilspace in particular) we have that $f_0-g(f_0)$ is an element of the $(i+1)$-th term in the filtration of $F=\prod_{j=1}^k \mc{D}_j(\mb{Z}^{a_j}\times \mb{R}^{b_j})$, that is, we have $f_0-g(f_0) \in \prod_{j=1}^i \{0\}\times \prod_{j=i+1}^k \mc{D}_j(\mb{Z}^{a_j}\times \mb{R}^{b_j})$. In particular, by Theorem \ref{thm:decrip-trans-group} the map $h:F\to F$, $f\mapsto f+f_0-g(f_0)$ is in $\tran_{i+1}(F)=G_{i+1}$. Hence $hg\in K$ (as $hg(f_0)=g(f_0)+f_0-g(f_0)=f_0$) and thus $g = h^{-1}hg\in KG_{i+1}$. Since  $\stab_{\tran(F_i)}(\pi_i(f_0))$ is clearly closed, so is $KG_{i+1}$.

Next let us prove $(ii)$.  We claim that the map $\psi:K\backslash G\to F$, $Kg \mapsto g^{-1}(f_0)$ is a nilspace isomorphism. First note that $\psi$ is well-defined, since $Kg=Kg'$ means that there exists $k\in K$ such that $kg=g'$, and then ${g'}^{-1}(f_0)=(kg)^{-1}(f_0)=g^{-1}k^{-1}(f_0)=g^{-1}(f_0)$ by definition of $K$. The map $\psi$ is injective, since $g^{-1}(f_0)={g'}^{-1}(f_0)$ implies $g{g'}^{-1}\in K$. Moreover $\psi$ is surjective, since for every $f\in F$ the map $g_f:F\to F$, $x\mapsto x-f$ (using addition on the abelian group $F$) is a translation in $G$ and $\psi(Kg_{f_0-f})=g^{-1}_{f-f_0}(f_0)=f$. Next let us prove that both $\psi$ and $\psi^{-1}$ are nilspace morphisms. Any element of $\cu^n(K\backslash G)$ is of the form $K\q$ for some $\q\in \cu^n(G)$. Thus $\psi\co (K\q) = \q^{-1} \co f_0^{\db{n}}$ where $f_0^{\db{n}}\in \cu^n(F)$ is the constant cube equal to $f_0$. As $\cu^n(G)$ is a group we have that $\q^{-1}\in \cu^n(G)$ and hence $\q^{-1}\co f_0^{\db{n}}\in\cu^n(F)$. To prove that $\psi^{-1}$ is a morphism, note that given any cube in $\cu^n(F)$ we can write it as $\q^*\co f_0^{\db{n}}$ for some $\q^*\in \cu^n(G)$ where $\q^*$ is just a product of maps of the form $g_f:F\to F$ defined as $x\mapsto x-f$.\footnote{ If $f\in \prod_{i=1}^j\{0\}\times \prod_{i=j+1}^k \mc{D}_i(\mb{Z}^{a_i}\times \mb{R}^{b_i})\subset \prod_{i=1}^k \mc{D}_i(\mb{Z}^{a_i}\times \mb{R}^{b_i})$ then it is easy to check that $g_f\in \tran_{j+1}(F)$. By definition, letting $0^{\db{n}}\in \cu^n(F)$ be the constant 0 function, any element of $\cu^n(F)$ is of the form $(\prod_{d=1}^\ell g_{f_d}^{C_d} )\co 0^{\db{n}}$ where for any $d\in[\ell]$, $C_d\subset \db{n}$ is a face of $\db{n}$ and $f_d\in \prod_{i=1}^{\codim(C_d)-1}\{0\}\times \prod_{i=\codim(C_d)}^k \mc{D}_i(\mb{Z}^{a_i}\times \mb{R}^{b_i})$ and $g_{f_d}^{C_d}$ means that we apply the function $g_{f_d}$ only of the face $C_d$.} Therefore, $\psi^{-1} \co (\q^*\co f_0^{\db{n}}) = K(\q^*)^{-1}\in \cu^n(K\backslash G)$ and the result follows. This shows that $K\backslash G$ and $F$ are isomorphic algebraically as nilspaces. Finally, note that the map $\psi$ is continuous. Indeed, if $Kg_n\to Kg$ in $K\backslash G$ then  (using that the quotient map $G\to K\backslash G$ is open) there exists $k_n\in K$ such that $k_ng_n\to g$ in $G$. Thus $g_n^{-1}(f_0) = (k_ng_n)^{-1}(f_0)\to g^{-1}(f_0)$, so $\psi$ is indeed continuous. The map $\psi^{-1}$ can be defined as $f\mapsto Kg_{f_0-f}$, which is clearly continuous as well. This completes the proof of $(ii)$.

Property $(iii)$ follows from the definitions: for any $h,g\in G$ we have $\psi(Kgh^{-1})=hg^{-1}(f_0) = h(\psi(Kg))$.
\end{proof}

\begin{proof}[Proof of Theorem \ref{thm:cfr=double-coset}]
By Theorem \ref{thm:groupequivrep} there is a free nilspace $F$ and a  fiber-transitive group $\Gamma\subset \tran(F)$ which is fiber-discrete and  fiber-cocompact, such that $\pi_\Gamma(F)\cong \ns$. Let $G=\tran(F)$, which is a Lie group by Corollary \ref{cor:trans&structFree}, and let $G_\bullet$ be the filtration $(G_i:=\tran_i(F))_{i\geq 0}$. Let $\Gamma_\bullet$ be the filtration $(\Gamma_i:=\Gamma\cap G_i)_{i\geq 0}$. Fix any point $f_0\in F$ and let $K:=\stab_G(f_0)=\{g\in G:g(f_0)=f_0\}$, a closed subgroup of $G$. Similarly, for each $i\ge 0$ let $K_i:=K\cap G_i$. By Lemma \ref{lem_aux-for-dou-coset} we already know that $\psi:K\backslash G\to F$, $Kg\mapsto g^{-1}(f_0)$ is an equivariant nilspace isomorphism. In particular, it induces an isomorphism $\widehat{\psi}:\tran(K\backslash G)\to G$ between the translation groups of $K\backslash G$ and $F$.

We now claim that $(K,\Gamma)$ is a closed right-discrete groupable nilpair in $(G,G_\bullet)$. Via the equivariance of $\psi$, the action of $\Gamma$ on $F$ by translations corresponds to the right-action of $\Gamma$ on $K\backslash G$ defined by $(Kg,\gamma)\mapsto Kg\gamma^{-1}$. As  $\Gamma$ is a fiber-transitive group on $K\backslash G$ (by right-multiplication), by Lemma \ref{lem:nil-pair-good-fil} we have that $(K,\Gamma)$ is a groupable nilpair. To complete the proof of our claim, we now show that the groupable nilpair $(K,\Gamma)$ is closed right-discrete. To prove this, we note that the proof of Lemma \ref{lem_aux-for-dou-coset} already gives us that for every $i\in [k]$ the subgroup $KG_{i+1}/G_{i+1}$ of $G/G_{i+1}$ is closed. Thus it only remains to prove that for every $i\in [k]$ the subgroup $K_i\Gamma_i G_{i+1}/(K_i G_{i+1})$ of $G_i/(K_iG_{i+1})$ is discrete. To prove this we shall use again that $\psi$ is equivariant. Indeed, note that the $i$-th structure group of the $i$-step nilspace factor of $K\backslash G$ is isomorphic as a topological group to $(G_i/G_{i+1})\,/\, \big((KG_{i+1}/G_{i+1})\cap (G_i/G_{i+1})\big)$ by (the proof of) Theorem \ref{thm:coset-quot-closed}. This latter group in turn equals $(G_i/G_{i+1})\,/\, (K_iG_{i+1}/G_{i+1})$. Finally, using the third isomorphism theorem for topological groups we conclude that the $i$-th structure group of the $i$-step  factor of $K\backslash G$ is isomorphic to $G_i/(K_iG_{i+1})$. Using the equivariance of $\psi$ (which implies straightforwardly that the maps $\psi_i:(K\backslash G)_i\to F_i$ between factors are similarly equivariant) we see that the action of $\eta_i(\widehat{\psi^{-1}}(\Gamma_i))$ on the $i$-step factor of $K\backslash G$ is given by addition of the group  $\Gamma_iK_iG_{i+1}/(K_iG_{i+1})$ (which is a subgroup of $G_i/(K_iG_{i+1})$, the $i$-th structure group of the $i$-step factor of $K\backslash G$). Therefore, if $\eta_i(\Gamma_i)$, seen as a subgroup of the $i$-th structure group of $F_i$, is discrete, then so is $\Gamma_iK_iG_{i+1}/(K_iG_{i+1})$ as a subgroup of $G_i/(K_iG_{i+1})$ (simply because $\psi$ is an isomorphism of \textsc{lch} nilspaces and hence it induces isomorphisms between all the structure groups of $F$ and $K\backslash G$). This proves our claim.

Thus $(K,\Gamma)$ is a closed right-discrete groupable nilpair in $(G,G_\bullet)$, so by Theorem \ref{thm:top-double-coset} and Remark \ref{rem:closed-right-discrete-implies-right-closed}, $K\backslash G / \Gamma$ is a Lie-fibered nilspace. Moreover, the map $\psi$ induces a nilspace isomorphism $\psi_\Gamma:K\backslash G/\Gamma \to \pi_\Gamma(F)$ defined as $\psi_\Gamma(Kg\Gamma):=\pi_H(\psi(Kg))$. Using the equivariance of $\psi$, it is straightforwardly checked that $\psi_\Gamma$ is an isomorphism of \textsc{lch} nilspaces. As $\ns\cong \pi_\Gamma(F)$ we conclude that $\ns\cong K\backslash G /\Gamma$.
\end{proof}

\subsection{Inverse theorems in terms of free nilspaces and double-coset nilspaces}\hfill\smallskip\\
\noindent We now use the main results of the paper to obtain new versions of the inverse theorem.
\begin{proof}[Proof of Theorem \ref{thm:inv-double-coset-intro}]
The result follows by combining \cite[Theorem 1.6]{CSinverse}  with Theorem \ref{thm:cfr=double-coset} and noting that by the proof of Theorem \ref{thm:groupequivrep}, there is a deterministic way of obtaining one such representation of each \textsc{cfr} nilspace as a double coset nilspace. Thus, we can collect the countably many resulting triples $(G,K,\Gamma)$ to obtain the family $\mc{N}$ claimed in the theorem.
\end{proof}

\begin{proof}[Proof of Theorem \ref{thm:inv-high-ord-latti-intro}]
This proof is very similar to the previous one, except that it uses Theorem \ref{thm:groupequivrep} (instead of Theorem \ref{thm:cfr=double-coset}) to represent any $k$-step \textsc{cfr} nilspace as a quotient of a free nilspace by a fiber-transitive group action that is also fiber-discrete and fiber-cocompact. 
\end{proof}
\noindent Note that similar arguments will yield new versions of  regularity theorems for Gowers norms (such as \cite[Theorem 1.5]{CSinverse}) analogous to Theorems \ref{thm:inv-double-coset-intro} and \ref{thm:inv-high-ord-latti-intro}. The proofs are  similar to the above except we would start from the regularity result \cite[Theorem 1.5]{CSinverse} instead of the inverse theorem \cite[Theorem 1.6]{CSinverse}. We omit the details.

\subsection{On a result of Jamneshan and Tao}\hfill\smallskip\\
\noindent Conjecture 1.11 in \cite{J&T} was confirmed in \cite[Theorem 1.10]{J&T} for the $U^3$-norm. Here we give an alternative confirmation.

\begin{theorem}[Inverse theorem for $U^3(\ab)$ norm on finite abelian groups; Theorem 1.10 in \cite{J&T}]\label{thm:jam-tao}
Let $\epsilon>0$, let $\ab$ be a finite abelian group. Let $f:\ab\to \mb{C}$ be a 1-bounded function with $\|f\|_{U^{3}(\ab)}\ge \epsilon$. Then there exists a degree-2 connected nilmanifold $G/\Gamma$ drawn from some finite collection $\mc{N}_\epsilon$ of such manifolds (where each such nilmanifold is  endowed with an arbitrary compatible metric), a $O_\epsilon(1)$-Lipschitz function $F:G/\Gamma\to \mb{C}$ and a nilspace morphism $\phi:\mc{D}_1(\ab)\to G/\Gamma$, such that $|\mb{E}_{x\in \ab} f(x) \overline{F(\phi(x))}|\gg_{\epsilon} 1$.
\end{theorem}

\begin{proof}
First we apply Theorem \ref{thm:groupequivrep} to any 2-step \textsc{cfr} nilspace. Recall that in the case $k=2$ of Theorem \ref{thm:groupequivrep}, on the resulting free nilspace $F'$, the fiber-transitive filtration $K_\bullet$ that we obtain is also pure, by Corollary \ref{cor:k=2-str-good}. Hence, by Proposition \ref{prop:embeb-strong-good} there is an injective morphism $\iota:\pi_K(F')\hookrightarrow H/\Gamma$ for some compact nilmanifold $H/\Gamma$. Indeed, by Remark \ref{rem:isometric-embedding}, given any metric $d$ on $H/\Gamma$ we can choose an appropriate metric on $\pi_K(F')$ in such a way that $\iota$ is an isometry. Thus, without loss of generality we apply Theorem \ref{thm:inv-high-ord-latti-intro} with a metrization $D$ such that each $2$-step \textsc{cfr} nilspace is endowed with such an appropriate metric.\footnote{Recall that in Theorem \ref{thm:inv-high-ord-latti-intro} we can choose an arbitrary metrization. The chosen metric for each 2-step \textsc{cfr} nilspace $\pi_K(F')$ is the one that makes $\iota:\pi_K(F')\to H/\Gamma$ an isometry.} Hence, the Lipschitz map $F:\pi_K(F')\to \mb{C}$ can be interpreted as a map $\tilde{F}:\iota(\pi_K(F'))\to \mb{C}$, $\iota(x)\mapsto F(x)$. Moreover, as $\iota$ is an isometry, relative to the (restriction of the) metric $d$ on the subset $\iota(\pi_K(F'))\subset H/\Gamma$ we have that $\tilde{F}$ is also Lipschitz with the same Lipschitz constant as $F$.

We can extend $\tilde{F}$ to the whole space $H/\Gamma$ in such a way that its Lipschitz constant is at most twice as large as the Lipschitz constant of $F$. A standard way of doing so (essentially using Kirszbraun's theorem) consists in separating the real ($\Re$) and imaginary ($\Im$) parts and defining $\tilde{F}^*:H/\Gamma\to \mb{C}$ as
\[
\tilde{F}^*(x):=\inf_{y\in \iota(\pi_k(F'))}(\Re(\tilde{F}(y))+\Lip(\Re(\tilde{F}))d(x,y))+i\inf_{y\in \iota(\pi_K(F'))}(\Im(\tilde{F}(y))+\Lip(\Im(\tilde{F}))d(x,y))
\]
where $\Lip(\cdot)$ is the Lipschitz constant of a function. This function extends $\tilde{F}$ to $H/\Gamma$ and has a Lipschitz constant $\Lip(\tilde{F}^*)\le \sqrt{2}\Lip(\tilde{F})=\sqrt{2}\Lip(F)$.

Thus, if $F\co \phi$ was the correlating function given by Theorem \ref{thm:inv-high-ord-latti-intro} we have that $F\co \phi = \tilde{F}\co\iota\co\phi$. Now note that $\iota\co \phi$ is a morphism and as the Lipschitz constant of $\tilde{F}$ is bounded in terms of the Lipschitz constant of $F$ the result follows. 
\end{proof}

\section{Compact nilspaces as quotients of pro-free nilspaces and as double-coset nilspaces}\label{sec:cpct-nil}
\noindent In this final section we prove Theorems  \ref{thm:cpct-as-quo-of-free-inf} and \ref{thm:cpct-nil-as-dou-coset}, arguing as outlined in Subsection \ref{subsec:introCompExt}. Recall the notion of pro-free graded nilspaces from Definition \ref{def:free-graded-nilspace}.

\begin{remark}\label{rem:not-gen-theory}
Note that pro-free nilspaces, unlike all previous topological nilspaces in this paper, may fail to be \textsc{lch} (indeed an infinite product of copies of $\mb{R}$ is not locally-compact). However they are still \emph{Polish} spaces. In this section we will not generalize the concept of \textsc{lch} nilspaces to more general \emph{Polish} nilspaces, as we will not need to do so for our main results. Indeed, this topology on a pro-free nilspace will not be used in the results below until Theorem \ref{thm:cpct-as-quotient-of-infinite-free}. In the latter theorem and subsequent two results we will carefully indicate how this topology is used.
\end{remark}

\begin{defn}[Graded translation group]\label{def:gr-tran}
Let $d\in \mb{N}\cup\{\omega\}$. For every $i< d$ let $F_i$ be a free nilspace and let $F=\prod_{i<d}F_i$ be the associated graded pro-free nilspace. The \emph{graded translation group} $\tran(F,(F_i)_{i< d})$ consists of maps $\alpha: F\to F$ such that, for every projection $\pi_j:F\to \prod_{i=1}^j F_i$, $j\in \mb{N}$, there is a translation $\alpha_j\in \tran(\prod_{i=1}^j F_i)$ satisfying $\alpha_j\co \pi_j = \pi_j\co \alpha$.
\end{defn}
\noindent Note that it follows from this definition that such a map $\alpha$ is continuous (by the last equality and definition of the product topology on $F$), and also that it is a translation on $F$ (by checking the translation property in each component $F_i$). Moreover, it is not hard to deduce from the definition that  these maps do form a group (under composition).
\begin{lemma}\label{lem:top-graded-trans-gr}
For any graded pro-free nilspace $F=\prod_{i<\omega} F_i$, the group $\tran(F,(F_i)_{i\in \mb{N}})$ equipped with the compact-open topology is a $k$-step nilpotent pro-Lie group. Moreover, this topology on $\tran(F,(F_i)_{i\in \mb{N}})$ is equivalent to the  topology as an inverse limit of the graded groups of translations $\tran(\prod_{i=1}^jF_i,(F_i)_{i\in [j]})$ for $j\in \mb{N}$.
\end{lemma}

\begin{proof}
For any positive integers $j\le \ell$ consider the projection $p_{j,\ell}:\prod_{i=1}^\ell F_i \to \prod_{i=1}^j F_i$. By the proof of Theorem \ref{thm:LFnstransLie}, the group $C_{j,\ell}:=\{\alpha\in \tran(\prod_{i=1}^\ell F_i): \forall x,y\in \ns, \text{ if } p_{j,\ell}(x)=p_{j,\ell}(y) \text{ then } p_{j,\ell}(\alpha(x))=p_{j,\ell}(\alpha(y))\}$ is a closed subgroup of $\tran(\prod_{i=1}^\ell F_i)$. Therefore, the group $C_\ell:=\cap_{j=1}^\ell C_{j,\ell}$ is also a closed subgroup of $\tran(\prod_{i=1}^\ell F_i)$ and hence a Lie group, by Cartan's theorem. Note $C_\ell = \tran(\prod_{i=1}^\ell F_i,(F_i)_{i\in [\ell]})$. Let $\widehat{p_{j,\ell}}:\tran(\prod_{i=1}^\ell F_i,(F_i)_{i\in [\ell]})\to \tran(\prod_{i=1}^j F_i)$ be the homomorphism induced by $p_{j,\ell}$. Then $\widehat{p_{\ell-1,\ell}}:\tran(\prod_{i=1}^\ell F_i,(F_i)_{i\in [\ell]})\to \tran(\prod_{i=1}^{\ell-1} F_i,(F_i)_{i\in [\ell-1]})$ is a well-defined continuous surjective homomorphism.\footnote{$\widehat{p_{\ell-1,\ell}}$ is onto: $\forall\, \alpha\in \tran(\prod_{i=1}^{\ell-1} F_i,(F_i)_{i\in [\ell-1]})$ the map $(\alpha,\id)$ on $(\prod_{i=1}^{\ell-1} F_i)\times F_\ell$ is in $\tran(\prod_{i=1}^{\ell} F_i,(F_i)_{i\in [\ell]})$.} This defines an inverse system of $k$-step nilpotent Lie groups, so $\tran(F,(F_i)_{i\in \mb{N}})=\varprojlim \tran(\prod_{i=1}^\ell F_i,(F_i)_{i\in [\ell]})$ is a pro-Lie group.

Finally, we have to show that the topology on $\tran(F,(F_i)_{i\in \mb{N}})$ as a subset of the set of continuous functions from $F$ to $F$ with the compact-open topology is equal to the topology generated by the limit maps for the inverse system defined by  $\widehat{p_{\ell-1,\ell}}$ for $\ell\in \mb{N}$. By \cite[Theorem 46.8]{Mu} the compact-open topology in both $C(F,F)$ and $C(\tran(\prod_{i=1}^\ell F_i,(F_i)_{i\in [\ell]}),\tran(\prod_{i=1}^\ell F_i,(F_i)_{i\in [\ell]}))$ for $\ell\in \mb{N}$ coincides with the topology of compact convergence. Note that given any metric $d_\ell$ in $\prod_{i=1}^\ell F_i$ for $\ell\in \mb{N}$, we can define a bounded equivalent metric just by taking $\max(1,d_\ell(\cdot,\cdot))$. Furthermore, if $d_\ell$ is 1-bounded for all $\ell\in \mb{N}$ on $F=\prod_{i<\omega} F_i$ we can define the metric $d_F(x,x'):=\sum_{\ell=1}^\infty \frac{d_\ell(\pi_\ell(x),\pi_\ell(x'))}{2^\ell}$ where $\pi_\ell:F\to \prod_{i=1}^\ell F_i$ is the projection map. Also note that by construction of the maps $\pi_\ell$ we have the following property: for any compact $D_\ell\subset \prod_{i=1}^\ell F_i$ there exists a compact $D\subset F$ such that $\pi_\ell(D)=D_\ell$. Using this it follows that the topology of $\tran(F,(F_i)_{i\in \mb{N}})$ as an inverse limit and as a subset of $C(F,F)$ coincide. We omit the details.
\end{proof}

\noindent Recall that by \cite[Theorem 2.7.3]{Cand:Notes2} every compact $k$-step nilspace $\ns$ is the inverse limit of $k$-step \textsc{cfr} nilspaces $\ns_i$, $i\in \mb{N}$, and by Theorem \ref{thm:gpcongrep-intro} each $\ns_i$ is itself isomorphic to $F_i/H_i$ for some free nilspace $F_i$. We thus have a diagram as follows, where $\ns$ is the inverse limit of the  top row:
\begin{equation}\label{eq:inv-lim-1}
\begin{aligned}[c]
\begin{tikzpicture}
  \matrix (m) [matrix of math nodes,row sep=1em,column sep=4em,minimum width=2em]
  { & \cdots & \ns_2 & \ns_1 \\
      & \cdots & F_2 & F_1. \\};
  \path[-stealth]
    (m-1-2) edge node [above] {} (m-1-3)
    (m-1-3) edge node [right] {} (m-1-4)
    (m-2-3) edge node [right] {} (m-1-3)
    (m-2-4) edge node [above] {} (m-1-4);
\end{tikzpicture}
\end{aligned}
\end{equation}

\vspace{-0.3cm}

\noindent The idea now is to modify inductively the above diagram by replacing the nilspaces $F_i$ by new free nilspaces $\tilde{F_i}$ forming an inverse system. First we will replace $F_2$ with the \emph{fiber-product} of $F_1$ and $F_2$, denoted $\tilde{F_2}$, and prove that $\tilde{F_2}$ is also a free nilspace which has a natural fiber-transitive group acting on it, with the corresponding quotient being isomorphic to $\ns_2$. Then we will replace $F_3$ with the fiber-product $\tilde{F_3}$ of $\tilde{F_2}$ and $F_3$, and so on. This will yield the desired inverse system; we will then prove that its inverse limit is a graded pro-free nilspace, and that $\ns$ is the quotient of this graded nilspace by a pro-discrete fiber-transitive group.

To carry out the above plan we will use the following  results.
\begin{proposition}\label{prop:fiber-prod-of-free}
Let $\ns$, $\ns'$  be $k$-step \textsc{cfr} nilspaces, and let $\varphi:\ns'\to\ns$ be a fibration. Let $F,F'$ be free $k$-step nilspaces and $G,G'$ be fiber-transitive, fiber-cocompact and fiber-discrete subgroups of $\tran(F),\tran(F')$ respectively, such that $\ns\cong F/G$ and $\ns'\cong F'/G'$. Let $\tilde{F}$ be the fiber-product nilspace $\tilde{F}=\{(f,f')\in F\times F': \varphi\co\pi_{G'}(f')= \pi_{G}(f)\}$, where $\pi_G$, $\pi_{G'}$ are the quotient fibrations $F\to F/G$, $F'\to F'/G'$ respectively. Then $\tilde{F}$ is a sub-nilspace of $F\times F'$ which is itself a free nilspace. Moreover, letting the group $G\times G'$ act the natural way on $\tilde{F}$ \textup{(}namely $(g,g')(f,f')=(gf,g'f')$\textup{)}, we embed  $G\times G'$ as a subgroup of $\tran(\tilde{F})$ which is fiber-transitive on $\tilde{F}$ with $\ns'\cong \tilde{F}/H$, and we have the following commutative diagram, where $p_1,p_2$ are the restrictions to $\tilde{F}$ of the 1-st and 2-nd coordinate projections on $F\times F'$:
\begin{equation}\label{eq:k-stepdiag}
\begin{aligned}[c]
\begin{tikzpicture}
  \matrix (m) [matrix of math nodes,row sep=2em,column sep=4em,minimum width=2em]
  { \tilde{F} & F' & \ns'  \\
    & F & \ns, \\};
  \path[-stealth]
    (m-1-1) edge node [above] {$ p_2$} (m-1-2)
    (m-1-2) edge node [above] {$ \pi_{G'}$} (m-1-3)
    (m-1-1) edge node [right] {$p_1$} (m-2-2)
    (m-1-3) edge node [right] {$\varphi$} (m-2-3)
    (m-2-2) edge node [above] {$\pi_G$} (m-2-3);
\end{tikzpicture}
\end{aligned}
\end{equation}
\end{proposition}

We split the proof into several lemmas. 

\begin{lemma}\label{lem:fiber-prod-of-free-is-free}
Under the same assumptions as in Proposition \ref{prop:fiber-prod-of-free}, the fiber-product $\tilde{F}$ is a free sub-nilspace of $F\times F'$.    
\end{lemma}
\begin{proof}
    We argue by induction on the step $k$. For $k=1$, it follows from classical results on torsion-free abelian groups that a 1-step sub-nilspace of a free 1-step nilspace is again a free nilspace. Indeed, note that up to the addition of a constant, we can consider $\varphi\co \pi_{G'}$ and $\pi_G$ as surjective homomorphisms where we identify $F$ and $F'$ with two abelian groups of the form $\mb{Z}^a\times \mb{R}^b$ and $\mb{Z}^{a'}\times \mb{R}^{b'}$ respectively. Thus, the 1-step nilspace $\tilde{F}$ is identified with the following abelian group: $\ab:=\{(f,f')\in (\mb{Z}^a\times \mb{R}^b)\times (\mb{Z}^{a'}\times \mb{R}^{b'}):\varphi\co \pi_{G'}(f')=\pi_G(f)\}$. But now we have a short exact sequence $0\to \ker(\pi_G)\to \ab\to \mb{Z}^{a'}\times \mb{R}^{b'}\to 0$ where $f\in \ker(\pi_G)\mapsto (f,0)\in \ab$ and $(f,f')\in \ab\mapsto f'\in \mb{Z}^{a'}\times \mb{R}^{b'}$. By \cite[Theorem 3.3]{Mosk} the group $\mb{Z}^{a'}\times \mb{R}^{b'}$ is projective and thus $\ab$ is isomorphic (as a topological group) to $\mb{Z}^{a'}\times \mb{R}^{b'}\times \ker(\pi_G)$ by \cite[Theorem 3.5]{Mosk}. The closed subgroup $\ker(\pi_G)\leq \mb{Z}^{a}\times \mb{R}^b$ is also of the form $\mb{Z}^{a''}\times \mb{R}^{b''}$. The case $k=1$ follows.

    For $k>1$, we begin by observing that by Lemma \ref{lem:diff-cont} (see also \cite[Proposition A.20]{CGSS-p-hom}) the $(k-1)$-th factor $\tilde{F}_{k-1}$ of $\tilde{F}$ is isomorphic to the fiber-product $F_{k-1}\times_{\ns_{k-1}} F'_{k-1}$. Note that we have again a commutative diagram exactly as \eqref{eq:k-stepdiag}, but at step $k-1$ and with $\tilde{F}_{k-1}$. By induction, we can thus assume that $\tilde{F}_{k-1}$ is a free nilspace. Since $\tilde{F}$ is a degree-$k$ extension of this \emph{free} nilspace $\tilde{F}_{k-1}$, by Theorem \ref{thm:splitext} we have $\tilde{F}\cong \tilde{F}_{k-1}\times \mc{D}_k(\ab)$ where $\ab=\ab_k(\tilde{F})$, the $k$-th structure group of $\tilde{F}$. Moreover $\ab_k(\tilde{F})=\{(z,z')\in \ab_k(F)\times \ab_k(F'): \phi_k(z)=\phi_k'(z')\}$, by Lemma \ref{lem:diff-cont} (where $\ab_k(F)$ and $\ab_k(F')$ are the $k$-th structure groups of $F$ and $F'$ respectively, and $\phi_k:\ab_k(F)\to \ab_k(\ns)$, $\phi_k':\ab_k(F')\to \ab_k(\ns)$ are the $k$-th structure homomorphisms of $\pi_G$ and $\varphi\co \pi_{G'}$ respectively). By the facts for torsion-free abelian groups used in the case $k=1$ above, we have that  $\ab_k(\tilde{F})\cong \mb{Z}^a\times \mb{R}^b$, for some $a,b\in\mb{Z}_{\geq 0}$.
\end{proof}

\begin{lemma}\label{lem:structure-of-fiber-prod-of-free}
Under the assumptions of Proposition \ref{prop:fiber-prod-of-free}, the group $G\times G'$ leaves $\tilde{F}$ globally invariant, and \textup{(}viewed as a subgroup of $\tran(\tilde{F})$\textup{)} is fiber-transitive, fiber-discrete and fiber-cocompact on $\tilde{F}$. Moreover the quotient $\tilde{F}/(G\times G')$ is nilspace-isomorphic to $\ns'$.  
\end{lemma}

\begin{proof}
Let $(f,f')\in \tilde{F}$ and $(g,g')\in G\times G'$. Now $(f,f')\in \tilde{F}$ means by definition that $\varphi\co\pi_{G'}(f')=\pi_G(f)$, and then clearly $\varphi\co\pi_{G'}(g'f')=\pi_G(gf)$, so $G\times G'$ indeed leaves $\tilde{F}$ globally invariant. In particular we have 
$p_2(gf,g'f')=g'f'=g'p_2(f,f')$, and similarly with $p_1$ instead of $p_2$. Hence $G\times G'$ is consistent with both $p_1,p_2$ in the sense of \cite[Definition 1.2]{CGSS}.

We want to prove that $G\times G'$ is a fiber-transitive group on $\tilde{F}$ and that the quotient nilspace $\tilde{F}/(G\times G')$ is a compact nilspace isomorphic (as a compact nilspace) to $\ns'$.

Let us check the fiber-transitivity property. We have by assumption that $G$ is fiber-transitive on $F$ and so is $G'$ on $F'$. Suppose that $(f,f'), (h,h')$ in $\tilde{F}$ satisfy $\pi_{i,\tilde{F}} (f,f') = \pi_{i,\tilde{F}}(h,h')$ for some $i\in[k]$ and there is $(g,g')\in G\times G'$ such that $(g,g')(f,f')=(h,h')$. Letting $(p_1)_i:\pi_i(\tilde{F})\to \pi_i(F)$ be the morphism such that $(p_1)_i\co \pi_{i,\tilde{F}} = \pi_{i,F}\co p_1$, we deduce that $\pi_{i,F}(f)=\pi_{i,F}\co p_1(f,f') = (p_1)_i\co \pi_{i,\tilde{F}}(f,f') = (p_1)_i\co \pi_{i,\tilde{F}}(h,h') = \pi_{i,F}(h)$. Applying $p_1$ to both sides of $(g,g')(f,f')=(h,h')$ we also deduce that $gf = h$. Similarly, we have $\pi_{i,F'}(f')=\pi_{i,F'}(h')$ and $g' f'= h'$. By the fiber-transitivity of $G,G'$, we deduce that there exist $g_{i+1}\in G_{i+1}:=G\cap \tran_{i+1}(F)$, $g_{i+1}'\in G_{i+1}'$, such that $g_{i+1}f=h$ and $g_{i+1}'f'=h'$, whence $(g_{i+1},g_{i+1}')\in G_{i+1}\times G_{i+1}'$ satisfies $(g_{i+1},g_{i+1}')(f,f')=(h,h')$.

Next, we need to check that if $G,G'$ are fiber-discrete on $F,F'$ respectively then so is $G\times G'$ on $\tilde{F}$. Note that it suffices to check this for the last structure group, as then the result for lower factors follows by induction. Clearly $G_{k}\times G'_k$ is discrete (as a product of two discrete groups). 

To see that $G\times G'$ is fiber-cocompact, recall that by Lemma \ref{lem:diff-cont} we have $\ab_k(\tilde{F})=\{(z,z')\in \ab_k(F)\times \ab_k(F'):\phi_k(z)=\phi'_k(z')\}$ where $\phi_k,\phi_k'$ are the $k$-th structure homomorphisms of $\pi_G$ and $\varphi\co \pi_{G'}$ respectively. Then consider the map $\psi_k:\ab_k(\tilde{F})\to \ab_k(F')/G_k'$ defined by $(z,z')\mapsto z'\mod G_k'$. The map $\psi_k$ is clearly a continuous surjective homomorphism. We claim that its kernel equals precisely $G_k\times G_k'$. The inclusion $G_k\times G_k'\subset \ker(\psi_k)$ is clear. Now let $(z,z')\in \ab_k(\tilde{F})$ and suppose that $0=\psi_k(z,z')=z'\mod G_k'$. Hence $z'\in G_k'$. Fix any $(f,f')\in \tilde{F}$ and note that $\varphi\co \pi_{G'}\co p_2(f+z,f'+z') = \varphi\co \pi_{G'}(f'+z') = \varphi(\pi_{G'}(f')) = \pi_G(f)$ (by definition of $(f,f')$ being in $\tilde{F}$). On the other hand, using the commutativity of \eqref{eq:k-stepdiag} we have  $\varphi\co \pi_{G'}\co p_2(f+z,f'+z') = \pi_G(p_1(f+z,f'+z')) = \pi_G(f+z) = \pi_G(f)+(z\mod G_k)$. Combining this with the previous formula we conclude that $z\mod G_k=0$ and thus $z\in G_k$, which proves our claim. Hence, since $\ab_k(F')/G_k'$ is compact and isomorphic to $\ab_k(\tilde F)/(G_k\times G_k')$, we obtain that  $G_k\times G_k'$ is cocompact as desired.\footnote{Note that we have not used the hypothesis that $G$ is fiber-cocompact. This is not strictly necessary because if we are in a situation as in Proposition \ref{prop:fiber-prod-of-free}, the fact that $\ns'$ is compact forces $\ns$ to be compact. Since $\ns\cong F/G$, this implies directly that $G$ is fiber-cocompact.}

Combining the above facts we obtain that $G\times G'$ is a subgroup of $\tran(\tilde{F})$ which is  fiber-transitive, fiber-discrete and fiber-cocompact group on $\tilde{F}$. Moreover, the morphism $p_2:\tilde{F}\to F'$ is equivariant relative to the actions of $G\times G'$ and $G'$. This induces a well-defined morphism $\overline{p_2}:\tilde{F}/(G\times G')\to F'/G'$. By the previous paragraphs all structure groups are isomorphic, so $\overline{p_2}$ is an isomorphism.
\end{proof}

\begin{remark}\label{rem:lifting-of-consistent-trans} Note that under the same assumptions as in Proposition \ref{prop:fiber-prod-of-free}, we have the following property: if $(f,f'),(h,h')\in \tilde{F}$ satisfy $\pi_{G\times G'}(f,f')= \pi_{G\times G'}(h,h')$ and there exists $g\in G$ such that $gf=h$, then there exists $g'\in G'$ such that $(g,g')(f,f')=(h,h')$. To prove this, note that by definition of $\tilde{F}$ we can take any $g'\in G'$ such that $g'f'=h'$ and the pair $(g,g')$ works. Moreover, if $\pi_{i,\tilde{F}}(f,f')=\pi_{i,\tilde{F}}(h,h')$ for some $i\in[k]$, and $\pi_{G\times G'}(f,f')= \pi_{G\times G'}(h,h')$, and  for some $g_{i+1}\in G_{i+1}=G\cap \tran_i(F)$, $g_{i+1}f=h$, then using the fiber-transitive property, the same proof shows that there is $g_{i+1}'\in G'_{i+1}=G'\cap \tran_{i+1}(F')$ such that $(g_{i+1},g_{i+1}')(f,f')=(h,h')$. 
\end{remark}

\begin{lemma}\label{lem:strong-surjectivity}
Under the assumptions of Proposition \ref{prop:fiber-prod-of-free}, for any integer $n\geq 0$ and any cubes $\q_1\in\cu^n(\ns')$ and $\q_2\in \cu^n(F)$ such that $\varphi\co \q_1 = \pi_G\co \q_2$, there exists $\q_3\in \cu^n(\tilde{F})$ such that $\pi_{G'}\co p_2 \co \q_3 = \q_1$ and $p_1\co \q_3 = \q_2$.
\end{lemma}

\begin{proof}
As $\pi_{G'}$ is a fibration, take any $\q'\in \cu^n(F')$ such that $\pi_{G'}\co \q' = \q_1$. Then note that $\q_3:=(\q_2,\q')$ satisfies the requirements of the lemma.
\end{proof}

\begin{proof}[Proof of Proposition \ref{prop:fiber-prod-of-free}] The result follows immediately from Lemmas \ref{lem:fiber-prod-of-free-is-free} and \ref{lem:structure-of-fiber-prod-of-free}.
\end{proof}
\noindent Applying Proposition \ref{prop:fiber-prod-of-free} inductively for each $i\in \mb{N}$ to the sequence of nilspaces $F_i$ in diagram \eqref{eq:inv-lim-1}, we obtain the following commutative diagram:
\begin{equation}\label{eq:inv-lim-2}
\begin{aligned}[c]
\begin{tikzpicture}
  \matrix (m) [matrix of math nodes,row sep=1.5em,column sep=4em,minimum width=2em]
  {& \cdots & \ns_2 & \ns_1 \\
      & \cdots & \tilde{F}_2 & \tilde{F}_1. \\};
  \path[-stealth]
    (m-1-2) edge node [above] {} (m-1-3)
    (m-1-3) edge node [right] {} (m-1-4)
    (m-2-3) edge node [right] {} (m-1-3)
    (m-2-4) edge node [above] {} (m-1-4)
    (m-2-2) edge node [right] {} (m-2-3)
    (m-2-3) edge node [right] {} (m-2-4);
\end{tikzpicture}
\end{aligned}
\end{equation}

\vspace{-0.3cm}

\noindent Moreover, we have a sequence of discrete groups $G_i\le \tran(\tilde{F}_i)$ for $i\in \mb{N}$ such that $\ns_i\cong \tilde{F}_i/G_i$ and all the groups $G_i$ are consistent with the fibrations $\psi_{i,j}:\tilde{F}_j\to \tilde{F}_i$ for any $i\le j$. In fact we have that $\widehat{\psi_{i,j}}(G_j)=G_i$. However, the maps $\psi_{i,j}$ may have  polynomial expressions that make it difficult to see directly that the inverse limit of the $(\tilde{F}_i,\psi_{i,j})$ is a graded pro-free nilspace. Indeed, a priori the maps $\psi_{i,j}$ need not be coordinate projections, they are just fibrations between free nilspaces. To see that the inverse limit of the $\tilde{F}_i$ form a graded pro-free nilspace, we need to rearrange the situation so that the maps in the inverse system can be regarded as projections to some of the coordinates. We do this as follows.

\begin{lemma}\label{lem:strong-splitting}
Let $F_1,F_2$ be free nilspaces and let $\phi:F_2\to F_1$ be a continuous fibration. Then there is a free nilspace $F_3$ such that, letting $p:F_1\times F_3\to F_1$ denote the projection to the first coordinate, there is a nilspace isomorphism $\tau:F_1\times F_3\to F_2$ such that $\phi\co\tau=p$.
\end{lemma}

\begin{proof}
We argue by induction on  $k:=\max(k_1,k_2)$ where for $i\in[2]$, $F_i$ is $k_i$-step. 

For $k=1$ the result follows by an argument similar to the start of the proof of Lemma \ref{lem:fiber-prod-of-free-is-free}. 

For $k>1$, for $i=1,2$ let us write $F_i=\pi_{k-1}(F_i)\times \mathcal{D}_k(A_i)$ where $\pi_{k-1}(F_i)$ is $(k-1)$-step and $A_i$ is some abelian group.

By Lemmas \ref{lem:Taylor-discrete} and \ref{lem:poly-free-to-R} we can write the map $\phi:\pi_{k-1}(F_2)\times H_2\to \pi_{k-1}(F_1)\times H_1$ as $(x,y)\mapsto (f(x),g(x,y))$. By induction on $k$, there exists a free nilspace $F_3'$ and a nilspace isomorphism $\tau':\pi_{k-1}(F_1)\times F_3' \to \pi_{k-1}(F_2)$ such that $f\co \tau':\pi_{k-1}(F_1)\times F_3'\to \pi_{k-1}(F_1)$ is just the projection to the first coordinate.  Letting $(\tau',\id):\pi_{k-1}(F_1)\times  F_3'\times \mc{D}_k(A_2)\to \pi_{k-1}(F_2)\times \mc{D}_k(A_2)$ denote the map $ (a,b,c)\mapsto (\tau'(a,b),c)$  then $(\tau',\id)$ is an isomorphism and $\phi\co (\tau',\id)(s,r,y)=(s,g'(s,r,y))$ for some morphism $g':\pi_{k-1}(F_1)\times F_3' \times \mc{D}_k(A_2)\to \mc{D}_k(A_1)$. Let us define $\phi':=\phi\co (\tau',\id)$.

As both $F_3'$ and $\pi_{k-1}(F_1)$ are at most $(k-1)$-step, the map $g'(s,r,y)$ equals $g'(s,r,0)+\alpha(y)$ for some homomorphism $\alpha:A_2\to A_1$ (see \cite[Definition 3.3.1]{Cand:Notes1}). As $\phi$ is a fibration, so is $\phi\co (\tau',\id)$ and hence $\alpha$ is a surjective homomorphism. Since $\alpha$ is a surjective homomorphism between abelian Lie groups of the form $\mathbb{Z}^n\times\mathbb{R}^m$, by the case $k=1$ of the lemma we have that $A_2$ is isomorphic as an abelian Lie group to $A_1\times B$. In particular there exists a homomorphism $c:A_1\to A_2$ such that $\alpha\co c =\id_{A_1}$. Let us now consider the map $\tau'':\pi_{k-1}(F_1)\times  F_3'\times \mc{D}_k(A_2)\to \pi_{k-1}(F_1)\times  F_3' \times \mc{D}_k(A_2)$ defined as $(s,r,y)\mapsto(s,r,y-c\co g'(s,r,0))$. This is clearly a nilspace isomorphism and, letting $\phi'':=\phi'\co \tau''$,  we have $\phi''(s,r,y)=(s,\alpha(y))$.

Finally, again by the case $k=1$ of the lemma, there exists an isomorphism $\tau''':A_1\times B\to A_2$ such that $\alpha \co \tau'''(y',y'')=y'$. If we define $(\id,\tau'''):\pi_{k-1}(F_1)\times  F_3' \times \mc{D}_k(A_1\times B)\to \pi_{k-1}(F_1)\times  F_3' \times \mc{D}_k(A_2)$ as $(s,r,y',y'')\mapsto (s,r,\tau'''(y,y'))$ we get that $\phi''\co (\id,\tau''')$ is just the map $(s,r,y',y'')\mapsto (s,y')$ which completes the proof (just let $\tau:=(\tau',\id)\co \tau''\co(\id,\tau''')$).\end{proof}

\begin{corollary}\label{cor:trans-in-graded-trans-group}
Under the assumptions of Lemma \ref{lem:strong-splitting}, suppose that we have a subgroup $G\le \tran(F_2)$ consistent with $\phi$ and such that $\widehat{\phi}(G)=G'\le \tran(F_1)$. Then, under the isomorphism $\widehat{\tau}^{-1}:\tran(F_2)\to\tran(F_1\times F_3) $ the elements of $G$ are of the form $(x,y)\in F_1\times F_3\mapsto (\alpha(x),\beta(x,y))$ where $\alpha\in G'$ and $(\alpha,\beta)\in \tran(F_1\times F_3)$. In other words, $\widehat{\tau}^{-1}(G)$ is a subgroup of $\tran(F_1\times F_3,(F_1,F_3))$ the graded translation group of $F$ with grading $(F_1,F_3)$.
\end{corollary}
We can now prove the following result.
\begin{theorem}\label{thm:covering-of-inverse limit}
Let $\ns$ be a compact $k$-step nilspace, and let $(\ns_j)_{j\in\mb{N}}$ be any inverse system of $k$-step \textsc{cfr} nilspaces such that $\ns=\varprojlim \ns_j$. Then there is a graded pro-free nilspace $(F',(F_i')_{i\in\mb{N}})$ such that for every $j\in \mb{N}$ there is a fiber-discrete, fiber-cocompact  subgroup $G_j$ of the graded translation group of $(\prod_{i=1}^j F_i',(F_i')_{i\le j})$ such that $(\ns_j\cong \prod_{i=1}^j F_i')/G_j$. Moreover the groups $G_j$ together with the homomorphisms $\widehat{p_{\ell,j}}:\tran(\prod_{i\le j} F_i',(F_i')_{i\le j})\to \tran(\prod_{i\le \ell} F_i',(F_i')_{i\le \ell})$,  $j,\ell\in \mb{N}$, form an inverse system of topological groups with $\widehat{p_{\ell,j}}(G_j)=G_\ell$ for all $j\ge \ell$.
\end{theorem}

\begin{proof}
Starting from diagram \eqref{eq:inv-lim-2} we apply successively Lemma \ref{lem:strong-splitting} to obtain that the sequence $\tilde F_1,\tilde F_2,\ldots$ from \eqref{eq:inv-lim-2} can be written (modulo an isomorphism) as a sequence of projections for some free nilspaces $F_1'=\tilde F_1, F_2',F_3',\ldots$ where for every $j\in \mb{N}$ we have $\tilde F_j\cong \prod_{i=1}^j F_i'$. Moreover, by Corollary \ref{cor:trans-in-graded-trans-group}, for every $j\in \mb{N}$ the fiber-transitive group acting on  $\tilde F_j$ is isomorphic to a subgroup of the graded translation group of $(\prod_{i=1}^j F_i',(F_i')_{i\le j})$. Recall that, by the paragraph below \eqref{eq:inv-lim-2} combined with Corollary \ref{cor:trans-in-graded-trans-group}, the fiber-transitive subgroups that we obtain are already fiber-discrete, fiber-cocompact and consistent with the maps $\widehat{p_{\ell,j}}$ for all $j\ge \ell$.
\end{proof}

\noindent We now consider the graded pro-free nilspace $F'=\prod_{i<\omega}F_i'$ and inverse system of groups $(G_i)_{i\in \mb{N}}$ from Theorem \ref{thm:covering-of-inverse limit}, we let $G:=\varprojlim G_i$ be the corresponding inverse limit group, and we shall now prove that $\ns$ is isomorphic as a compact nilspace to $F'/G$. To this end, we first make the following observation.

\begin{remark}\label{rem:lift-throu-inv-lim} The case $n=0$ of Lemma \ref{lem:strong-surjectivity} implies that in diagram \eqref{eq:inv-lim-2}, if $(x_i)_{i\in \mb{N}}\in \varprojlim \ns_i$ and for some $j\in \mb{N}$ we fix any preimage $y_j\in \tilde F_j$ of $x_j\in \ns_j$, then there exists  $(y_i')_{i\in \mb{N}}\in \varprojlim \tilde F_i$ such that $y_j'=y_j$ and $y_i'\mapsto x_i$ for all $i\in \mb{N}$. The proof consists in constructing inductively the sequence $y_i'$. Given $y'_j$, note that by taking its image through the maps $\tilde{F}_j\to \tilde{F}_{j-1}\to \cdots$ we can form $y'_i$ for $i\le j$ satisfying the requirements of the lemma. For larger $i$, inductively we apply Lemma \ref{lem:strong-surjectivity} we have that there exists $y'_{j+1}$ satisfying that $y'_{j+1}\mapsto x_{j+1}$. Repeating this process countably many times we find our element $(y_i')_{i\in \mb{N}}\in \varprojlim \tilde{F}_i$.
\end{remark}

\begin{theorem}\label{thm:cpct-as-quotient-of-infinite-free}
Let $\ns$ be a compact $k$-step nilspace. Then there exists a graded pro-free nilspace $(F'=\prod_{i<\omega} F_i',(F_i')_{i<\omega})$ and a sequence of subgroups $G_j\le \tran(\prod_{i\le j} F_i',(F_i')_{i\le j})$ satisfying the following. For every $j\in \mb{N}$, $G_j$ is fiber-discrete and fiber-cocompact on $(\prod_{i\le j} F_i',(F_i')_{i\le j})$, and for every $j\ge \ell$ we have that $\widehat{p_{\ell,j}}(G_j)=G_\ell$ where $\widehat{p_{\ell,j}}$ is defined as in Theorem \ref{thm:covering-of-inverse limit}. Moreover, letting $G:=\varprojlim G_j$ and endowing $F'$ with the product topology, $\ns$ is isomorphic as a compact nilspace to $F'/G$ (where $F'/G$ is endowed with the quotient topology).  
\end{theorem}
This implies Theorem \ref{thm:cpct-as-quo-of-free-inf}.
\begin{proof}
Let $(\ns_i)_{i\in\mb{N}}$ be a sequence of $k$-step \textsc{cfr} nilspaces such that $\ns=\varprojlim \ns_i$ (with limit maps $\zeta_i:\ns\to \ns_i$), let $F'$ and $G_i$ ($i\in\mb{N}$) be the objects resulting from applying Theorem \ref{thm:covering-of-inverse limit}, and let $G:=\varprojlim G_i$. For any $j\in \mb{N}$ we have the following commutative diagram:
\begin{equation}\label{eq:glueing-1}
\begin{aligned}[c]
\begin{tikzpicture}
  \matrix (m) [matrix of math nodes,row sep=2em,column sep=4em,minimum width=2em]
  {F' & F'/G & \ns \\
      \prod_{i=1}^jF_i'&  & \ns_j. \\};
  \path[-stealth]
    (m-1-1) edge node [above] {$\pi_G$} (m-1-2)
    (m-1-2) edge node [above] {$\varphi$} (m-1-3)
    (m-1-1) edge node [right] {$p_j$} (m-2-1)
    (m-2-1) edge node [above] {$\psi_j$} (m-2-3)
    (m-1-3) edge node [right] {$\zeta_j$} (m-2-3);
\end{tikzpicture}
\end{aligned}
\end{equation}

\vspace{-0.3cm}

We have to prove that $\varphi$ is a well-defined map and an isomorphism of compact nilspaces. 

We first define a map $\tilde\varphi:F'\to \ns$ using the inverse limit expression of $\ns$, namely $\tilde\varphi: (f_i)_{i\in \mb{N}}\in F'=\prod_{i=1}^\infty F_i' \mapsto (\psi_j(p_j((f_i)_{i\in \mb{N}}))))_{j\in \mb{N}}$ where $p_j$ is just the projection to the first $j$ coordinates. We claim that $\tilde \varphi$ factors through $\pi_G$, thus yielding a well-defined map $\varphi:F'/G\to \ns$ such that $\tilde\varphi=\varphi\co\pi_G$. To prove this, suppose that $\pi_G((f_i)_{i\in \mb{N}})=\pi_G((h_i)_{i\in \mb{N}})$. Then there exists $g = (g_j)_{j\in \mb{N}}\in G$ where $g_j\in G_j\subset \tran(\prod_{i=1}^j F_i',(F_i')_{i\in[j]})$ such that for all $j\in \mb{N}$, $g_j(f_1,\ldots,f_j)=(h_1,\ldots,h_j)$. As $\psi_j$ is invariant under the action of $G_j$ the claim follows.

Let us now prove that the well-defined map $\varphi$ is bijective. For injectivity, suppose that $\varphi(\pi_G((f_i)_{i\in \mb{N}}))=\varphi(\pi_G((h_i)_{i\in \mb{N}}))$. Then $\psi_j(f_1,\ldots,f_j)= \psi_j(h_1,\ldots,h_j)$ for every $j\in \mb{N}$. Starting with $\psi_1(f_1)= \psi_1(h_1)$, by definition of $\psi_1$ there is $g_1\in G_1$ such that $g_1(f_1)=h_1$. Then, by Remark \ref{rem:lifting-of-consistent-trans} there is $g_2\in G_2$ such that $g_2\mapsto g_1$ and $g_2(f_1,f_2)=(h_1,h_2)$. Continuing this process, we deduce that there exists $g=(g_j)\in \varprojlim G_j$ with $g(f_i)_{i\in \mb{N}}=(h_i)_{i\in \mb{N}}$, whence $\pi_G((f_i)_{i\in \mb{N}})=\pi_G((h_i)_{i\in \mb{N}})$, which proves injectivity. Surjectivity follows from Lemma \ref{lem:strong-surjectivity}. 

By the second part of Remark \ref{rem:lifting-of-consistent-trans} we have that $G=\varprojlim G_i$ is fiber-transitive on $F'$.

Lemma \ref{lem:strong-surjectivity} implies that $\varphi$ is an isomorphism of nilspaces (in the sense that $\varphi$ and $\varphi^{-1}$ send cubes to cubes). To see that $\varphi$ is a homeomorphism, it suffices prove that it is continuous and open. To prove that $\varphi$ is continuous note that by definition of inverse limit and the quotient topology on $F'/G$ it suffices to prove that $\zeta_j\co \varphi\co \pi_G$ is continuous for every $j\in \mb{N}$. However this map equals $\psi_j\co p_j$ which is continuous. To prove that $\varphi$ is open note first that the sets of the form $p_j^{-1}(U)$ for any $j\in \mb{N}$ and any $U\subset \prod_{i=1}^j F_i'$ form a basis of the topology of $F'$. Thus it is enough to prove that for any such $p_j^{-1}(U)$,  $\varphi\co \pi_G(p_j^{-1}(U))$ is open.  By Lemma \ref{lem:strong-surjectivity} (for $n=0$) we have that $\varphi(\pi_G(p_j^{-1}(U))) = \zeta_j^{-1}(\psi_j(U))$. As $\psi_j$ is an open map the result  follows.

Finally, we need to prove that for every $n\in \mb{N}$, the cube set $\cu^n(F'/G)$ is a closed subset of $(F'/G)^{\db{n}}$. However, as $\varphi^{\db{n}}:(F'/G)^{\db{n}}\to \ns^{\db{n}}$ is a homeomorphism  and $\varphi^{\db{n}}(\cu^n(F'/G))=\cu^n(\ns)$, the closure of $\cu^n(F'/G)$ follows from that of $\cu^n(\ns)$.
\end{proof}
This shows in particular that $F'/G\cong \varprojlim [(\prod_{i=1}^jF_i')/G_j]$ as  compact nilspaces.

\noindent Now, to go further and represent $\ns\cong F'/G$ as a double-coset space, we first extend the results in Lemma \ref{lem_aux-for-dou-coset} to the present setting of  graded pro-free nilspaces and graded translation groups.
\begin{lemma}\label{lem:inf-free-as-quo-of-trans-by-k}
Let $(F',(F_i')_{i=1}^\infty)$ be a graded pro-free nilspace, and let $G=\tran(F',(F_i')_{i=1}^\infty)$. Fix any $f_0\in F'$ and let $K:=\stab_G(f_0)=\{\alpha\in G: \alpha(f_0)=f_0\}$. The following properties hold:
\setlength{\leftmargini}{0.7cm}
\begin{enumerate}
    \item $\psi:K\backslash G\to F'$,  $Kg\mapsto g^{-1}(f_0)$ is a homeomorphism and a nilspace isomorphism.\footnote{Note that as we have not defined a category of \emph{Polish} nilspaces (as we highlighted in Remark \ref{rem:not-gen-theory}), here we simply show that $\psi$ is an isomorphism as an algebraic nilspace morphism  and also that it is a homeomorphism when $F'$ and $K\backslash G$ are endowed with the product and the compact-open topology respectively, see Lemma \ref{lem:top-graded-trans-gr}.}
    \item The map $\psi$ is equivariant with respect to the action of $G$ on $K\backslash G$ defined as $(Kg,g')\in (K\backslash G)\times G\mapsto Kg{g'}^{-1}$ and on $F'$ defined as  $(f,g)\in F'\times \tran(F',(F_i')_{i=1}^\infty)\mapsto g(f)$. 
\end{enumerate} 
\end{lemma}

\begin{proof}
Many parts of this result follow by the same arguments as in Lemma \ref{lem_aux-for-dou-coset}. In particular, all algebraic statements ($\psi$ being a nilspace isomorphism and its claimed equivariance) follow by the same arguments. Hence, it only remains to check that  $\psi$ is a homeomorphism. To do so, note that for any $j\in \mb{N}$, letting $F_j:=\prod_{i=1}^j F_i'$, the following diagram commutes:
\begin{equation}\label{eq:inf-free-as-coset}
\begin{aligned}[c]
\begin{tikzpicture}
  \matrix (m) [matrix of math nodes,row sep=1.5em,column sep=4em,minimum width=2em]
  {\tran(F',(F_i')_{i\in \mb{N}})  & F' \\
      \tran(F_j,(F_i')_{i=1}^j) &  F_j. \\};
  \path[-stealth]
    (m-1-1) edge node [above] {$\psi'$} (m-1-2)
    (m-1-1) edge node [right] {$p_j'$} (m-2-1)
    (m-2-1) edge node [above] {$\psi_j'$} (m-2-2)
    (m-1-2) edge node [right] {$p_j$} (m-2-2);
\end{tikzpicture}
\end{aligned}
\end{equation}

\vspace{-0.3cm}

\noindent Here $p_j'$ is just the restriction to the $F_j$ coordinates, $p_j$ is the usual projection to $F_j$,  $\psi'(g):=g^{-1}(f_0)$, and $\psi_j'(p_j'(g)):=p_j'(g^{-1})(p_j(f_0))$. In order to analyze this diagram, it is useful to analyze first what happens between the different factors of the inverse limits of $F'$ and $\tran(F',(F_i')_{i\in \mb{N}})$. In fact, for any $j\in \mb{N}$ we have the following diagram:
\begin{equation}\label{eq:inf-free-as-coset-2}
\begin{tikzpicture}
  \matrix (m) [matrix of math nodes,row sep=1.5em,column sep=4em,minimum width=2em]
  {\tran(F',(F_i')_{i\in \mb{N}}) & \tran(F_j,(F_i')_{i=1}^j) & \tran(F_{j-1},(F_i')_{i=1}^{j-1}) &  \tran(F_1,(F_1')) \\
      F' & F_j & F_{j-1} & F_1. \\};
  \path[dotted,->]
    (m-1-1) edge node [above] {} (m-1-2)
    (m-2-1) edge node [above] {} (m-2-2)
    (m-1-3) edge node [above] {} (m-1-4)
    (m-2-3) edge node [above] {} (m-2-4);
   \path[-stealth]
    (m-1-2) edge node [above] {$p_{j-1,j}'$} (m-1-3)
    (m-1-2) edge node [right] {$\psi_j'$} (m-2-2)
    (m-2-2) edge node [above] {$p_{j-1,j}$} (m-2-3)
    (m-1-3) edge node [right] {$\psi_{j-1}'$} (m-2-3)
    (m-1-4) edge node [right] {$\psi_1'$} (m-2-4)
    (m-1-1) edge node [right] {$\psi'$} (m-2-1);
\end{tikzpicture}
\end{equation}
Now let $f\in F'$ and for any $j\in \mb{N}$ let $g^{(j-1)}\in \tran(F_{j-1},(F_i')_{i=1}^{j-1})$ be such that $\psi'_{j-1}(g)=p_j(f)$. We claim that there exists  $g\in \tran(F',(F_i')_{i\in \mb{N}})$ such that $\psi'(g)=f$ and $p_j'(g)=g^{(j-1)}$. To see this, note that by induction it suffices to lift $g^{(j-1)}$ to an element in $\tran(F_{j},(F_i')_{i=1}^{j})$ with the desired properties. In fact, since we are considering the graded translations, it is clear that we can find such an element. That is, any element of $F_{j}=(F_i')_{i=1}^{j}$ is of the form $(f',f^{(j-1)})\in F_j'\times F_{j-1}$ and we know that $(g^{(j-1)})^{-1}(f_0)=f^{(j-1)}$. The translation $g^{(j)}\in \tran(F_j,(F_i')_{i=1}^j)$ defined as $(a,b)\in F_j'\times F_{j-1}\mapsto(a+f_0-f',g^{(j-1)}(b))$ works. 

The map $\psi'$ factorizes as $\psi\co \pi_K$, where $\pi_K:G\to K\backslash G$ is the quotient map to the space of right cosets (recall that $G=\tran(F',(F_i')_{i\in \mb{N}})$). To prove that $\psi$ is a homeomorphism it suffices to check that $\psi'$ is open and continuous. Continuity follows from the fact that by definition of $F'$ it suffices to prove that $p_j\co \psi'$ is continuous for every $j\in \mb{N}$. Since this equals $\psi_j'\co p_j'$, which is clearly continuous, the continuity of $\psi'$ follows. To prove openness, note that it is enough to prove it for the basis of the topology on $G$ consisting of sets of the form ${p_j'}^{-1}(U)$ where $U\subset \tran(F_j,(F_i')_{i=1}^j)$ is open and $j\in \mb{N}$. Arguing just as in the proof of Theorem \ref{thm:cpct-as-quotient-of-infinite-free}, by the previous paragraph we have that $\psi'({p_j'}^{-1}(U)) = p_j^{-1}(\psi_j(U))$, and the openness follows.
\end{proof}
We can now prove the following result, which implies Theorem \ref{thm:cpct-nil-as-dou-coset}.
\begin{theorem}
Let $\ns$ be a $k$-step compact nilspace. Then there exists a $k$-step graded pro-free nilspace $(F',(F_i')_{i\in \mb{N}})$ such that, letting $G$ be  the degree-$k$ filtered pro-Lie group $\tran(F',(F_i')_{i\in \mb{N}})$, and letting $K$ be the closed pro-Lie subgroup\footnote{Here we took the stabilizer of the constant $\underline{0}$ element, but this can be replaced by any element of $F'=(F_i')_{i\in \mb{N}}$.} $\stab_G(\underline{0})$, there is a fiber-transitive pro-discrete subgroup $\Gamma$ of $G$ such that the following properties hold:
\setlength{\leftmargini}{0.7cm}
\begin{enumerate}
    \item For $j\in \mb{N}$ let $p_j':\tran(F',(F_i')_{i\in \mb{N}})\to \tran(\prod_{i=1}^jF_i',(F_i')_{i\le j})$ be the projection to the $j$-th factor. Then $p_j'(\Gamma)$ acting on $\prod_{i=1}^jF_i'$ is fiber-transitive, fiber-discrete and fiber-cocompact.
    \item The double-coset nilspace $K\backslash G/\Gamma$ endowed with the quotient topology is isomorphic as a compact nilspace to $\ns$.
\end{enumerate}
\end{theorem} 

\begin{proof}
This follows from combining Lemma \ref{lem:inf-free-as-quo-of-trans-by-k} with Theorem \ref{thm:cpct-as-quotient-of-infinite-free}.
\end{proof}

\appendix

\section{Polynomials between free nilspaces}\label{app:Taylor}

\noindent The goal of this section is to understand continuous morphisms $\phi:F\to F'$ where $F,F'$ are free nilspaces and morphisms $\phi':F\to \mc{D}_k(A)$ where $A$ is an abelian Lie group (recall that for us this also means compactly generated). In particular, we will prove Lemmas \ref{lem:Taylor-discrete}, \ref{lem:poly-free-to-R} and \ref{lem:Taylor-cont}. As any abelian Lie group is isomorphic to $\mb{Z}^r\times \ab\times \mb{R}^s\times \mb{T}^{s'}$, where $\ab$ is a finite abelian group, composing with projections it suffices to study morphisms $\phi:F\to \mc{D}_k(A)$ where $A$ is either $\mb{T},\mb{R}$ or a finitely generated discrete abelian group.

Let $\phi:F\to \mc{D}_k(A)$ be a morphism (where $A$ is an abelian Lie group). Then we have the decomposition $\phi = \phi_k \co \pi_k$, where $\pi_k$ is the canonical projection to the $k$-th step nilspace factor $\prod_{i=1}^k \mc{D}_i(\mb{Z}^{a_i}\times \mb{R}^{b_i})$ of $F$ (thus $\pi_k$ is the projection to the coordinates of degree at most $k$ in $F$), and $\phi_k:\prod_{i=1}^k \mc{D}_i(\mb{Z}^{a_i}\times \mb{R}^{b_i})\to \mc{D}_k(A)$ is a morphism.

Using the above description of $F$, the analysis of $\phi$ naturally splits into two separate cases, the discrete case and the continuous case. In the case $A$ is discrete, note that a continuous morphism $\phi:F\to \mc{D}_k(A)$ cannot depend on the continuous part of $g\in F$, by the following simple observation.

\begin{lemma}\label{lem:non-dep-cont-coord}
Let $F=\prod_{i=1}^k \mc{D}_i(\mb{Z}^{a_i}\times \mb{R}^{b_i})$ be a free nilspace where $a_i,b_i\ge 0$ for $i\in[k]$. Let $A$ be a discrete finitely generated abelian group. Let $\phi:F\to \mc{D}_k(A)$ be a continuous morphism. Then  $\phi(x,y)=\phi(x,0)$
for every $x\in \prod_{i=1}^k \mc{D}_i(\mb{Z}^{a_i})$ and $y\in \prod_{i=1}^k \mc{D}_i(\mb{R}^{b_i})$. \footnote{Here we wrote $0$ for the element of $\prod_{i=1}^k \mc{D}_i(\mb{R}^{b_i})$ with all coordinates equal to zero.}
\end{lemma}
\begin{proof}
The morphism property of $\phi$ plays no role, this is purely a topological result. For any $x\in \prod_{i=1}^k\mc{D}_i(\mb{Z}^{a_i})$ let $\phi_x:\prod_{i=1}^k\mc{D}_i(\mb{R}^{b_i})\to \mc{D}_k(A)$ be the map $\phi_x(y):=\phi(x,y)$. Clearly, $\phi_x$ is also continuous, and as $\prod_{i=1}^k\mc{D}_i(\mb{R}^{b_i})$ is connected and $\mc{D}_k(A)$ is discrete we have that $\phi_x$ must be constant. Thus $\phi_x(y)=\phi_x(0)$ for any $y\in \prod_{i=1}^k\mc{D}_i(\mb{R}^{b_i})$, as claimed.
\end{proof}

\noindent For our analysis of morphisms $\phi:F\to \mc{D}_k(A)$ where $A$ is a discrete, finitely generated abelian group, Lemma \ref{lem:non-dep-cont-coord} implies that it suffices to study the case where $F=\prod_{i=1}^k \mc{D}_i(\mb{Z}^{a_i})$. In order to describe such morphisms we need some additional results.

The first one is the following lemma describing morphisms between higher-degree abelian groups. Recall that by Remark \ref{rem:D0} the  $\mc{D}_0(\ab)$-valued morphisms are just the constant maps.

\begin{lemma}\label{lem:higher-ab-hom-set}
Let $\ab$, $\ab'$ be abelian groups and let $k,\ell\in\mb{N}$. Then
\begin{equation}
 \hom(\mc{D}_{\ell}(\ab),\mc{D}_k(\ab'))=   \hom(\mc{D}_1(\ab),\mc{D}_{\lfloor k/\ell\rfloor}(\ab')).
\end{equation}
\end{lemma}

\begin{proof}
Assume first that $\ell\le k$. To see that $\hom(\mc{D}_{\ell}(\ab),\mc{D}_k(\ab'))\subset \hom(\mc{D}_1(\ab),\mc{D}_{\lfloor k/\ell\rfloor}(\ab'))$, fix any $f\in \hom(\mc{D}_\ell(\ab),\mc{D}_k(\ab'))$ and note that by \cite[Theorem 2.2.14]{Cand:Notes1} it suffices to check that for every $z_1,\ldots,z_t\in \ab$ the derivative $\partial_{z_1}\cdots \partial_{z_t}f$ takes values in the $t$-th term of the filtration defining  $\mc{D}_k(\ab')$. Note that this holds trivially if any of the $z_i$ is 0 (as then the derivative vanishes). Note that it is also clear if $t\le \lfloor k/\ell\rfloor$, since then we have to check that the derivative is $\ab'$-valued, which holds trivially. Thus, the only non-trivial case is when $t\ge \lfloor k/\ell\rfloor+1$, in which case we have to check that $\partial_{z_1}\cdots \partial_{z_t}f$ is the 0 function. Since $f\in \hom(\mc{D}_\ell(\ab),\mc{D}_k(\ab'))$, viewing the elements $z_i$ as elements of the $\ell$-th term of the filtration defining $\mc{D}_\ell(\ab)$ (this term being $\ab$), we see that $\partial_{z_1}\cdots \partial_{z_t}f$ takes values in the $t\ell$-th term of the filtration $\mc{D}_k(\ab')$. By assumption $t\ell\ge \ell(\lfloor k/\ell\rfloor+1)\ge k+1$, so the derivative takes values in the $(k+1)$-th term of the filtration associated with $\mc{D}_k(\ab')$, which is $\{0\}$ as required. This proves the desired inclusion. To see the opposite inclusion, let $f\in \hom(\mc{D}_1(\ab),\mc{D}_{\lfloor k/\ell\rfloor}(\ab'))$ and note that, again, the only non-trivial case consists in checking that for every $z_1,\ldots,z_t\in \ab$ with $\ell t\ge k+1$ we have $\partial_{z_1}\cdots \partial_{z_t}f=0$. Note that the last inequality implies that $t\ge \lceil (k+1)/\ell\rceil$, and it is easily shown that we always have $\lceil (k+1)/\ell\rceil=\lfloor k/\ell\rfloor+1$, so $t\ge \lfloor k/\ell\rfloor+1$ and therefore, since $f\in \hom(\mc{D}_1(\ab),\mc{D}_{\lfloor k/\ell\rfloor}(\ab'))$, the derivative in question is indeed 0. This proves the case $\ell\leq k$.

To complete the proof, suppose that  $\ell>k$, and  fix any $f\in \hom(\mc{D}_l(\ab),\mc{D}_k(\ab'))$. To see that $f$ must be constant, fix any $z\in \ab$ and consider $\q_z\in \cu^{k+1}(\ab)$ the cube such that $\q_z(1^{k+1})=z$ and $\q_z(v)=0$ for any $v\neq 1^{k+1}$. Then $f\co \q_z\in \cu^{k+1}(\ab')$. By unique corner-completion in the latter nilspace, we have that $f\co\q_z$ is the constant cube with value $f(0)$, so in particular $f\co \q_z(1^{k+1})=f(z)=f(0)$. Since $z$ was arbitrary, we deduce that $f$ is indeed constant.
\end{proof}

\begin{lemma}\label{lem:binompoly}
Let $n$ be a non-negative integer, and let  $\binom{\cdot}{n}$ denote the map $\mb{Z}\to \mb{Z}$ sending $x$ to the binomial coefficient $\binom{x}{n}$. Then for every $i\ge 1$ the map $\binom{\cdot}{n}$ is a morphism $\mc{D}_i(\mb{Z})\to \mc{D}_{ni}(\mb{Z})$.
\end{lemma}
\begin{proof}
By Lemma \ref{lem:higher-ab-hom-set}, it suffices to prove the case $i=1$, namely, that $\binom{\cdot}{n}$ is a polynomial map of degree $n$. This is clear (note that $\partial_1 \binom{x}{n}=\binom{x}{n-1}$ for any $x\in \mb{Z}$).
\end{proof}

\begin{lemma}\label{lem:prodpolys}
For $i\in\{0,1\}$, let $(G^{(i)},G^{(i)}_\bullet)$ be  a filtered abelian group and for some $q_i\in \mb{N}$ let $\phi^{(i)}:G^{(i)}\to \mc{D}_{q_i}(\mb{Z})$ be a morphism. Let $\ns$ be the nilspace associated with the group $G^{(0)}\times G^{(1)}$ equipped with the product filtration $G^{(0)}_\bullet\times G^{(1)}_\bullet$. Then the map $\psi:\ns \to \mc{D}_{q_0+q_1}(\mb{Z})$, $(x_0,x_1)\mapsto\phi^{(0)}(x_0)\,\phi^{(1)}(x_1)$ is a morphism.
\end{lemma}

\begin{proof}
By \cite[Theorem 2.2.14]{Cand:Notes1} it suffices to show that $\psi$ is a polynomial map between the corresponding filtered groups. Note that if $g_i\in G^{(i)}$ for $i\in\{0,1\}$ we have that $\partial_{(g_0,g_1)}\psi(x_0,x_1)=\partial_{(0,g_1)}\psi(x_0+g_0,x_1)+\partial_{(g_0,0)}\psi(x_0,x_1)$. Thus it suffices to prove that the derivatives lie in the correct term in the filtration when we derivate only with respect to elements that are 0 in either the first or the second coordinate. By construction of $\psi$, if we take the derivative with respect to the elements $(g_0^1,0),\ldots,(g_0^t,0),(0,g_1^1),\ldots,(0,g_1^s)\in G^{(0)}\times G^{(1)}$ for some integers $t,s\ge 0$ we have that
\[
\partial_{(g_0^1,0)}\cdots \partial_{(g_0^t,0)}\partial_{(0,g_1^1)}\cdots \partial_{(0,g_1^s)}\psi(x_0,x_1)= \partial_{(g_0^1,0)}\cdots \partial_{(g_0^t,0)}\phi^{(0)}(x_0) \partial_{(0,g_1^1)}\cdots \partial_{(0,g_1^s)}\phi^{(1)}(x_1).
\]
Thus, if the elements $g_0^i\in G^{(0)}_{u_i}$ for all $i\in [t]$ for some integers $u_i\ge 1$ and $g_1^j\in G^{(1)}_{v_j}$ for all $j\in [t]$ for some integers $v_j\ge 0$ we have $
\begin{cases}
 \partial_{(g_0^1,0)}\cdots \partial_{(g_0^t,0)}\phi^{(0)}(x_0)\in G^{(0)}_{u_1+\cdots+u_t} \\
 \partial_{(0,g_1^1)}\cdots \partial_{(0,g_1^s)}\phi^{(1)}(x_1) \in G^{(1)}_{v_1+\cdots+v_s}.
\end{cases}$. Hence, if $(u_1+\cdots+u_t)+(v_1+\cdots+v_s)>q_0+q_1$ it must happen that either $(u_1+\cdots+u_t)>q_0$ or $(v_1+\cdots+v_s)>q_1$. In either case we conclude that the corresponding derivative of $\psi$ is $0$.
\end{proof}
\noindent Given a free nilspace $F=\prod_{i=1}^k \mc{D}_i(\mb{Z}^{a_i})$ and $m:=((m_{i,j})_{j\in a_i})_{i\in[k]} \in \prod_{i=1}^k \mb{Z}_{\ge 0}^{a_i}$, recall from Definition \ref{def:height-lattice} the \emph{filtered degree} $|m|:=\sum_{i=1}^k i\sum_{j=1}^{a_i}m_{i,j}$. Combining Lemmas \ref{lem:binompoly} and \ref{lem:prodpolys}, we deduce the following. 
\begin{corollary}\label{cor:binom-is-poly}
Let $F$ be the free nilspace $\prod_{i=1}^k \mc{D}_i(\mb{Z}^{a_i})$, let $m \in \prod_{i=1}^k \mb{Z}_{\ge 0}^{a_i}$. Then the map $((x_{i,j})_{j\in a_i})_{i\in [k]} \mapsto \binom{x}{m}=\prod_{i=1}^k\prod_{j=1}^{a_i} \binom{x_{i,j}}{m_{i,j}}$ is a morphism $F\to \mc{D}_{|m|}(\mb{Z})$.\footnote{Note that we write $m$ instead of $(m,0)$ because the free nilspace $F$ does not have continuous components.}
\end{corollary} 

\begin{theorem}\label{thm:Taylor-discrete-app}
Let $F=\prod_{i=1}^k \mc{D}_i(\mb{Z}^{a_i})$ be a discrete nilspace for some $k\ge 0$ and some integers $a_i\ge 0$, $i\in[k]$. Let $A$ be an abelian group and $\phi:F\to \mc{D}_k(A)$ be a morphism. Then there exists elements $a_m\in A$ for $m= ((m_{i,j})_{j\in a_i})_{i\in[k]} \in \prod_{i=1}^k \mb{Z}_{\ge 0}^{a_i}$ such that
\begin{equation}\label{eq:ex-taylor-from-discrete}
\phi(x)=\sum_{|m|\le k}a_m \binom{x}{m} =\sum_{|m|\le k}a_m \prod_{i=1}^k\prod_{j=1}^{a_i} \binom{x_{i,j}}{m_{i,j}}.
\end{equation}
\end{theorem}
\begin{remark}\label{rem:Taylor-discrete-app}
Theorem \ref{thm:Taylor-discrete-app} holds even if $F$ has continuous components, provided that $A$ is finitely generated and discrete. The continuous components then do not appear in the Taylor expansion by Lemma \ref{lem:non-dep-cont-coord}. Moreover, by Corollary \ref{cor:binom-is-poly} any expression of the form \eqref{eq:ex-taylor-from-discrete} is a morphism from $F$ to $\mc{D}_k(A)$. In particular, this yields Lemma \ref{lem:Taylor-discrete}.
\end{remark}

\begin{proof}[Proof of Theorem \ref{thm:Taylor-discrete-app}]
First note that the set $T:=\{m\in \prod_{i=1}^k \mb{Z}_{\ge 0}^{a_i}: |m|\le k\}$ is a finite set. Let us define $e_{i_0,j_0}\in \prod_{i=1}^k \mb{Z}^{a_i}$ as the element that equals 1 only at the coordinate indexed by $i_0\in[k]$ and $j_0\in[a_{i_0}]$ and 0 otherwise. Note that if $m\notin T$, by \cite[Theorem 2.2.14]{Cand:Notes1} we have $\partial_{e_{1,1}}^{m_{1,1}}\cdots \partial_{e_{k,a_k}}^{m_{k,a_k}}\phi(x)=0$.
Given two elements $m,m'\in \prod_{i=1}^k \mb{Z}_{\ge 0}^{a_i}$, we write $m\le m'$ if $m_{i,j}\le m'_{i,j}$ for all $i\in[k]$ and $j\in[a_i]$. We write $m<m'$ if $m\not=m'$ and $m\le m'$. Let $S\subset \prod_{i=1}^k \mb{Z}_{\ge 0}^{a_i}$ be any finite subset. We say that $S$ is \emph{simplicial} if for every $m\in S$ and every $m'\le m$ we have $m'\in S$. 

For any function $\psi:\prod_{i=1}^k \mb{Z}^{a_i}\to A$, let us define the \emph{support} of $\psi$ as follows:
\[
\supp(\psi):=\prod_{i=1}^k \mb{Z}_{\ge 0}^{a_i}\setminus\left\{m \in \prod_{i=1}^k \mb{Z}_{\ge 0}^{a_i}: \partial_{e_{1,1}}^{m_{1,1}}\cdots \partial_{e_{k,a_k}}^{m_{k,a_k}}\phi(x)=0 \text{ for all }x\in \prod_{i=1}^k \mb{Z}^{a_i}\right\}.
\]
Note that if $\phi:F\to \mc{D}_k(A)$ is a morphism then its support is finite, since $\supp(\phi)\subset \{m \in \prod_{i=1}^k \mb{Z}_{\ge 0}^{a_i}:|m|\le k\}$.  Note also that $\supp(\psi)$ is always a simplicial set.

We shall now prove that if $\psi:\prod_{i=1}^k \mb{Z}^{a_i}\to A$ has finite support then $\psi(x)=\sum_{m\in \supp(\psi)}a_m \binom{x}{m}$ for some coefficients $a_m\in A$. In particular this will imply the result for morphisms on discrete free nilspaces. The proof will argue by induction on the size of $\supp(\psi)$. Clearly, if $|\supp(\psi)|=0$ then the result holds trivially.

Suppose then that $k:=|\supp(\psi)|>0$, and let $m\in \supp(\psi)$ be any element such that there is no other $m'\in \supp(\psi)$ with $m<m'$. We claim that $\psi'(x):=\partial_{e_{1,1}}^{m_{1,1}}\cdots \partial_{e_{k,a_k}}^{m_{k,a_k}}\psi(x)$ is a constant  function of $x\in \prod_{i=1}^k \mb{Z}^{a_i}$. Indeed, note that for any $a\in \prod_{i=1}^k \mb{Z}^{a_i}$ we have $\partial_a \psi'(x)=\psi'(x+a)-\psi'(x)=0$ for all $x\in \prod_{i=1}^k \mb{Z}^{a_i}$, since we can decompose this derivative as a sum of derivatives with respect to the elements $e_{i,j}$. As those are all 0,  so is $\partial_a\psi'(x)$, and therefore $\psi'(x)$ is indeed constant. Let us denote this constant by $\lambda$. 
We then claim that $\psi^*(x):=\psi(x)-\lambda\binom{x}{m}
$ is a map from $\prod_{i=1}^k \mb{Z}^{a_i}$ to $A$ such that $\supp(\psi^*)\subset \supp(\psi)\setminus \{m\}$.

The fact that $n$ is not included in the support of $\psi^*$ follows by definition and noting that $\partial_{e_{i,j}} \binom{x_{i,j}}{t} = \binom{x_{i,j}}{t-1}$ for any integer $t\ge 0$. Furthermore, using the latter formula it is easy to check that the term
\[
\partial_{e_{1,1}}^{m'_{1,1}}\cdots \partial_{e_{k,a_k}}^{m'_{k,a_k}}\binom{x}{m}=\partial_{e_{1,1}}^{m'_{1,1}}\cdots \partial_{e_{k,a_k}}^{m'_{k,a_k}}\left(\lambda\prod_{i=1}^k\prod_{j=1}^{a_i} \binom{x_{i,j}}{n_{i,j}}\right) = \lambda\prod_{i=1}^k\prod_{j=1}^{a_i}\partial_{e_{i,j}}^{n'_{i,j}}\binom{x_{i,j}}{m_{i,j}}
\]
vanishes as soon as there exists some $(i,j)\in [k]\times [a_i]$ such that $m'_{i,j}>m_{i,j}$. This confirms that $\supp(\psi^*)\subset \supp(\psi)\setminus \{m\}$. By induction we have $\psi^*(x)=\sum_{m'\in \supp(\psi^*)}a_{m'} \binom{x}{m'}$, and then $\psi(x) =\lambda\binom{x}{m}+ \sum_{m'\in \supp(\psi^*)}a_{m'} \binom{x}{m'}$. The result follows.
\end{proof}
Our next goal is to prove Lemmas \ref{lem:poly-free-to-R} and \ref{lem:Taylor-cont}. To this end we need some auxiliary results.

\begin{lemma}\label{lem:fact-hom}
Let $a,b\ge 0$ be two integers and let $\phi:\mb{Z}^a\times \mb{R}^b\to \mb{T}$ be a continuous homomorphism. Then there exists a continuous homomorphism $\psi:\mb{Z}^a\times \mb{R}^b\to \mb{R}$ of the form $(x,y)\mapsto \alpha\cdot x+\beta\cdot y$ for some $\alpha\in\mb{R}^a$ and $\beta\in \mb{R}^b$ such that $\phi = \pi\co \psi$ where $\pi:\mb{R}\to \mb{T}$ is the quotient map.
\end{lemma}

\begin{proof}
Splitting the problem between the different coordinates it is enough to prove the result for $\phi:\mb{R}\to \mb{T}$ (a homomorphism from $\mb{Z}$ to $\mb{T}$ is trivially of the desired form). As $\mb{R}$ is the universal covering of $\mb{T}$, we have that $\phi$ factorizes though $\mb{R}$. That is, there exists a continuous homomorphism $\psi:\mb{R}\to \mb{R}$ such that $\pi\co\psi = \phi$, which completes the proof.
\end{proof}

\begin{lemma}\label{lem:poly-cont}
Let $n\ge 0$ be an integer and let $\binom{\cdot}{n}:\mb{R}\to \mb{R}$ be the map $y\mapsto \frac{y(y-1)\cdots(y-n+1)}{n!}$. Then $\binom{\cdot}{n}\in \hom(\mc{D}_1(\mb{R}),\mc{D}_n(\mb{R}))$.
\end{lemma}

\begin{proof}
We leave this to the reader.
\end{proof}

\begin{corollary}\label{cor:cont-binom-poly}
Let $F=\prod_{i=1}^k \mc{D}_i(\mb{Z}^{a_i}\times \mb{R}^{b_i})$ and let $(m,n)\in \prod_{i=1}^k \mb{Z}_{\ge 0}^{a_i}\times \mb{Z}_{\ge 0}^{b_i}$. Then the map $(x,y)\in F\mapsto (x,y)^{(m,n)}$ is a morphism $F\to \mc{D}_{|(m,n)|}(\mb{R})$.
\end{corollary}

\begin{proof}
Note that the proof of Lemma \ref{lem:prodpolys} works the same if we consider that the morphisms involved are $\mb{R}$-valued instead of $\mb{Z}$-valued. The result then follows by combining this with Lemmas \ref{lem:higher-ab-hom-set} and \ref{lem:poly-cont}.
\end{proof}
\noindent We shall now prove Lemma \ref{lem:Taylor-cont}, and then deduce Lemma \ref{lem:poly-free-to-R}. We shall use the following notation. In the sequel, given an abelian group $A$ and a map $\phi:\mb{Z}^a\times \mb{R}^b\to A$, we may want to take derivatives with respect to variables other than the canonical ones; for example, letting $\phi:\mb{R}\to A$, $x\mapsto \phi(x)$, then we may take its derivative with respect to $x\in \mb{R}$ as usual, $\partial_a\phi(x)=\phi(x+a)-\phi(x)$, but then we may want to take the derivative with respect to  $a$, i.e., $\text{``}\partial_b\text{''}\partial_a\phi(x)=\phi(x+a+b)-\phi(x+a)$. To avoid confusion, we adopt the following notation from \cite{GT08}: given a function $\phi$ that depends on the variables $x_1,\ldots,x_n$, then $(h\cdot \nabla_{x_i})\phi(x_1,\ldots,x_n):=\phi(x_1,\ldots,x_{i-1},x_i+h,x_{i+1},\ldots,x_n)-\phi(x_1,\ldots,x_n)$. When there is no risk of confusion, we may still write $\partial_{h}$ instead.

\begin{proof}[Proof of Lemma \ref{lem:Taylor-cont}]
Let us denote $D:=\prod_{i=1}^k \mb{Z}^{a_i} \times \prod_{i=1}^k  \mb{R}^{b_i}$ and $L:=\prod_{i=1}^k \mb{Z}_{\ge 0}^{a_i}\times \prod_{i=1}^k \mb{Z}_{\ge 0}^{b_i}$.
Similarly as before, given a function $\psi:D\to \mb{T}$ we define its support as follows:
\begin{multline*}
\supp(\psi):=L\setminus\{(m,n)\in L: \text{for all }(t_{i,\ell})_{(i,\ell)\in[k]\times [b_i]}\in \mb{R} \text{ and } (x,y)\in D, \\ \partial_{e_{1,1}}^{m_{1,1}}\cdots \partial_{e_{k,a_k}}^{m_{k,a_k}}\partial_{t_{1,1}w_{1,1}}^{n_{1,1}}\cdots \partial_{t_{k,b_k}w_{k,b_k}}^{n_{k,b_k}}\phi(x)=0 \}
\end{multline*}
where $e_{i,j}\in D$ is everywhere 0 except for the coordinate $(i,j)$ in the discrete part of $D$ (where it is 1) and similarly with $w_{i,\ell}$ for the continuous part. Note that if $\phi:F\to \mb{T}$ is a morphism then $\supp(\phi)\subset \{(m,n)\in L:|(m,n)|\le k\}$.

We want to prove that if $\psi$ has finite support then $\psi(x,y)=\pi\left(\sum_{(m,n)\in \supp(\psi)}\lambda_{m,n} \binom{(x,y)}{(m,n)}\right)$ for some coefficients $\lambda_{m,n}\in \mb{R}$ (provided that $\psi$ is continuous).

We prove this by induction on the size of $\supp(\psi)$. Note that if the support of $\psi$ is 0 then this automatically means that the function equals 0, so there is nothing to prove. For the general case, let $(m,n)\in \supp(\psi)$ be an element such that if $(m',n')\in L$ satisfies $(m,n)<(m',n')$ then $(m',n')\notin \supp(\psi)$. Assume without loss of generality that $n_{i_0,\ell_0}>0$ for some $(i_0,\ell_0)\in[k]\times [b_i]$. We leave as an exercise to the reader to check the easier case when for some $(i_0,j_0)\in[k]\times[a_i]$ we have $m_{i_0,j_0}>0$. Note that if none of these exists (neither in the continuous part nor in the discrete part), then $\partial_{(u,v)}\psi(x,y)=0$ for all $(u,v)\in D$. Thus $\psi$ in this case is constant and the result follows.

Now, let us consider $(hw_{i_0,j_0}\cdot \nabla_{(x,y)})\psi(x,y)$. Note that this is a function (in the variables $(x,y)$) from $D$ to $\mb{T}$ with support strictly smaller than that of $\psi$. Indeed, note that \[\supp((hw_{i_0,j_0}\cdot \nabla_{(x,y)})\psi(x,y))\subset \{(m',n')\in L:(m',n')+w_{i_0,\ell_0}\in \supp(\psi)\}=:T.\] Note that $T$ is independent of $h$ and that $|T|<|\supp(\psi)|$. Thus by induction we can write
\begin{equation}\label{eq:form-deri}
(hw_{i_0,j_0}\cdot \nabla_{(x,y)})\psi(x,y)=\pi\left(\sum_{(m',n')\in T}\lambda_{m,n}(h) \binom{(x,y)}{(m,n)}\right)
\end{equation}
where now the coefficients $\lambda_{m,n}(h)\in \mb{R}$ may depend on $h$.

The element $(m^*,n^*):=(m,n)-w_{i_0,\ell_0}$ is in $T$ by definition. Next, we take the derivative of $(hw_{i_0,j_0}\cdot \nabla_{(x,y)})\psi(x,y)$ as many times as necessary in order to cancel all terms but $\lambda_{m^*,n^*}$. That is,
\begin{equation}\label{eq:proof-fact-free-cont-mor}
\partial_{e_{1,1}}^{m^*_{1,1}}\cdots \partial_{e_{k,a_k}}^{m^*_{k,a_k}}\partial_{w_{1,1}}^{n^*_{1,1}}\cdots \partial_{w_{k,b_k}}^{n^*_{k,b_k}} (hw_{i_0,j_0}\cdot \nabla_{(x,y)})\psi(x,y) = \pi(\lambda_{m^*,n^*}(h)),
\end{equation}
where $\partial_{e_{i,j}}$ is the operator $(e_{i,j}\cdot \nabla_{(x,y)})$ (and similarly with $w_{i,\ell}$). Note \eqref{eq:proof-fact-free-cont-mor} holds for all $(x,y)\in D$ and all $h\in \mb{R}$. Evaluating at $(x,y)=(0,0)$ yields that $\pi\co \lambda_{m^*,n^*}:\mb{R}\to \mb{T}$ is continuous.

Moreover, we can prove that this function is indeed affine. That is, it is a morphism $\mc{D}_1(\mb{R})\to \mc{D}_1(\mb{T})$. In order to do so, first apply the operator $(z_1\cdot \nabla_h)$ to both sides of \eqref{eq:proof-fact-free-cont-mor}. Note that in the left hand side we end up with an expression that consists of sums and differences of terms of the form
\[
\psi\left((x,y)+hw_{i_0,\ell_0}+\xi z_1 w_{i_0,\ell_0}+\iota_{1,1}e_{1,1}+\cdots+\iota_{k,a_k}e_{k,a_k}+\tau_{1,1}w_{1,1}+\cdots+\tau_{k,b_k}w_{k,b_k}
\right).
\]
Where $\iota_{i,j},\tau_{i,\ell},\xi\in \{0,1\}$ for all $i\in[k]$, $j\in[a_i]$ and $\ell\in[b_i]$. Hence we have the following equality: $(z_1\cdot \nabla_h)\partial_{e_{1,1}}^{m^*_{1,1}}\cdots \partial_{e_{k,a_k}}^{m^*_{k,a_k}}\partial_{w_{1,1}}^{n^*_{1,1}}\cdots \partial_{w_{k,b_k}}^{n^*_{k,b_k}} (hw_{i_0,j_0}\cdot \nabla_{(x,y)})\psi(x,y) = (z_1\cdot \nabla_h)\pi(\lambda_{m^*,n^*}(h))$, for all $(x,y)\in D$, $h,z_1\in \mb{R}$. But making the change of variables $(x,y)\mapsto (x,y)-hw_{i_0,\ell_0}$ we have that the left hand side of the previous equation does not depend on $h$. Therefore applying $(z_2\cdot \nabla_h)$ we get 0 on the  left. Hence
 $0 = (z_2\cdot \nabla_h)(z_1\cdot \nabla_h)\pi(\lambda_{n^*,m^*}(h))$. By Lemma \ref{lem:fact-hom} we have  indeed $\pi(\lambda_{m^*,n^*}(h)) = \pi(\alpha h+\beta)$ for some constants $\alpha,\beta\in \mb{R}$. We can use this information to eliminate the term indexed by $(m,n)$ in the support of $\psi$. Indeed, we want to check that
\begin{equation}\label{eq:reduc-degree}
\psi(x,y)-\pi\left(\alpha \binom{(x,y)}{(m,n)}-\beta \binom{(x,y)}{(m^*,n^*)}\right)
\end{equation}
is function on $D$ such that its support is strictly smaller than the support of $\psi$.

In order to prove this, it is enough to check that taking the correct number of derivatives according to $(m,n)$ we have 0. Note that the function $\zeta:\mb{R}\to \mb{T}$, $y\mapsto \pi\left(\binom{y}{d}\right)$ satisfies that $\partial_t\zeta(y) = \pi\left(t\binom{y}{d-1}\right)+\textbf{p}(y,t)$ where $\textbf{p}(y,t)$ is a polynomial in $y$ of degree at most $d-2$ and thus it vanishes if we take $d-1$ derivatives on $y$. Using this fact and  \eqref{eq:form-deri} it follows that \eqref{eq:reduc-degree} does not contain $(m,n)$ in its support and so we can conclude the result by induction.

To complete the proof we still need to prove that any function of the form \eqref{eq:Taylor-cont} is a morphism from $F$ to $\mc{D}_k(\mb{T})$. By Corollary \ref{cor:cont-binom-poly}, any map of the form \eqref{eq:poly-free-to-R} is a morphism from $F$ to $\mc{D}_k(\mb{R})$. As the quotient map $\pi:\mb{R}\to \mb{T}$ is a homomorphism, in particular it is in $\hom(\mc{D}_1(\mb{R}),\mc{D}_1(\mb{T}))$ and by Lemma \ref{lem:higher-ab-hom-set} in particular it is in $\hom(\mc{D}_k(\mb{R}),\mc{D}_k(\mb{T}))$. The result follows.
\end{proof}

\begin{proof}[Proof of Lemma \ref{lem:poly-free-to-R}]
We already know that any function of the form \eqref{eq:poly-free-to-R} is a morphism from $F$ to $\mc{D}_k(\mb{R})$. To prove that those are all possibilities, let $\phi:F\to \mc{D}_k(\mb{R})$ be a morphism. Then, if $\pi:\mb{R}\to \mb{T}$ is the usual quotient map, we have that $\pi\co \phi\in \hom(F,\mc{D}_k(\mb{T}))$. By Lemma \ref{lem:Taylor-cont} there exists $\phi'\in \hom(F,\mc{D}_k(\mb{R}))$ such that $\pi\co \phi = \pi\co \phi'$. Hence, $\phi-\phi'\in \hom(F,\mc{D}_k(\mb{Z}))$ and we can apply Lemma \ref{lem:Taylor-discrete} to conclude that $\phi-\phi'$ has the form \eqref{eq:Taylor-discrete}. The result follows.
\end{proof}

\section{Some results in topology}

\begin{lemma}\label{lem:quo-by-cont-gr-act}
Let $X$ be an \textsc{lch} space and let $G$ be a topological group acting continuously on $X$ and such that $\{(x,x')\in X\times X:\exists g\in G \text{ such that }x=gx'\}$ is a closed subset of $X\times X$. Then $X/G$ with the quotient topology is also an \textsc{lch} space.
\end{lemma}

\begin{proof}
The fact that $G$ acts continuously on $X$ implies that the quotient map $\pi:X\to X/G$ is open. Hence it is $X/G$ is second-countable and the fact that it is \textsc{lch} follows from \cite[\S 8.3, Proposition 8 and \S 10.4 Proposition 10]{BourbakiGT1}.
\end{proof}

\begin{lemma}\label{lem:conv-quot}
Let $X$ be an \textsc{lch} space and let $G$ be a group acting continuously on $X$ and such that $\{(x,x')\in X\times X:\exists g\in G \text{ such that }x=gx'\}$ is a closed subset of $X\times X$. Then $\pi(x_n)\to \pi(x)$ as $n\to\infty$ in $X/G$ if and only if there exists a sequence $g_n\in G$ such that $g_nx_n\to x$ as $n\to\infty$ in $X$.
\end{lemma}

\begin{proof}
By Lemma \ref{lem:quo-by-cont-gr-act} we already know that both $X$ and $X/G$ are metric spaces and the quotient map $X\to X/G$ is open. The result follows by \cite[Proposition 2.4]{Si}.
\end{proof}

A similar result holds for Polish groups.

\begin{lemma}\label{lem:conv-polish-group}
Let $H,G$ be Polish groups and let $\varphi:G\to H$ be a surjective continuous homomorphism. Then for every sequence $(\varphi(g_n))_{n\ge 0}\subset H$ such that $\varphi(g_n)\to \varphi(g)$ for some $g\in G$ there exists a sequence $r_n\in \ker(\varphi)$ such that $g_n r_n\to g$ in $G$.
\end{lemma}

\begin{proof}
By the open mapping theorem for Polish groups \cite[Theorem A.1]{HK-non-conv} we have that $\varphi$ is open. The result follows by \cite[Proposition 2.4]{Si}.
\end{proof}

\begin{lemma}\label{lem:quot-polish-gr-lie-gr}
Let $\varphi:G\to H$ be a continuous surjective homomorphism where $G$ is Polish and $H$ is Lie. Then $G/\ker(\varphi)\cong H$ and in particular, it is Lie.
\end{lemma}

\begin{proof}
By the open mapping theorem for Polish groups $\varphi$ is open. Thus, the map $G/\ker(\varphi)\to H$, $g\ker(\varphi)\mapsto \varphi(g)$ is a homeomorphism. As $H$ is Lie so is $G/\ker(\varphi)$.
\end{proof}

Recall that \textsc{lch} spaces are by definition locally compact, Hausdorff and second-countable. Hence, in particular they are compactly generated. By \cite[Lemma 46.4]{Mu} we know then that $f\in C(X,Y)$ if and only if for each compact set $K\subset X$, $f|_K$ is continuous.

\begin{proposition}\label{prop:cont-LCH-polish}
Let $X,Y$ be \textsc{lch} spaces. Then the space $C(X,Y)$ is Polish.
\end{proposition}

\begin{proof}
As $X$ is \textsc{lch}, by \cite[\S 46 Ex. 10. (b)]{Mu} there exists a sequence of nested compact subsets $K_n\subset X$ such that $X=\cup_{n=1}^\infty K_n^\circ$ where $A^\circ$ is the interior of the set $A$. Then consider the map $\iota:C(X,Y)\to \prod_{n=1}^\infty C(K_n,Y)$ that sends $f\mapsto \prod_{n=1}^{\infty} f|_{K_n}$. Then, if $\phi_{n,n+1}:C(K_{n+1},Y)\to C(K_n,Y)$ is the map $g\mapsto g|_{K_n}$ it is easy to see that this map is continuous. Furthermore, we have 
$\iota(C(X,Y))= \{(f_n)_{n\ge 1}\in \prod_{n=1}^\infty C(K_n,Y): \phi_{n,n+1}(f_{n+1})=f_n\}$. As the maps $\phi_{n,n+1}$ are continuous we get that $\iota(C(X,Y))$ is a closed subset of $\prod_{n=1}^\infty C(K_n,Y)$. Moreover, the map $\iota$ is injective so there exists a well-defined map $\iota^{-1}:\iota(C(X,Y))\to C(X,Y)$. Using that $X=\cup_{n=1}^\infty K_n^\circ$ it follows that this map is indeed continuous. Hence $C(X,Y)$ is homeomorphic to a closed subset of $\prod_{n=1}^\infty C(K_n,Y)$. By \cite[Theorem 4.19]{Kechris} each $C(K_n,Y)$ is Polish. Hence so is $\prod_{n=1}^\infty C(K_n,Y)$ and also $\iota(C(X,Y))$, being a closed subset of a Polish space.
\end{proof}

\subsection*{Funding.} All authors used funding from project PID2020-113350GB-I00 financed by  MICIU/AEI/10.13039/501100011033/FEDER, EU. The second-named author received funding from project Momentum (Lend\"ulet) 30003 of the Hungarian Government. The second and third-named authors also supported partially by the NKFIH ``\'Elvonal” KKP 133921 grant and partially by the Hungarian Ministry of Innovation and Technology NRDI Office within the framework of the Artificial Intelligence National Laboratory Program. 

\subsection*{Acknowledgements.} We thank Asgar Jamneshan, Or Shalom and Terence Tao for sharing their recent preprints \cite{JST1,JST2}  related to the present work.

\end{document}